\documentclass[preprint]{imsart}

\RequirePackage{amsthm,amsmath,amsfonts,amssymb}
\RequirePackage[numbers]{natbib}

\usepackage{booktabs}
\usepackage{multirow}
\usepackage{multicol}
\usepackage{mathrsfs}
\usepackage{enumitem}
\usepackage{bbm}
\usepackage{algpseudocode}
\usepackage{algorithm}
\newtheorem{theorem}{Theorem}
\newtheorem{lemma}{Lemma}[section]
\newtheorem{proposition}{Proposition}

\newtheorem{corollary}{Corollary}
\newtheorem{assumption}{Assumption}[section]
\theoremstyle{definition}
\newtheorem{example}{Example}
\newtheorem{remark}{Remark}

\usepackage[table]{xcolor}
\usepackage{adjustbox}
\usepackage{amsmath}

\newcommand{\mcse}[3]{%
    \ensuremath{#1\,(#2)\!\times\!10^{#3}}%
}

\newcommand{\bestmcse}[3]{%
    \ensuremath{\boldsymbol{#1\,(#2)\!\times\!10^{#3}}}%
}

\def\argmin{{\arg\min}}

\def\bE{\mathbf{E}}

\def\bM{\mathbf{M}}

\def\bY{\mathbf{Y}}

\def\bbD{\mathbb{D}}
\def\bbE{\mathbb{E}}

\def\bbP{\mathbb{P}}

\def\bbR{\mathbb{R}}
\def\bbS{\mathbb{S}}

\def\cE{\mathcal{E}}
\def\cF{\mathcal{F}}
\def\cG{\mathcal{G}}

\def\cL{\mathcal{L}}

\def\cN{\mathcal{N}}
\def\cO{\mathcal{O}}

\def\cR{\mathcal{R}}

\def\cU{\mathcal{U}}
\def\cV{\mathcal{V}}
\def\cW{\mathcal{W}}

\def\Cov{{\sf Cov}}

\def\det{{\sf det}}
\def\diag{{\sf diag}}

\def\ind{{\mathbbm{1}}}

\def\KL{{\sf KL}}

\def\norm#1{\left\|#1\right\|}

\def\rank{{\sf rank}}

\def\tr{{\sf tr}}

\def\TV{{\sf TV}}

\def\Var{{\sf Var}}

\def\vol{{\sf vol}}

\def\ind{\mathbbm{1}}
\def\tr{\text{tr}}

\RequirePackage[colorlinks,citecolor=blue,urlcolor=blue]{hyperref}
\RequirePackage{graphicx}

\startlocaldefs
\numberwithin{equation}{section}

\endlocaldefs

\begin{document}

\begin{frontmatter}
\title{Factor-Assisted Density Estimation: When Ambient Dimension Enables Oracle Rates}
\runtitle{Factor-Assisted Density Estimation}

\begin{aug}
\author[A]{\fnms{Sunrit}~\snm{Chakraborty}\ead[label=e1]{sunrit.chakraborty@duke.edu}}
\author[A]{\fnms{David}~\snm{Dunson}\ead[label=e2]{dunson@duke.edu}}

\address[A]{Department of Statistical Science, Duke University\printead[presep={,\ }]{e1,e2}}

\end{aug}

\begin{abstract}
We study density estimation in the high-dimensional factor model, whose fixed-dimensional latent factors have an unknown smooth density. We propose factor-assisted density estimation (FADE), a non-iterative procedure that applies kernel density estimation to spectral estimates of the latent scores. The score proxies are dependent and do not follow the latent density, so standard kernel-density arguments do not apply. A refined rowwise expansion separates error into a conditionally centered Gaussian perturbation and a uniformly controlled remainder. 
We show that FADE achieves the oracle rate in sufficiently high ambient dimension and
prove a minimax lower bound separating ordinary nonparametric difficulty from a Gaussian deconvolution obstruction. Together, the upper and lower bounds establish rate optimality of FADE in specified high-dimensional regimes.
We also derive $L^1$ upper and lower bounds for the induced observation density, distinguishing latent-distribution uncertainty from the cost of estimating a growing number of parameters. These results formalize a blessing of ambient dimension: sufficiently many observed coordinates can make a simple spectral reduction statistically lossless even for the nonparametric task of recovering an unobserved low-dimensional density.
\end{abstract}



\end{frontmatter}

\section{Introduction}

Nonparametric density estimation is severely affected by ambient dimension. Under smoothness assumptions alone, the minimax rate deteriorates rapidly with the dimension of the observations. A natural response is to exploit lower-dimensional structure. We consider the centered factor model
$Y = \Lambda \eta + \epsilon$, with $\eta\sim g$ and $\epsilon \sim N_D(0, \Sigma)$, where $Y\in \bbR^D$, $\eta\in \bbR^d$, $\Lambda\in \bbR^{D\times d}$, $\Sigma$ is diagonal, and $\eta$ is independent of $\epsilon$. The latent density $g$ is otherwise unspecified, apart from smoothness and tail conditions, and may be multimodal, asymmetric, and dependent across coordinates. We study the low-intrinsic-dimensional regime in which $d$ is fixed while $n,D\to\infty$. 

Our primary target is the latent density $g$, which characterizes population heterogeneity in the intrinsic factor space. Since the observation density is invariant to simultaneous rotation of the scores and loading matrix, we take the orbit of $g$ under orthogonal transformations as the target, and we measure the error modulo rotation. A secondary target is the induced observation density
\begin{equation}\label{eq: observation density model}
    p_{g,\Lambda,\Sigma}(y) = \int_{\bbR^d} \phi_D(y\mid \Lambda z, \Sigma) \, g(z) \, dz,
\end{equation}
which is invariant under simultaneous rotation of $g$ and $\Lambda$ and hence fully identifiable. It provides a flexible, dimension-reduced representation of the high-dimensional observation distribution.

Estimating a latent distribution is related to classical nonparametric estimation of mixture \citep{kw1956}, random-effect \citep{beran1996nonparametric, comte2012nonparametric}, and measurement-error or multi-factor distributions \citep{bonhomme2010generalized, ben2018identification, kankanala2023quasi}. Spectral and simultaneous-diagonalization methods also recover finite latent structures from repeated measurements \citep{bonhomme2016non}. These important approaches typically exploit known or structured measurement operators, independence restrictions, characteristic-function inversion, or joint regularization of the inverse problem. Our setting instead has an unknown loading matrix and noise variances, potentially dependent factor coordinates, and an ambient dimension that itself increases the information available about each latent score.

Indeed, the $D$ coordinates of $Y$ are noisy measurements of the same low-dimensional vector $\eta$: $n$ governs learning the population distribution from latent samples, whereas $D$ governs recovery of those samples. Related blessing-of-dimension phenomena have been established for parametric Gaussian factor models \citep{li2018embracing, chattopadhyay2024blessing}, item-response models \citep{chen2019}, nonlinear factor models \citep{feng2023optimal}, latent-position network models \citep{rubin2022statistical}, and large Gaussian mixtures with favorable separation geometry \citep{anderson2014more}. We establish a different mechanism: ambient dimension can make a simple nonparametric estimator of a general latent factor density attain the oracle rate available when the latent scores are observed.

Figure \ref{fig: demonstration} illustrates this phenomenon. For fixed $n$ and $d$, a density estimate based on spectrally estimated scores is blurred when $D$ is small. As $D$ increases, recovery of latent scores improves and the density estimate approaches the oracle kernel density estimator (KDE) computed from the true scores. 
Motivated by this observation, we propose factor-assisted density estimation (FADE), a simple two-stage procedure. A rank $d$ singular value decomposition of the centered data matrix produces score proxies $\widehat \eta_1,\dots,\widehat \eta_n$. We then apply $d-$dimensional KDE to these proxies, and estimate the loadings and noise variances through the corresponding surrogate regressions by plugging in the score proxies in the factor model, and finally substitute the resulting estimates into \eqref{eq: observation density model} to obtain a simple plug-in observation density estimate. The resulting procedure is non-iterative and bypasses joint optimization. 

\begin{figure}
    \centering
    \includegraphics[width=1.0\linewidth, clip, trim=0 0 0 2.2em]{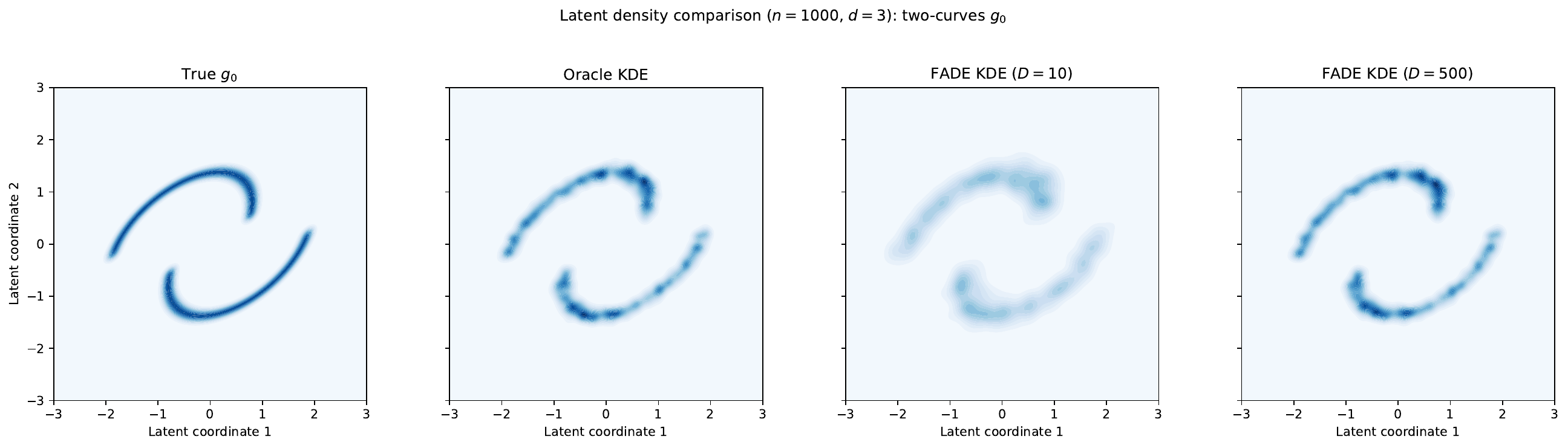}
    \caption{Illustration of the effect of ambient dimension on latent-density recovery with $n=1000$ and $d=3$. From left to right: the true latent density, the oracle kernel density estimator based on the latent scores, and FADE based on spectral score estimates for $D=10$ and $D=500$. Each panel displays the bivariate marginal density of the first two latent coordinates. Increasing $D$ sharpens recovery of the scores and brings FADE closer to the oracle estimator.}
    \label{fig: demonstration}
\end{figure}

This spectral preprocessing is related to the use of estimated factors in factor-augmented regression, forecasting, covariance estimation, and other downstream tasks \citep{stock2002forecasting, bernanke2005, bai2006confidence, chattopadhyay2024blessing, mauri2025factor, Mauri2026, mauri2025spectral}. It also connects to inference from estimated or noisy latent quantities \citep{jochmans2024inference} and to classical and vanishing-noise deconvolution \citep{carroll1988optimal, stefanski1990deconvolving, fan1991optimal, fan1992deconvolution, van2010asymptotic, durot2026minimax}. These literatures address important but different problems: our score proxies are jointly estimated from the same data matrix, dependent across observations, and affected by an unknown latent subspace.

The key technical step is a refined row-wise score expansion. After a common measurable rotation $R$, the score error admits a decomposition of the form 
$\widehat \eta_i - R^\top \eta_i = \widehat R^\top\xi_i + \Delta_i$,
into a conditionally centered Gaussian leading term with covariance of order $D^{-1}$ and a uniformly controlled remainder $\Delta_i$; see Theorem \ref{thm: spectral recovery}. Retaining this centered leading perturbation inside a kernel expansion, rather than treating each score error deterministically, produces stochastic cancellation across observations. Building on double-asymptotic principal-component theory \citep{bai2002determining, bai2003, bai2016, fan2013} and recent latent-variable results \citep{cui2026marginal}, this expansion makes it possible to analyze nonparametric density estimation from estimated scores.

If the scores were observed, a $\beta$-smooth density on $\bbR^d$ has squared-$L^2$ oracle rate $n^{-2\beta/(2\beta+d)}$. Our upper bounds show how this error competes with the cost of spectral score recovery. Above a dimension- and smoothness-dependent threshold, that cost is dominated by the oracle density-estimation error: FADE attains the oracle rate for $d\geq3$, and up to logarithmic factors for $d=2$ and $\beta\geq1/2$. A complementary minimax lower bound separates the ordinary nonparametric difficulty from a finite-$D$ Gaussian-deconvolution obstruction. Thus, in these sufficiently high-dimensional regimes, spectral reduction followed by ordinary KDE is rate-lossless for latent-density estimation. We do not obtain matching bounds throughout the full range of $D$.

As a secondary contribution, we study the observation density estimation in \eqref{eq: observation density model}. Gaussian convolution makes this density smooth, but its estimation is not automatically parametric when both the mixing density and a growing number of factor parameters are unknown. We derive $L^1$ upper and lower bounds that distinguish latent-density uncertainty from factor-parameter estimation. They match at the polynomial-rate level in several high-dimensional regimes, but we do not claim a complete minimax characterization.

FADE complements joint approaches that model the latent distribution and factor parameters together, including mixtures of factor analyzers \citep{mfa, mclachlan2003modelling, bouveyron2014model}, flexible methods for non-Gaussian item-response populations \citep{woods2006, woods2009item, li2026estimating, cui2026multidimensional}, and Bayesian nonparametric or identifiable nonlinear factor models \citep{Dunson2010, chandra2023escaping, xu2023identifiable}. These approaches offer richer joint modeling and uncertainty propagation and can be advantageous at low or moderate $D$, but generally require iterative optimization or sampling. FADE addresses the complementary question of when such coupling is unnecessary. It is also distinct from projection-pursuit, independent-component, low-rank product-density, and manifold methods \citep{friedman1984projection, samarov2004nonparametric, amato2010noisy, song2014nonparametric, vandermeulen2021beyond, chhor2025generalized, pelletier2005kernel, Berenfeld2021, Berenfeld2024}: rather than merely replacing ambient dimension by an intrinsic dimension, we exploit increasing ambient dimension to recover each latent observation more accurately.

In summary, the paper makes four contributions. First, it introduces FADE, a scalable spectral estimator of a general joint factor density, together with a plug-in estimator of the observation density. Second, it develops a row-wise score expansion that preserves first-order stochastic cancellation. Third, it gives latent-density upper and minimax lower bounds that characterize the transition from a factor-recovery-limited regime to an oracle nonparametric regime, with minimax optimality in the stated sufficiently high-dimensional cases. Fourth, it provides upper and lower bounds for observation-density estimation that separate uncertainty in the latent distribution from the cost of estimating growing-dimensional factor parameters. Together, these results show that sufficiently many observed coordinates can make a computationally elementary spectral reduction statistically lossless, at the rate level, for recovering an unobserved low-dimensional density.

The remainder of the paper is organized as follows. Section \ref{sec: method} introduces the model, identifiability conventions, and FADE. Section \ref{sec: latent theory} develops spectral recovery and the upper and lower bounds for the latent density estimation problem. Section \ref{sec: observation density theory} studies the observation density. Section \ref{sec: numerical experiments} presents numerical experiments to support the theoretical findings, and Section \ref{sec: discussion} concludes. Complete proofs and details regarding the numerical experiments are provided in the Appendix.

\section{Model and Factor-Assisted Density Estimation}\label{sec: method}

We formulate the model and its identifiable targets, then describe FADE. Practical selection of $d$ is discussed in Section~\ref{subsec: implementation}.

\subsection{Factor model and inferential targets}\label{subsec: model-targets}

We consider the centered factor model 
\begin{equation}\label{eq: factor model}
    Y_i = \Lambda \eta_i + \epsilon_i, \qquad \eta_i \overset{\rm iid}{\sim} g, \qquad \epsilon_i \overset{\rm iid}{\sim} N_D(0,\Sigma),
\end{equation}
where $Y_i\in \bbR^D$, $\eta_i\in \bbR^d$ is independent of the idiosyncratic noise $\epsilon_i\in \bbR^D$, $\Lambda\in \bbR^{D\times d}$ is the loading matrix, and $\Sigma = \diag(\sigma_1^2,\dots, \sigma_D^2)$. The unknown density $g$ may be multimodal, asymmetric, and dependent across coordinates. We impose $\bbE_g[\eta_i]=0$ and $\Cov_g(\eta_i)=I_d$ to fix location and scale; the remaining ambiguity is orthogonal rotation. For every $R\in \bbR^{d\times d}$ with $R^\top R=I_d$,

\begin{equation}\label{eq: rotation-equivalence}
    \Lambda \eta_i = (\Lambda R) (R^\top \eta_i).
\end{equation}
If $g^R(z)=g(R z)$ denotes the density of $R^\top \eta_i$, then $(g,\Lambda)$ and $(g^R, \Lambda R)$ generate the same observation distribution.

There are two inferential targets. The first is the latent density modulo rotation,
$[g] = \{g^R : R\in \cO(d)\}.$
Under the squared $L^2$ loss, we assess an estimator $\widehat{g}$ through 
\begin{equation}\label{eq: rotation-loss}
    d_{\cO}^2(\widehat{g}, g) := \inf_{R\in \cO(d)} \|\widehat g - g^R\|_{L^2}^2.
\end{equation}
Rotations preserve distances, density values, modes, and level-set geometry, so the relevant features are properties of $[g]$. The second target is the observation density in \eqref{eq: observation density model}, where $\phi_D(\cdot\mid m,V)$ denotes a $D$-dimensional Gaussian density. By \eqref{eq: rotation-equivalence}, $p_{g,\Lambda,\Sigma}=p_{g^R, \Lambda R, \Sigma}$ and is therefore unaffected by the latent orientation.

\begin{remark}[Unknown observation mean]\label{remark: unknown mean}
    Centering is for notational clarity. If $Y_i^{\rm raw}=\mu + \Lambda \eta_i + \epsilon_i$ with unknown $\mu\in\bbR^D$, we center by $\widehat\mu=\overline Y$. More generally, let $m_g=\bbE_g[\eta_i]$ and $V_g=\Cov_g(\eta_i)$ be positive definite. The transformation
$$\eta_i^* = V_g^{-1/2}(\eta_i - m_g), \qquad \Lambda^*=\Lambda V_g^{1/2}, \qquad \mu^*=\mu + \Lambda m_g$$
leaves the distribution unchanged and gives $\bbE[\eta_i^*]=0$ and $\Cov(\eta_i^*)=I_d$. Thus these normalizations are without loss of generality. With $H=I_n - n^{-1}\mathbf 1_n\mathbf 1_n^\top$, the centered data satisfy
    $$HY^{\rm raw} = (HM) \Lambda^\top + HE.$$
    Sample centering removes only the intercept direction and replaces $\eta_i$ by $\eta_i-\overline\eta$, and $\epsilon_i$ by $\epsilon_i-\overline \epsilon$. Appendix \ref{app: sample centering} shows that the spectral-recovery and density-estimation rates in Sections \ref{subsec: spectral recovery}, \ref{subsec: latent density upper bound}, and \ref{subsec: obs density upper bound} are unchanged. We therefore use the centered formulation throughout.
\end{remark}

\subsection{Spectral estimation of scores and parameters}\label{subsec: spectral-estimation}

Stacking observations as rows gives $Y=M\Lambda^\top+E$, where the $i$th rows of $M\in\bbR^{n\times d}$ and $E\in \bbR^{n\times D}$ are $\eta_i^\top$ and $\epsilon_i^\top$. Equivalently, for the $j$th columns $y^{(j)}$ and $\epsilon^{(j)}$ of $Y$ and $E$ respectively, and the $j$th row $\lambda_j^\top$ of $\Lambda$,
\begin{equation}\label{eq: surrogate regression population}
    y^{(j)} = M \lambda_j + \epsilon^{(j)}, \qquad \epsilon^{(j)}\sim N_n(0, \sigma_j^2 I_n), \qquad j=1,\dots,D.
\end{equation}

Let the singular value decomposition of the observation matrix be
\begin{equation*}
    Y = U S V^\top + U_\perp S_\perp V_\perp^\top,
\end{equation*}
where $S$ and $U\in\bbR^{n\times d}$ contain the $d$ largest singular values and corresponding left singular vectors. We estimate $M$ by
\begin{equation}\label{eq: score estimator}
    \widehat M = \sqrt{n} U
\end{equation}
and denote its $i$th row by $\widehat{\eta}_i^\top$. This scaling gives $n^{-1}\widehat M^\top \widehat M=I_d$, matching $\Cov_g(\eta_i)=I_d$.

Substituting $\widehat M$ into \eqref{eq: surrogate regression population} gives $D$ surrogate regressions with a common design matrix and closed-form least-squares estimates:
\begin{equation}\label{eq: factor parameter estimators}
    \begin{aligned}
        \widehat \lambda_j &= \left(\widehat M^\top \widehat M\right)^{-1} \widehat M^\top y^{(j)} = \frac{1}{\sqrt{n}} U^\top y^{(j)}\\
        \widehat \sigma_j^2 &= \frac1n \left\|y^{(j)} - \widehat M \widehat \lambda_j\right\|^2 = \frac1n \left\|\left(I_n - UU^\top\right) y^{(j)}\right\|^2.
    \end{aligned}
\end{equation}
Equivalently, $\widehat \Lambda=Y^\top U/\sqrt{n}$ and $\widehat\Sigma=\diag(\widehat{\sigma}_1^2,\dots,\widehat \sigma_D^2)$, yielding the rank-$d$ reconstruction $\widehat M\widehat\Lambda^\top=UU^\top Y=USV^\top$. Scores and loadings use the same latent coordinates: $U\mapsto UR$ maps $(\widehat M,\widehat\Lambda)$ to $(\widehat M R,\widehat\Lambda R)$ without changing their product. This approach of estimating $M$ and $\Lambda$ is commonly known as the method of asymptotic principal components.

\subsection{Density estimators}\label{subsec: density estimator}

Let $K:\bbR^d\to \bbR$ be a kernel satisfying $\int_{\bbR^d} K(u) du = 1$. For an isotropic bandwidth $h>0$, define
\begin{equation}\label{eq: fade latent kde}
    \widehat g_h(x) = \frac{1}{nh^d} \sum_{i=1}^n K\left(\frac{x-\widehat \eta_i}{h}\right), \qquad x\in \bbR^d.
\end{equation}
We use $\widehat g_h$ to estimate $g$. With a spherically symmetric kernel, rotating all estimated scores only rotates the density estimate. Higher-order kernels are generally signed \citep{wand1994kernel}, so $\widehat g_h$ need not be a probability density. When needed, we use the positive, renormalized version
\begin{equation}\label{eq: positive correction}
    \widetilde g_h(x) = \frac{\widehat g_h^+(x)}{\int \widehat g_h^+(z)\, dz}, \qquad \widehat g_h^+(x) = \max\{\widehat g_h(x), 0\}.
\end{equation}
For a non-negative $K$, $\widetilde g_h=\widehat g_h$. The later observation-density analysis shows that the correction preserves the required $L^1$ rate. Although we analyze KDE in this paper, the spectral preprocessing could be paired with another density estimator, such as a finite or Dirichlet-process Gaussian mixture fitted to $\widehat\eta_1,\dots,\widehat\eta_n$.

\medskip
\noindent
The induced estimator of the observation density is 
\begin{equation}\label{eq: fade observation density}
    \widehat p_h(y) = \int_{\bbR^d} \phi_D(y\mid \widehat \Lambda z,\widehat \Sigma) \widetilde g_h(z) \, dz,
\end{equation}
obtained by plugging the component estimators into \eqref{eq: observation density model}. It is invariant to simultaneous rotation of $\widehat g_h$ and $\widehat\Lambda$. For a Gaussian kernel, no positive correction is needed and the integral is explicit:
\begin{equation*}
    \widehat p_h(y) = \frac1n \sum_{i=1}^n \phi_D\big( y \, \big| \, \widehat \Lambda \widehat \eta_i,\; \widehat\Sigma + h^2 \widehat \Lambda \widehat \Lambda^\top\big).
\end{equation*}
Thus, $\widehat p_h$ is a Gaussian mixture centered at rank-$d$ reconstructions, with common low-rank-plus-diagonal covariance.

\subsection{Implementation}\label{subsec: implementation}

Algorithm \ref{alg: fade} summarizes the non-iterative procedure. Its dominant operation is a rank-$d$ truncated SVD; smoothing is standard KDE in $d$ dimensions, and the loading and variance estimates are closed form. In practice, $h$ may be selected by cross-validation in score space, and $d$ by the joint-likelihood information criterion of \cite{chen2019, chen2021} using spectrally estimated factors as in \cite{chattopadhyay2024blessing}.

\begin{algorithm}[t]
\caption{Factor-assisted density estimation (FADE)}\label{alg: fade}
\begin{algorithmic}
\Require Centered observation matrix $Y\in \bbR^{n\times D}$, latent dimension $d< D$, kernel bandwidth $h>0$.
\State 1. Compute the rank-$d$ SVD $USV^\top$ of $Y$.
\State 2. Set $\widehat M=\sqrt{n} U$, and extract its rows $\widehat \eta_i^\top$, $i=1,\dots,n$.
\State 3. Compute the parameter estimates $\widehat \Lambda$ and $\widehat \Sigma$ from \eqref{eq: factor parameter estimators}.
\State 4. Choose $h$ if not provided. Construct $\widehat g_h$ from \eqref{eq: fade latent kde}. Apply \eqref{eq: positive correction} when a signed kernel is used, and a probability density is required.
\State 5. For the observation density, construct $\widehat{p}_h$ from \eqref{eq: fade observation density}.
\end{algorithmic}
\end{algorithm}

\section{Latent density estimation}\label{sec: latent theory}

We establish spectral recovery, a row-wise score expansion, latent-density upper bounds and oracle transitions, and a minimax lower bound.

\medskip
\noindent
\textit{Notation:}
For $m\in\mathbb{N}$, let $[m]=\{1,\ldots,m\}$. Vector and matrix norms are generally Euclidean and spectral, respectively, with $\|\cdot\|_F$ denoting the Frobenius norm; $s_k(A)$ is the $k$th largest singular value of $A$. The orthogonal group is $\cO(d)$. We write $a_{n,D}\lesssim b_{n,D}$ when $a_{n,D}\leq Cb_{n,D}$ for a constant independent of $n,D$ but possibly depending on fixed model-class constants and $d$, and write $a_{n,D}\asymp b_{n,D}$ when both inequalities hold. Lastly, $a_{n,D}\ll b_{n,D}$ or $a_{n,D}=o(b_{n,D})$ denotes $a_{n,D}/b_{n,D}\to 0$. 

\subsection{Assumptions and asymptotic framework}\label{subsec: assumptions}
Suppose the observations satisfy
\begin{equation}
    Y = M_0 \Lambda_0^\top +E_0
\end{equation}
where the rows of $M_0\in\bbR^{n\times d}$ are iid from density $g_0$, the rows of $E_0$ are iid $N_D(0,\Sigma_0)$ independently of $M_0$, and $\Sigma_0=\diag(\sigma_{01}^2,\dots,\sigma_{0D}^2)$. Let $\bbP_0$ and $\bbP_0^{(n)}$ denote the laws of one pair $(\eta_i,\epsilon_i)$ and of $(M_0,E_0)$, respectively.

\begin{assumption}[Asymptotic regime]\label{assume: dimensions}
    The latent dimension $d$ is known and fixed, while $n,D\to\infty$. In addition, $\log (nD)\lesssim n$ and $\log n = o(D)$.
\end{assumption}

\begin{assumption}[Pervasive loadings]\label{assume: loadings}
    There exist fixed constants $0<c_\Lambda <C_\Lambda<\infty$ such that $c_\Lambda \sqrt{D} \leq s_d(\Lambda_0) \leq \|\Lambda_0\| \leq C_\Lambda\sqrt{D}$. Moreover, for a sequence $\ell_D\geq 1$, assume $\max_{j\in [D]} \|\lambda_{0j}\|^2 \leq C_\ell \ell_D$ for some fixed constant $C_\ell>0$, and $\ell_D\lesssim \sqrt{n\log(nD)}$, where $\lambda_{0j}^\top$ is the $j$th row of $\Lambda_0$.
\end{assumption}

\begin{assumption}[Idiosyncratic noise]\label{assume: noise}
    The noise covariance is diagonal and uniformly well-conditioned: for fixed constants $0<c_\sigma <C_\sigma <\infty$, we assume $c_\sigma \leq \min_{j\in [D]} \sigma_{0j}^2 \leq \max_{j\in[D]} \sigma_{0j}^2 \leq C_\sigma$.
\end{assumption}

\begin{assumption}[Latent distribution]\label{assume: latent tail}
    The latent vectors are centered and standardized, so that $\bbE_{g_0}[\eta_i]=0$ and $\Cov_{g_0}(\eta_i)=I_d$, and are uniformly sub-Gaussian: for a fixed $C_\eta<\infty$, $\|\eta_i\|_{\psi_2}:=\sup_{u\in \bbS^{d-1}} \|u^\top \eta_i\|_{\psi_2} \leq C_\eta$, where $\|X\|_{\psi_2}$ is the sub-Gaussian norm for random variable $X$. 
\end{assumption}

Assumption~\ref{assume: dimensions} specifies the central low-intrinsic-dimensional regime; no limiting ratio of $n$ and $D$ is imposed. The conditions hold, for example, when $D=n^{\gamma+o(1)}$, and ensure simultaneous concentration over coordinates and a vanishing maximum first-order score perturbation. The stated conditions are stronger than needed for some global spectral conclusions because the later KDE analysis requires a uniform row-wise expansion.

Assumption~\ref{assume: loadings} is the standard pervasive-loadings condition: aggregate signal grows with $D$ and every factor direction is represented across coordinates. Consequently, the signal singular values are of order $\sqrt{nD}$ whereas the noise norm is of order $\sqrt n+\sqrt D$, yielding improving score recovery as $D$ grows. The row bound governed by $\ell_D$ controls residual variances uniformly; it allows bounded loadings ($\ell_D=O(1)$), logarithmic growth under light-tailed random loadings, and coordinates carrying little factor signal.

For spectral recovery, only the upper noise bound is essential; the lower bound ensures a nondegenerate observation model for the later observation-density analysis. Diagonality permits coordinatewise residual-variance estimation. Assumption~\ref{assume: latent tail} allows non-Gaussian, dependent, multimodal factors. Its sub-Gaussian condition controls $n^{-1}M_0^\top M_0$ and $\max_i\|\eta_i\|$; it could be replaced by direct high-probability conditions on those quantities. Standard covariance concentration results are given by \cite{vershynin2012close, vershynin2019high}. Smoothness and kernel conditions are discussed later in Section \ref{subsec: latent density upper bound}, separating spectral recovery from the nonparametric analysis.

\subsection{Spectral recovery and score decomposition}\label{subsec: spectral recovery}

We quantify the spectral estimators under a common rotation. Part (i) controls scores, loadings, and noise variances; part (ii) gives the conditionally Gaussian leading score perturbation needed below. The loading and variance bounds are also used for observation-density estimation.

Let $Y_0=M_0\Lambda_0^\top=U_0S_0V_0^\top$ be the compact SVD of the signal, with $P_0=U_0U_0^\top$ and $P_0^\perp=I_n-P_0$. Write $M_0=\overline U_0\overline D_0\overline V_0^\top$, where $\overline U_0=U_0\overline R^\top$ for some $\overline R\in\cO(d)$; this is possible because $M_0$ and $Y_0$ have the same column space. For empirical left singular vectors $U$, choose the measurable Procrustes alignment
\begin{equation*}
    \widehat R \in \argmin_{\Omega\in \cO(d)} \|U\Omega^\top - U_0\|_F
\end{equation*}
and define the common rotation as
\begin{align}\label{eq: common rotation}
    R = \overline V_0 \overline R \widehat R.
\end{align}
This aligns both $\widehat M$ and $\widehat\Lambda$ with their population counterparts.

\begin{theorem}\label{thm: spectral recovery}
    Suppose Assumptions~\ref{assume: dimensions}--\ref{assume: latent tail} hold. There is a measurable event $\cE_{n,D}$ satisfying
    $$\bbP_0^{(n)}(\cE_{n,D}) \geq 1 - \frac{C}{n}$$
    for some constant $C>0$ such that the rotation $R$ in \eqref{eq: common rotation} can be chosen measurably and the following statements hold on $\cE_{n,D}$.
    \begin{enumerate}[label=(\roman*)]
        \item The factor spectral estimators satisfy
        \begin{equation}\label{eq: parameters high probability}
        \begin{aligned}
            \|\widehat \Sigma - \Sigma_0\| &\lesssim \sqrt{\frac{\log(nD)}{n}} + \frac{\ell_D}{D}, \\
            \|\widehat M - M_0 R\|_F &\lesssim \sqrt{n}\left(\frac{1}{\sqrt{D}} + \sqrt{\frac{\log n}{n}}\right),\\
            \|\widehat \Lambda - \Lambda_0 R\|_F &\lesssim \sqrt{D}\left(\frac1D + \sqrt{\frac{\log(nD)}{n}}\right).
        \end{aligned}
        \end{equation}
        \item For each $i\in[n]$, define
        \begin{equation}\label{eq: leading score noise}
            \xi_i = \sqrt{n} S_0^{-1} V_0^\top E_0^\top P_0^\perp e_i,
        \end{equation}
        where $e_i$ is the standard basis vector in $\bbR^n$ with $e_{ii}=1$. Then the rowwise score error admits the decomposition
        \begin{equation}\label{eq: structured score expansion}
        \begin{aligned}
            \widehat \eta_i - R^\top \eta_i = &\widehat R^\top \xi_i + \Delta_i \;\quad \text{for } i\in [n], \\ &\text{ where } \max_{i\in [n]} \|\Delta_i\|\lesssim \frac{\sqrt{\log (nD)}}{D} + \frac{\log n}{\sqrt{n}}.
        \end{aligned}
        \end{equation}
    \end{enumerate}
    The constants implicit in $\lesssim$ are independent of $n$ and $D$, but may depend on $d$ and the fixed constants in the assumptions.
\end{theorem}

\begin{remark}
    Conditional on $M_0$, the vectors $\xi_1,\dots,\xi_n$ are jointly centered Gaussian with 
    \begin{equation}\label{eq: score noise covariance}
        \begin{aligned}
            \Cov(\xi_i, \xi_j \mid M_0) = (P_0^\perp)_{ij} \Sigma_{\xi}, \qquad \Sigma_\xi = nS_0^{-1} V_0^\top \Sigma_0 V_0 S_0^{-1}
        \end{aligned}
    \end{equation}
    and $\|\Sigma_\xi\|\lesssim D^{-1}$ with high probability. After intersecting $\cE_{n,D}$ with an additional concentration event without changing its probability order, this leads to $\max_i \|\xi_i\| \lesssim \sqrt{\frac{\log n}{D}}$.
\end{remark}

Part (i) displays the asymmetric roles of $n$ and $D$. Since $\|M_0\|_F\asymp\sqrt n$ and $\|\Lambda_0\|_F\asymp\sqrt D$, the normalized errors are
\begin{align*}
    \frac{\|\widehat M - M_0R\|_F}{\|M_0\|_F} &\lesssim \frac{1}{\sqrt{D}} + \sqrt{\frac{\log n}{n}}, \\
    \frac{\|\widehat \Lambda - \Lambda_0 R\|_F}{\|\Lambda_0\|_F} &\lesssim \frac1D + \sqrt{\frac{\log (nD)}{n}}.
\end{align*}
The score estimator inherits a first-order $D^{-1/2}$ perturbation of the left singular space. In contrast, projecting observations onto that space to estimate $\Lambda$ annihilates the first-order alignment error, yielding a $D^{-1}$ normalized loading contribution; the residual projection similarly benefits $\widehat\Sigma$. The remaining $n^{-1/2}$ terms reflect averaging over observations. These conclusions refine Davis--Kahan or Wedin-type subspace bounds, which do not retain this downstream projection geometry \citep{wedin1972perturbation}. The common rotation and conditionally centered row-wise expansion, rather than separate deterministic bounds on each error, are what the density analysis requires.

For intuition, the empirical and population projectors satisfy
\begin{equation}\label{eq: projector expansion sketch}
    P - P_0 = P_0^\perp E_0 Y_0^+ + (Y_0^+)^\top E_0^\top P_0^\perp + \cR(E_0),
\end{equation}
where $Y_0^+=V_0S_0^{-1}U_0^\top$ and $\cR(E_0)$ collects higher-order terms. Applying this expansion to $U_0$ and resolving Procrustes alignment yields \eqref{eq: leading score noise}; $\Delta_i$ collects the higher-order projector, alignment, and normalization errors. Full bounds are deferred to the proof.

The leading errors have scale $D^{-1/2}$ and are dependent through $P_0^\perp$. Although $\widehat R$ depends on the noise, after applying the common rotation,
$$R\widehat R^\top \xi_i = \overline V_0 \overline R \,\xi_i,$$
 $\overline V_0\overline R$ is fixed conditionally on $M_0$. The first-order term in the KDE expansion therefore remains conditionally centered Gaussian, allowing cancellation under averaging that maximum- or Frobenius-norm bounds would lose.

Observation-density analysis instead needs expected Frobenius errors. Direct moment calculations remove logarithmic factors and give the following consequence.

\begin{corollary}\label{cor: expected factor error}
    Under the conditions of Theorem \ref{thm: spectral recovery}, using the same measurable rotation $R$ and event $\cE_{n,D}$, we have the following.
    \begin{align*}
        \bbE_0 \left[\|\widehat \Lambda - \Lambda_0 R\|_F \mid \cE_{n,D}\right] &\lesssim \sqrt{\frac{D}{n}} + \frac{1}{\sqrt{D}}, \\
        \bbE_0\left[\|\widehat \Sigma - \Sigma_0\|_F\mid \cE_{n,D}\right] &\lesssim \sqrt{\frac{D}{n}} + \sqrt{\frac{\ell_D}{D}}.
    \end{align*}
\end{corollary}

The theorem supplies the common high-probability event for the score expansion; the corollary supplies the aggregate parameter control needed under the observation-density $L^1$ loss.

\subsection{Latent density upper bound}\label{subsec: latent density upper bound}

We study the squared-$L^2$ error of \eqref{eq: fade latent kde}. Although it has the form of an ordinary $d$-dimensional KDE, its inputs are dependent score proxies jointly estimated from the data matrix, rather than draws from $g_0$ or a standard additive-error sample. The oracle benchmark, using true scores under the same rotation, is
$$\widetilde g_h^R(x) = \frac{1}{nh^d}\sum_{i=1}^n K\left(\frac{R^\top \eta_i - x}{h}\right).$$
Define the corresponding smoothed rotated density by
$$\overline g_h^R(x) = (K_h \star g_0^R) (x), \qquad K_h(x) =h^{-d} K(x/h).$$
For observed scores, the usual bias--variance decomposition gives
$$\bbE_0[\|\widetilde g_h^R - g_0^R\|_{L^2}^2 ] \lesssim h^{2\beta} + \frac{1}{nh^d}.$$
We quantify the additional error from replacing $R^\top\eta_i$ by $\widehat\eta_i$. For an integrable $f$, let
$$\varphi_f(t) = \int_{\bbR^d} e^{-it^\top x} f(x) dx$$
denote its Fourier transform.

\begin{assumption}[Smoothness of $g_0$]\label{assume: smoothness of g_0}
    For some $\beta>0$ and $L<\infty$,
    $$\int_{\bbR^d} (1 + \|t\|^2)^{\beta} |\varphi_{g_0}(t)|^2 dt \leq L^2.$$
\end{assumption}
This standard Fourier characterization of an $L^2$-Sobolev ball yields the $h^{2\beta}$ squared-bias bound; see \cite{tsybakov2009introduction}. It concerns high-frequency regularity, whereas the earlier sub-Gaussian condition controls tails for spectral recovery. The class is rotation invariant since $\varphi_{g^R}(t)=\varphi_g(Rt)$. The score perturbation analysis requires additional kernel derivatives for a Taylor expansion.

\begin{assumption}[Kernel regularity]\label{assume: kernel}
    The kernel $K:\bbR^d\to \bbR$ is integrable, belongs to $C^3\cap L^2$, and satisfies $\int_{\bbR^d} K(u) du = 1$. It is spherically symmetric:
    $$K(Qu) = K(u) \qquad \text{for every } u\in \bbR^d, \; Q\in \cO(d).$$
    In addition:
    \begin{enumerate}
        \item $K$ is of order $\beta$ in the sense that, for some $C_K<\infty$,
        $$|\varphi_K(t) - 1| \leq C_K \min \{1, \|t\|^\beta\}, \qquad t\in \bbR^d.$$
        \item Its first two derivatives satisfy
        $$\|\nabla K\|_{L^p}<\infty, \qquad \int_{\bbR^d} \|\nabla^2 K(u) \|_F^p du<\infty, \qquad p\in \{1,2\}.$$
        \item Its third derivative satisfies
        $$\int_{\bbR^d} \|\nabla^3 K(u)\|_{\rm op}^2 du <\infty,$$
        where $\|\nabla^3 K(u)\|_{\rm op}$ denotes the operator norm of the trilinear form $\nabla^3 K(u)$.
    \end{enumerate}
    All kernel norms are allowed to depend on the fixed latent dimension $d$, but not on $n$ or $D$.
\end{assumption}

The Fourier-order condition gives the bias bound; for integer $\beta$, sufficiently many vanishing moments suffice. The derivative conditions control score substitution and are standard for kernel estimation with generated quantities \citep{samb2010nonparametric}. A Gaussian kernel satisfies them for $0<\beta\leq2$; for larger $\beta$, suitable higher-order radial kernels are available \citep{wand1994kernel} and are generally signed. This is harmless under $L^2$ loss. For observation-density estimation, the positive correction in \eqref{eq: positive correction} preserves the required $L^1$ rate (Section~\ref{sec: observation density theory}).

Let $\cE_{n,D}$ be the event in Theorem \ref{thm: spectral recovery}, augmented by the conditional Gaussian concentration event on which $\max_i\|\xi_i\|\lesssim\sqrt{(\log n)/D}$. This does not change $\bbP_0^{(n)}(\cE_{n,D}^c)\lesssim n^{-1}$. Define
$$\rho_{n,D} := \frac{\sqrt{\log (nD)}}{D} + \frac{\log n}{\sqrt n},$$
so that $\max_i\|\Delta_i\|\lesssim\rho_{n,D}$ on $\cE_{n,D}$.

\begin{theorem}\label{thm: new kde}
    Suppose the conditions of Theorem \ref{thm: spectral recovery}, Assumption~\ref{assume: smoothness of g_0}, and Assumption~\ref{assume: kernel} hold. Let $R=R(M_0, E_0)\in \cO(d)$ be the measurable rotation appearing in Theorem \ref{thm: spectral recovery}. Then for every bandwidth $h>0$,
    $$\bbE_0\left[\|\widehat g_h - g_0^R\|_{L^2}^2 \mid \cE_{n,D}\right] \lesssim \cR_{n,D}(h),$$
    where 
    \begin{equation}\label{eq: kde upper bound full}
        \begin{aligned}
            \cR_{n,D}(h) &= \underbrace{h^{2\beta} + \frac{1}{nh^d}}_{\text{oracle bias and variance}}   + \underbrace{\frac{1}{nDh^{d+2}} + \rho_{n,D}^2\left(\frac{1}{nh^{d+2}} + \frac{1}{h^2}\right)}_{\text{linear score perturbation}} \\
            &\qquad + \underbrace{\left(\frac{1}{D^2} + \rho_{n,D}^4\right) \left(\frac{1}{nh^{d+4}} + \frac{1}{h^4}\right)}_{\text{quadratic score perturbation}}  + \underbrace{\frac{\log^3 n}{D^3 h^{d+6}} + \frac{\rho_{n,D}^6}{h^{d+6}}}_{\text{third order remainder}}.
        \end{aligned}
    \end{equation}
    The implicit constant may depend on $d, \beta, L$, the constants in the factor-model assumptions, and the kernel norms, but not on $n, D$ or $h$.
\end{theorem}

For the rotation-invariant loss, this implies
$$\bbE_0\left[d_{\cO}^2 (\widehat g_h, g_0)\right] \lesssim \cR_{n,D}(h).$$
Indeed, the complement has probability $O(n^{-1})$ and contributes at most a constant multiple of $h^{-d}+1$. The unrestricted bound displays every score-recovery contribution before bandwidth optimization. Its proof decomposes
$$\widehat g_h - g_0^R = \left(\overline g_h^R - g_0^R\right) + \left(\tilde g_h^R - \overline g_h^R\right) + \left(\widehat g_h - \tilde g_h^R\right).$$
The first two terms give ordinary KDE bias and variance; the third is the score-substitution error.

If one uses only the average spectral recovery bound from Theorem \ref{thm: spectral recovery}(i), a mean-value argument gives
$$\|\widehat g_h - \tilde g_h^R\|_{L^2}^2 \lesssim \frac{1}{h^{d+2}} \sum_{i=1}^n \|\widehat \eta_i - R^\top \eta_i\|^2 \lesssim \frac{1}{Dh^{d+2}} + \frac{\log n}{nh^{d+2}}.$$
This treats score errors separately and loses their centering and covariance structure; the resulting weaker bound is recorded as Theorem \ref{thm: weak kde g0} in the Appendix.

Instead, we substitute $\widehat \eta_i-R^\top\eta_i=\widehat R^\top\xi_i+\Delta_i$ into a second-order Taylor expansion. With $z_i(x)=(R^\top\eta_i-x)/h$ and $\delta_i=\widehat\eta_i-R^\top\eta_i$,
\begin{align*}
    K\bigg(z_i(x) + \frac{\delta_i}{h}\bigg) - K(z_i(x)) = \frac1h \nabla K(z_i(x))^\top \delta_i + \frac{1}{2h^2}\delta_i^\top \nabla^2 K(z_i(x)) \delta_i + r_i(x),
\end{align*}
where $r_i(x)$ is controlled by $\nabla^3K$ and $\max_i\|\delta_i\|$. Spherical symmetry permits a rotation of the $L^2$ integral under which the leading errors remain conditionally centered Gaussian. Together with $0\preceq P_0^\perp\preceq I_n$ and $\|\Sigma_\xi\|\lesssim D^{-1}$, stochastic averaging yields $1/(nDh^{d+2})$, rather than the deterministic $1/(Dh^{d+2})$ term. The $D^{-2}$ term is the second-order effect of score noise, while the remaining terms are controlled by $\max_i\|\Delta_i\|$. The proof handles conditioning on $\cE_{n,D}$ through a score-only event so that exact conditional Gaussian moments remain available.

\begin{remark}[Uniformity of upper bound]\label{remark: uniform upper bound}
    Let the class of all parameters satisfying the stated assumptions be denoted as $\Theta_*=\{\theta=(g,\Lambda, \Sigma): \theta \text{ satisfies}$ $\text{Assumptions } \ref{assume: loadings}- \ref{assume: smoothness of g_0}\}$. Then, the results in Theorems \ref{thm: spectral recovery} and \ref{thm: new kde} can be stated uniformly over $\Theta_*$, with constants depending only on $d,\beta, L$, the fixed constants in the factor-model assumptions, and the kernel norms. In particular, for the latent density estimation results, there exists a constant $C>0$ such that for every $\theta\in\Theta_*$, there is a measurable rotation $R_\theta\in\cO(d)$ and an event $\cE_{n,D}(\theta)$ satisfying $\inf_{\theta\in \Theta_*} \bbP_\theta \{\cE_{n,D}(\theta)\}\geq 1-C/n$, for which, for any deterministic bandwidth $h>0$, we have
    $$\bbE_\theta [\|\widehat g_h - g^{R_\theta}\|_{L^2}^2 \mid \cE_{n,D}(\theta)] \leq C\cR_{n,D}(h),$$
    where $\cR_{n,D}(h)$ is given in Equation \eqref{eq: kde upper bound full}. Consequently, this implies the uniform bound $\sup_{\theta\in \Theta_*} \bbE_\theta d_{\cO}^2(\widehat g, g) \leq C \cR_{n,D}(h)$. To see this uniformity, note that although the event $\cE_{n,D}$ itself depends on $\theta$, its probability guarantee is uniform as all the concentration events in the proof of Theorem \ref{thm: spectral recovery} depend only on the constants in the assumptions and $d$, which are fixed. For the proof of Theorem \ref{thm: new kde}, all constants depend on the relevant concentration events, whose probabilities are controlled uniformly via the assumptions on the parameters, the kernel norms, and $\|g\|_{L^2}^2$, which is bounded by Parseval's inequality $\|g\|_{L^2}^2\lesssim \int |\varphi_g(t)|^2 dt \leq L^2$ by Assumption \ref{assume: smoothness of g_0}. Therefore, we shall henceforth write $\cR_{\rm FADE}(n,D)$ to denote the uniform bound $\sup_{\theta\in \Theta_*} \bbE_\theta [d_{\cO}^2(\widehat{g}_{h^*}, g)]\lesssim \cR_{\rm FADE}(n,D)$, where $h^*$ is the optimized bandwidth discussed in the next section and $\widehat{g}_h$ is the FADE density estimator constructed from $n$ observations $Y_1,\dots,Y_n$ in $\bbR^D$.
\end{remark}

\subsection{Optimized rate and oracle transition}\label{subsec: latent density optimized rate}

Theorem \ref{thm: new kde} separates ordinary nonparametric error from score-recovery error. For observed rotated scores, balancing $h^{2\beta}$ and $(nh^d)^{-1}$ gives
$$h_{\rm oracle} \asymp n^{-1/(2\beta+d)}, \qquad  \cR_{\rm oracle}^* \asymp n^{-2\beta/(2\beta+d)}.$$
For FADE, score perturbations require $h^*\asymp h_D^*\vee h_n^*$, where $h_D^*$ is determined by finite-$D$ error and $h_n^*$ by sampling and the $n$-dependent remainder. Once $h_D^*\lesssim h_n^*$, finite-$D$ uncertainty no longer controls the bound; whether the remaining rate is oracle depends on $d$ and $\beta$.

We impose $\log\log n=o(\log D)$ only to order slowly varying terms; it holds, for example, when $D=n^{\gamma+o(1)}$.

\begin{corollary}\label{cor: optimized latent kde rate}
    Under the conditions of Theorem \ref{thm: new kde} and the condition $\log \log n =o(\log D)$, using $q=2\beta+d+6$, define 
    \begin{align*}
        h_D^* \asymp \begin{cases}
            D^{-1/(\beta+2)}, &\beta>d, \\
            D^{-3/q} (\log n)^{3/q}, &0<\beta \leq d,
        \end{cases}
    \end{align*}
    and 
    \begin{align*}
        h_n^* \asymp \begin{cases}
            n^{-1/(2\beta+d)}, &d\geq 3,\\
            n^{-1/(2\beta+2)}(\log n)^{1/(\beta+1)}, &d\leq 2,\; \beta\geq d/4, \\
            n^{-3/q}(\log n)^{6/q}, &d\leq 2, \; 0<\beta<d/4.
        \end{cases}
    \end{align*}
    Then the bandwidth $h^*\asymp h_D^* \vee h_n^*$ satisfies
    $$\bbE_0[d_{\cO}^2(\widehat g_{h^*}, g_0)] \lesssim (h_D^*)^{2\beta} \vee (h_n^*)^{2\beta}.$$
\end{corollary}

Appendix \ref{app: bandwidth} compares all terms in $\cR_{n,D}(h)$ by balancing $c_{n,D}h^{-s}$ against $h^{2\beta}$. The two dominant finite-$D$ balances yield $D^{-1/(\beta+2)}$ for $\beta>d$ and for $\beta\leq d$, $D^{-3/q}(\log n)^{3/q}$. Among the $n$-dependent terms, the oracle balance dominates for $d\geq3$; for $d\leq2$, the uniform remainder can require a larger bandwidth, with boundary $\beta=d/4$.

\medskip
\noindent
\textbf{Ambient-dimension threshold} Define $D_c(n)$ by equating $h_D^*$ and $h_n^*$. This is the ambient dimension at which the optimized upper bound changes from being $D-$limited to being $n-$limited. With $q=2\beta+d+6$, we obtain

\begin{equation}\label{eq: D_c}
    \begin{aligned}
        D_c(n) \asymp \begin{cases}
            n^{(\beta+2)/(2\beta+d)}, &d\geq 3, \; \beta>d,\\
            n^{q/(6\beta+3d)} \log n, &d\geq 3,\; 0<\beta\leq d,\\
            n^{(\beta+2)/(2\beta+2)} (\log n)^{-(\beta+2)/(\beta+1)}, &d\leq 2, \; \beta > d,\\
            n^{q/(6\beta+6)} (\log n)^{(\beta-d-3)/(3\beta+3)}, &d\leq 2, \; d/4\leq \beta \leq d,\\
            n(\log n)^{-1}, & d\leq 2, \; 0<\beta <d/4.
        \end{cases}
    \end{aligned}
\end{equation}

If $D\lesssim D_c(n)$, finite-$D$ score recovery controls the bound; if $D\gtrsim D_c(n)$, its $n$-dependent component controls it.

\begin{corollary}
    Under the conditions of the preceding corollary, we have the following. For $d\geq 3$, 
    \begin{align*}
        \bbE_0[d_{\cO}^2(\widehat g_{h^*}, g_0)] \lesssim \begin{cases}
            n^{-2\beta/(2\beta+d)}, &D\gtrsim D_c(n),\\
            D^{-2\beta/(\beta+2)}, & D\lesssim D_c(n),\; \beta>d,\\
            D^{-6\beta/q} (\log n)^{6\beta/q}, &D\lesssim D_c(n), \; \beta \leq d.
        \end{cases}
    \end{align*}
    For $d\leq 2$,
    \begin{align*}
        \bbE_0[d_{\cO}^2(\widehat g_{h^*}, g_0)] \lesssim \begin{cases}
            n^{-\beta/(\beta+1) }(\log n)^{2\beta/(\beta+1)}, &D\gtrsim D_c(n), \; \beta \geq d/4,\\
            n^{-6\beta/q} (\log n)^{12\beta/q}, &D\gtrsim D_c(n), \; 0<\beta <d/4,\\
            D^{-2\beta/(\beta+2)}, &D\lesssim D_c(n), \; \beta>d,\\
            D^{-6\beta/q}(\log n)^{6\beta/q}, &D\lesssim D_c(n), \; 0<\beta\leq d.
        \end{cases}
    \end{align*}
    The statements at $D\asymp D_c(n)$ agree up to constants and the displayed logarithmic factors.
\end{corollary}

The principal oracle conclusion is especially clean when $d\geq3$.

\begin{example}
    If $d\geq 3$ and $D\gtrsim D_c(n)$, where (obtained from \eqref{eq: D_c})
\begin{align}\label{eq: D_c for d >= 3}
    D_c(n) \asymp \begin{cases}
    n^{(\beta+2)/(2\beta+d)}, &\beta>d,\\
    n^{(2\beta+d+6)/(6\beta + 3d)}\log n, &0<\beta \leq d,
\end{cases}
\end{align}
then we have 
$\bbE_0[d_{\cO}^2(\widehat g_{h^*}, g_0)]\lesssim n^{-2\beta /(2\beta+d)}.$
For such cases, FADE attains the same squared-$L^2$ rate as the oracle KDE that directly observes the latent scores. Thus, in this regime, beyond the ambient dimension threshold as mentioned, spectral preprocessing incurs no additional statistical cost. By the uniformity of the upper bound as discussed in Remark \ref{remark: uniform upper bound}, we have
\begin{equation}\label{eq: uniform upper bound D >= 3}
    d\geq 3, \; D\gtrsim D_c(n)\;\Longrightarrow \sup_{(g,\Lambda,\Sigma)} \bbE_{g,\Lambda,\Sigma} d_{\cO}^2(\widehat g_{h^*}, g) \lesssim n^{-2\beta/(2\beta+d)},
\end{equation}
where the supremum is over all parameters $g,\Lambda, \Sigma$ satisfying the assumptions in the above corollary. For $d=2$ and $\beta\geq 1/2$, the corresponding conclusion holds up to a logarithmic factor
$$d=2, \beta\geq 1/2, \; D\gtrsim D_c(n)\;\Longrightarrow \sup_{(g,\Lambda,\Sigma)} \bbE_{g,\Lambda,\Sigma}d_{\cO}^2(\widehat g_{h^*}, g)\lesssim n^{-\beta/(\beta+1)} (\log n)^{2\beta/(\beta+1)}.$$
Since the oracle rate for $d=2$ is $n^{-\beta/(\beta+1)}$, the polynomial part is oracle-optimal, but our upper bound retains the displayed logarithmic loss.
\end{example} 

For $d=2$ and $0<\beta<1/2$, and for $d=1$, the present analysis does not recover the oracle polynomial rate. This does not establish intrinsic suboptimality: the gap may reflect uniform remainder control, the higher-order Taylor analysis, or a genuine cost of two-stage estimation. The comparatively favorable $d\geq3$ result does not mean that density estimation is easier in higher dimension; rather, the oracle bandwidth is larger and can absorb the $n$-dependent score-estimation remainder.

The upper bounds therefore reveal two qualitatively distinct regimes for FADE:
\begin{align*}
    D \lesssim D_c(n) \quad &: \quad \text{latent-score recovery contributes to the leading rate} \\
    D\gtrsim D_c(n) \quad &: \quad \text{finite-}D \text{ score uncertainty is asymptotically negligible.}
\end{align*}
For $d\geq3$, the second regime is exactly the oracle regime. We next compare these bounds with minimax risk based on full observations.

\subsection{Lower bound and optimality}\label{subsec: latent density lower bound}
The oracle comparison does not determine whether FADE's $D$-dependent error is intrinsic or arises from spectral preprocessing. We therefore derive a minimax lower bound over all estimators based on $Y_1,\dots,Y_n$, separating ordinary nonparametric difficulty from finite-$D$ Gaussian attenuation. For fixed $\beta,L>0$, define
$$\cW_\beta(L):=\left\{ g\in L^2(\bbR^d)\; : \; \int_{\bbR^d} (1+\|t\|^2)^\beta |\varphi_g (t)|^2 dt \leq L^2 \right\}.$$
We use the density subclass 
\begin{equation}
    \begin{aligned}
        \cG_{\beta}(L,C):=\left\{g\in \cW_\beta(L)\; : \; \right.& g\geq 0, \; \int_{\bbR^d} g(x) dx = 1, \; \bbE_g(\eta)=0, \; \bbE_g(\eta\eta^\top)=I_d, \\
        & \left. g(x) \leq C\exp\left(-c\|x\|^2\right) \text{ for every } x\in \bbR^d \right\},
    \end{aligned}
\end{equation}
where $c,C>0$ are fixed and $c$ is suppressed from notation. This Gaussian envelope is stronger than the earlier sub-Gaussian assumption, so a lower bound over $\cG_{\beta}(L,C)$ applies to the larger class considered above.

For the lower bound, specialize to bounded loadings $\ell_D\asymp1$. Let $\Theta_{D,\beta}(L,C)$ collect triples $(g,\Lambda,\Sigma)$ satisfying these density, loading, and noise conditions, and define
\begin{equation}
    \mathfrak R_{n,D}^{\rm lat}:= \inf_{\widehat g} \sup_{(g,\Lambda,\Sigma)\in \Theta_{D,\beta}(L,C)} \bbE_{g,\Lambda,\Sigma}\left[d_{\cO}^2(\widehat g, g)\right],
\end{equation}
where the infimum is over measurable estimators based on $Y_1,\dots,Y_n$.

\begin{theorem}\label{thm: latent lower bound}
    Let $d\geq 1$ and $\beta>0$ be fixed. Suppose $L, C$ are sufficiently large so that $\cG_{\beta}(L,C)$ contains local perturbations of the standard Gaussian density. Then, for all sufficiently large $n,D$, 
    $$\mathfrak R_{n,D}^{\rm lat} \gtrsim n^{-2\beta/(2\beta+d)} \vee (D\log n)^{-\beta}.$$
    The implicit constant depends on the fixed constants defining the model and density classes.
\end{theorem}

\textbf{Interpretation:} The first term is the ordinary nonparametric lower bound, including when scores are observed. The proof adapts the usual Assouad--Fano construction so that alternatives remain separated modulo the fixed-dimensional rotation group; quotienting by rotations does not change the polynomial rate.

The second term is the finite-$D$ inverse-problem obstruction. Restrict to the admissible submodel $\Sigma=I_D$, $\Lambda_D^\top\Lambda_D=DI_d$, with bounded loading rows, and reveal $(\Lambda_D,\Sigma)$ to the estimator. Then
$$X_i = \eta_i + D^{-1/2} Z_i, \qquad Z_i \sim N_d(0, I_d).$$
The component of $Y_i$ orthogonal to $\Lambda_D$ contains no information about $g$, so this favorable submodel is equivalent to
$$X_i \overset{\rm iid}{\sim} g \star \phi_{D^{-1} I_d}.$$
Any lower bound here applies to the original problem. Classical Gaussian deconvolution has logarithmic minimax rates because its characteristic function decays exponentially \citep{carroll1988optimal, stefanski1990deconvolving, fan1991optimal, fan1992deconvolution}. Here the noise variance is $D^{-1}$, with characteristic function $\exp\{-\|t\|^2/(2D)\}$. The proof uses moment-corrected oscillatory alternatives of the form
$$g_{\pm, T}(x) \approx \phi_d(x) \left\{1 \pm a T^{-\beta} \cos(Tx_1)\right\}.$$
Their squared-$L^2$ separation is $T^{-2\beta}$ and convolution attenuates the oscillation by $e^{-T^2/(2D)}$. Choosing $T^2\asymp D\log n$ makes the convolved experiments indistinguishable while retaining separation $(D\log n)^{-\beta}$; a two-point Le Cam argument gives the result. Thus ambient dimension progressively weakens, rather than removes, the supersmooth obstruction. This does not model FADE scores as independent additive-error observations: the deconvolution experiment is only a favorable submodel for the lower bound. The upper bound analyzes the actual dependent score proxies.

Equating the two terms in the lower bound in Theorem \ref{thm: latent lower bound}  gives
$n^{-2\beta/(2\beta+d)}\asymp (D\log n)^{-\beta}$. This yields the scale
\begin{equation}\label{eq: D IT}
    D_{\rm IT}(n):= \frac{n^{2/(2\beta+d)}}{\log n}.
\end{equation}
We call $D_{\rm IT}(n)$ the information-theoretic lower-bound threshold: it is the crossover in the proved lower bound, not a complete minimax transition. If $D\ll D_{\rm IT}(n)$, no estimator based on full observations can attain the oracle rate. If $D\gtrsim D_{\rm IT}(n)$, the lower bound reduces to the oracle term but does not establish its attainability. Thus, the $D$ dependence is not solely an artifact of spectral preprocessing.

\medskip
\noindent
\textbf{Optimality of FADE in high dimension} For $d\geq3$, \eqref{eq: uniform upper bound D >= 3} gives
$$\sup_{(g,\Lambda, \Sigma)} \bbE_{g,\Lambda,\Sigma} \left[d_{\cO}^2(\widehat g_{h^*}, g)\right] \lesssim n^{-2\beta/(2\beta+d)}$$
whenever $D\gtrsim D_c(n)$. Since the minimax lower bound always contains the oracle term,
$$\text{For }d\geq 3, \; D\gtrsim D_c(n): \qquad \mathfrak R_{n,D}^{\rm lat} \asymp n^{-2\beta/(2\beta+d)}, \quad \cR_{\rm FADE}(n,D) \asymp \mathfrak R_{n,D}^{\rm lat}.$$
Thus FADE is minimax-rate optimal and statistically lossless, at the rate level, in this sufficiently high-dimensional regime.

For $d=2$ and $\beta\geq1/2$, FADE is minimax-rate optimal up to a logarithmic factor when $D\gtrsim D_c(n)$; for $d=2,\beta<1/2$ and $d=1$, the current bounds establish neither oracle nor minimax optimality.

\medskip
\noindent
The lower and upper bounds yield different thresholds: $D_{\rm IT}(n)$ in \eqref{eq: D IT} and $D_c(n)$ in \eqref{eq: D_c}. In the oracle upper-bound regimes, $D_{\rm IT}(n)\ll D_c(n)$, up to boundary cases and logarithms. For example, if $d\geq3$ and $\beta>d$,
$$D_{\rm IT}(n) = n^{2/(2\beta+d)}(\log n)^{-1} \ll n^{(\beta+2)/(2\beta+d)}=D_c(n).$$
This leaves $D_{\rm IT}(n)\lesssim D\lesssim D_c(n)$, where the lower bound permits the oracle rate but the FADE upper bound does not attain it. The gap may reflect upper- or lower-bound looseness, a cost of spectral point estimates, an advantage of joint estimation, or some combination; the present results do not distinguish these possibilities.

For $d\geq 3$, the conclusions can be summarized as
\begin{align*}
    D\ll D_{\rm IT}(n) &\quad : \quad \text{the oracle rate is impossible for any estimator}, \\
    D_{\rm IT}(n) \lesssim D \lesssim D_c(n) &\quad : \quad \text{the optimal rate remains unresolved},\\
    D\gtrsim D_c(n) &\quad : \quad \text{FADE is minimax optimal and attains oracle rate}.
\end{align*}
Ambient dimension progressively weakens an intrinsic Gaussian inverse-problem obstruction; once it is sufficiently large, the simple two-stage FADE estimator is minimax optimal.

\section{Observation-density estimation}\label{sec: observation density theory}

We now study the induced observation density
$$p_0(y) = p_{g_0, \Lambda_0, \Sigma_0}(y):=\int_{\bbR^d} \phi_D(y \mid \Lambda_0 z, \Sigma_0)\, g_0(z) \, dz, \qquad y\in \bbR^D.$$
Unlike $g_0$, which is defined on fixed-dimensional intrinsic space and assessed only modulo rotation, $p_0$ is a fully identifiable density on growing-dimensional $\bbR^D$. We estimate it by the plug-in estimator $\widehat p_{h^*}$ in \eqref{eq: fade observation density}. Ambient dimension has competing effects: it improves latent-score recovery but increases the number of loadings and noise variances to estimate.

Although Gaussian convolution makes $p_0$ real analytic, estimation is not automatically parametric because the mixing density and a growing-dimensional factor parameter are unknown. We use $L^1$ loss,
$$\|p - q\|_{L^1} = \int_{\bbR^D} |p(y) - q(y)| dy.$$
This distance remains bounded as $D$ changes and contracts under the Gaussian observation kernel, yielding a clean decomposition into latent-density and factor-parameter error.

\subsection{Upper bound for the observation density}\label{subsec: obs density upper bound}

Two preliminary facts connect latent squared-$L^2$ bounds to observation-density $L^1$ loss. First, impose the kernel-tail condition
\begin{equation}\label{assume: K tail}
    \int_{\|u\|>t} |K(u)| \, du \leq C_K e^{-c_Kt^2}, \qquad t\geq 0,
\end{equation}
for fixed $c_K,C_K>0$. This holds for Gaussian, polynomially modified Gaussian, and compactly supported kernels. Appendix \ref{app: additional lemmas} combines it with sub-Gaussian tails of $g_0$ to show that if $\bbE_0[\|\widehat g_h-g_0^R\|_{L^2}^2\mid\cE_{n,D}]\lesssim r_{n,D}$, then
$$\bbE_0[\|\widehat g_h - g_0^R\|_{L^1} \mid \cE_{n,D}] \lesssim r_{n,D}^{1/2} \{\log_+ (1/r_{n,D})\}^{d/4},$$
where $\log_+x=1\vee\log x$; the logarithm follows from truncation and Cauchy--Schwarz. Second, Appendix \ref{app: additional lemmas} also shows that the positive correction $\widetilde g_h$ in \eqref{eq: positive correction} preserves this rate and has uniformly bounded first moment, which controls loading-estimation error.

Let $\cR_{n,D}^* = \cR_{n,D}(h^*)$ be the optimized squared $L^2$ upper bound from Section \ref{subsec: latent density optimized rate}, and define
$$\cL_{n,D}^* := (\cR_{n,D}^*)^{1/2} \{\log_+ (1 / \cR_{n,D}^*)\}^{d/4}.$$
Thus,
$$\bbE_0\left[\|\widetilde g_{h^*} - g_0^R\|_{L^1}\mid \cE_{n,D}\right] \lesssim \cL_{n,D}^*.$$

\begin{theorem}\label{thm: observation kde rate}
    Suppose the conditions of Theorem \ref{thm: new kde} and the kernel-tail condition \eqref{assume: K tail} hold. Construct the FADE observation density estimator $\widehat p(y):=\widehat p_{h^*}(y)$ from \eqref{eq: fade observation density} using the optimal $h^*$ from Section \ref{subsec: latent density optimized rate}. Then
    $$\bbE_0\|\widehat p - p_0 \|_{L^1} \lesssim \cL_{n,D}^* + \sqrt{\frac{D}{n}} + \sqrt{\frac{\ell_D}{D}} + \frac1n.$$
\end{theorem}

The terms are, respectively, the latent nonparametric component, the growing-dimensional parametric component, the finite-$D$ spectral remainder, and the exceptional-event contribution. As with Remark \ref{remark: uniform upper bound}, the upper bound in the above theorem can be stated uniformly over parameters $\theta\in\Theta_*$ subject to the stated assumptions, with the constant depending on all the constants in the assumptions and $d$.

The proof uses the triangle inequality to decompose
$$\|\widehat p - p_0\|_{L^1} \leq \|\widehat p - \widetilde p\|_{L^1} + \|\widetilde p - p_0\|_{L^1}, \qquad \widetilde p(y):=\int_{\bbR^d} \phi_D(y|\Lambda_0Rz, \Sigma_0)\widetilde g_{h^*}(z) dz.$$
The Markov-kernel contraction gives $\|\widetilde p-p_0\|_{L^1}\leq\|\widetilde g_{h^*}-g_0^R\|_{L^1}$; a Gaussian perturbation bound controls the first term through parameter distances.

In particular, if $d\geq 3$ and $D\gtrsim D_c(n)$, then Section \ref{subsec: latent density optimized rate} gives $\cR_{n,D}^* \lesssim n^{-2\beta/(2\beta+d)}$. Consequently, assuming the bounded-loading regime $\ell_D\asymp 1$, up to the negligible $n^{-1}$ term, we have
$$\bbE_0 \|\widehat p - p_0\|_{L^1} \lesssim n^{-\beta/(2\beta+d)}(\log n)^{d/4} + \sqrt{\frac{D}{n}} + \frac{1}{\sqrt{D}}.$$

Increasing $D$ reduces latent-score error but increases parameter-estimation cost. Thus consistency requires sufficiently informative ambient dimension together with $D=o(n)$; more generally, the finite-$D$ remainder vanishes when $\ell_D=o(D)$.

\subsection{Lower bound and statistical regimes}\label{subsec: obs density lower bound}

We now identify which upper-bound components are fundamental. Although every admissible observation density is real analytic, the experiment changes with $D$: projection reveals finer features of $g$, while the number of unknown loadings and variances grows. The lower bound makes these competing difficulties explicit. Recall $\Theta_{D,\beta}$ and $\cG_{\beta}(L,C)$ from Section \ref{subsec: latent density lower bound}.

\begin{theorem}\label{thm: obs lower bound}
    Let $d\geq 1$ and $\beta>0$ be fixed. Suppose $L,C$ are sufficiently large so that $\cG_\beta(L,C)$ contains a neighborhood of the standard Gaussian density. Then there exists a constant $c>0$ independent of $n,D$ such that
    \begin{equation}\label{eq: obs density lower bound}
    \begin{aligned}
        \inf_{\widehat p} \sup_{(g,\Lambda, \Sigma)\in \Theta_{D,\beta}} &\bbE_{g,\Lambda, \Sigma} \left[\|\widehat p - p_{g,\Lambda, \Sigma}\|_{L^1}\right] \geq  \\ &\qquad c\max\left\{\min\left\{n^{-\beta/(2\beta+d)}, \frac{D^{d/4}}{\sqrt{n}}\right\}, \; \min \left\{1, \sqrt{\frac Dn}\right\}\right\}
    \end{aligned}
    \end{equation}
    where the infimum is over all estimators based on $Y_1,\dots,Y_n\sim p_{g,\Lambda,\Sigma}$.
\end{theorem}

The first term fixes $(\Lambda,\Sigma)$ and varies $g$. The informative experiment reduces to $X_i=\eta_i+D^{-1/2}Z_i$, so an Assouad construction at resolution $n^{-1/(2\beta+d)}$ yields $D^{d/4}/\sqrt n$ until $D\gtrsim n^{2/(2\beta+d)}$, after which it reaches the oracle $L^1$ rate $n^{-\beta/(2\beta+d)}$. The second term fixes $(g,\Lambda)$ and varies diagonal $\Sigma$: estimating $D$ variances gives the growing-dimensional parametric obstruction and rules out uniform consistency when $D$ is of order $n$.

For $d=1,2$, the parametric component dominates, up to constants $\mathfrak R_{n,D}^{\rm obs}\gtrsim\min\{1,\sqrt{D/n}\}$, where $\mathfrak R_{n,D}^{\rm obs}$ is the left hand side in \eqref{eq: obs density lower bound}. For $d\geq3$, there are three consistent regimes.

\begin{equation}\label{eq: obs density lower bound regimes}
    \mathfrak R_{n,D}^{\rm obs} \gtrsim \begin{cases}
        D^{d/4}/\sqrt n, &D \lesssim n^{2/(2\beta+d)}, \\
        n^{-\beta/(2\beta+d)}, &n^{2/(2\beta+d)}\lesssim D \lesssim n^{d/(2\beta+d)}, \\
        \sqrt{D/n}, & n^{d/(2\beta+d)}\lesssim D \lesssim n,
    \end{cases}
\end{equation}
with risk bounded away from zero when $D\gtrsim n$. Unlike the latent target, increasing $D$ also changes the observation-density family: finer features of $g$ become visible. The middle regime is therefore nonparametric despite analyticity; only in the third does factor-parameter estimation dominate.

Comparing with Theorem \ref{thm: observation kde rate}, whenever
\begin{equation}
    \mathcal L_{n,D}^*
\vee
\sqrt{\frac{\ell_D}{D}}
\lesssim
\sqrt{\frac Dn},
\qquad D=o(n),
\end{equation}
the FADE plug-in estimator satisfies
$$\sup_{\theta\in \Theta_*} \,\bbE_\theta\|\widehat p - p_\theta\|_{L^1} \asymp \mathfrak R_{n,D}^{\rm obs} \asymp \sqrt{D/n}.$$
Thus FADE is minimax optimal once latent and score-recovery terms fall below growing-dimensional parametric error. For $d\geq3$, bounded loading rows, and $D\gtrsim D_c(n)$, $\cL_{n,D}^*\lesssim n^{-\beta/(2\beta+d)}(\log n)^{d/4}$, so the sufficient condition
\begin{equation}\label{eq: observation parametric condition}
    D\gtrsim D_c(n) \vee n^{d/(2\beta+d)}(\log n)^{d/2} \vee n^{1/2}
\end{equation}
places FADE in the minimax-optimal parametric regime. Analogous thresholds follow by substituting the corresponding $\cL_{n,D}^*$ values; the key point is that spectral preprocessing and plug-in estimation incur no rate loss there. The following example demonstrates an interesting regime, where the upper and lower bounds nearly match, but not at the $\sqrt{D/n}$ part.

\medskip
\noindent
\begin{example}
There is also a regime in which the nonparametric component dominates. Suppose $d\geq4$, $0<\beta<d-3$, and $\ell_D\asymp1$. Then
the interval $D_c(n)
\vee
\left[
n^{2\beta/(2\beta+d)}(\log n)^{-d/2}
\right]
\lesssim D
\lesssim
n^{d/(2\beta+d)}$ is nonempty at polynomial order. There the latent component dominates, while \eqref{eq: obs density lower bound regimes} gives the nonparametric lower bound. FADE matches it at polynomial order, with the additional logarithmic factor from the $L^2$--$L^1$ conversion. For $d=4$, $\beta=1/2$, the window is $n^{11/15}\log n\lesssim D\lesssim n^{4/5}$, and in this case,
$$n^{-1/10}\lesssim \mathfrak R_{n,D}^{\rm obs} \leq \sup_{\theta\in\Theta_*} \,\bbE_\theta\|\widehat p - p_\theta\|_{L^1} \lesssim n^{-1/10} \log n.$$
\end{example}

Thus the nonparametric term is not merely an artifact of the FADE decomposition: analyticity of each density does not imply a uniform parametric rate when the family grows. The effect of ambient dimension on observation density estimation in this model is complex; on one hand, it resolves latent structure, while on the other it enlarges the parameter space itself. For sufficiently large $D$, FADE attains the minimax parametric rate; for $d\geq3$, an intermediate regime remains nonparametric, and FADE matches it up to logarithms in some cases.

\section{Numerical experiments}\label{sec: numerical experiments}

In this section, we complement our theoretical findings with simulated numerical experiments. Simulation studies concerning spectral recovery, latent density estimation and observation density estimation are provided in Sections \ref{subsec: simulations spectral recovery}, \ref{subsec: simulations latent density} and \ref{subsec: simulations obs density} respectively, with details in Appendix \ref{app: simulation details}.

\subsection{Spectral recovery}\label{subsec: simulations spectral recovery}
\begin{figure}
    \centering
    \includegraphics[width=1.0\linewidth]
    {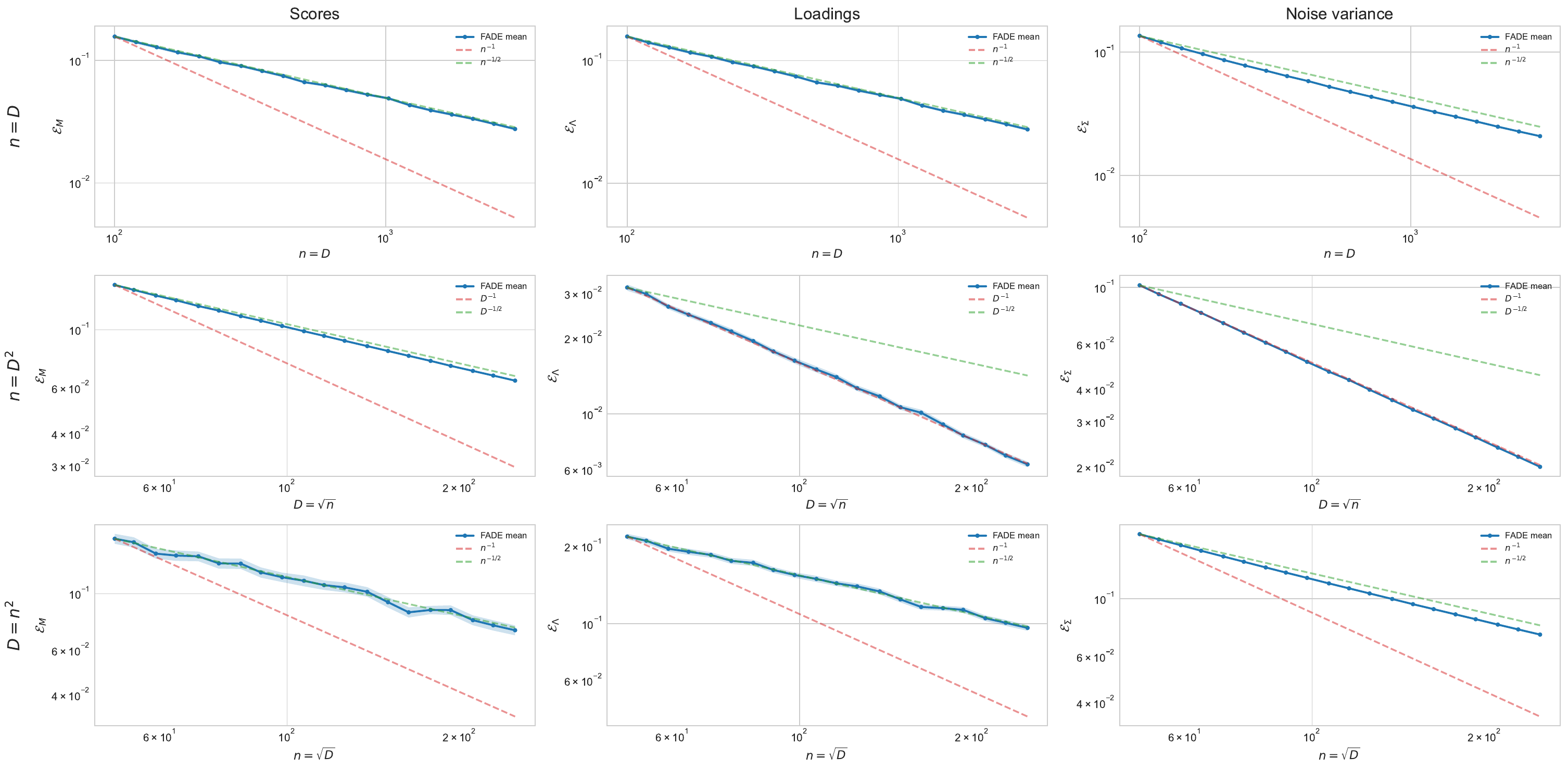}
    \caption{
    Spectral recovery under three joint-growth regimes. Rows correspond to
    $(n=D=M)$, $(n=M^2,D=M)$, and $(n=M,D=M^2)$; columns show relative
    score error, relative loading error, and average absolute noise-variance
    error. Solid blue curves give Monte Carlo means over \(100\) repetitions,
    and shaded bands show one Monte Carlo standard error. Dashed curves are
    reference rates with slopes \(M^{-1}\) and \(M^{-1/2}\), rescaled to
    agree with the empirical curve at the smallest value of \(M\). When
    \(n=M^2\) and \(D=M\), the score error follows the \(D^{-1/2}\)
    reference rate, whereas the loading and noise-variance errors follow the
    faster \(D^{-1}\) rate. When \(n=D\) or \(D=n^2\), the
    \(n^{-1/2}\) contribution is dominant.
    }
    \label{fig: parameter estimation simulation gmm truth}
\end{figure}

We first illustrate the asymmetric roles of $n$ and $D$, particularly the faster $D^{-1}$ contribution to loading and noise-variance estimation.

We use $d=5$, a standardized ten-component Gaussian mixture $g_0$, $\Sigma_0=I_D$, and bounded spherical loading rows ($\ell_D\asymp1$). The three regimes are
$$(n,D): \qquad  (M,M), \qquad (M^2,M), \qquad (M,M^2)$$
corresponding to $n=D$, $n=D^2$, and $D=n^2$. We report Monte Carlo averages over $100$ replications. Ignoring logarithmic factors, theory gives
\begin{align*}
    \cE_M &:= \frac{\|\widehat M-M_0\widehat R\|_F}
         {\|M_0\|_F} \lesssim \frac1{\sqrt n}+\frac1{\sqrt D},\\
         \cE_{\Lambda} &:= \frac{\|\widehat\Lambda-\Lambda_0\widehat R\|_F}
         {\|\Lambda_0\|_F} \lesssim \frac1{\sqrt n}+\frac1D, \\
         \cE_{\Sigma} &:=\frac1D\sum_{j=1}^D
    |\widehat\sigma_j^2-\sigma_{0j}^2| \lesssim \frac1{\sqrt n}+\frac1D.
\end{align*}

The resulting polynomial orders are
\[
\begin{array}{c|ccc}
(n,D)
&
\mathcal E_M
&
\mathcal E_\Lambda
&
\mathcal E_\Sigma
\\ \hline
(M,M)
&
M^{-1/2}
&
M^{-1/2}
&
M^{-1/2}
\\
(M^2,M)
&
M^{-1/2}
&
M^{-1}
&
M^{-1}
\\
(M,M^2)
&
M^{-1/2}
&
M^{-1/2}
&
M^{-1/2}.
\end{array}
\]

Figure~\ref{fig: parameter estimation simulation gmm truth} follows these slopes. In the informative middle row ($n=D^2$), score error is approximately $D^{-1/2}$ whereas loading and average noise-variance errors are approximately $D^{-1}$, supporting the projected-estimator cancellation underlying the density analysis.

\subsection{Latent density estimation}\label{subsec: simulations latent density}

We next compare FADE with an oracle KDE using true scores, isolating the cost of spectral score recovery. We use the preceding $d=5$ mixture, $\Sigma_0=I_D$, bounded spherical loadings, isotropic Gaussian kernels, and separate five-fold held-out likelihood bandwidth selection.

We use two nested designs: fix $D\in\{10,100,500\}$ and increase $n$ from $200$ to $5000$, or fix $n\in\{100,500,5000\}$ and increase $D$ from $10$ to $500$. We use $50$ replications and integrated squared error modulo rotation $\|\widehat g - g_0^{\widehat R}\|_{L^2}^2$, where $\widehat R$ is the Procrustes rotation used to align the scores. 

\begin{figure}
    \centering
    \includegraphics[width=1.0\linewidth]{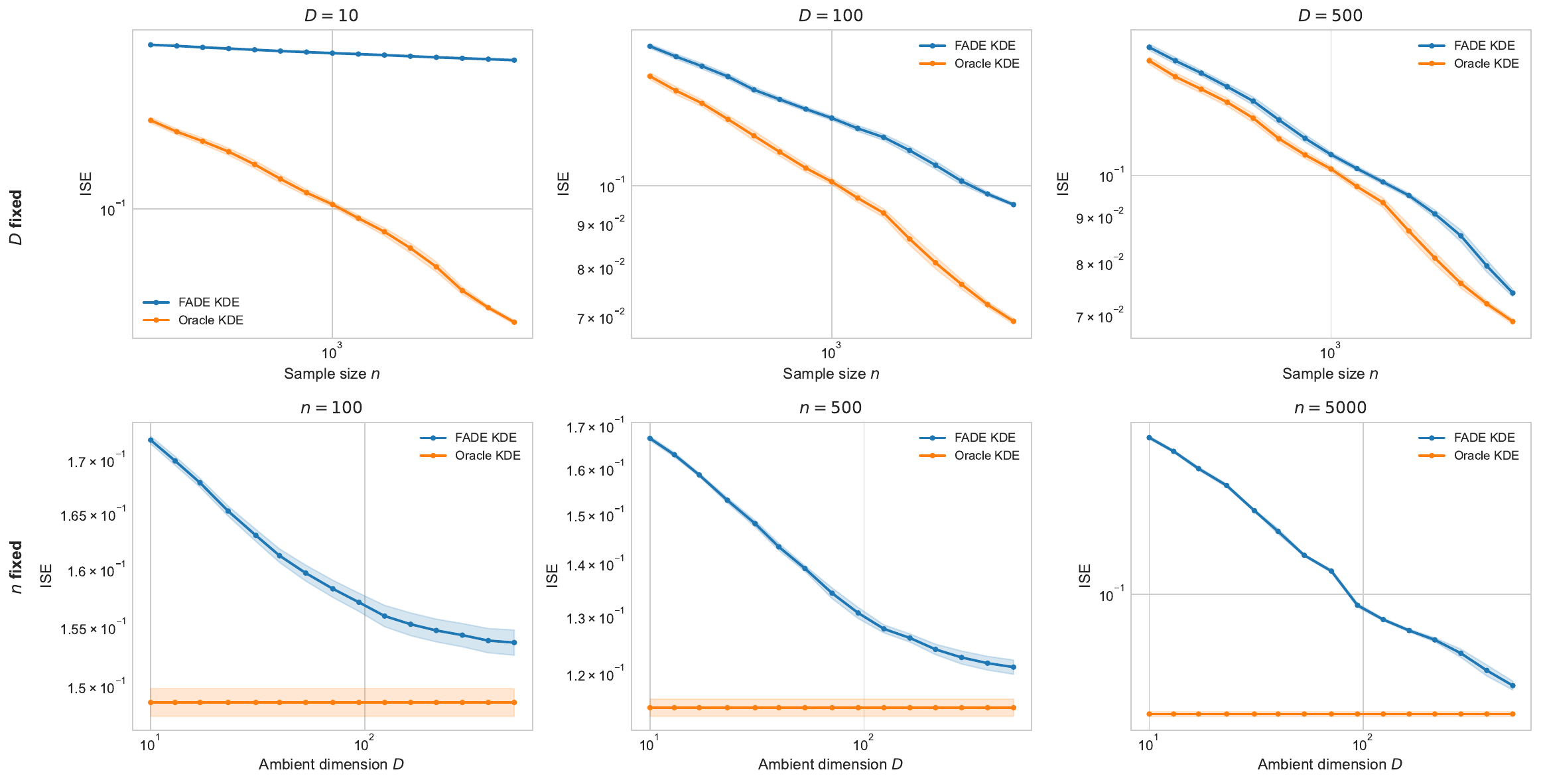}
    \caption{Latent-density estimation in $d=5$, measured via Integrated Squared Error (ISE) after Procrustes alignment of the scores. The top row fixes $D$ and
    increases $n$; the bottom row fixes $n$ and increases $D$.
    Curves show mean integrated squared error over $50$ replications, with
    shaded bands representing one Monte Carlo standard error. FADE uses
    spectrally estimated scores, whereas the oracle KDE observes the true
    latent scores. At fixed small $D$, increasing $n$ exposes the
    latent-score recovery bottleneck. At fixed $n$, increasing $D$
    reduces the FADE error toward the oracle benchmark. Vertical scales
    differ across panels.}
    \label{fig: latent density estimation simulation}
\end{figure}

In the top row of Figure~\ref{fig: latent density estimation simulation}, low $D$ produces an increasingly visible FADE--oracle separation as $n$ grows: more subjects reduce ordinary KDE error but do not improve each score. At $D=500$, FADE is close to the oracle. In the bottom row, the oracle is unchanged within a replication as $D$ varies, whereas FADE approaches it as added features improve score recovery. Thus increasing $n$ alone need not remove the latent-recovery bottleneck, while increasing $D$ does.

\subsection{Observation density estimation}
\label{subsec: simulations obs density}

Our final experiment studies observation-density estimation using a non Gaussian $d=3$ latent density obtained by smoothing two curved trajectories and standardizing to mean zero and identity covariance (Figure~\ref{fig: demonstration}). We compare FADE with an oracle KDE, a 10-component jointly estimated latent Gaussian mixture (J-GMM), a 10-component mixture of factor analyzers (MFA), and a standard Gaussian factor model. We consider
$$n\in \{100, 500, 5000\},\qquad D\in \{10, 100, 500\},$$
and evaluate per-coordinate excess log loss on $5000$ independent test observations.

\begin{table}[t]

\centering
\caption{
Observation-density estimation under per-coordinate excess held-out log
loss. Each entry is written as \(a(b)\times10^k\), where
\(a\times10^k\) is the Monte Carlo mean and \(b\times10^k\) is its Monte
Carlo standard error, computed over \(30\) independent replications. Lower
values are better. Bold denotes the best implementable estimator in each
setting. The shaded oracle KDE is an unattainable benchmark and is excluded
from this comparison.
}
\label{tab: observation density estimation k=10}

\begingroup
\small
\setlength{\tabcolsep}{5pt}
\renewcommand{\arraystretch}{1.1}

\begin{adjustbox}{
    max width=0.9\textwidth
}
\begin{tabular*}{\textwidth}{
    @{\extracolsep{\fill}}
    lccc
    @{}
}
\toprule
& \multicolumn{3}{c}{Ambient dimension \(D\)}\\
\cmidrule(lr){2-4}
Method & \(10\) & \(100\) & \(500\)\\
\midrule

\multicolumn{4}{@{}l}{\textit{Panel A: Training sample size \(n=100\)}}\\
\addlinespace[2pt]
\rowcolor{gray!10}
Oracle KDE
& \mcse{4.941}{0.510}{-3}
& \mcse{4.024}{0.200}{-3}
& \mcse{1.764}{0.085}{-3}\\
FADE KDE
& \mcse{9.444}{0.240}{-2}
& \bestmcse{3.084}{0.031}{-2}
& \bestmcse{2.383}{0.018}{-2}\\
J-GMM(10)
& \mcse{9.261}{0.650}{-2}
& \mcse{3.516}{0.042}{-2}
& \mcse{3.032}{0.022}{-2}\\
MFA(10)
& \mcse{5.981}{0.330}{-1}
& \mcse{4.720}{0.360}{-1}
& \mcse{2.403}{0.170}{-1}\\
Standard FA
& \bestmcse{6.991}{0.270}{-2}
& \mcse{4.968}{0.039}{-2}
& \mcse{3.372}{0.017}{-2}\\

\midrule
\multicolumn{4}{@{}l}{\textit{Panel B: Training sample size \(n=500\)}}\\
\addlinespace[2pt]
\rowcolor{gray!10}
Oracle KDE
& \mcse{1.092}{0.089}{-3}
& \mcse{8.669}{0.260}{-4}
& \mcse{5.166}{0.110}{-4}\\
FADE KDE
& \mcse{5.603}{0.130}{-2}
& \mcse{8.377}{0.080}{-3}
& \bestmcse{4.953}{0.036}{-3}\\
J-GMM(10)
& \bestmcse{1.208}{0.041}{-2}
& \bestmcse{6.822}{0.067}{-3}
& \mcse{5.722}{0.036}{-3}\\
MFA(10)
& \mcse{8.695}{0.160}{-2}
& \mcse{5.482}{0.039}{-2}
& \mcse{4.655}{0.024}{-2}\\
Standard FA
& \mcse{5.084}{0.230}{-2}
& \mcse{2.734}{0.013}{-2}
& \mcse{1.138}{0.0034}{-2}\\

\midrule
\multicolumn{4}{@{}l}{\textit{Panel C: Training sample size \(n=5000\)}}\\
\addlinespace[2pt]
\rowcolor{gray!10}
Oracle KDE
& \mcse{1.132}{0.140}{-4}
& \mcse{1.455}{0.058}{-4}
& \mcse{1.185}{0.020}{-4}\\
FADE KDE
& \mcse{4.249}{0.091}{-2}
& \mcse{3.092}{0.022}{-3}
& \bestmcse{7.531}{0.039}{-4}\\
J-GMM(10)
& \bestmcse{1.360}{0.057}{-3}
& \bestmcse{1.279}{0.020}{-3}
& \mcse{8.173}{0.045}{-4}\\
MFA(10)
& \mcse{5.538}{0.110}{-3}
& \mcse{6.609}{0.041}{-3}
& \mcse{5.339}{0.019}{-3}\\
Standard FA
& \mcse{4.787}{0.270}{-2}
& \mcse{2.290}{0.014}{-2}
& \mcse{6.732}{0.0085}{-3}\\
\bottomrule
\end{tabular*}
\end{adjustbox}

\endgroup
\end{table}


Table~\ref{tab: observation density estimation k=10} shows marked improvement in FADE as $D$ grows: for $n=5000$, excess loss falls from $0.04249$ at $D=10$ to $0.00075$ at $D=500$. At small $D$, J-GMM can be substantially more accurate, consistent with the potential value of jointly handling latent uncertainty in the unresolved intermediate-dimensional regime. At $D=500$, FADE is best among implementable methods for every $n$. The standard Gaussian factor model wins only at $n=100,D=10$, where its restrictive structure offers a favorable bias--variance tradeoff; MFA performs poorly in most settings.

The simulation uses per-coordinate excess log loss, whereas theory uses unnormalized $L^1$ loss; it therefore does not directly verify the rates and suppresses part of the growing-dimensional parameter cost. It nevertheless supports the central message: joint modeling can help at small or moderate $D$, while sufficiently accurate spectral score recovery makes simple two-stage FADE highly competitive at large $D$.

\section{Discussion}
\label{sec: discussion}
High ambient dimension can be a source of information rather than only complexity. In a linear factor model, the $D$ coordinates provide repeated noisy measurements of the same low-dimensional latent vector. FADE exploits this by combining spectral score recovery with ordinary $d$-dimensional density estimation. The score proxies are dependent and are not draws from the latent density, but, after a common rotation, their leading errors are conditionally centered with covariance of order $D^{-1}$. Retaining this structure in the kernel expansion yields sharper control than deterministic score-error bounds.

For latent-density estimation, when $d\geq3$ and $D$ exceeds the identified threshold, FADE attains the oracle squared-$L^2$ rate $n^{-2\beta/(2\beta+d)}$; for $d=2$ and $\beta\geq1/2$, it does so up to logarithmic factors. The lower bound shows that this conclusion is not merely algorithmic: finite $D$ introduces a Gaussian inverse-problem obstruction of order $(D\log n)^{-\beta}$, so no estimator can attain the oracle rate when $D$ is too small. In the sufficiently high-dimensional regimes, FADE is minimax-rate optimal.

Observation-density estimation gives a complementary picture. Although convolution with Gaussian noise makes each target density analytic, the problem need not be uniformly parametric: increasing $D$ reveals finer mixing-distribution features while increasing the number of factor parameters. The resulting lower bound has both nonparametric and growing-dimensional parametric components. In sufficiently high dimensions with $D=o(n)$, FADE attains the minimax parametric rate $\sqrt{D/n}$; in some intermediate regimes, the nonparametric component dominates and FADE matches it up to logarithmic factors.

Several questions remain. The current bounds leave gaps for low $d$ and moderate $D$; these may reflect the analysis, the lower bound, or information lost by spectral point estimates relative to joint procedures. A theory of data-dependent bandwidth selection is also needed, since the optimal bandwidth depends on both $n$ and $D$. More broadly, it is natural to ask whether mixture-based or Bayesian nonparametric estimators applied after spectral preprocessing exhibit analogous oracle transitions. Finally, allowing $d$ to grow requires a new density-estimation analysis, even if much of the spectral theory extends by tracking dimension-dependent constants.

The broader message is that the effect of dimension depends on what it represents. For an unrestricted density on $\bbR^D$, it produces the usual curse of dimensionality. When coordinates provide repeated information about a fixed-dimensional latent vector, however, ambient dimension can improve the effective quality of the data. FADE formalizes when a simple spectral reduction followed by low-dimensional nonparametric estimation is rate-equivalent to observing the latent variables themselves.

%
%

\begin{acks}[Declaration of AI assistance]
The authors used ChatGPT (GPT-6.1 Sol Light) to improve the manuscript's wording and grammar, after finalizing the results. All mathematical proofs, statistical methods, and conclusions were developed by the authors. The authors have reviewed the final text and accept full responsibility for its accuracy.
\end{acks}
\begin{funding}
This project was supported by the U.S. National Institutes of Health (RO1ES035625), the Office of Naval Research (Grant N00014-24-1-2626) and National Science Foundation under Grant Number 242676.
\end{funding}

\bibliographystyle{spr-ims-number} 
\bibliography{ref}       

@article{Berenfeld2024,
title={Estimating a density near an unknown manifold: A {B}ayesian nonparametric approach},
author={Berenfeld, C and Rosa, P and Rousseau, J},
journal={The Annals of Statistics},
volume={52},
issue={5},
pages={2081-2111},
year={2024}
}

@article{Berenfeld2021,
title={Density estimation on an unknown submanifold},
author={Berenfeld, C and Hoffmann, N},
journal={Electronic Journal of Statistics},
volume={15},
pages={2179-2223},
year={2021}
}

@article{Dunson2010, 
title={Compressive sensing on manifolds using a nonparametric mixture of factor analyzers: Algorithm and performance bounds},
author={Dunson, D and Chen, M and Silva, J and Paisley, J and Wang, C and Carin, L},
journal={IEEE Signal Processing},
volume={58},
pages={6140-6155},
year = {2010}
}

@article{chen2019,
title={Joint maximum likelihood estimation for high-dimensional exploratory item factor analysis},
author={Chen, Y and Li, X and Zhang, S},
journal={Psychometrika},
volume={84},
issue={1},
pages={124-146},
year={2019},
}

@article{chen2021,
title={Determining the number of factors in high-dimensional generalized latent factor models},
author={Chen, Y and Li, X},
journal={Biometrika},
volume={109},
issue={3},
pages={769-782},
year={2021}
}

@article{bai2003,
title={Inferential theory for factor models of large dimension},
author={Bai, J},
journal={Econometrica},
volume={71},
issue={1},
pages={135-171},
year={2003}
}

@article{bai2016,
title={Maximum likelihood estimation and inference for approximate factor models of high dimension},
author={Bai, J and Li, K},
journal={The Review of Economics and Statistics},
volume={98},
issue={2},
pages={298-309},
year={2016}
}

@article{fan2013,
title={Large covariance estimation by thresholding principal orthogonal complements},
author={Fan, J and Liao, Y and Mincheva, M},
journal={Journal of the Royal Statistical Society Series B},
volume={75},
issue={4},
pages={603-680},
year={2013}
}

@article{kw1956,
title={Consistency of the maximum likelihood estimator in the presence of infinitely many incidental parameters}, 
author={Kiefer, J and Wolfowitz, J},
journal={Annals of Mathematical Statistics},
volume={27},
issue={4},
pages={887-906},
year={1956}
}

@article{mfa,
title={Mixtures of factor analyzers},
author={McLachlan, GJ and Peel, D},
journal={Proceedings of the seventeenth international conference on machine learning},
pages = {599-606},
year={2000}
}

@article{woods2006,
title={Item response theory with estimation of the latent population distribution using spline-based densities},
author={Woods, CM and Thissen, D},
journal={Psychometrika}, 
volume={71},
issue={2},
pages={281-301},
year={2006}
}

@article{Mauri2026,
title={Inference on covariance structure in high-dimensional multiview data},
author={Mauri, L and Dunson, DB},
journal = {Biometrika}, 
volume = {to appear},
year = {2026}
}

@article{bernanke2005,
title={Measuring the effects of monetary policy: A factor-augmented vector regression},
author={Bernanke, BS and Boivin, J and Eliasz, P},
journal={The Quarterly Journal of Economics}, 
volume={120},
issue={1},
pages={387-422},
year={2005}
}

@article{xu2023identifiable,
  title={Identifiable and interpretable nonparametric factor analysis},
  author={Xu, Maoran and Winter, Steven and Herring, Amy H and Dunson, David B},
  journal={arXiv preprint arXiv:2311.08254},
  year={2023}
}

@article{chandra2023escaping,
  title={Escaping the curse of dimensionality in {B}ayesian model-based clustering},
  author={Chandra, Noirrit Kiran and Canale, Antonio and Dunson, David B},
  journal={Journal of Machine Learning Research},
  volume={24},
  number={144},
  pages={1--42},
  year={2023}
}

@article{chattopadhyay2024blessing,
  title={Blessing of dimension in {B}ayesian inference on covariance matrices},
  author={Chattopadhyay, Shounak and Zhang, Anru R and Dunson, David B},
  journal={arXiv preprint arXiv:2404.03805},
  year={2024}
}

@article{mauri2025spectral,
  title={Spectral decomposition-assisted multi-study factor analysis},
  author={Mauri, Lorenzo and Anceschi, Niccol{\`o} and Dunson, David B},
  journal={arXiv preprint arXiv:2502.14600},
  year={2025}
}

@article{mauri2025factor,
  title={Factor pretraining in {B}ayesian multivariate logistic models},
  author={Mauri, Lorenzo and Dunson, David B},
  journal={Biometrika},
  volume={112},
  number={4},
  pages={asaf056},
  year={2025},
  publisher={Oxford University Press}
}

@book{tsybakov2009introduction,
  title={Introduction to nonparametric estimation},
  author={Tsybakov, Alexandre B and Zaiats, Vladimir},
  volume={11},
  year={2009},
  publisher={Springer}
}

@article{fan1991optimal,
  title={On the optimal rates of convergence for nonparametric deconvolution problems},
  author={Fan, Jianqing},
  journal={The Annals of Statistics},
  pages={1257--1272},
  year={1991},
  publisher={JSTOR}
}

@article{vershynin2019high,
  title={High-dimensional probability},
  author={Vershynin, Roman},
  journal={Cambridge Series in Statistical and Probabilistic Mathematics},
  volume={47},
  year={2019}
}

@article{vershynin2012close,
  title={How close is the sample covariance matrix to the actual covariance matrix?},
  author={Vershynin, Roman},
  journal={Journal of Theoretical Probability},
  volume={25},
  number={3},
  pages={655--686},
  year={2012},
  publisher={Springer}
}

@article{luo2021schatten,
  title={A Schatten-q low-rank matrix perturbation analysis via perturbation projection error bound},
  author={Luo, Yuetian and Han, Rungang and Zhang, Anru R},
  journal={Linear Algebra and its Applications},
  volume={630},
  pages={225--240},
  year={2021},
  publisher={Elsevier}
}

@article{davis1970rotation,
  title={The rotation of eigenvectors by a perturbation. III},
  author={Davis, Chandler and Kahan, William Morton},
  journal={SIAM Journal on Numerical Analysis},
  volume={7},
  number={1},
  pages={1--46},
  year={1970},
  publisher={SIAM}
}

@article{bai2002determining,
  title={Determining the number of factors in approximate factor models},
  author={Bai, Jushan and Ng, Serena},
  journal={Econometrica},
  volume={70},
  number={1},
  pages={191--221},
  year={2002},
  publisher={Wiley Online Library}
}

@book{stewart1990matrix,
  title={Matrix perturbation theory},
  author={Stewart, Gilbert W and Sun, Ji Guang},
  year={1990},
  publisher={Academic Press}
}

@article{bonhomme2010generalized,
  title={Generalized non-parametric deconvolution with an application to earnings dynamics},
  author={Bonhomme, St{\'e}phane and Robin, Jean-Marc},
  journal={The Review of Economic Studies},
  volume={77},
  number={2},
  pages={491--533},
  year={2010},
  publisher={Wiley-Blackwell}
}

@article{amato2010noisy,
  title={Noisy Independent Factor Analysis Model for Density Estimation and Classification},
  author={Amato, Umberto and Antoniadis, Anestis and Samarov, Alexander and Tsybakov, Alexandre},
  journal={Electronic Journal of Statistics},
  volume={4},
  pages={707--736},
  year={2010}
}

@article{pelletier2005kernel,
  title={Kernel density estimation on Riemannian manifolds},
  author={Pelletier, Bruno},
  journal={Statistics \& probability letters},
  volume={73},
  number={3},
  pages={297--304},
  year={2005},
  publisher={Elsevier}
}

@article{stock2002forecasting,
  title={Forecasting using principal components from a large number of predictors},
  author={Stock, James H and Watson, Mark W},
  journal={Journal of the American statistical association},
  volume={97},
  number={460},
  pages={1167--1179},
  year={2002},
  publisher={Taylor \& Francis}
}

@book{wand1994kernel,
  title={Kernel smoothing},
  author={Wand, Matt P and Jones, M Chris},
  year={1994},
  publisher={Chapman and Hall/CRC},
  DOI={https://doi.org/10.1201/b14876}
}

@article{samb2010nonparametric,
  title={Nonparametric kernel estimation of the probability density function of regression errors using estimated residuals},
  author={Samb, Rawane},
  journal={arXiv preprint arXiv:1010.0439},
  year={2010}
}

@article{carroll1988optimal,
  title={Optimal rates of convergence for deconvolving a density},
  author={Carroll, Raymond J and Hall, Peter},
  journal={Journal of the American Statistical Association},
  volume={83},
  number={404},
  pages={1184--1186},
  year={1988},
  publisher={Taylor \& Francis}
}

@article{stefanski1990deconvolving,
  title={Deconvolving kernel density estimators},
  author={Stefanski, Leonard A and Carroll, Raymond J},
  journal={Statistics},
  volume={21},
  number={2},
  pages={169--184},
  year={1990},
  publisher={Taylor \& Francis}
}

@article{fan1992deconvolution,
  title={Deconvolution with supersmooth distributions},
  author={Fan, Jianqing},
  journal={Canadian Journal of Statistics},
  volume={20},
  number={2},
  pages={155--169},
  year={1992},
  publisher={Wiley Online Library}
}

@article{friedman1984projection,
  title={Projection pursuit density estimation},
  author={Friedman, Jerome H and Stuetzle, Werner and Schroeder, Anne},
  journal={Journal of the American Statistical Association},
  volume={79},
  number={387},
  pages={599--608},
  year={1984},
  publisher={Taylor \& Francis}
}

@article{samarov2004nonparametric,
  title={Nonparametric independent component analysis},
  author={Samarov, Alexander and Tsybakov, Alexandre},
  journal={Bernoulli},
  volume={10},
  number={4},
  pages={565--582},
  year={2004},
  publisher={Bernoulli Society for Mathematical Statistics and Probability}
}

@article{bai2006confidence,
  title={Confidence intervals for diffusion index forecasts and inference for factor-augmented regressions},
  author={Bai, Jushan and Ng, Serena},
  journal={Econometrica},
  volume={74},
  number={4},
  pages={1133--1150},
  year={2006},
  publisher={Wiley Online Library}
}

@inproceedings{song2014nonparametric,
  title={Nonparametric estimation of multi-view latent variable models},
  author={Song, Le and Anandkumar, Animashree and Dai, Bo and Xie, Bo},
  booktitle={International Conference on Machine Learning},
  pages={640--648},
  year={2014},
  organization={PMLR}
}

@article{vandermeulen2021beyond,
  title={Beyond smoothness: Incorporating low-rank analysis into nonparametric density estimation},
  author={Vandermeulen, Robert A and Ledent, Antoine},
  journal={Advances in Neural Information Processing Systems},
  volume={34},
  pages={12180--12193},
  year={2021}
}

@article{chhor2025generalized,
  title={Generalized multi-view model: Adaptive density estimation under low-rank constraints},
  author={Chhor, Julien and Klopp, Olga and Tsybakov, Alexandre B},
  journal={Journal of Machine Learning Research},
  volume={26},
  number={236},
  pages={1--52},
  year={2025}
}

@article{mclachlan2003modelling,
  title={Modelling high-dimensional data by mixtures of factor analyzers},
  author={McLachlan, Geoffrey J and Peel, David and Bean, Richard W},
  journal={Computational Statistics \& Data Analysis},
  volume={41},
  number={3-4},
  pages={379--388},
  year={2003},
  publisher={Elsevier}
}

@article{bouveyron2014model,
  title={Model-based clustering of high-dimensional data: A review},
  author={Bouveyron, Charles and Brunet-Saumard, Camille},
  journal={Computational Statistics \& Data Analysis},
  volume={71},
  pages={52--78},
  year={2014},
  publisher={Elsevier}
}

@article{kankanala2023quasi,
  title={Quasi-Bayes in Latent Variable Models},
  author={Kankanala, Sid},
  journal={arXiv preprint arXiv:2311.06831},
  year={2023}
}

@article{woods2009item,
  title={Item response theory with estimation of the latent density using Davidian curves},
  author={Woods, Carol M and Lin, Nan},
  journal={Applied Psychological Measurement},
  volume={33},
  number={2},
  pages={102--117},
  year={2009},
  publisher={Sage Publications Sage CA: Los Angeles, CA}
}

@article{wedin1972perturbation,
  title={Perturbation bounds in connection with singular value decomposition},
  author={Wedin, Per-{\AA}ke},
  journal={BIT Numerical Mathematics},
  volume={12},
  number={1},
  pages={99--111},
  year={1972},
  publisher={Springer}
}

@article{beran1996nonparametric,
  title={On nonparametric estimation of intercept and slope distributions in random coefficient regression},
  author={Beran, Rudolf and Feuerverger, Andrey and Hall, Peter},
  journal={The Annals of Statistics},
  volume={24},
  number={6},
  pages={2569--2592},
  year={1996},
  publisher={Institute of Mathematical Statistics}
}

@article{ben2018identification,
  title={Identification of joint distributions in dependent factor models},
  author={Ben-Moshe, Dan},
  journal={Econometric Theory},
  volume={34},
  number={1},
  pages={134--165},
  year={2018},
  publisher={Cambridge University Press}
}

@article{bonhomme2016non,
  title={Non-parametric estimation of finite mixtures from repeated measurements},
  author={Bonhomme, St{\'e}phane and Jochmans, Koen and Robin, Jean-Marc},
  journal={Journal of the Royal Statistical Society Series B: Statistical Methodology},
  volume={78},
  number={1},
  pages={211--229},
  year={2016},
  publisher={Oxford University Press}
}

@article{comte2012nonparametric,
  title={Nonparametric estimation of random-effects densities in linear mixed-effects model},
  author={Comte, Fabienne and Samson, Adeline},
  journal={Journal of Nonparametric Statistics},
  volume={24},
  number={4},
  pages={951--975},
  year={2012},
  publisher={Taylor \& Francis}
}

@article{jochmans2024inference,
  title={Inference on a distribution from noisy draws},
  author={Jochmans, Koen and Weidner, Martin},
  journal={Econometric Theory},
  volume={40},
  number={1},
  pages={60--97},
  year={2024},
  publisher={Cambridge University Press}
}

@article{van2010asymptotic,
  title={Asymptotic normality of the deconvolution kernel density estimator under the vanishing error variance},
  author={van Es, Bert and Gugushvili, Shota},
  journal={Journal of the Korean Statistical Society},
  volume={39},
  number={1},
  pages={103--115},
  year={2010},
  publisher={Springer}
}

@article{durot2026minimax,
  title={Minimax optimal rates of convergence in monotone shuffled and unlinked regression models under vanishing noise},
  author={Durot, Cecile and Mukherjee, Debarghya},
  journal={Bernoulli},
  volume={32},
  number={3},
  pages={2244--2270},
  year={2026},
  publisher={Bernoulli Society for Mathematical Statistics and Probability}
}

@article{li2018embracing,
  title={Embracing the blessing of dimensionality in factor models},
  author={Li, Quefeng and Cheng, Guang and Fan, Jianqing and Wang, Yuyan},
  journal={Journal of the American Statistical Association},
  volume={113},
  number={521},
  pages={380--389},
  year={2018},
  publisher={Taylor \& Francis}
}

@article{feng2023optimal,
  title={Optimal estimation of large-dimensional nonlinear factor models},
  author={Feng, Yingjie},
  journal={arXiv preprint arXiv:2311.07243},
  year={2023}
}

@article{rubin2022statistical,
  title={A statistical interpretation of spectral embedding: the generalised random dot product graph},
  author={Rubin-Delanchy, Patrick and Cape, Joshua and Tang, Minh and Priebe, Carey E},
  journal={Journal of the Royal Statistical Society Series B: Statistical Methodology},
  volume={84},
  number={4},
  pages={1446--1473},
  year={2022},
  publisher={Oxford University Press}
}

@inproceedings{anderson2014more,
  title={The more, the merrier: the blessing of dimensionality for learning large Gaussian mixtures},
  author={Anderson, Joseph and Belkin, Mikhail and Goyal, Navin and Rademacher, Luis and Voss, James},
  booktitle={Conference on Learning Theory},
  pages={1135--1164},
  year={2014},
  organization={PMLR}
}

@article{li2026estimating,
  title={Estimating Latent Distribution of Item Response Theory Using Kernel Density Method},
  author={Li, Seewoo and Lee, Guemin},
  journal={Psychometrika},
  volume={91},
  number={2},
  pages={620--640},
  year={2026},
  publisher={Cambridge University Press}
}

@article{cui2026multidimensional,
  title={Multidimensional Item Response Theory under General Latent Distributions},
  author={Cui, Chengyu and Chen, Taoyi and Wang, Chun and Xu, Gongjun},
  journal={arXiv preprint arXiv:2605.30471},
  year={2026}
}

@article{cui2026marginal,
  title={Marginal maximum likelihood estimation and asymptotic theory for latent variable models in high dimensions},
  author={Cui, Chengyu and Xu, Gongjun},
  journal={arXiv preprint arXiv:2609.16653},
  year={2026}
}

@article{fade,
  title={Factor-assisted density estimation: when ambient dimension enables oracle rates},
  author={Chakraborty, Sunrit and Dunson, David},
  doi={to be filled},
  year={2026}
}

\newpage
\begin{appendix}

\section{Proofs for Section \ref{sec: latent theory}}
\label{app: proof of estimation result}

This section discusses the proofs for our results on latent density estimation. Recall $\bbP_0$ is the joint law of $(\eta, \epsilon)$, the latent factor and idiosyncratic noise. Denote $Y_0=M_0\Lambda_0^\top$ to be the true signal, and then the observed data is $Y=Y_0+E_0$. Let $Y_0=U_0S_0V_0^\top$ be the SVD of the true signal with $U_0\in \bbR^{n\times d}$. Also, let $M_0=\bar{U}_0\bar{D}_0\bar{V}_0^\top$ be the SVD of the factor score matrix with $\bar{U}\in\bbR^{n\times d}$. In general, $\bar{U}_0$ and $U_0$ are not identical, but since $\Lambda_0$ is full rank $d$, the underlying column spaces coincide $\rm col(U_0)=\rm col(M_0)=\rm col(\bar{U}_0)$. Since these are two different orthonormal bases for the  same $d-$dimensional population subspace, they are related by an $M_0-$measurable rotation matrix $\bar{R}\in\cO(d)$, i.e., $\bar{U}_0=U_0\bar{R}^\top$. 

Next, at the observation level, we have the rank-$d$ SVD of the observation matrix $Y=USV^\top$ with $U\in\bbR^{n\times d}$, corresponding to the top $d$ singular values. Recall, our estimate for the score matrix is $\widehat{M}=\sqrt{n} U$. We will define a sample-level measurable rotation matrix $\hat{R}$ that optimizes the Procrustes problem
$$\hat{R} \in  \argmin_{\Omega\in\cO(d)} \|U\Omega^\top - U_0\|_F^2.$$
Thus, $\hat{R}$ is a random rotation that aligns the noisy empirical $U$ to the true signal counterpart $U_0$. Finally, as we demonstrate in the following proof, the (random) orthogonal matrix $R=\bar{V}_0\bar{R}\hat{R}\in \cO(d)$ is precisely the rotation required to align \textit{both} the score matrix $\widehat{M} \approx M_0 R$ and the loading matrix $\widehat{\Lambda} \approx \Lambda_0 R$. 

\medskip
\noindent
\textbf{Subspace perturbation decompositions: } Let the SVD of the observed data matrix be $Y=USV^\top + U_\perp S_\perp V_\perp^\top$, where $U\in\bbR^{n\times d}, S\in \bbR^{d\times d}$ and $V\in\bbR^{D\times d}$ represent the top $d$ components. Let $P=UU^\top$ and $P_0=U_0U_0^\top$. We define the rank $d$ truncated pseudo-inverse as $Y^+=VS^{-1}U^\top$, meaning $YY^+ = P$. Also, define $Y_0^+ = V_0S_0^{-1}U_0^\top$ to be the pseudo-inverse of $Y_0$. Let $P_0^\perp=I_n-P_0$ and $P^\perp=I_n-P$ be the projections onto the respective orthogonal complements.

Since $Y_0=P_0Y_0$, its orthogonal complement annihilates the signal $P_0^\perp Y_0 = 0$. Projecting the observed matrix yields
\begin{equation}\label{eq: p0 complement y}
    P_0^\perp Y = P_0^\perp(Y_0+E_0) = P_0^\perp E_0.
\end{equation}
Post-multiplying by $Y^+$ gives
\begin{equation}\label{eq: P0 perp P}
    P_0^\perp YY^+ = P_0^\perp E_0 Y^+ \Longrightarrow P_0^\perp P = P_0^\perp E_0 Y^+.
\end{equation}
Taking the transpose, we also have $PP_0^\perp = (Y^+)^\top E_0^\top P_0^\perp$. Now, we split $P-P_0$ into four orthogonal blocks by expanding $P-P_0=I(P-P_0)I$ and $I=P_0+P_0^\perp$:
$$P-P_0 = P_0^\perp P P_0 + P_0 P P_0^\perp + P_0^\perp P P_0^\perp - P_0 P^\perp P_0,$$
noting that $P_0PP_0 - P_0 = -P_0(I-P)P_0=-P_0P^\perp P_0$. Substituting the exact expressions for $P_0^\perp P$ and $PP_0^\perp$ obtained above into this identity yields the exact decomposition:
\begin{align*}
    P-P_0 &= P_0^\perp E_0 Y^+ P_0 +P_0(Y^+)^\top E_0^\top P_0^\perp + P_0^\perp E_0 Y^+ P_0^\perp - P_0P^\perp P_0.
\end{align*}
Now, we linearize this via $Y^+ = Y_0^+ + (Y^+ - Y_0^+)$:
\begin{equation}
\begin{aligned}\label{eq: decomposition P-P0}
    P-P_0 &= P_0^\perp E_0 Y^+ P_0 +P_0(Y^+)^\top E_0^\top P_0^\perp + P_0^\perp E_0 Y^+ P_0^\perp - P_0P^\perp P_0 \\
    &= P_0^\perp E_0 Y_0^+ P_0 + P_0(Y_0^+)^\top E_0^\top P_0^\perp + (P_0^\perp E_0 Y^+)P_0^\perp - P_0P^\perp P_0 + \dots \\
    &= \underbrace{P_0^\perp E_0 Y_0^+ + (Y_0^+)^\top E_0^\top P_0^\perp}_{\cL(E_0)} + \underbrace{P_0^\perp PP_0^\perp - P_0P^\perp P_0 + \dots}_{\cR(E_0)}
\end{aligned}
\end{equation}
where the last line used $Y_0^+P_0 = Y_0^+$ and replaced $P_0^\perp E_0Y^+$ by $P_0^\perp P$ as obtained in Equation \eqref{eq: P0 perp P}. The complete remainder term $\cR$ is given below:
\begin{equation}\label{eq: remainder P-P0}
    \cR(E_0) = P_0^\perp E_0(Y^+ - Y_0^+)P_0 + P_0(Y^+-Y_0^+)^\top E_0^\top P_0^\perp + P_0^\perp P P_0^\perp - P_0 P^\perp P_0.
\end{equation}

\medskip
\noindent
\textbf{Another useful identity for geometric annihilation:} By definition, $Y=USV^\top + U_\perp S_\perp V_\perp^\top$. So, $Y^\top U =V S$ and hence, $YY^\top U = US^2$ since $U, U_\perp$ and $V, V_\perp$ are orthogonal. This gives the following identity
\begin{equation}
    \left(\frac{1}{nD} YY^\top\right) U = U V_n, \quad V_n = \frac{1}{nD} S^2.
\end{equation}
Left multiply the above identity with $P_0^\perp = I_n - U_0U_0^\top$:
$$P_0^\perp\left(\frac{1}{nD} YY^\top\right)U = P_0^\perp U V_n \Longrightarrow P_0^\perp U = \frac{1}{nD} \left(P_0^\perp Y\right) Y^\top U V_n^{-1},$$
This is possible since on the good event constructed below, $V_n$ is invertible. Substitute $P_0^\perp Y = P_0^\perp E_0$ from equation \eqref{eq: p0 complement y} into the above to get
\begin{equation}\label{eq: p0 complement u}
    P_0^\perp U = \frac{1}{nD} P_0^\perp E_0 Y^\top U V_n^{-1}.
\end{equation}

\medskip
\noindent
\textbf{Control on $V_n^{-1}$:} The diagonal entries $v_i$ of the diagonal matrix $V_n$ are $v_i = s_i^2(Y)/nD$. Now, by Lemma \ref{lemma: subgaussian concentration}, there is an event with $g_0-$probability at least $1-2/n$ on which $s_d(M_0)\asymp \|M_0\|\asymp \sqrt{n}$. Similarly, assumption \ref{assume: loadings} gives $s_d(\Lambda_0)\asymp \|\Lambda_0\|\asymp \sqrt{D}$. Since $Y_0=M_0\Lambda_0^\top$, we have $s_d(Y_0)\geq s_d(M_0)s_d(\Lambda_0)\asymp \sqrt{nD}$, while $\|Y_0\|\leq \|M_0\|\|\Lambda_0\|\asymp \sqrt{nD}$. Consequently, $s_i(Y_0)\asymp \sqrt{nD}$ for $i=1,\dots, d$.

Now, consider the event on which $\|E_0\|\lesssim \sqrt{n}+\sqrt{D}$. This event is included in the construction of the good event $\cE_{n,D}$ below. Since $Y=Y_0+E_0$, Mirsky's theorem \ref{thm: mirsky} gives $|s_i(Y)-s_i(Y_0)|\leq \|E_0\|$ for all $i$. Moreover,
$$\frac{\|E_0\|}{s_d(Y_0)} \lesssim \frac{\sqrt n + \sqrt D}{\sqrt{nD}} = \frac{1}{\sqrt{D}} + \frac{1}{\sqrt{n}}=o(1).$$
It follows that, for all sufficiently large $n,D$, 
$$s_d(Y) \geq s_d(Y_0) -\|E_0\| \asymp \sqrt{nD},$$
and similarly 
$$s_1(Y) \leq s_1(Y_0) + \|E_0\|\asymp \sqrt{nD}.$$
Therefore, for all sufficiently large $n,D$, we have $s_i(Y)\asymp \sqrt{nD}$ for $i=1,\dots, d$, on an event with probability at least $1-C/n$. In particular, $V_n$ is invertible on this event and 
\begin{equation}\label{eq: control v_i}
    \|V_n^{-1}\| = \frac{1}{\min_{1\leq i\leq d} v_i} = \frac{nD}{s_d^2(Y)} \lesssim 1.
\end{equation}

\subsection{Proofs for Section \ref{subsec: spectral recovery} (Spectral recovery)}

\subsubsection*{Proof of Theorem \ref{thm: spectral recovery}(i)}

\begin{proof}\label{proof of thm 1(i)}
We divide the proof into several parts.

\medskip
    \noindent
    \textbf{Step 1: Good event}

    Firstly, by Lemma \ref{lemma: subgaussian concentration}, we have event $A_1$ with probability $\bbP_0^{(n)}(A_1) \geq 1-2/n$, such that on $A_1$, we have 
    $$|s_i(M_0) - \sqrt{n}| \lesssim \sqrt{\log n}, \quad i=1,\dots,d.$$
    Secondly, by standard Gaussian matrix results, we have
    $$\|E_0\| \leq C(\sqrt{n} + \sqrt{D} + t), \quad \text{w.p.}\geq 1-2e^{-ct^2},$$
    taking $t=\sqrt{\log n}$, we have event $A_2$ with $\bbP_0^{(n)}(A_2) \geq 1-c/n$, such that on $A_2$, $\|E_0\|\lesssim \sqrt{n} + \sqrt{D}$.

    Next, for each column $\epsilon^{(j)}\sim N(0, \sigma_{0j}^2 I)$, then $(1/n) \|\epsilon^{(j)}\|^2 \overset{d}{=}\sigma_{0j}^2 \chi^2_n / n$. Applying the standard chi-square concentration, and union bound over $j=1,\dots,D$, we have an event 
$A_3$, with $\bbP_0^{(n)}$-probability at least $1-c_3/n$, such that on $A_3$, we have
    $$\max_j \left|\frac{1}{n}\|\epsilon^{(j)}\|^2 - \sigma_{0j}^2 \right| \lesssim
\sqrt{\frac{\log(nD)}n}
+
\frac{\log(nD)}n\lesssim \sqrt{\frac{\log (nD)}{n}},$$
using $\log(nD) \lesssim n$.

    Finally, $U_0^\top \epsilon^{(j)} \mid M_0\sim N(0, \sigma_{0j}^2 I_d)$, since $U_0^\top U_0=I_d$. Noting that the conditional distribution does not depend on $M_0$, we have $\|U_0^\top \epsilon^{(j)}\|^2 \sim \sigma_{0j}^2 \chi^2_d$. Since $d$ is fixed, we can again use chi-square concentration to obtain $A_4$ with $\bbP_0^{(n)}(A_4)\geq 1-c_4/n$ such that on $A_4$, 
    $$\max_j \frac{1}{n} \|U_0^\top \epsilon^{(j)}\|^2 \lesssim \frac{\log (nD)}{n}.$$
    Finally, by Lemma \ref{lemma: concentration w quadratic}, there exists $A_5$ with $\bbP_0^{(n)}(A_5)\geq 1-c_5/n$, such that 
    $$\max_{j\in [D]} \bigg|\sum_{k\neq j} w_{j,k}\epsilon^{(k)\top}(I_n-U_0U_0^\top)\epsilon^{(j)}\bigg| \lesssim \sqrt{n\log(nD)} \|w_j\|_2, \quad w_j^\top=\lambda_{0j}^\top R_0^\top S_0^{-1} V_0^\top,$$
    where $R_0=U_0^\top M_0$; see Lemma \ref{lemma: concentration w quadratic} for details. 
    
    Take $\cE_{n,D}=A_1\cap A_2\cap A_3\cap A_4\cap A_5$ and apply union bound to get the final result. Thus, we have shown that there exists a set $\cE_{n,D}$ whose $\bbP_0^{(n)}$ measure is at least $1-C/n$ for some constant $C>0$, such that on $\cE_{n,D}$ we have
    \begin{align*}
        |s_i(M_0) - \sqrt{n}| &\lesssim \sqrt{\log n}, \quad \text{for } i=1,\dots, d\\
        \|E_0\| &\lesssim \sqrt{n} + \sqrt{D} \\
        \max_{j\in [D]}\left|\frac{1}{n}\|\epsilon^{(j)}\|^2 - \sigma_{0j}^2\right|&\lesssim \sqrt{\frac{\log (nD)}{n}}\\
        \max_{j\in [D]}\frac{1}{n}\|U_0^\top \epsilon^{(j)}\|^2 &\lesssim \frac{\log (nD)}{n} \\
        \max_{j\in [D]} |\sum_{k\neq j} w_{jk}\epsilon^{(k)\top}(I_n-U_0U_0^\top)\epsilon^{(j)}| &\lesssim \sqrt{n\log(nD)} \|w_j\|_2, \quad w_j^\top=\lambda_{0j}^\top R_0^\top S_0^{-1} V_0^\top.
    \end{align*}
    Note that $\|E_0^\top M_0\|_F^2 = \sum_{j=1}^D \|M_0^\top \epsilon^{(j)}\|^2$.
    Since $M_0=U_0 R_0$ (same column space) for invertible $R_0$  with the same singular values as $M_0$, hence $\|R_0\|\asymp \sqrt{n}$ and implying $\|M_0^\top \epsilon^{(j)}\|\leq \|R_0\| \|U_0^\top \epsilon^{(j)}\|$. By the fourth condition in the above display, we immediately get $\|E_0^\top M_0\|_F^2 \lesssim Dn\log (nD)$, and hence, $(1/n)\|E_0^\top M_0\|_F \lesssim \sqrt{(D\log (nD))/n}$. Also, the concentration of $\|\epsilon^{(j)}\|$ implies that $\|\epsilon^{(j)}\|\lesssim \sqrt{n}$.

    Recall $y^{(j)}=M_0 \lambda_{0j} + \epsilon^{(j)}$. This gives $\|y^{(j)}\|\leq \|M_0\|\|\lambda_{0j}\| + \|\epsilon^{(j)}\|$. The first term is controlled via $\cE_{n,D}$ and assumption on $\Lambda_0$, $\|M_0\|\|\lambda_{0j}\|\lesssim \sqrt{n\ell_D}$. The last term is again controlled via $\cE_{n,D}$, giving the overall bound $\|y^{(j)}\|\lesssim \sqrt{n\ell_D}$.
    
    \medskip
    \noindent
    \textbf{Step 2: Subspace estimation} We have $Y=Y_0+E_0$ with $Y_0=M_0\Lambda_0^\top$ being the true signal. We start with 
    \begin{align*}
        \|Y_0^\top UU^\top Y_0 - Y_0^\top U_0U_0^\top Y_0\| &= \|\Lambda_0 M_0^\top (UU^\top - U_0U_0^\top) M_0\Lambda_0^\top\| 
    \end{align*}
    Both $\Lambda_0$ and $M_0$ have exactly $d$ positive singular values since rank is $d$. Let $W=M_0^\top (UU^\top - U_0U_0^\top)M_0$. $W$ is a square $d\times d$ matrix. We first show that $\|\Lambda_0 W\Lambda_0^\top\|\geq s_{\min}^2(\Lambda_0)\|W\|$. Let $u^*\in\bbR^d$ be the top right singular vector of $W$, then by definition, $\|u^*\|_2=1$ and $\|Wu^*\|_2 = \|W\|$. Define $x=\Lambda_0(\Lambda_0^\top \Lambda_0)^{-1}u^*$ (note that $\Lambda_0(\Lambda_0^\top\Lambda_0)^{-1}=(\Lambda_0^\top)^+$, the pseudo-inverse of $\Lambda_0^\top$). By construction, $\Lambda_0^\top x = \Lambda_0^\top(\Lambda_0^\top)^+u^*=u^*$. Hence, we get
    $$\|x\|_2 \leq \|(\Lambda_0)^+\| \|u^*\|_2 \leq 1/s_{\min}(\Lambda_0).$$
    Define $x^*=x/\|x\|_2$. Then
    $$\|\Lambda_0W\Lambda_0^\top x^*\|_2 = \frac{1}{\|x\|_2} \|\Lambda_0 W \Lambda_0^\top x\| = \frac{1}{\|x\|_2}\|\Lambda_0(Wu^*)\|_2\geq s_{\min}^2(\Lambda_0)\|W\|.$$

    Now, we lower bound $\|W\|$. Since $\rm col(M_0)=\rm col(U_0)$, we can write $M_0=U_0 Q$ where $Q$ is a $d\times d$ invertible matrix. The singular values of $M_0$ are identical to those of $Q$ as $U_0$ has orthonormal columns. Thus, $s_{\min}(M_0)=s_{\min}(Q)$. Now, we have
    $$W = Q^\top U_0^\top(UU^\top - U_0U_0^\top)U_0Q = Q^\top(U_0^\top U U^\top U_0 - I_d)Q.$$
    Since $Q$ is square and invertible, we have $\|W\|\geq s_{\min}^2(Q) \|U_0^\top UU^\top U_0 - I_d\|$. Let $\theta_{\max}$ be the maximum principal angle between the true subspace $\rm col(U_0)$ and the estimated subspace $\rm col(U)$. The standard projection distance is $\|UU^\top - U_0U_0^\top\|=\sin (\theta_{\max})$. Furthermore, the eigenvalues of $U_0^\top UU^\top U_0$ are precisely $\cos^2(\theta_i)$, which implies that the maximum eigenvalue of $(I_d-U_0^\top UU^\top U_0)$ is $1-\cos^2(\theta_{\max})=\sin^2(\theta_{\max})$. Therefore,
    $$\|I_d - U_0^\top UU^\top U_0\| = \sin^2(\theta_{\max}) = \|UU^\top - U_0U_0^\top\|^2.$$

    Thus, combining, we obtain
    \begin{align*}
        \|Y_0^\top UU^\top Y_0 - Y_0^\top U_0U_0^\top Y_0\| &= \|\Lambda_0 M_0^\top (UU^\top - U_0U_0^\top) M_0\Lambda_o^\top\| 
        \\
        &\geq s_{\min}^2(\Lambda_0) s_{\min}^2(M_0) \|UU^\top - U_0U_0^\top\|^2.
    \end{align*}
     Since the Schatten $\infty-$norm is identical to the spectral norm, Theorem $2$ in \cite{luo2021schatten} implies
    $$\|Y_0^\top UU^\top Y_0 - Y_0^\top U_0U_0^\top Y_0\| = \|Y_0^\top UU^\top Y_0 - Y_0^\top Y_0\|=\|(I_n - UU^\top)Y_0\|^2 \leq 4\|E_0\|^2.$$
    The first equality uses the fact that $U_0U_0^\top Y_0=Y_0$. Combining the last two displays, we have
    $$\|UU^\top - U_0 U_0^\top \| \leq \frac{2\|E_0\|}{s_{\min}(\Lambda_0)s_{\min}(M_0)}.$$
    
    Under $\cE_{n,D}$, $s_{\min}(M_0)\asymp \sqrt{n}$. Under the assumptions, $s_{\min}(\Lambda_0)\asymp \|\Lambda_0\|\asymp \sqrt{D}$. Also, under $\cE_{n,D}$, $\|E_0\|\lesssim \sqrt{n} + \sqrt{D}$. Thus, under $\cE_{n,D}$, we have
    \begin{equation} \label{eq: diff subspace}
        \|UU^\top - U_0U_0^\top \|\lesssim \frac{(\sqrt{n}+\sqrt{D})}{\sqrt{nD}}= \frac{1}{\sqrt{n}} + \frac{1}{\sqrt{D}}.
    \end{equation}

    \medskip
    \noindent
    \textbf{Step 3: Noise variance rate} In this step, we utilize the decomposition \eqref{eq: decomposition P-P0} to control the estimation rate for the noise variance $\sigma_j^2$.

    \medskip
    \noindent
    \textbf{Step 3(a): Controlling the remainder in the decomposition of $P-P_0$:} Let us consider the remainder term $\cR(E_0)$ in the decomposition of $P-P_0$  as obtained in  Equations \eqref{eq: decomposition P-P0}, \eqref{eq: remainder P-P0} -- we show that $\|\cR(E_0)\| \lesssim 1/n + 1/D$ in this part, under the assumptions and in the good event $\cE_{n,D}$ defined in Step 1 in the proof above. Firstly, under $\cE_{n,D}$, we have $\|E_0\|\lesssim \sqrt{n} + \sqrt{D}$ and under the assumptions, we have $\sigma_d(Y_0)\asymp \sqrt{nD}$. By Mirsky's inequality Theorem \ref{thm: mirsky}, as $Y=Y_0+E_0$, we have $\sigma_d(Y)\geq \sigma_d(Y_0)-\|E_0\| \asymp \sqrt{nD}$.
    
By Lemma \ref{lemma: spectral stability truncated pseudo-inverse}, we also have
$$\|Y^+ - Y_0^+\| \leq 2\|E_0\| \left(\frac{1}{\sigma_d(Y)^2} + \frac{1}{\sigma_d(Y_0)^2} + \frac{1}{\sigma_d(Y)\sigma_d(Y_0)} \right).$$
Using the bounds $\|E_0\|\lesssim \sqrt{n} + \sqrt{D}$ and $\sigma_d(Y)\asymp \sqrt{nD}$ and $\sigma_d(Y_0)\asymp \sqrt{nD}$, we get
\begin{equation} \label{eq: diff pseudoinv}
    \|Y^+ - Y_0^+\| \lesssim \frac{\sqrt{n}+\sqrt{D}}{nD}.
\end{equation}
We now bound the operator norm of all the four terms in $\cR(E_0)$ from Equation \eqref{eq: remainder P-P0}.

For the first two terms, note that
$$\|P_0^\perp E_0 (Y^+-Y_0^+)P_0\| \leq \|E_0\|\|Y^+-Y_0^+\| \lesssim (\sqrt{n}+\sqrt{D})\frac{(\sqrt{n}+\sqrt{D})}{nD}\lesssim\frac{1}{n}+\frac{1}{D}.$$

For the third term, note that $\|P_0^\perp PP_0^\perp\| = \|P_0^\perp UU^\top P_0^\perp\|=\|P_0^\perp U\|^2$. Now, since $P_0^\perp P = P_0^\perp E_0 Y^+$ from Equation \eqref{eq: P0 perp P}, we have $\|P_0^\perp UU^\top\| = \|P_0^\perp E_0 Y^+\|$. Since $U$ is an isometry, $\|P_0^\perp UU^\top\| = \|P_0^\perp U\|$. Hence,
$$\|P_0^\perp P P_0^\perp\| = \|P_0^\perp U\|^2 = \|P_0^\perp E_0 Y^+\|^2 \leq (\|E_0\|\|Y^+\|)^2 \lesssim \left(\frac{\sqrt{n}+\sqrt{D}}{\sqrt{nD}}\right)^2\asymp \frac1n + \frac1D.$$

For the last term, since $P_0$ and $P^\perp$ are orthogonal projections, we have $\|P_0P^\perp P_0\|=\|P_0P^\perp\|^2$ and $\|P^\perp\|=1$. Now, note that $(P_0-P)P^\perp = P_0P^\perp - 0$. Hence,
$$\|P_0P^\perp P_0\| = \|P_0P^\perp\|^2 = \|(P_0-P)P^\perp\|^2 \leq \|P_0-P\|^2\|P^\perp\|^2\lesssim \left(\frac{1}{\sqrt{n}} + \frac{1}{\sqrt{D}}\right)^2 \asymp \frac1n+\frac1D.$$
Combining all the terms, we conclude that under the assumptions and good set $\cE_{n,D}$, we have 
\begin{equation}\label{eq: bound remainder P-P0 decomposition}
    \|\cR(E_0)\| \lesssim \frac1n + \frac1D.
\end{equation}
    
    \medskip
    \noindent
    \textbf{Step 3(b): Controlling the noise} 
    Recall $\hat{\sigma}_j^2 = \frac1n y^{(j)\top}(I_n - UU^\top)y^{(j)}.$ Thus, we have
    \begin{equation}\label{eq: sigma error}
        |\hat{\sigma}_j^2 - \sigma_{0j}^2| \leq \left|\frac{1}{n}y^{(j)\top} (I -U_0U_0^\top)y^{(j)} - \sigma_{0j}^2 \right|+ \frac{1}{n}\left|y^{(j)\top}(U_0U_0^\top - UU^\top)y^{(j)}\right|.
    \end{equation}
    
    For the first term, note that $(I-U_0U_0^\top)y^{(j)}=(I-U_0U_0^\top)\epsilon^{(j)}$. Thus, we get
    $$\left|\frac{1}{n}\epsilon^{(j)\top}(I-U_0U_0^\top)\epsilon^{(j)} - \sigma_{0j}^2\right| \leq \left|\frac{1}{n}\|\epsilon^{(j)}\|^2 -\sigma_{0j}^2\right| + \frac{1}{n}\|U_0^\top \epsilon^{(j)}\|^2.$$
    The first term is bounded by $\sqrt{(\log (nD))/n}$ under $\cE_{n,D}$. The second term is bounded by $\log (nD)/n$ under $\cE_{n,D}$. Thus, this term is bounded by $\sqrt{(\log (nD))/ n}$, since $\log (nD) \lesssim n$.

    Now, for the second term in Equation \eqref{eq: sigma error}, let $T_j = \frac{1}{n}y^{(j)\top}(U_0U_0^\top - UU^\top) y^{(j)}$. Expanding $y^{(j)}=M_0\lambda_{0j} + \epsilon^{(j)}$, we break $T_j$ into three components:
    \begin{equation}\label{eq: T decomposition}
    \begin{aligned}
        T_j = \frac{1}{n}(M_0\lambda_{0j})^\top(U_0U_0^\top - UU^\top)M_0\lambda_{0j} &+ \frac{2}{n}(M_0\lambda_{0j})^\top(U_0U_0^\top - UU^\top)\epsilon^{(j)} \\
        &\qquad+ \frac{1}{n} \epsilon^{(j)\top}(U_0U_0^\top - UU^\top)\epsilon^{{(j)}}.
    \end{aligned}
    \end{equation}
    Call these terms $T_{j,1}$, $T_{j,2}$ and $T_{j,3}$ respectively, capturing the signal-signal, signal-noise and noise-noise terms. For the signal-signal term, since $\rm col(M_0)=\rm col(U_0)$, we have $U_0U_0^\top M_0 = M_0$. Thus, 
    $$T_{j,1}=\frac{1}{n}\|(I_n - UU^\top)M_0\lambda_{0j}\|^2\leq \|UU^\top - U_0 U_0^\top\|^2 \left(\frac{1}{n}\|M_0\lambda_{0j}\|^2\right)\lesssim \ell_D(1/n + 1/D).$$

    Now, we turn to the noise-noise term $T_{j,3}$. Denote $\bbD=U_0U_0^\top - UU^\top$. By standard orthogonal (Procrustes) alignment, there exists orthogonal matrix $O\in \cO(d)$ such that $\|U-U_0O\|\leq \sqrt{2}\|\bbD\|$. Let $E_U = U - U_0O$. This gives
    $$\bbD = U_0U_0^\top - UU^\top = -U_0OE_U^\top - E_UO^\top U_0^\top - E_U E_U^\top.$$
    Substituting this back into the quadratic form and using triangle inequality, we have
    \begin{equation}\label{eq: sigma T_j3 decompose}
        |T_{j,3}|\leq \frac{2}{n}\left|\epsilon^{(j)\top} U_0 O E_U^\top \epsilon^{(j)}\right| + \frac{1}{n}\left|\epsilon^{(j)\top}E_UE_U^\top \epsilon^{(j)}\right|.
    \end{equation}
    We now bound each piece using the properties controlled by $\cE_{n,D}$. First,
    $$\frac{1}{n}\left|\epsilon^{(j)\top}E_UE_U^\top \epsilon^{(j)}\right|\leq \frac{1}{n}\|E_U\|^2 \|\epsilon^{(j)}\|_2^2 \leq \frac{2}{n}\|\bbD\|^2\|\epsilon^{(j)}\|^2 \lesssim \|\bbD\|^2 \lesssim 1/n + 1/D.$$

    Next, the first term can be bounded as
    $$\frac{2}{n}\left|\epsilon^{(j)\top} U_0 O E_U^\top \epsilon^{(j)}\right|\leq \frac{2}{n}\|U_0^\top\epsilon^{(j)}\|_2 \|O\|\|E_U\|\|\epsilon^{(j)}\|_2 \lesssim \frac{1}{n}\sqrt{\log(nD)}(1/\sqrt{n}+1/\sqrt{D})\sqrt{n}.$$
    Applying the inequality $2ab\leq a^2 + b^2 $ to the term $\sqrt{\log(nD)}/\sqrt{nD}\leq \log(nD) / n + 1/D$. Combining all the pieces, we get
    $$|T_{j,3}| \lesssim \frac{\log (nD)}{n} + \frac{1}{D}.$$

    Finally, we turn to the signal-noise term $T_{j,2}$ from the decomposition in Equation \eqref{eq: T decomposition}. Recall 
    $$T_{j,2} =\frac{2}{n}(M_0\lambda_{0j})^\top (U_0U_0^\top - UU^\top) \epsilon^{(j)}.$$ 
    Using the decomposition from Equation \eqref{eq: decomposition P-P0}, we know that $P-P_0 = L(E_0) + \cR(E_0)$, where $\|\cR(E_0)\| \lesssim 1/n + 1/D$, as shown in the previous step Equation \eqref{eq: bound remainder P-P0 decomposition}.

    Thus, combining the decomposition of $P-P_0$ and the form of $T_{j,2}$, we have
    \begin{equation}\label{eq: sigma T_j2 decompose}
        T_{j,2}=-\frac{2}{n}(M_0\lambda_{0j})^\top(L(E_0) + \cR(E_0)) \epsilon^{(j)}.
    \end{equation}

    Using the bound for the remainder term, we already have a bound for the second term
    $$\frac{2}{n}|(M_0\lambda_{0j})^\top \cR(E_0)\epsilon^{(j)}|\lesssim \frac{1}{n}\sqrt{n}\sqrt{\ell_D}(1/n + 1/D)\sqrt{n} = \sqrt{\ell_D}\left(\frac1n + \frac1D\right).$$
    So we can focus on the remaining part
    $$T_{j,2}^{\rm linear} = -\frac{2}{n}(M_0\lambda_{0j})^\top \left(P_0^\perp E_0 Y_0^+ + (Y_0^+)^\top E_0^\top P_0^\perp\right)\epsilon^{(j)}.$$
    Since $\rm col(M_0)=\rm col(U_0)$, the vector $M_0\lambda_{0j}$ lies entirely within the true subspace spanned by $P_0$. Hence $(M_0\lambda_{0j})^\top P_0^\perp = 0$, which kills the first term from the above display and leaves only
    $$T_{j,2}^{\rm linear} = -\frac{2}{n}(M_0\lambda_{0j})^\top (Y_0^+)^\top E_0^\top P_0^\perp \epsilon^{(j)}.$$
    Using the SVD $Y_0=U_0S_0V_0^\top$ and $M_0=U_0R_0$ where $R_0=\bar{R}^\top \bar{D}_0\bar{V}_0^\top$ with $\|R_0\|=\|\bar{D}_0\|=\|M_0\|\asymp \sqrt{n}$ on $\cE_{n,D}$, we can simplify the deterministic vector on the left
    $$w_j^\top = (M_0\lambda_{0j})^\top (Y_0^+)^\top = \lambda_{0j}^\top R_0^\top U_0^\top(U_0S_0^{-1}V_0^\top)=\lambda_{0j}^\top R_0^\top S_0^{-1}V_0^\top.$$
    This gives 
    $$\|w_j\|_2 = \|S_0^{-1}R_0\lambda_{0j}\|\leq \|S_0^{-1}\|\|R_0\|\|\lambda_{0j}\|_2 \lesssim \frac{1}{\sqrt{nD}}\sqrt{n}\sqrt{\ell_D}=\sqrt{\ell_D}/\sqrt{D},$$
    where we use the fact that $Y_0=M_0\Lambda_0^\top$ with $\|\Lambda_0\|\asymp \sqrt{D}$ and under $\cE_{n,D}$, $\|M_0\|\asymp \sqrt{n}$. Furthermore, $w_j=V_0 z_j$, where $z_j=S_0^{-1}U_0^\top M_0\lambda_{0j}.$ Since $V_0$ has orthonormal columns, it acts as an isometry, hence $\|w_j\|_2=\|V_0z_j\|_2=\|z_j\|_2 \lesssim \sqrt{\ell_D}/\sqrt{D}$. Note that $V_0^\top=S_0^{-1}U_0^\top Y_0 = S_0^{-1}U_0^\top M_0\Lambda_0^\top$ and hence extracting the $k-$th row of $V_0$, we have $v_{0k}^\top=\lambda_{0k}^\top(M_0^\top U_0 S_0^{-1})$. Hence, for every $k=1,\dots,D$
    $$\|v_{0k}\|_2 \leq \|\lambda_{0k}\|_2 \|M_0^\top U_0 S_0^{-1}\|\lesssim \sqrt{\ell_D}\sqrt{n} \frac{1}{\sqrt{nD}} = \frac{\sqrt{\ell_D}}{\sqrt{D}}$$
    using properties established above. Thus, by Cauchy-Schwarz inequality, we have
    $$|w_{jk}| =|v_{0k}^\top z_j| \leq \|v_{0k}\|_2\|z_j\|_2 \lesssim \frac{\ell_D}{D}.$$
    Recall, our linear term
    $$T_{j,2}^{\rm linear} = -\frac{2}{n}w_j^\top E_0^\top P_0^\perp \epsilon^{(j)}= -\frac2n \sum_{k=1}^D w_{jk} \epsilon^{(k)\top} P_0^\perp \epsilon^{(j)}.$$

    For the diagonal term ($k=j$), as $\|P_0^\perp \epsilon^{(j)}\|_2^2 = \|\epsilon^{(j)}\|_2^2 - \|U_0^\top \epsilon^{(j)}\|_2^2 \lesssim n$, we have
    $$\left|-\frac2n w_{jj} \epsilon^{(j)\top} P_0^\perp \epsilon^{(j)}\right| \lesssim \frac1n \frac{\ell_D}{D}  n = \frac{\ell_D}{D}.$$

    Under $\cE_{n,D}$, the off-diagonal terms are bounded as
    $$\frac2n \left|\sum_{k\neq j} w_{jk} \epsilon^{(k)\top}P_0^\perp \epsilon^{(j)}\right| \lesssim \frac1n \sqrt{n\log(nD)} \|w_j\|_2 \lesssim \frac{\|\lambda_{0j}\|\sqrt{\log(nD)}}{\sqrt{nD}}\lesssim \sqrt{\frac{\log(nD)}{n}}.$$
    The last inequality used $\|w_j\|\lesssim \|\lambda_{0j}\|/\sqrt{D}$ and $\|\lambda_{0j}\| \leq \|\Lambda_0\| \lesssim \sqrt D$. Thus, we have
    $$|T_{j,2}| \lesssim \sqrt{\frac{\log(nD)}n} + \frac{\ell_D}{D} + \sqrt{\ell_D}\left(\frac1n+\frac1D\right).$$

    Thus, starting from the bound Equation \eqref{eq: sigma error}, we have
    $$|\hat{\sigma}_j^2 - \sigma_{0j}^2| \leq \sqrt{\frac{\log (nD)}{n}} + |T_{j,1}|+|T_{j,2}|+|T_{j,3}| $$
    and we showed
    \begin{align*}
        |T_{j,1}| &\lesssim \frac{\ell_D}{n} + \frac{\ell_D}{D} \\
        |T_{j,2}| &\lesssim \sqrt{\frac{\log(nD)}n} + \frac{\ell_D}{D} + \sqrt{\ell_D}\left(\frac1n+\frac1D\right) \\
        |T_{j,3}| &\lesssim \frac{\log(nD)}{n} + \frac1D.
    \end{align*}
    Hence, combining all parts and under our assumption on $\ell_D$, we get
    $$\max_j |\hat{\sigma}_j^2 - \sigma_{0j}^2| \lesssim \sqrt{\frac{\log (nD)}{n}} + \frac{\ell_D}{D}.$$

    \medskip
    \noindent
    \textbf{Step 4: Controlling the factor scores}  By standard Procrustes inequality \citep{davis1970rotation}
    $$\min_{Q\in \cO(d)}\|U-U_0Q\|_F \lesssim \|UU^\top - U_0U_0^\top\|_F \leq \sqrt{2d}\|UU^\top  -U_0U_0^\top\|,$$
    and let $\hat{R}$ be the optimal $Q$ above. Recall $\widehat{M}=\sqrt{n} U$, and we defined the SVD $M_0=\bar{U}_0 \bar{D}_0 \bar{V}_0^\top$ where $\bar{U}_0=U_0\bar{R}^\top$ for some orthogonal $\bar{R}$ (recall $U_0$ consisted of the top $d$ singular vectors of $M_0\Lambda_0^\top$). Let $R=\bar{V}_0\bar{R} \hat{R}$.
    \begin{align*}
        \|\widehat{M} - M_0 R \|_F &= \|\sqrt{n} U - \bar{U}_0 \bar{D}_0 \bar{V}_0^\top R\|_F = \|\sqrt{n}U - U_0 \bar{R}^\top \bar{D}_0 \bar{R}\hat{R}\|_F \\
        &\leq \|\sqrt{n}(U - U_0 \hat{R})\|_F + \|\sqrt{n} U_0 \hat{R} - U_0 \bar{R}^\top \bar{D}_0 \bar{R} \hat{R}\|_F \\
        &\leq \sqrt{n} \|U -U_0 \hat{R}\|_F + \|\sqrt{n} I - \bar{R} ^\top \bar{D}_0 \bar{R}\|_F \\
        &\lesssim \sqrt{nd}\left(\frac{1}{\sqrt{n}}+\frac{1}{\sqrt{D}}\right) + \left(\sum_{i=1}^d (s_i(M_0) - \sqrt{n})^2\right)^{1/2} \\
        &\leq \sqrt{nd}(1/\sqrt{n} + 1/\sqrt{D}) + \sqrt{d\log n} \lesssim  \sqrt{n}\left(\frac{1}{\sqrt{D}} + \sqrt{\frac{\log n}{n}}\right).
    \end{align*}
    
    \medskip
    \noindent
    \textbf{Step 5: Controlling loading estimation} 
    Recall $\widehat{\Lambda} = Y^\top U / \sqrt{n} = \frac1n Y^\top \widehat{M}$. Hence, using $R=\bar{V}_0\bar{R}\hat{R}$ from Step 4, we have
$$\widehat{\Lambda} - \Lambda_0R = \frac1n Y^\top \widehat{M} - \Lambda_0R$$
Now,
    \begin{align*}
        \frac1n Y^\top \widehat{M} - \Lambda_0R 
        &= \frac1n \left(M_0\Lambda_0^\top + E_0\right)^\top \widehat{M} - \Lambda_0 R \left(\frac1n \widehat{M}^\top \widehat{M}\right), &\text{since } \widehat{M}=\sqrt{n}U, \text{ so } \widehat{M}^\top\widehat{M} = nI \\
        &= \frac1n \Lambda_0(RR^\top)M_0^\top \widehat{M} - \frac1n \Lambda_0 R \widehat{M}^\top \widehat{M} + \frac1n E_0^\top  \widehat{M} \\
        &= \frac1n \Lambda_0 R \left(M_0R - \widehat{M}\right)^\top \widehat{M} + \frac1n E_0^\top \widehat{M}.
    \end{align*}
    Thus, we get the upper bound
    \begin{equation}\label{eq: point estimate Lambda decompose}
    \|\widehat{\Lambda} - \Lambda_0 R\|_F \;\leq \;\underbrace{\frac1n \|\Lambda_0  R(M_0 R - \widehat{M})^\top \widehat{M}\|_F}_{\text{Term A (Signal alignment error)}} \quad+ \underbrace{\frac1n \|E_0^\top \widehat{M}\|_F}_{\text{Term B (Noise projection error)}}.
    \end{equation}
    Now, we look at the two terms separately. 

    \medskip
    \noindent
    \textbf{Bounding Term A:} Since $M_0=U_0\bar{R}^\top \bar{D}_0\bar{V}_0^\top$ and $R=\bar{V}_0\bar{R}\hat{R}$ , we have
    $$M_0 R = U_0\bar{R}^\top \bar{D}_0 \bar{R} \hat{R}.$$
    Plugging this and $\widehat{M}=\sqrt{n} U$ into term A, we get
    $$\frac1n \Lambda_0 R (M_0R - \widehat{M})^\top \widehat{M} = \Lambda_0\bar{V}_0 \bar{R}\left(\frac{1}{\sqrt{n}} \bar{R}^\top \bar{D}_0 \bar{R} U_0^\top U - \hat{R}\right).$$
    Thus, using another triangle inequality, we get
    \begin{equation}\label{eq: point estimate Lambda decompose A}
        \text{Term A} \leq \left\|\Lambda_0\bar{V}_0 \bar{R} \left[\frac{1}{\sqrt{n}}\bar{R}^\top \bar{D}_0 \bar{R}(U_0^\top U - \hat{R})\right]\right\|_F + \left\|\Lambda_0 \bar{V}_0 \bar{R}\left[\left(\frac{1}{\sqrt{n}}\bar{R}^\top \bar{D}_0 \bar{R} - I_d\right)\hat{R}\right]\right\|_F
    \end{equation}
    and let us call these terms $A_1$ and $A_2$. Firstly, note that $\|\Lambda_0\bar{V}_0 \bar{R}\|=\|\Lambda_0\|\asymp \sqrt{D}$ by assumption on $\Lambda_0$.

    For the term $A_1$, we have $\|\bar{D}_0/\sqrt{n}\|\asymp 1$ (using the control on singular values of $M_0$ under $\cE_{n,D})$. Now, to control the term $\|U_0^\top U - \hat{R}\|_F$, we recall that $\hat{R}=U_1V_1^\top$, since it is the optimal Procrustes rotation where $U_0^\top U = U_1\Sigma_1 V_1^\top$, where $\Sigma_1=\diag(\cos\theta_1,\dots,\cos \theta_d)$ containing the principal angles between the estimated and true subspaces. So we have
    \begin{align*}
        \|U_0^\top U - \hat{R}\|_F = \|U_1(\Sigma_1 - I_d) V_1^\top \|_F &= \|\Sigma_1 - I_d\|_F \\
        &=\sqrt{\sum_{i\in[d]} (1 - \cos \theta_i)^2} \\
        &\leq \sum_{i=1}^d (1 - \cos \theta_i), \qquad \text{using } \ell_2 \leq \ell_1\\
        &\leq \sum_{i=1}^d \sin^2 \theta_i, \qquad \text{as }1-\cos \theta = 2\sin^2(\theta/2) \leq \sin^2 \theta \\
        &= \frac12 \|UU^\top - U_0U_0^\top\|_F^2 \\
        &\lesssim (1/\sqrt{n} + 1/\sqrt{D})^2 = 1/n + 1/D.
    \end{align*}
    where the last bound uses Step 2, and we used $\|UU^\top - U_0U_0^\top\|_F^2 = 2\sum_i \sin^2(\theta_i)$, a standard result for principal angles. Hence, we have
    $$\text{Term } A_1 \leq \|\Lambda_0\bar{V}_0 \bar{R}\|\|\bar{D}_0/\sqrt{n}\|\|U_0^\top U - \hat{R}\|_F \lesssim \sqrt{D} (1/n + 1/D).$$

    For the term $A_2$, note that $\|\Lambda_0\|_F \asymp \sqrt{D}$ and on $\cE_{n,D}$, $\|\bar{D}_0/\sqrt{n} - I_d\| \lesssim \sqrt{\log n / n}$, which gives
    $$\text{Term } A_2 \leq \|\Lambda_0\|_F \left\|\frac{1}{\sqrt{n}}\bar{D}_0 - I_d\right\| \|\hat{R}\| \lesssim \sqrt{D}\sqrt{\frac{\log n}{n}}.$$
    Combining, we have
    \begin{equation}
        \text{Term A} \lesssim \sqrt{D} \left(\sqrt{\frac{\log n}{n}} + \frac{1}{D}\right).
    \end{equation}

    \medskip
    \noindent
    \textbf{Bounding Term B:} Now, we turn to the noise projection error term. We split based on the true space $P_0=U_0U_0^\top$ and its complement
    \begin{equation}\label{eq: point estimate Lambda decompose B}
        \frac1n E_0^\top \widehat{M} = \frac1n E_0^\top [P_0 + (I-P_0)]\sqrt{n} U = \frac{1}{\sqrt{n}}E_0^\top U_0U_0^\top U + \frac{1}{\sqrt{n}}E_0^\top P_0^\perp U.
    \end{equation}
    For the first term, note that under $\cE_{n,D}$, $\|E_0^\top U_0\|_F \lesssim \sqrt{D\log(nD)}$. Since $\|U_0^\top U\|\leq 1$, we have
    $$\left\|\frac{1}{\sqrt{n}}E_0^\top U_0U_0^\top U\right\|_F \leq \frac{1}{\sqrt{n}}\|E_0^\top U_0\|_F \|U_0^\top U\| \lesssim \sqrt{\frac{D\log(nD)}{n}}.$$
    For the second term, we use the identity \eqref{eq: p0 complement u} to get
    $$E_0^\top P_0^\perp U = \frac{1}{nD}E_0^\top P_0^\perp E_0 Y^\top U V_n^{-1}$$
    which gives
    $$\left\|\frac{1}{\sqrt{n}}E_0^\top P_0^\perp U\right\|_F=\frac{1}{nD\sqrt{n}} \|E_0^\top P_0^\perp E_0 Y^\top U V_n^{-1}\|_F \leq \frac{1}{nD\sqrt{n}}\|E_0^\top P_0^\perp E_0\|\|Y^\top U\|_F \|V_n^{-1}\|.$$
    Now, $\|E_0^\top P_0^\perp E_0\|\leq \|E_0\|^2 \lesssim n+D$, and $\|V_n^{-1}\|\lesssim 1$ (since $V_n=S^2/nD$ and we showed $v_i\gtrsim 1$ before the start of the proof, Equation \eqref{eq: control v_i}). Finally, $\|Y^\top U\|_F = \|V S\|_F = \|S\|_F \asymp \sqrt{nD}$. Together, we have
    $$\left\|\frac{1}{\sqrt{n}}E_0^\top P_0^\perp U\right\|_F \lesssim \frac{1}{nD\sqrt{n}} (n+D) \sqrt{nD} = \sqrt{D}\left(\frac1n + \frac1D\right).$$
    Combined, we have
    \begin{equation}
        \text{Term B} \lesssim \sqrt{D}\left(\sqrt{\frac{\log(nD)}{n}} + \frac{1}{D}\right).
    \end{equation}
The two numbered equations above give the following upper bound for the loading matrix estimation
$$\|\widehat{\Lambda} - \Lambda_0 R\|_F \lesssim \sqrt{D}\left(\sqrt{\frac{\log(nD)}{n}} + \frac1D\right).$$
    
\end{proof}

\begin{lemma}\label{lemma: concentration w quadratic}
    Let $Y_0=M_0\Lambda_0^\top = U_0S_0V_0^\top$ be the SVD of the signal and also $M_0=U_0R_0$ be the representation of $M_0$ in the orthonormal basis $U_0$. Define $w_j^\top = \lambda_{0j}^\top R_0^\top S_0^{-1} V_0^\top$. Also define
    $$Z_j = \sum_{k\neq j} w_{jk} \epsilon^{(k)\top} (I_n - U_0U_0^\top) \epsilon^{(j)}.$$
    Then with probability $1-c/n$ with respect to $\bbP_0^{(n)}$, we have
    $$ |Z_j| \lesssim \sqrt{n\log(nD)} \|w_j\|_2, \quad \text{for all } j=1,\dots,D.$$
\end{lemma}
\begin{proof}
For each $j$, define $g_j=P_0^\perp \epsilon^{(j)}$. Since $P_0^\perp$ is an orthogonal projection matrix of rank $n-d$, we have $\|g_j\|^2/\sigma_{0j}^2 \sim \chi_{n-d}^2$ conditionally on $M_0$. In particular, 
    $$\bbE_0[\|g_j\|^2 \mid M_0] = \sigma_{0j}^2 \tr(P_0^\perp) = \sigma_{0j}^2(n-d).$$
    Let $\cG_j=\{\|g_j\|^2 \leq C_0 n\}$. By chi-square concentration and uniform upper bound of $\sigma_{0j}^2$, $C_0$ can be chosen sufficiently large that $\bbP_0^{(n)}(\cG_j^c \mid M_0) \leq (nD)^{-3}$ uniformly over $j$, using $\log (nD) \lesssim n$ (any sufficiently large fixed polynomial power in place of 3 would work). 

    Next, condition on both $M_0$ and $\epsilon^{(j)}$. For $K\neq j$, the vectors $\epsilon^{(k)}$ are independent of $\epsilon^{(j)}$ and hence
    $$\epsilon^{(k)\top} g_j \mid M_0, \epsilon^{(j)} \sim N(0, \sigma_{0k}^2 \|g_j\|^2).$$
    These variables are also independent across $k\neq j$. Therefore,
    $$Z_j \mid M_0, \epsilon^{(j)} \sim N\bigg(0, \|g_j\|^2 \sum_{k\neq j} w_{jk}^2 \sigma_{0k}^2\bigg).$$
    On $\cG_j$, Assumption \ref{assume: noise} gives
    \begin{align*}
        \Var_0\left(Z_j\mid M_0, \epsilon^{(j)}\right) &= \|g_j\|^2 \sum_{k\neq j} w_{jk}^2 \sigma_{0k}^2 \\
        &\leq C_0 C_\sigma n \sum_{k\neq j} w_{jk}^2 \leq C_0 C_\sigma n\|w_j\|^2.
    \end{align*}
    If $\|w_j\|=0$, then $Z_j=0$ and the desired inequality holds automatically. Otherwise, define $t_j = C\sqrt{n\log(nD)} \|w_j\|$. The conditional Gaussian tail bound gives, on $\cG_j$,
    $$\bbP_0^{(n)}\left(|Z_j| >t_j \mid M_0, \epsilon^{(j)}\right) \leq 2 \exp \left\{-\frac{t_j^2}{2C_0 C_\sigma n \|w_j\|^2}\right\}.$$
    Substituting the definition of $t_j$, we have 
    $$\bbP_0^{(n)}\left(|Z_j|>t_j \mid M_0, \epsilon^{(j)}\right) \leq 2(nD)^{-C^2/(2C_0C_\sigma)},$$
    on $\cG_j$. We now integrate over $\epsilon^{(j)}$, while continuing to condition on $M_0$. This gives
    \begin{align*}
        \bbP_0^{(n)}(|Z_j| >t_j , \cG_j \mid M_0) &= \bbE_0\left[\textbf{1}_{\cG_j} \bbP_0^{(n)}(|Z_j|>t_j\mid M_0, \epsilon^{(j)}) \mid M_0\right] \\
        &\leq 2(nD)^{-C^2/(2C_0C_\sigma)}.
    \end{align*}
    Consequently, 
    \begin{align*}
        \bbP_0^{(n)}(|Z_j|>t_j\mid M_0) &\leq \bbP_0^{(n)}(|Z_j|>t_j, \cG_j\mid M_0) + \bbP_0^{(n)}(\cG^c\mid M_0) \\
        &\leq 2(nD)^{-C^2/(2C_0C_\sigma)} + (nD)^{-3}.
    \end{align*}
    Choose $C$ sufficiently large so that $C^2 \geq 3(2C_0C_\sigma)$ and applying a union bound over $j=1,\dots,D$, conditionally on $M_0$, yields
    \begin{align*}
        \bbP_0^{(n)}\left(\exists j\in[D]:|Z_j|>C\sqrt{n\log(nD)} \|w_j\| \mid M_0\right) \leq 3D(nD)^{-3} \leq c/n.
    \end{align*}
    This bound is uniform on $M_0$. Integrating with respect to $M_0$, therefore gives
    $$\bbP_0^{(n)}\left(|Zj|\leq C\sqrt{n\log(nD)}\|w_j\| \text{ for every } j\in[D]\right) \geq 1 - \frac{c}{n},$$
    which proves the result.
\end{proof}

\subsubsection*{Proof of Theorem \ref{thm: spectral recovery}(ii)}

We prove the score decomposition
in \eqref{eq: structured score expansion}. We present the proof in steps.

\begin{proof}[Proof of decomposition \eqref{eq: structured score expansion}]
Note that $\hat{\eta}_i = \widehat{M}^\top e_i = \sqrt{n}U^\top e_i\in\bbR^d$ and $\eta_i = M_0^\top e_i\in \bbR^d$ and the target is to decompose $\delta_i = \hat{\eta}_i - R^\top \eta_i$ (where $e_i\in \bbR^n$ is the standard basis vector with $e_{ii}=1$). First, applying the operator $P-P_0$ directly to the true basis matrix $U_0$, we have
    $$(P-P_0)U_0 = UU^\top U_0 - U_0=(U\hat{R}^\top - U_0) + U(U^\top U_0 - \hat{R}^\top).$$
    Rearranging and isolating $U\hat{R}^\top - U_0$, and multiplying the entire equation from the right by $\hat{R}$, we get
    \begin{align}
        U-U_0\hat{R} = (P-P_0)U_0\hat{R} - U(U^\top U_0 \hat{R} - I_d).
    \end{align}

    Now, we substitute the decomposition $P-P_0=\cL(E_0)+\cR(E_0)$ from equation \eqref{eq: decomposition P-P0} to get
    $$U-U_0\hat{R} = \cL(E_0)U_0\hat{R} + \cR(E_0)U_0\hat{R} - U(U^\top U_0 \hat{R} - I_d).$$
    Now, we use  $P_0^\perp U_0 = 0$ and the form of $Y_0^+=V_0S_0^{-1}U_0^\top$ to get the following simplification for the linear term
    \begin{align*}
        \cL(E_0) U_0 &= [P_0^\perp E_0 Y_0^+ + (Y_0^+)^\top E_0^\top P_0^\perp] U_0 \\
        &= P_0^\perp E_0 Y_0^+ U_0 + 0 \\
        &= P_0^\perp E_0 V_0S_0^{-1}U_0^\top U_0 \\
        &= P_0^\perp E_0 V_0 S_0^{-1}.
    \end{align*}
    Plugging this back, we get
    $$U- U_0\hat{R} =P_0^\perp E_0 V_0 S_0^{-1}\hat{R} + \cR(E_0)U_0 \hat{R} - U(U^\top U_0 \hat{R} - I_d).$$

    Isolating the $i-$th row and scaling by $\sqrt{n}$ we get
    \begin{equation}
        \begin{aligned}\label{eq: decompose eta_hat 1}
            \hat{\eta}_i - \sqrt{n}\hat{R}^\top U_0^\top e_i &= \sqrt{n}(U-U_0\hat{R})^\top e_i \\
            &= \hat{R}^\top (\underbrace{\sqrt{n}S_0^{-1}V_0^\top E_0^\top P_0^\perp e_i}_{\xi_i}) + \hat{R}^\top (\underbrace{\sqrt{n}U_0^\top \cR(E_0)^\top e_i}_{\delta_{i,\rm rem}}) \\
            &+\qquad\qquad(\underbrace{-(\hat{R}^\top U_0^\top U - I_d)\hat{\eta}_i)}_{\zeta_i}.
        \end{aligned}
    \end{equation}
    Note that $\xi_i$ corresponds to the linear row noise term, $\delta_{i, \rm rem}$ corresponds to the higher-order subspace remainder term and $\zeta_i$, which corresponds to the in-subspace Procrustes distortion term, uses $\hat{\eta}_i = \sqrt{n}U^\top e_i$.

    Next, we rewrite the shift vector $\sqrt{n}\hat{R}^\top U_0^\top e_i$ in terms of our target true parameter coordinates $R^\top \eta_i$. We simplify as follows
    \begin{align*}
        R^\top \eta_i &= (\bar{V}_0\bar{R}\hat{R})^\top (M_0^\top e_i) \\
        &= \hat{R}^\top \bar{R}^\top \bar{V}_0^\top (\bar{U}_0 \bar{D}_0 \bar{V}_0^\top)^\top e_i \\
        &= \hat{R}^\top \bar{R}^\top \bar{D}_0\bar{U}_0^\top e_i \\
        &= \hat{R}^\top \bar{R}^\top \bar{D}_0(\bar{R} U_0^\top) e_i = \hat{R}^\top\bar{R}^\top\bar{D}_0\bar{R} U_0^\top e_i.
    \end{align*}

    From Equation \eqref{eq: decompose eta_hat 1} and the above, we have
    \begin{align*}
        \hat{\eta}_i - R^\top \eta_i &= \left(\sqrt{n}\hat{R}^\top U_0^\top e_i - R^\top \eta_i \right) + \hat{R}^\top \xi_i + \hat{R}^\top\delta_{i, \rm rem} + \zeta_i \\
        &=\underbrace{\hat{R}^\top\bar{R}^\top (\sqrt{n} I_d - \bar{D}_0) \bar{R}U_0^\top e_i}_{\omega_i} + \hat{R}^\top \xi_i + \hat{R}^\top\delta_{i, \rm rem} + \zeta_i.
    \end{align*}

    This gives the exact decomposition that we sought to prove, with $\Delta_i = \omega_i + \hat{R}^\top\delta_{i, \rm rem} + \zeta_i$.
\end{proof}

\subsubsection*{Bounding $\Delta_i$}
We prove the upper bound for the term $\Delta_i = \hat{R}^\top \delta_{i,\rm rem} + \zeta_i + \omega_i$, as obtained above. Towards this, we first refine the good event considered in the proof of Theorem \ref{thm: spectral recovery}(i), based on some additional concentrations we require for this part. In particular, we introduce the leave-one-out noise and observation matrices. The resulting refined event does not change the probability order, as we show below. Before proceeding, to simplify notations, let
$$a_{n,D}:=\frac{1}{\sqrt n} +\frac{1}{\sqrt D}, \qquad b_{n,D}:=\frac{\sqrt n+\sqrt D}{nD}, \qquad L_{n,D}:=\sqrt{\log (nD)}.$$

We first collect the uniform concentration and leave-one-out bounds used to control the remainder in the score expansion. Let $\cE_{n,D}^{(1)}$ denote the event constructed in the proof \ref{proof of thm 1(i)}. It satisfies $\bbP_0^{(n)}(\cE_{n,D}^{(1)})\geq 1 - C/n$ for some constant $C>0$. Among the conclusions already established on this event are
$$\max_{j\in[d]} |s_k(M_0) - \sqrt{n}|\lesssim \sqrt{\log n}, \qquad \|E_0\|\lesssim \sqrt n + \sqrt D,$$
and consequently,
\begin{align*}
    s_d(Y_0)\asymp \sqrt{nD}, &\qquad s_d(Y)\asymp \sqrt{nD} ,\\
    \|Y_0^+\| \lesssim (nD)^{-1/2}, &\qquad \|Y^+\|\lesssim (nD)^{-1/2},\\
    \|Y^+-Y_0^+\|\lesssim b_{n,D}, &\qquad \|Y_\perp\|\lesssim \sqrt n + \sqrt D.
\end{align*}
The proof of part (i) also established
$$\|P-P_0\|\lesssim a_{n,D}, \qquad \|\cR(E_0)\|\lesssim \frac1n + \frac1D.$$
The event $\cE_{n,D}^{(1)}$ additionally contains columnwise concentration and quadratic-form bounds used to prove the loading and noise variance conclusions in part (i). Although those bounds are not all needed directly below, we retain them by constructing the event for part (ii) as a refinement of $\cE_{n,D}^{(1)}.$

For each $i\in [n]$, define the leave-one-out noise and observation matrices
\begin{equation}\label{eq: def leave-one-out}
    E_0^{(i)}=(I_n - e_ie_i^\top) E_0, \qquad Y^{(i)}=Y_0 + E_0^{(i)} = Y - e_ie_i^\top E_0,
\end{equation}
where $e_i\in\bbR^n$ is the standard bases vector with $e_{ii}=1$. Thus, the $i$th row of the noise matrix is set equal to 0, while the signal matrix remains unchanged. Conditionally on $M_0$, $Y^{(i)}$ is independent of the $i$th noise vector $\epsilon_i = E_0^\top e_i$. Write the singular-value decomposition as
\begin{equation}
    Y^{(i)} = Y_d^{(i)}+Y_\perp^{(i)} = U^{(i)}S^{(i)}V^{(i)\top} + U_\perp^{(i)}S_\perp^{(i)}V_\perp^{(i)\top},
\end{equation}
where $Y_d^{(i)}$ contains the top $d$ singular components. Define
\begin{equation}
    Y^{(i)+} = V^{(i)} (S^{(i)})^{-1} U^{(i)\top}.
\end{equation}
Thus, $Y^{(i)+}$ is the Moore-Penrose pseudoinverse of the rank-$d$ truncation $Y_d^{(i)}$, rather than the pseudoinverse of the full matrix $Y^{(i)}$.

\begin{lemma}[Uniform bounds for score expansion]\label{lemma: uniform bounds for score expansion}
    Under the conditions of Theorem \ref{thm: spectral recovery}, there exists an event $\cE_{n,D}\subseteq \cE_{n,D}^{(1)}$ such that $\bbP_0^{(n)}(\cE_{n,D}^c) \leq C/n$ for some constant $C>0$ and the following statements hold \textit{simultaneously} on $\cE_{n,D}$.

    First, the population singular vectors and noise rows satisfy
    \begin{equation}\label{eq: good event incoherence and noise rows}
        \begin{aligned}
            \max_{i\in[n]} \|e_i^\top U_0\| &\lesssim \sqrt{\frac{\log n}{n}}, \\
            \max_{i\in[n]} \|\epsilon_i\| &\lesssim \sqrt D + \sqrt{\log n}, \\
            \|U_0^\top E_0\| &\lesssim \sqrt D.
        \end{aligned}
    \end{equation}
    Consequently,
    \begin{equation}
        \max_{i\in[n]} \|e_i^\top P_0^\perp E_0\| \lesssim \sqrt D + \sqrt{\log n}.
    \end{equation}
    Second, uniformly over $i\in[n]$, 
    \begin{equation}\label{eq: good event singular values of leave-one-out}
        \begin{aligned}
            s_d(Y^{(i)}) &\asymp \sqrt{nD}, \\
            \|Y^{(i)+}\| &\lesssim \frac{1}{\sqrt{nD}},\\
            \|Y_\perp^{(i)}\|&\lesssim \sqrt n + \sqrt D,\\
            \|Y^{(i)+} - Y_0^+\| &\lesssim b_{n,D},\\
            \|Y^+ - Y^{(i)+}\|&\lesssim b_{n,D}.
        \end{aligned}
    \end{equation}
    The corresponding singular subspaces satisfy
    \begin{equation}\label{eq: good event singular subspaces}
    \begin{aligned}
        \|\sin \Theta(U_0, U^{(i)})\| \vee \|\sin \Theta(V_0, V^{(i)})\| &\lesssim a_{n,D}, \\
        \|\sin \Theta(U, U^{(i)})\| \vee \|\sin \Theta(V, V^{(i)})\| &\lesssim a_{n,D}, \\
        \max_{i\in [n]} \|e_i^\top U^{(i)} \| &\lesssim \sqrt{\frac{\log n}{n}} + \frac{1}{\sqrt D}.
    \end{aligned}
    \end{equation}
    In particular,
    $$\max_{i\in[n]} \|Y^{(i)+} e_i\| \lesssim \frac{L_{n,D}}{n\sqrt{D}} + \frac{1}{D\sqrt n}.$$
    Finally, the following leave-one-out Gaussian concentration bounds hold simultaneously
    \begin{equation}\label{eq: good event gaussian bounds}
        \begin{aligned}
            \max_{i\in[n]} \|\epsilon_i^\top (Y^{(i)+} - Y_0^+) U_0\| &\lesssim L_{n,D} b_{n,D}, \\
            \max_{i\in[n]} \|\epsilon_i^\top Y^{(i)+} U_0\| &\lesssim \frac{L_{n,D}}{\sqrt{nD}}, \\
            \max_{i\in[n]} \left|\epsilon_i^\top Y^{(i)+} e_i \right| &\lesssim \frac{L_{n,D}^2}{n\sqrt{D}} + \frac{L_{n,D}}{D\sqrt n}.
        \end{aligned}
    \end{equation}
\end{lemma}

\begin{proof}
    The first rowwise bound follows from the singular value decomposition of $M_0$. Recall $M_0=\bar U_0 \bar D_0\bar V_0^\top$ and $\bar U_0 = U_0 \bar R^\top$. It follows that $U_0=M_0\bar V_0 \bar D_0^{-1} \bar R$. Therefore, $\|e_i^\top U_0\| \leq \|\eta_i\| \|\bar D_0^{-1}\|$. On $\cE_{n,D}^{(1)}$, $\|\bar D_0^{-1}\| \lesssim n^{-1/2}$. Since $d$ is fixed, and the $\eta_i$'s are uniformly sub-Gaussian, $\max_i \|\eta_i\| \lesssim \sqrt{\log n}$ with probability at least $1-C/n$. Hence, $\max_i \|e_i^\top U_0\| \lesssim \sqrt{(\log n)/n}$. For the noise rows, Gaussian norm concentration and a union bound gives $\max_i \|\epsilon_i\| \lesssim \sqrt D + \sqrt{\log n}$ with probability at least $1-C/n$. Next, conditional on $M_0$, the columns of $U_0^\top E_0$ are independent centered Gaussian vectors with covariance matrices $\sigma_{0j}^2 I_d$. Since $d$ is fixed and the noise variances are uniformly bounded, $\|U_0^\top E_0\|\lesssim \sqrt D + \sqrt{\log n} \lesssim \sqrt D$ with probability at least $1-C/n$, where the final inequality uses $\log n = o(D)$. Moreover,
    $e_i^\top P_0^\perp E_0 = e_i^\top E_0 - e_i^\top U_0U_0^\top E_0$. Thus,
    \begin{align*}
        \|e_i^\top P_0^\perp E_0\| &\leq \|\epsilon_i\| + \|e_i^\top U_0\|\|U_0^\top E_0\| \\
        &\lesssim \sqrt D + \sqrt{\log n} + \sqrt{\frac{\log n}{n}}\sqrt D \lesssim \sqrt D + \sqrt{\log n}.
    \end{align*}

    Next, we prove the leave-one-out spectral bounds. Since $I_n-e_ie_i^\top$ is an orthogonal projection, $\|E_0^{(i)}\|\leq \|E_0\| \lesssim \sqrt n + \sqrt D$ uniformly over $i$. Mirsky's theorem \ref{thm: mirsky} therefore gives
    $$\|Y_\perp^{(i)}\| = s_{d+1}(Y^{(i)})\leq \|E_0^{(i)}\| \lesssim \sqrt n + \sqrt D.$$
    Consequently,
    $$\|Y^{(i)+}\| = \frac{1}{s_d(Y^{(i)})}\lesssim \frac{1}{\sqrt{nD}}.$$
    Applying Lemma \ref{lemma: spectral stability truncated pseudo-inverse} to $Y^{(i)}=Y_0+E_0^{(i)}$ gives
    $$\|Y^{(i)+} - Y_0^+\|\lesssim \frac{\sqrt n+\sqrt D}{nD} = b_{n,D}.$$
    To compare $Y^+$ and $Y^{(i)+}$, note that 
    $$Y_d - Y_d^{(i)} = (Y-Y_\perp) - (Y^{(i)} - Y_\perp^{(i)}) = e_i\epsilon_i^\top -Y_\perp +Y_\perp^{(i)}.$$
    It follows that
    $$\|Y_d - Y_d^{(i)}\| \leq \|\epsilon_i\| + \|Y_\perp \| +\|Y_\perp^{(i)}\| \lesssim \sqrt n+ \sqrt D$$
    uniformly over $i$. Applying the pseudoinverse perturbation bound from Lemma \ref{lemma: spectral stability truncated pseudo-inverse} to the two rank-$d$ matrices $Y_d$ and $Y_d^{(i)}$, whose smallest nonzero singular values are both of order $\sqrt{nD}$, gives
    $$\|Y^+ - Y^{(i)+}\| \lesssim b_{n,D}.$$

    Wedin's theorem \ref{thm: wedin} applied to ($Y_0$ and $Y^{(i)}$), and ($Y_d$ and $Y_d^{(i)}$) respectively gives the first two bounds in \eqref{eq: good event singular subspaces}. For the rowwise leave-one-out singular vector bound, decompose
    $$e_i^\top U^{(i)} = (e_i^\top U_0)(U_0^\top U^{(i)}) + e_i^\top P_0^\perp U^{(i)}.$$
    Therefore,
    \begin{align*}
        \|e_i^\top U^{(i)}\| &\leq \|e_i^\top U_0\| + \|P_0^\perp U^{(i)}\| \\
        &\lesssim \sqrt{\frac{\log n}{n}} + a_{n,D} \lesssim \sqrt{\frac{\log n}{n}} + \frac{1}{\sqrt D},
    \end{align*}
    where the bound for the second term follows from the bound on the corresponding singular subspaces. It follows that
    \begin{align*}
        \|Y^{(i)+}e_i\| &= \|V^{(i)} (S^{(i)})^{-1} U^{(i)\top}e_i\| \\
        &\leq \|(S^{(i)})^{-1}\| \|U^{(i)\top} e_i\| \\
        &\lesssim \frac{1}{\sqrt{nD}}\left(\sqrt{\frac{\log n}{n}}+\frac{1}{\sqrt D}\right) \lesssim \frac{L_{n,D}}{n\sqrt D} + \frac{1}{D\sqrt n}.
    \end{align*}

    It remains to establish the three Gaussian leave-one-out bounds. We do this using conditional random thresholds. Let $\cF_i=\sigma(M_0, E_0^{(i)})$. If $A_i\in \bbR^{D\times q}$, for fixed $q$, is $\cF_i$-measurable, then conditional on $\cF_i$, $\epsilon_i^\top A_i \sim N_q(0, A_i^\top \Sigma_0 A_i)$. Because $\|\Sigma_0\|\leq C_\sigma$, Gaussian concentration gives, for a sufficiently large constant $C$,
    $$\bbP_0^{(n)}\left(\|\epsilon_i^\top A_i\| > C\sqrt{\log n} \|A_i\|_F\mid \cF_i\right) \leq C' n^{-3}.$$
    Integrating the conditional probability and taking a union bound over $i$ gives
    $$\bbP_0^{(n)}\left(\|\epsilon_i^\top A_i\|>C\sqrt{\log n} \|A_i\|_F \text{ for some }i \right) \leq C' n^{-2}.$$
    Crucially, the concentration event is first formulated with the random threshold $C\sqrt{\log n}\|A_i\|_F$. Only after unconditioning we intersect it with the high-probability event on which $\|A_i\|_F$ has a deterministic upper bound. Apply this result first with $A_i=(Y^{(i)+} - Y_0^+)U_0$. Since $d$ is fixed, $\|A_i\|_F \lesssim Y^{(i)+}-Y_0^+\|\lesssim b_{n,D}$. Hence, 
    $$\max_i \|\epsilon_i^\top (Y^{(i)+}-Y_0^+) U_0\| \lesssim L_{n,D} b_{n,D}.$$

    Next, take $A_i=Y^{(i)+}U_0$. Then, $\|A_i\|_F \lesssim \|Y^{(i)+}\|\lesssim 1/\sqrt{nD}$, which yields
    $$\max_i \|\epsilon_i^\top Y^{(i)+}U_0\|\lesssim \frac{L_{n,D}}{\sqrt{nD}}.$$
    Finally, take the $D-$dimensional vector $A_i = Y^{(i)+}e_i$. Using the previously established rowwise bound,
    \begin{align*}
        \max_i |\epsilon_i^\top Y^{(i)+}e_i| &\lesssim L_{n,D} \max_i \|Y^{(i)+} e_i\| \\
        &\lesssim \frac{L_{n,D}^2}{n\sqrt D} + \frac{L_{n,D}}{D\sqrt n}.
    \end{align*}
    Intersecting $\cE_{n,D}^{(1)}$ with the finitely many (not depending on $n,D$) rowwise concentration, leave-one-out spectral and conditional Gaussian events constructed above produces $\cE_{n,D}$. A union bound gives $\bbP_0^{(n)}(\cE_{n,D}^c)\leq C/n$. This proves the lemma.
\end{proof}

We work throughout on the event $\cE_{n,D}$ from the preceding lemma. 

\begin{lemma}[Bound for remainder term $\delta_{i,\rm rem}$]\label{lemma: bound delta_i}
    On the event $\cE_{n,D}$ constructed in Lemma~\ref{lemma: uniform bounds for score expansion}, we have 
    $$\max_{i\in [n]} \|\delta_{i,\rm rem}\| \lesssim \sqrt{\log(nD)} \left(\frac1n + \frac1D\right).$$
    Consequently, the bound holds with $\bbP_0^{(n)}$-probability at least $1-C/n$.
\end{lemma}

\begin{proof}
    To recall, we had
    $$\delta_{i, \rm rem} = \sqrt{n} U_0^\top \cR(E_0)^\top e_i$$
    where 
    $$\cR(E_0) = P_0^\perp E_0(Y^+-Y_0^+)P_0 + P_0(Y^+-Y_0^+)^\top E_0^\top P_0^\perp + P_0^\perp P P_0^\perp - P_0 P^\perp P_0.$$

    Using $P_0U_0=U_0$ and $P_0^\perp U_0=0$, we have
    $$e_i^\top \cR(E_0)U_0 = \underbrace{e_i^\top P_0^\perp E_0(Y^+ - Y_0^+)U_0}_{\text{Term I}} - \underbrace{e_i^\top U_0U_0^\top P^\perp U_0}_{\text{Term II}}.$$
    We bound the two terms uniformly over $i\in[n]$. The second term is easier to deal with, while for the first, we use a leave-one-out analysis.

    \medskip
    \noindent
    \textbf{Bounding Term II:} Firstly,
    $$\|e_i^\top U_0U_0^\top P^\perp U_0\| \leq \|e_i^\top U_0\|  \|U_0^\top P^\perp U_0\|.$$
    On $\cE_{n,D}$, the first term is bounded by $\sqrt{(\log n)/n}$ uniformly over $i$. For the second term, note that 
    $$U_0^\top P^\perp U_0 = I - U_0^\top P U_0$$
    The eigenvalues of $U_0^\top P U_0 = U_0^\top U U^\top U_0$ are precisely $\cos^2(\theta_i)$, where $\theta_i$ are principal angles between the subspaces. Hence, $U_0^\top P^\perp U_0$ has eigenvalues $\sin^2(\theta_i)$. Thus, $\|U_0^\top P^\perp U_0\| = \sin^2(\theta_{\max})$. But, $\|P-P_0\|=\sin(\theta_{\max})$. Combined with the fact that $\|P-P_0\|\lesssim 1/\sqrt{n} + 1/\sqrt{D}$ under $\cE_{n,D}$, we get
    $$\|U_0^\top P^\perp U_0\| \lesssim \frac1n + \frac1D.$$
    Hence, combining, we get the following bound for term II:
\begin{align}\label{eq: bound for Term II}
        \sqrt{n}\max_{i\in [n]} \|e_i^\top U_0U_0^\top P^\perp U_0\| \lesssim \sqrt{n} \sqrt{\frac{\log n}{n}}\left(\frac1n + \frac1D\right) = \sqrt{\log n}\left(\frac1n + \frac1D\right).
    \end{align}
    
    \medskip
    \noindent
    \textbf{Bounding Term I:} We decompose the term using $P_0^\perp = I - U_0U_0^\top$. This gives
    $$\text{Term I} = e_i^\top E_0(Y^+ - Y_0^+)U_0 - e_i^\top U_0U_0^\top E_0(Y^+-Y_0^+)U_0.$$
    The second term in the above display can be bounded using the bounds we have already established under $\cE_{n,D}$ \begin{equation}\label{eq: bound for second subterm in Term I}
    \begin{aligned}
        \sqrt{n}\|e_i^\top U_0U_0^\top E_0(Y^+-Y_0^+)U_0\| &\leq \sqrt{n}\|e_i^\top U_0\|\|U_0^\top E_0\| \|Y^+-Y_0^+\| \\
        &\lesssim \sqrt{n}\sqrt{\frac{\log n}{n}}  \sqrt{D} \frac{\sqrt{n}+\sqrt{D}}{nD} \\
        &\lesssim \sqrt{\log n}\left(\frac1n + \frac1D\right),
    \end{aligned}
    \end{equation}
    where we used $1/\sqrt{nD} \lesssim 1/n+1/D$.

    Now, we are only left with the first subterm from Term I. This term $e_i^\top E_0(Y^+-Y_0^+)U_0$ contains direct statistical dependence between the $i-$th row of noise $\epsilon_i^\top=e_i^\top E_0$ and the full pseudo-inverse $Y^+$. We break this dependence by a leave-one-out analysis. By inserting $Y^{(i)+}$, we expand the target term as follows
    $$e_i^\top E_0(Y^+-Y_0^+)U_0 = \underbrace{e_i^\top E_0 (Y^{(i)+} - Y_0^+)U_0}_{T_{i,b}} + \underbrace{e_i^\top E_0(Y^+ - Y^{(i)+})U_0}_{T_{i,a}}.$$

    The first Gaussian bound \eqref{eq: good event gaussian bounds} in Lemma \ref{lemma: uniform bounds for score expansion} gives $\max_i \|T_{i,b}\|\lesssim L_{n,D} b_{n,D}$. Consequently,
    \begin{equation}\label{eq: bound for T_ib}
    \begin{aligned}
        \sqrt n \max_i \|T_{i,b}\| &\lesssim L_{n,D} \sqrt n \frac{\sqrt n + \sqrt D}{nD} \\
        &= L_{n,D} \left(\frac1D + \frac{1}{\sqrt{nD}}\right) \lesssim L_{n,D}\left(\frac1n + \frac1D\right).
    \end{aligned}
    \end{equation}
    It remains to bound $T_{i,a}$.

   \medskip
    \noindent
    \textbf{Bounding $T_{i,a}$ -- Decomposition:} Now we focus on $T_{i,a}=e_i^\top E_0(Y^+ - Y^{(i)+})U_0$. Let $E_a=Y_d - Y_d^{(i)}$, the difference of the rank$-d$ matrices whose pseudo-inverses are involved in the term. Using Lemma \ref{lemma: wedin identity}, we get the identity
    $$Y^+ - Y^{(i)+} = -Y^+E_aY^{(i)+} + Y^+(Y^+)^\top E_a^\top(I - Y_d^{(i)}Y^{(i)+}) + (I-Y^+Y_d)E_a^\top (Y^{(i)+})^\top Y^{(i)+}.$$
    Accordingly, this splits $T_{i,a}$ into three terms
    \begin{equation}\label{eq: 2.34}
        \begin{aligned}
            T_{i,a} &= T_{i,a}^{(1)} + T_{i,a}^{(2)} + T_{i,a}^{(3)} \\
            T_{i,a}^{(1)} &= -e_i^\top E_0 Y^+ E_a Y^{(i)+}U_0 \\
            T_{i,a}^{(2)} &= e_i^\top E_0 Y^+(Y^+)^\top E_a^\top (I - Y_d^{(i)} Y^{(i)+}) U_0 \\
            T_{i,a}^{(3)} &= e_i^\top E_0 (I - Y^+Y_d)E_a^\top (Y^{(i)+})^\top Y^{(i)+} U_0.
        \end{aligned}
    \end{equation}
    Expressing $E_a$ via the un-truncated matrices
    \begin{equation}\label{eq: expand E_a}
        E_a = (Y-Y_\perp) - (Y^{(i)} - Y^{(i)}_\perp) = e_ie_i^\top E_0 -Y_\perp + Y_\perp^{(i)}.
    \end{equation}

    \medskip
    \noindent
    \textbf{Bound for $T_{i,a}^{(1)}$:} Because $Y^+Y_\perp = 0$ and $Y_\perp^{(i)}Y^{(i)+}=0$, we have $-Y^+E_a Y^{(i)+}=-Y^+ e_i \epsilon_i^\top Y^{(i)+}$ using the expression \eqref{eq: expand E_a}. Thus, $T_{i,a}^{(1)}=-(\epsilon_i^\top Y^+ e_i)(\epsilon_i^\top Y^{(i)+} U_0)$. By Lemma \ref{lemma: uniform bounds for score expansion}, Equation \eqref{eq: good event gaussian bounds} (second bound), we have
    $$\max_i \|\epsilon_i^\top Y^{(i)+} U_0\|\lesssim L_{n,D}/\sqrt{nD}.$$
    For the first factor, write
    $$\epsilon_i^\top Y^+ e_i = \epsilon_i^\top Y^{(i)+}e_i + \epsilon_i^\top (Y^+ - Y^{(i)+}) e_i.$$
    Consequently, using Lemma \ref{lemma: uniform bounds for score expansion}, equations \eqref{eq: good event gaussian bounds}, we have
    \begin{align*}
        \max_i |\epsilon_i^\top Y^+ e_i| &\leq \max_i |\epsilon_i^\top Y^{(i)+}e_i| + \max_i |\epsilon_i^\top (Y^+ - Y^{(i)+}) e_i| \\
        &\lesssim \max_i |\epsilon_i^\top Y^{(i)+}e_i| + \max_i \|\epsilon_i \| \|Y^+ - Y^{(i)+}\| \\
        &\lesssim \left[\frac{L_{n,D}^2}{n\sqrt D} + \frac{L_{n,D}}{D\sqrt n}\right] + (\sqrt D + \sqrt{\log n}) b_{n,D} \\
        &= \left[\frac{L_{n,D}^2}{n\sqrt D} + \frac{L_{n,D}}{D\sqrt n}\right] + \left[\frac{1}{\sqrt{nD}} + \frac1n\right],
    \end{align*}
    using $\log n = o(D)$. It follows that 
    $$\max_i \|T_{i,a}^{(1)}\| \lesssim \frac{L_{n,D}}{\sqrt{nD}}\left[\frac{L_{n,D}^2}{n\sqrt D} + \frac{L_{n,D}}{D\sqrt n} + \frac{1}{\sqrt{nD}} + \frac1n\right].$$
    Using $\log (nD)\lesssim n$ and $\log nD = o(D)$, each term of the right, after multiplication by $\sqrt{n}$ is bounded by a constant multiple of $L_{n,D}(1/n + 1/D)$. Hence,
    \begin{equation}\label{eq: bound T_ia^1}
        \sqrt{n} \max_i \|T_{i,a}^{(1)}\| \lesssim L_{n,D}\left(\frac1n + \frac1D\right).
    \end{equation}

    \medskip
    \noindent
    \textbf{Bound for $T_{i,a}^{(2)}$:} Let $Q_i = I_n - U^{(i)}U^{(i)\top}$. Since $Y^+ (Y^+)^\top = V S^{-2} V^\top$ and $I-Y_d^{(i)} Y^{(i)+} = Q_i$, we have
    \begin{equation}
        T_{i,a}^{(2)} = \epsilon_i^\top V S^{-2} V^\top \left[\epsilon_i e_i^\top - Y_\perp^\top + (Y_\perp^{(i)})^\top\right]Q_i U_0 = \underbrace{(\epsilon_i^\top V S^{-2} V^\top \epsilon_i) e_i^\top Q_i U_0}_{T_{i,a}^{(2,1)}} + \underbrace{\epsilon_i^\top VS^{-2} V^\top (Y_\perp^{(i)})^\top Q_i U_0}_{T_{i,a}^{(2,2)}},
    \end{equation}
    since $V^\top Y_\perp^\top = 0$ in the second term. For the first term $T_{i,a}^{(2,1)}$, we have
    \begin{align*}
        \|(\epsilon_i^\top V S^{-2} V^\top \epsilon_i) e_i^\top Q_i U_0\| &\leq |\epsilon_i^\top V S^{-2} V^\top \epsilon_i| \|e_i^\top Q_i U_0\| \\
        &\leq \left(\|S^{-1}\|^2 \|\epsilon_i\|^2\right) \left(\|e_i^\top U_0\| + \|e_i^\top U^{(i)}\| \|U^{(i)\top} U_0\|\right) \\
        &\lesssim \frac1n \left(\frac{L_{n,D}}{\sqrt n} + \frac{1}{\sqrt{D}}\right),
    \end{align*}
    using the bounds that hold under $\cE_{n,D}$, established in Lemma \ref{lemma: uniform bounds for score expansion}, noting that these bounds are uniform in $i$. Therefore,
    $$\sqrt{n} \max_i \|T_{i,a}^{(2,1)}\| \lesssim \frac{L_{n,D}}{n} + \frac{1}{\sqrt{nD}} \lesssim L_{n,D} \left(\frac1n + \frac1D\right).$$

    For the second term $T_{i,a}^{(2,2)}$, the identity $(Y_\perp^{(i)})^\top Q_i = (Y_\perp^{(i)})^\top$ gives 
    $$T_{i,a}^{(2,2)} = \epsilon_i^\top V S^{-2} (Y_\perp^{(i)}V)^\top U_0.$$
    Furthermore, on $\cE_{n,D}$,
    \begin{align*}
        \|Y_\perp^{(i)}V\| &\leq \|S_\perp^{(i)}\| \|V_\perp^{(i)\top} V\| \\
        &= \|S_\perp^{(i)}\| \|\sin \Theta(V, V^{(i)})\| \\
        &\lesssim (\sqrt n + \sqrt D) a_{n,D} = \frac{(\sqrt n+\sqrt D)^2}{\sqrt{nD}}.
    \end{align*}
    It follows that
    \begin{align*}
        \|T_{i,a}^{(2,2)}\| &\leq \|\epsilon_i \| \|S^{-2}\| \|Y_\perp^{(i)}V\| \\
        &\lesssim \sqrt D \frac{1}{nD} \frac{(\sqrt n + \sqrt D)^2}{\sqrt{nD}} \lesssim \frac{1}{\sqrt{n}}(1/n + 1/D),
    \end{align*}
    using bounds that hold under $\cE_{n,D}$. Consequently, combining the bounds for the two parts,
    \begin{equation}\label{eq: bound T_ia^2}
        \sqrt{n} \max_i \|T_{i,a}^{(2)}\| \lesssim L_{n,D}\left(\frac1n + \frac1D\right).
    \end{equation}

    \medskip
    \noindent
    \textbf{Bound for $T_{i,a}^{(3)}$:} Let $Q_V = I_D - Y^+ Y_d = I_D - VV^\top$. Then,
    \begin{equation}
    \begin{aligned}
        T_{i,a}^{(3)} &= e_i^\top E_0 (I - Y^+Y_d)E_a^\top (Y^{(i)+})^\top Y^{(i)+} U_0 \\
        &= \epsilon_i^\top Q_V \left[\epsilon_i e_i^\top - Y_\perp^\top + (Y_\perp^{(i)})^\top\right] (Y^{(i)+})^\top Y^{(i)+} U_0\\
        &= \underbrace{(\epsilon_i^\top Q_V \epsilon_i)e_i^\top (Y^{(i)+})^\top Y^{(i)+} U_0}_{T_{i,a}^{(3,1)}} + \underbrace{\left(-\epsilon_i^\top Y_\perp^\top (Y^{(i)+})^\top Y^{(i)+} U_0\right)}_{T_{i,a}^{(3,2)}}
    \end{aligned}
    \end{equation}
    where the third equality uses $(Y^{(i)}_\perp)^\top (Y^{(i)+})^\top = 0$ and $Q_V Y_\perp^\top =Y_\perp^\top$. For the term $T_{i,a}^{(3,1)}$, $|\epsilon_i^\top Q_V \epsilon_i| \leq \|\epsilon_i\|^2 \lesssim D$. Thus, the bounds on $\cE_{n,D}$ by Lemma \ref{lemma: uniform bounds for score expansion} give
    \begin{align*}
        \|T_{i,a}^{(3,1)}\| &\leq D \|Y^{(i)+}e_i\| \|Y^{(i)+}U_0\| \\
        &\leq D \left(\frac{L_{n,D}}{n\sqrt D} + \frac{1}{D\sqrt{n}}\right) \frac{1}{\sqrt{nD}} \\
        &\lesssim \frac{L_{n,D}}{n\sqrt n} + \frac{1}{n\sqrt{D}}.
    \end{align*}
    Consequently, 
    $$\sqrt{n} \max_i \|T_{i,a}^{(3,1)}\| \lesssim \frac{L_{n,D}}{n} + \frac{1}{\sqrt{nD}}\lesssim L_{n,D} \left(\frac1n + \frac1D\right).$$

    For $T_{i,a}^{(3,2)}$, we first note the following bound on $\cE_{n,D}$:
    \begin{align*}
        \|Y_\perp^\top (Y^{(i)+})^\top\| &\leq \|S_\perp\| \|U_\perp^\top U_i\| \|(S^{(i)})^{-1}\| \\
        &= \|S_\perp\| \|\sin \Theta(U, U^{(i)})\| \|(S^{(i)})^{-1}\| \\
        &\lesssim (\sqrt n + \sqrt D) a_{n,D} \frac{1}{\sqrt{nD}} \lesssim \frac1n + \frac1D.
    \end{align*}
    Hence, on $\cE_{n,D}$, we have
    \begin{align*}
        \max_i \sqrt{n}\|T_{i,a}^{(3,2)} \| &\leq \sqrt{n} \max_i \|\epsilon_i\| \|Y_\perp^\top (Y^{(i)+})^\top \| \|Y^{(i)+} U_0\| \\
        &\lesssim  \sqrt n\sqrt D (1/n + 1/D) \frac{1}{\sqrt{nD}} \lesssim \frac1n + \frac1D.
    \end{align*}
    Thus, combining the two parts above, we have
    \begin{equation}\label{eq: bound T_ia^3}
        \sqrt{n} \max_i \|T_{i,a}^{(3)}\| \lesssim L_{n,D}\left(\frac1n + \frac1D\right).
    \end{equation}

    Hence, combining \eqref{eq: bound T_ia^1}, \eqref{eq: bound T_ia^2}, \eqref{eq: bound T_ia^3}, we have $\sqrt{n} \max_i \|T_{i,a}\| \lesssim L_{n,D} (1/n + 1/D)$. Finally, combining this with \eqref{eq: bound for Term II}, \eqref{eq: bound for T_ib}, we obtain the desired result.
\end{proof}

\begin{lemma}\label{lemma: bound zeta and omega}
    On $\cE_{n,D}$ constructed in Lemma \ref{lemma: uniform bounds for score expansion}, we have
    \begin{align*}
        \max_{i\in [n]} \|\zeta_i\| &\lesssim \sqrt{\log n} \left(\frac1n + \frac1D\right) \\
        \max_{i\in [n]} \|\omega_i\| &\lesssim \frac{\log n}{\sqrt{n}}.
    \end{align*}
    Consequently, both bounds hold simultaneously with $\bbP_0^{(n)}-$probability at least $1-C/n$.
\end{lemma}

\begin{proof}
    We work throughout on $\cE_{n,D}$.

    \medskip
    \noindent
    \textbf{Bounding $\zeta$:} Recall $\zeta_i=-(\widehat R^\top U_0^\top U - I_d) \widehat \eta_i$. Let $U_0^\top U = W \Gamma \Omega^\top$ be the singular value decomposition, where $\Gamma=\diag(\cos\theta_1,\dots,\cos\theta_d)$ and $\theta_1,\dots,\theta_d$ are the principal angles between the column spaces of $U_0$ and $U$. By the definition of the Procrustes rotation, $\widehat R = W \Omega^\top$. Therefore, $\widehat R^\top U_0^\top U = \Omega \Gamma \Omega^\top$ and hence, $\|\zeta_i\|=\|(\Omega\Gamma\Omega^\top - I_d)\widehat{\eta}_i\|\leq \|\Gamma-I_d\|\|\Omega^\top\widehat \eta_i\|$. We bound these two factors separately.

    For the first factor, since principal angles belong to $[0,\pi/2]$, $1-\cos\theta_k = \sin^2 \theta_k /(1+\cos \theta_k)\le \sin^2 \theta_k$. It follows that $\|\Gamma-I_d\|=\max_{k\in[d]} (1-\cos\theta_k)\leq \|\sin \Theta(U_0,U)\|^2 \lesssim 1/n + 1/D$ on $\cE_{n,D}$. 

    Next, since $\widehat \eta_i = \sqrt n U^\top e_i$ and $U_0^\top U\Omega = W\Gamma$, we have
    \begin{align*}
        U\Omega = (P_0U + P_0^\perp U) \Omega =U_0U_0^\top U\Omega + P_0^\perp U \Omega = U_0W\Gamma+P_0^\perp U \Omega.
    \end{align*}
    It follows that
    $$\Omega^\top \widehat \eta_i = \sqrt n \Gamma W^\top U_0^\top e_i + \sqrt n \Omega^\top U^\top P_0^\perp e_i.$$
    For the first term, Lemma \ref{lemma: uniform bounds for score expansion} gives
    \begin{align*}
        \sqrt n\|\Gamma W^\top U_0^\top e_i\| &\leq \sqrt n \|\Gamma\| \|U_0^\top e_i\| \\
        &\lesssim \sqrt n \sqrt{\frac{\log n}{n}} = \sqrt{\log n}
    \end{align*}
    uniformly over $i$. For the second term, the singular value decomposition of $Y$ gives $U=YVS^{-1}$. Since $P_0^\perp Y_0=0$, $P_0^\perp U = P_0^\perp E_0 V S^{-1}$. Therefore
    \begin{align*}
        \sqrt n \|\Omega^\top U^\top P_0^\perp e_i\| &\leq \sqrt n \|S^{-1}\| \|e_i^\top P_0^\perp E_0 \| \\
        &\lesssim \sqrt n \frac{1}{\sqrt{nD}} (\sqrt D + \sqrt{\log n}) =1 + \sqrt{\frac{\log n}{D} }\lesssim 1,
    \end{align*}
    using the bounds on $\cE_{n,D}$ from Lemma \ref{lemma: uniform bounds for score expansion}, and the last inequality using $\log n =o(D)$. Combining the two components gives $\max_i \|\Omega^\top \widehat\eta_i\|\lesssim \sqrt{\log n}$. Therefore,
    $$\max_{i\in [n]} \|\zeta_i\| \lesssim \sqrt{\log n} \left(\frac1n + \frac1D\right).$$

    \medskip
    \noindent
    \textbf{Bounding $\omega$:} Recall $\omega_i = \widehat R^\top \bar R^\top (\sqrt n I_d - \bar D_0)\bar R U_0^\top e_i$, where $\bar D_0$ contains the singular values of $M_0$. Since $\widehat R$ and $\bar R$ are orthogonal,
    $$\|\omega_i\| \leq \|\sqrt n I_d - \bar D_0\| \|U_0^\top e_i\|.$$
    On $\cE_{n,D}$, the singular value concentration inherited from the event constructed in the proof of Theorem \ref{thm: spectral recovery}(i) gives 
    $$\|\sqrt n I_d - \bar D_0\| = \max_{k\in [d]} |\sqrt n - s_k(M_0)| \lesssim \sqrt{\log n}.$$
    Lemma \ref{lemma: uniform bounds for score expansion} also gives $\max_i \|U_0^\top e_i\| \lesssim \sqrt{(\log n)/n}$. It follows that 
    $$\max_{i\in[n]} \|\omega_i\| \lesssim \sqrt{\log n} \sqrt{\frac{\log n}{n}} = \frac{\log n}{\sqrt n}.$$
\end{proof}

Finally, recall $\Delta_i = \zeta_i + \hat{R}^\top \delta_{i,\rm rem} + \omega_i$, so that $\delta_i = \hat{R}^\top \xi_i + \Delta_i$.
Hence, we get the following bound.
$$\max_i \|\Delta_i\| \leq \max_i \|\zeta_i\| + \max_i \|\delta_{i,\rm rem}\| + \max_i \|\omega_i\|.$$
Plugging in the bounds from Lemma \ref{lemma: bound delta_i} and Lemma \ref{lemma: bound zeta and omega}, we get
\begin{equation}
    \begin{aligned}
        \max_{i\in [n]} \|\Delta_i\| &\lesssim \sqrt{\log (nD)}\left(\frac1n + \frac1D\right) + \sqrt{\log n}\left(\frac1n + \frac1D\right) + \frac{\log n}{\sqrt{n}}\\
        &\lesssim\frac{\sqrt{\log (nD)}}{D} + \frac{\log n}{\sqrt{n}}.
    \end{aligned}
\end{equation}
using the assumed $\log(nD)\lesssim n$. The above uniform bound holds on a single event $\cE_{n,D}$ (see Lemma \ref{lemma: uniform bounds for score expansion}) such that $\bbP_0^{(n)}(\cE_{n,D}^c)\leq C/n$ for some constant $C>0$. This proves the uniform upper bound on $\|\Delta_i\|$ in \eqref{eq: structured score expansion} in Theorem \ref{thm: spectral recovery}(ii). Since the $\cE_{n,D}$ used here is a refinement of the event used in the proof of Theorem \ref{thm: spectral recovery}(i), the current $\cE_{n,D}$ can be used in part (i) as well. Thus, all conclusions in the theorem are valid under a single high-probability event $\cE_{n,D}$, as required.

\subsubsection*{Proof of Corollary \ref{cor: expected factor error}}

\begin{proof}
    \textbf{Part A: } (Control for $\Lambda$ part) From the proof (step 5) of Theorem \ref{thm: spectral recovery}(i), we have 
    $$\|\hat{\Lambda} - \Lambda_0 R\|_F \leq A_1 + A_2 + B_1 + B_2,$$
    see Equations \eqref{eq: point estimate Lambda decompose}, \eqref{eq: point estimate Lambda decompose A} and \eqref{eq: point estimate Lambda decompose B}. We bound the four terms separately.

    For $A_1$, exactly as in the proof of Theorem \ref{thm: spectral recovery}(i), 
    $$A_1 \lesssim \sqrt{D}(1/n + 1/D) \text{ on } \cE_{n,D} \;\Longrightarrow\; \bbE_0[A_1^2 \mid \cE_{n,D}] \lesssim D(1/n+1/D)^2.$$

    For $A_2$, following the proof of Theorem \ref{thm: spectral recovery}(i), we get
    $$A_2 \leq \|\Lambda_0\|_F \left\|\frac{\bar{D}_0}{\sqrt{n}} - I_d\right\| \lesssim \sqrt{D} \left\|\frac{\bar{D}_0}{\sqrt{n}} - I_d\right\|.$$
    Since the rows of $M_0$ are iid sub-Gaussian, centered with covariance $I_d$, and $d$ is fixed, the usual sample covariance moment bound gives
    $$\bbE_0\|(1/n)M_0^\top M_0 - I_d\|^2 \lesssim 1/n \Longrightarrow \bbE_0[A_2^2] \lesssim D  \frac1n.$$
    Since $\bbP_0(\cE_{n,D})\geq 1 - c/n$, 
    $$\bbE_0[A_2^2 \mid \cE_{n,D}] \leq \frac{\bbE_0[A_2^2]}{\bbP_0(\cE_{n,D})}\lesssim \frac{D}{n}.$$
    Thus,
    $$\bbE_0[(A_1+A_2)^2 \mid \cE_{n,D}] \lesssim D\left(\frac1D + \frac1n\right)^2 + \frac{D}{n} \lesssim \frac{D}{n} + \frac1D.$$

    Now, for $B_1$, we had
    $$B_1 \leq \frac{1}{\sqrt{n}} \|E_0^\top U_0\|_F$$
    since $\|U_0^\top U\|\leq 1$. Conditional on $M_0$, $U_0^\top \epsilon^{(j)}\sim N_d(0, \sigma_{0j}^2 I_d)$ and hence $\bbE_0[\|U_0^\top \epsilon^{(j)}\|^2 \mid M_0] = d\sigma_{0j}^2$. Since the true noise variances are uniformly bounded, $\bbE\|E_0^\top U_0\|_F^2 = \sum_{j=1}^D \bbE\|U_0^\top \epsilon^{(j)}\|^2 \lesssim D$. This gives
    $$\bbE_0[B_1^2 \mid \cE_{n,D}] \lesssim \frac{D}{n}.$$

    For $B_2$, we already had the pointwise bound $B_2\lesssim \sqrt{D}(1/n+1/D)$ on $\cE_{n,D}$ and hence 
    $$\bbE_0[B_2^2 \mid \cE_{n,D}] \lesssim D\left(\frac1n + \frac1D\right)^2 \lesssim \frac{D}{n^2} + \frac1D.$$

    Combining all the parts and applying Jensen's inequality, we have
    $$\bbE_0\left[\|\hat{\Lambda} - \Lambda_0 R\|_F\mid \cE_{n,D}\right] \lesssim \sqrt{\frac{D}{n}} + \frac{1}{\sqrt{D}}.$$

    \medskip\noindent
    \textbf{Part B: } (Control for $\Sigma$ part) As in Step 3(b) in the proof of Theorem \ref{thm: spectral recovery}(i), we have
    $$|\widehat{\sigma}_j^2 - \sigma_{0j}^2| \leq Q_j + |T_{j,1}| + |T_{j,2}| + |T_{j,3}|, \quad Q_j = \left|\frac1n \epsilon^{(j)\top} P_0^\perp \epsilon^{(j)} - \sigma_{0j}^2 \right|,$$
    see Equations \eqref{eq: sigma error} and \eqref{eq: T decomposition}. 

    For $Q_j$, note that conditional on $M_0$, $P_0^\perp$ is fixed with rank $n-d$ and independent of $\epsilon^{(j)}$. Therefore, $\epsilon^{(j)\top}P_0^\perp \epsilon^{(j)} / \sigma_{0j}^2 \sim \chi^2_{n-d}$. Consequently,
    $$\bbE_0\left[\frac1n \epsilon^{(j)\top}P_0^\perp \epsilon^{(j)} - \sigma_{0j}^2 \mid M_0\right]=-\frac{d\sigma_{0j}^2}{n}, \; \Var_0\left(\frac1n \epsilon^{(j)\top}P_0^\perp \epsilon^{(j)} - \sigma_{0j}^2 \mid M_0\right)=\frac{2(n-d)\sigma_{0j}^4}{n^2}.$$

    Since $d$ is fixed and the noise variances are uniformly bounded, $\bbE_0[Q_j^2] \lesssim 1/n$, and hence, $\bbE_0[\sum_j Q_j^2]\lesssim D/n$. Since $\bbP_0(\cE_{n,D})\geq 1-c/n$, we have 
    $$\bbE_0\left[\sum_{j=1}^D Q_j^2 \mid \cE_{n,D}\right] \lesssim \frac{D}{n}.$$

    For $T_{j,1}$, the proof of Theorem \ref{thm: spectral recovery}(i) shows $|T_{j,1}|\lesssim (1/n+1/D)\|\lambda_{0j}\|^2$ uniformly over $j$ on $\cE_{n,D}$. Thus,
    $$\sum_j T_{j,1}^2 \lesssim \left(\frac1n + \frac1D\right)^2 \sum_{j=1}^D \|\lambda_{0j}\|^4.$$
    Under the loading assumptions, $\max_j \|\lambda_{0j}\|^2 \lesssim \ell_D$ and $\sum_j \|\lambda_{0j}\|^2 = \|\Lambda_0\|_F^2 \lesssim D$. Hence, $\sum_j \|\lambda_{0j}\|^4 \lesssim D\ell_D$. This gives
    $$\bbE_0\left[ \sum_{j=1}^D T_{j,1}^2 \mid \cE_{n,D}\right] \lesssim D \ell_D\left(\frac{1}{n^2} + \frac{1}{D^2}\right).$$

    For $T_{j,3}$, by \eqref{eq: sigma T_j3 decompose}, we have
    $$|T_{j,3}| \leq \frac2n \|U_0^\top \epsilon^{(j)}\| \|E_U\| \|\epsilon^{(j)}\| + \frac1n \|E_U\|^2 \|\epsilon^{(j)}\|^2.$$
    On $\cE_{n,D}$, $\|E_U\|\lesssim 1/\sqrt{n} + 1/\sqrt{D}$. Hence
    $$T_{j,3}^2 \lesssim \frac{1}{n^2}\left(\frac1n + \frac1D\right) \|U_0^\top \epsilon^{(j)}\|^2 \|\epsilon^{(j)}\|^2 + \frac{1}{n^2} \left(\frac1n + \frac1D\right)^2 \|\epsilon^{(j)}\|^4, \quad \text{on } \cE_{n,D}.$$
    Since $U_0$ is independent of $E_0$, $d$ is fixed, and the noise variances are uniformly bounded,
    $$\bbE_0[\|U_0^\top \epsilon^{(j)}\|^2 \|\epsilon^{(j)}\|^2\mid M_0] \lesssim n, \quad \bbE_0\|\epsilon^{(j)}\|^4 \lesssim n^2.$$
    Using this in the last display, and summing over $j$,
    $$\bbE_0\left[\sum_{j=1}^D T_{j,3}^2 \mid \cE_{n,D}\right] \lesssim \frac{D}{n}(1/n+1/D) + D(1/n + 1/D)^2 \lesssim \frac{D}{n} + \frac1D.$$

    Finally, for $T_{j,2}$ use the decomposition $T_{j,2}=T_{j,2}^{\rm linear} + T_{j,2}^{\rm rem}$ from \eqref{eq: sigma T_j2 decompose}. For the remainder, Equation \eqref{eq: bound remainder P-P0 decomposition} gives $\|\cR(E_0)\|\lesssim 1/n + 1/D$ on $\cE_{n,D}$. Using $\|M_0\lambda_{0j}\|\lesssim \sqrt{n} \|\lambda_{0j}\|$ on $\cE_{n,D}$, we have $|T_{j,2}^{\rm rem}|\lesssim (1/n+1/D) \|\lambda_{0j}\| \|\epsilon^{(j)}\|/\sqrt{n}$. Thus, $\bbE_0[(T_{j,2}^{\rm rem})^2 \ind(\cE_{n,D})] \lesssim (1/n+1/D)^2 \|\lambda_{0j}\|^2$. Summing over $j$, we get
    $$\bbE_0\left[\sum_{j=1}^D (T_{j,2}^{\rm rem})^2 \mid \cE_{n,D}\right] \lesssim D\left(\frac{1}{n^2} + \frac{1}{D^2}\right).$$

    For the linear part, the argument in the proof of Theorem \ref{thm: spectral recovery}(i) gives
    $$T_{j,2}^{\rm linear} = -\frac2n \sum_{k=1}^D w_{jk} \epsilon^{(k)\top} P_0^\perp \epsilon^{(j)}, \quad \|w_j\|^2\lesssim \|\lambda_{0j}\|^2/D, \;|w_{jj}| \lesssim \|\lambda_{0j}\|^2/D.$$

    For the diagonal part, $L_{j,\rm diag} = -(2/n)w_{jj}\epsilon^{(j)\top}P_0^\perp \epsilon^{(j)}.$
    Since $\bbE_0[((1/n)\epsilon^{(j)}P_0^\perp \epsilon^{(j)})^2]\lesssim 1$, we obtain $\bE_0 L_{j,\rm diag}^2 \lesssim \|\lambda_{0j}\|^4/D^2$. Therefore,
    $$\sum_{j=1}^D \bbE_0[L_{j,\rm diag}^2] \lesssim \frac{1}{D^2} \sum_{j=1}^D \|\lambda_{0j}\|^4 \lesssim \frac{\ell_D}{D}.$$

    For the off-diagonal part, $L_{j, \rm off} = -(2/n) \sum_{k\neq j} w_{jk} \epsilon^{(k)\top}P_0^\perp \epsilon^{(j)}.$ Conditional on $M_0$ and $\epsilon^{(j)}$, the variables $\epsilon^{(k)\top}P_0^\perp \epsilon^{(j)}$ ($k\neq j$) are independent centered Gaussian variables. Hence
    $$\bbE_0[L_{j,\rm off}^2 \mid M_0, \epsilon^{(j)}] \lesssim \frac{1}{n^2} \sum_{k\neq j} w_{jk}^2 \|P_0^\perp \epsilon^{(j)}\|^2 \lesssim \frac{\|w_j\|^2}{n^2} \|P_0^\perp \epsilon^{(j)}\|^2.$$
    Taking expectation over $\epsilon^{(j)}$, $\bbE_0[L_{j,\rm off}^2] \lesssim \|w_j\|^2/n \lesssim \|\lambda_{0j}\|^2 / (nD)$. Summing over $j$,
    $$\sum_{j=1}^D \bbE_0[L_{j,\rm off}^2] \lesssim \frac{1}{nD} \sum_{j=1}^D \|\lambda_{0j}\|^2 \lesssim \frac1n.$$

    Combining all the bounds from the different parts, simplifying and using $\ell_D \lesssim n$ (since $\ell_D\lesssim \sqrt{n \log(nD)}$ and $\log(nD)\lesssim n$), we have
    $$\bbE_0\left[\|\hat{\Sigma} - \Sigma_0\|_F^2 \mid \cE_{n,D} \right] \lesssim \frac{D}{n} + \frac{\ell_D}{D}.$$
    Application of Jensen's inequality gives the result we set to prove.
\end{proof}

\subsection{Proofs for Section \ref{subsec: latent density upper bound} (KDE upper bound)}

Before the proof of the main result Theorem \ref{thm: new kde}, we establish a weaker upper bound, under weaker assumptions on the kernel. Here, we only need $K\in C^1$ with $\|\nabla K\|_{L^2}<\infty$. 

\begin{theorem}[Weak latent density upper bound]\label{thm: weak kde g0}
    Suppose we construct estimator $\hat{g}_h$ using a kernel $K$ with bandwidth $h$ from the pre-estimated factors in $\widehat{M}$. Assuming \eqref{assume: dimensions}--\eqref{assume: latent tail}, with $g_0$ of smoothness $\beta$, and normalized, integrable, spherically symmetric $K\in C^1\cap L^2$ of order $\beta$ with $\nabla K\in L^2$, then
        $$\bbE_0\left[\|\hat{g}_h - g_0^R\|_{L^2}^2\mid \cE_{n,D}\right] \lesssim h^{2\beta} + \frac{1}{nh^d} + \frac{1}{Dh^{d+2}} + \frac{\log n}{nh^{d+2}}.$$
    Here, $\cE_{n,D}$ is the event in Theorem \ref{thm: spectral recovery}(i). The constants appearing in $\lesssim$ are based on constants depending on $g_0$ and $K$, and thus free of $n$ or $D$. Using this upper bound, taking $h^*\asymp (\frac1D+\frac{\log n}{n})^{1/(2\beta+d+2)}$ yields the following rate:
    $$\bbE_0\left[\|\hat{g}_{h^*} - g_0^R\|_{L^2}^2 \mid \cE_{n,D}\right]\lesssim  \left(\frac1D + \frac{\log n}{n}\right)^{\frac{2\beta}{2\beta+d+2}}.$$
\end{theorem}

\begin{proof}
    Conditioned on $\cE_{n,D}$, we have the rotation $R$, hence we ignore it in the proof (replace $\hat{\eta}_i$ with $R\hat{\eta}_i$, noting that rotations preserve Euclidean distances, and $K$ is spherically symmetric). Write $\hat{\eta}_i = \eta_i + \delta_i$.  Then, we have
    $$\hat{g}_h(\eta) = \frac{1}{n}\sum_{i=1}^n K_h(\eta - \eta_i - \delta_i).$$
    Define the following kernel-smoothed version of $g_0$
    \begin{align*}
        g_h(\eta) &= (K_h \star g_0) (\eta) = \int K_h(\eta-z)g_0(z) dz.
    \end{align*}
    We start with the basic decomposition (pointwise):
    \begin{align*}
        \hat{g}_h(\eta)-g_0(\eta) = \underbrace{g_h(\eta) - g_0(\eta)}_{R_1\text{ (bias)}} + \underbrace{\frac{1}{n}\sum_{i=1}^n K_h(\eta-\eta_i) - g_h(\eta)}_{R_2 \text{ (variance)}} + \underbrace{\hat{g}_h(\eta) - \frac{1}{n}\sum_{i=1}^n K_h(\eta-\eta_i)}_{R_3\text{ (perturbation)}}.
    \end{align*}
    Triangle inequality yields
    $$\|\hat{g}-g_0\|_{L^2}^2 \lesssim \|R_1\|_{L^2}^2 + \|R_2\|_{L^2}^2 + \|R_3\|_{L^2}^2.$$
    We bound the bias term pointwise. The perturbation term is bounded in expectation conditioned on $\cE_{n,D}$. The variance term is bounded in expectation as well, but we use $1_{\cE_{n,D}}\leq 1$ to convert to an unconditional expectation, so that $\eta_i\sim g_0$ iid could be used to bound this term.

    \medskip
    \noindent
    \textbf{Bias term: }For the bias term, we have $\|R_1\|_{L^2}^2 = \int ((K_h\star g_0)(\eta) - g_0(\eta))^2 d\eta$. By Parseval's identity, we have
    $$\|R_1\|_{L^2}^2 = (2\pi)^{-d}\int_{\bbR^d} |\varphi_{K_h\star g_0}(s) - \varphi_{g_0}(s)|^2 ds \lesssim \int_{\bbR^d} |\varphi_{g_0}(s)|^2|\varphi_K(hs) - 1|^2 ds.$$
    Now, by definition of the smoothness order of the kernel $K$, we have
    $$\|R_1\|_{L^2}^2 \lesssim \int |\varphi_{g_0}(s)|^2 \|hs\|^{2\beta} ds = h^{2\beta} \int \|s\|^{2\beta}|\varphi_{g_0}(s)|^2 ds \lesssim h^{2\beta}.$$
    
    \medskip
    \noindent
    \textbf{Variance term: } Next, we focus on the variance term. Let $\tilde{g}_h(\eta) = (1/n)\sum_i K_h(\eta - \eta_i)$. Under $\bbP_0^n$, $\eta_1,\dots,\eta_n$ are i.i.d. from $g_0$. Hence, $\bbE_0[\tilde{g}_h(\eta)]=g_h(\eta)$. Thus, $\|\tilde{g}_h - g_h\|_{L^2}^2=\int (\tilde{g}_h(\eta) - \bbE_0[\tilde{g}_h(\eta)])^2 d\eta$. Under the (unconditional) expectation, this becomes $\bbE[\|\tilde{g}_h - g_h\|_{L^2}^2]=\int \Var(\tilde{g}_h(\eta)) d\eta$. Now,
    \begin{align*}
        \bbE_0[\|R_2\|_{L^2}^2 &\mid \cE_{n,D}] \leq \frac{1}{\bbP_0(\cE_{n,D})} \bbE_0[\|R_2\|_{L^2}^2] \\
        &\leq \frac{1}{1-c/n} \int \Var\left(\tilde{g}_h(\eta)\right) d\eta\\
        &\leq \frac{n}{n-c}  \int \frac{1}{n} \bbE_0[K_h(\eta - \eta_1)^2] d\eta \\
        &\leq  \frac{c'}{n}\, \bbE_0 \left[\int K_h(\eta - \eta_1)^2 d\eta\right] \\
        &\leq \frac{1}{nh^d} \int K(u)^2 du \lesssim \frac{1}{nh^d}.
    \end{align*}

    \medskip
    \noindent
    \textbf{Perturbation term: } Finally, we consider the perturbation term. Use mean value theorem and Cauchy Schwarz inequality to get
    $$|K_h(\eta-\eta_i-\delta_i) - K_h(\eta-\eta_i)| \leq |\delta_i| \int_0^1 \|\nabla K_h(\eta-\eta_i - t\delta_i)\|dt.$$

    Averaging over $i$ and squaring,
    \begin{align*}
        |R_3(\eta)| &\leq \frac1n \sum_{i=1}^n |\delta_i|\int_0^1 \|\nabla K_h(\eta-\eta_i - t\delta_i)\|dt \\
        \Longrightarrow |R_3(\eta)|^2 &\leq \frac1n \sum_{i=1}^n |\delta_i|^2 \left(\int_0^1 \|\nabla K_h(\eta-\eta_i - t\delta_i)\|dt\right)^2 \\
        &\leq \frac1n \sum_{i=1}^n |\delta_i|^2 \int_0^1 \|\nabla K_h(\eta-\eta_i-t\delta_i)\|^2 dt,
    \end{align*}
    where we used Jensen's inequality over the sum and over the integral in the last two steps. 
    Integrating over $\eta$, exchanging the integrals and using the translation invariance of $L^2$, we get
    \begin{align*}
        \|R_3\|_{L^2}^2 &= \int_{\bbR^d} |R_3(\eta)|^2 d\eta\\
        &\leq \frac{1}{n}\sum_i |\delta_i|^2 \int_{0}^1 \int_{\bbR^d} \|\nabla K_h(\eta-\eta_i-t\delta_i)\|^2 d\eta \, dt \\
        &= \frac{1}{n}\sum_i |\delta_i|^2 \int_0^1 \int_{\bbR^d} \|\nabla K_h(\eta)\|^2 d\eta \, dt \\
        &= \frac{1}{n}\sum_{i=1}^n |\delta_i|^2 \|\nabla K_h \|_{L^2}^2.
    \end{align*}
    Note that $K_h(\eta) = h^{-d}K(\eta/h)$, so $\nabla K_h(\eta) = h^{-(d+1)} \nabla K(\eta/h)$. Hence, 
    $$\|\nabla K_h\|_{L^2}^2 = h^{-2(d+1)}\int \|\nabla K(u/h)\|^2 du= h^{-2(d+1)} h^d \int \|\nabla K(u)\|^2 du = h^{-(d+2)}\|\nabla K\|_{L^2}^2.$$
    Thus, we get the bound
    $$\|R_3\|_{L^2}^2 \leq \frac{1}{h^{d+2}} \|\nabla K\|_{L^2}^2 \left(\frac{1}{n}\sum_{i=1}^n |\delta_i|^2\right).$$
    Conditioned on $\cE_{n,D}$, we have $(1/n)\sum_{i} |\delta_i|^2 \leq \tilde{\epsilon}_{n,D}^2$ from Theorem \ref{thm: spectral recovery}(i) where $\tilde{\epsilon}_{n,D}=1/\sqrt{D} + \sqrt{(\log n)/n}$. Thus, we have
    $$\bbE_0[\|R_3\|_{L^2}^2 \mid \cE_{n,D}] \lesssim \frac{\tilde{\epsilon}_{n,D}^2}{h^{d+2}}.$$
    Combining the three parts, the conclusion follows.
\end{proof}

\subsubsection*{Proof of Theorem \ref{thm: new kde}}
\begin{proof}[Proof of Theorem \ref{thm: new kde}]\label{proof of theorem new kde}
    Recall $g_0^R(x) =g_0(Rx)$ and
    $$\hat{g}_h(x) = \frac{1}{nh^d} \sum_{i=1}^n K\left(\frac{\hat{\eta}_i -x}{h}\right).$$
    Define
    \begin{align*}
        \tilde{g}_h(x) &= \frac{1}{nh^d}\sum_{i=1}^n K\left(\frac{R^\top \eta_i - x}{h}\right) &\text{(oracle KDE)}\\
        \bar{g}_h(x) &= \int_{\bbR^d} \frac{1}{h^d} K\left(\frac{y-x}{h}\right)g_0^R(y) dy = (K_h \star g_0^R)(x) &\text{(smoothed rotated density)}
    \end{align*}
    where $R=\bar{V}_0\bar{R}\hat{R}$ (as defined earlier). Also, recall from \eqref{eq: structured score expansion}
    $$\delta_i=\hat{\eta}_i - R^\top \eta_i = \hat{R}^\top \xi_i + \Delta_i, \quad \max_i \|\Delta_i\| \lesssim \frac{\sqrt{\log (nD)}}{D} + \frac{\log n}{\sqrt{n}},$$
    with $\xi_i = \sqrt{n}S_0^{-1}V_0^\top E_0^\top P_0^\perp e_i$. We use the same decomposition as used in the proof of Theorem \ref{thm: weak kde g0}:
    \begin{equation}
        \hat{g}_h(x) - g_0^R(x) = \underbrace{\left(\bar{g}_h(x) - g_0^R(x)\right)}_{R_1(x)} + \underbrace{\left(\tilde{g}_h(x) - \bar{g}_h(x)\right)}_{R_2(x)} + \underbrace{\left(\hat{g}_h(x) - \tilde{g}_h(x)\right)}_{R_3(x)}.
    \end{equation}
    This gives
    $$\|\hat{g}_h - g_0^R\|_{L^2}^2 \lesssim \|R_1\|_{L^2}^2 + \|R_2\|_{L^2}^2 + \|R_3\|_{L^2}^2.$$
    The first term (bias) can be bounded pointwise by $h^{2\beta}$, exactly as in the proof of Theorem \ref{thm: weak kde g0}. Also, the second term (variance) can also be bounded in the conditional expectation by $1/nh^d$ in the same fashion as in Theorem \ref{thm: weak kde g0} by first moving to the unconditional expectation and removing the rotation $R$ by using the fact that the kernel $K$ is spherically symmetric. We now focus on the last term (perturbation). This term is
    $$R_3(x) = \frac{1}{nh^d}\sum_{i=1}^n \left[K\left(\frac{\hat{\eta}_i - x}{h}\right) - K\left(\frac{R^\top \eta_i - x}{h}\right)\right].$$
    Let $z_i(x) = (R^\top\eta_i - x)/h$ be the unperturbed center, then the perturbed center becomes
    $$\frac{\hat{\eta}_i - x}{h} = z_i(x) + \frac{\delta_i}{h}.$$
    Using Taylor expansion for $K:\bbR^d\to \bbR$ up to the second order about $z=z_i(x)$ with $\nu=\delta_i/h$ with the integral form of the third-order remainder, we have 
    $$K(z+\nu) =K(z) + \nabla K(z)^\top \nu +\frac12\nu^\top\nabla^2 K \nu + \frac12\int_0^1 (1-t)^2\nabla^3K(z+t\nu)[\nu,\nu,\nu] dt.$$
    Substituting this back into the expression for $R_3$, we decompose it into three distinct components:
    \begin{align*}
        R_3(x) &= R_{3,1}(x) + R_{3,2}(x) + R_{3,\rm rem}(x)
    \end{align*}
    corresponding to the linear, quadratic and third-order remainder terms, where 
    \begin{equation}
        \begin{aligned}
            R_{3,1}(x) &= \frac{1}{nh^{d+1}}\sum_{i=1}^n \nabla K\left(\frac{R^\top \eta_i -x}{h}\right)^\top (\hat{R}^\top\xi_i + \Delta_i) \\
            R_{3,2}(x) &= \frac{1}{2nh^{d+2}}\sum_{i=1}^n (\hat{R}^\top\xi + \Delta_i)^\top \nabla^2 K\left(\frac{R^\top \eta_i -x}{h}\right) (\hat{R}^\top \xi_i + \Delta_i) \\
            R_{3,\rm rem}(x) &= \frac{1}{2nh^{d+3}}\sum_{i=1}^n \int_0^1 (1-t)^2 \nabla^3 K\left(z_i(x) + t\frac{\delta_i}{h}\right)[\delta_i, \delta_i, \delta_i] dt.
        \end{aligned}
    \end{equation}

    \medskip
    \noindent
    \textbf{Decoupling the random rotation:} For each of the terms, we are interested in bounded $\|R_{3,p}\|_{L^2}^2=\int_{\bbR^d} |R_{3,p}(x)|^2dx$ for $p\in\{1,2,\rm rem\}$. We apply a change of variable $x=R^\top y$, noting that the Jocabian determinant satisfies $|\det(R^\top)|=1$ and the domain of integration remains the entire space $\bbR^d$. Under this change of variables,
    $$z_i(R^\top y) = \frac{R^\top\eta_i - R^\top y}{h} = R^\top \left(\frac{\eta_i-y}{h}\right) = R^\top u_i(y).$$
    Since the kernel is spherically symmetric, we have $K(R^\top u)=K(u)$, which leads to $\nabla K(R^\top u) = R^\top \nabla K(u)$ and $\nabla^2 K(R^\top u) = R^\top \nabla^2 K(u) R$. Applying this transformation to $R_{3,1}$ and $R_{3,2}$, we get
    \begin{align*}
        \|R_{3,1}\|_{L^2}^2 &= \int_{\bbR^d} \left(\frac{1}{nh^{d+1}} \sum_{i=1}^n \nabla K\left(\frac{\eta_i - y}{h}\right)^\top\left(R\hat{R}^\top\xi_i + R\Delta_i\right)\right)^2 dy, \\
        \|R_{3,2}\|_{L^2}^2 &= \int_{\bbR^d} \left[\frac{1}{2nh^{d+2}}\sum_{i=1}^n \left(R\hat{R}^\top\xi_i + R\Delta_i\right)^\top \nabla^2 K\left(\frac{\eta_i - y}{h}\right) \left(R\hat{R}^\top\xi_i + R\Delta_i\right)\right]^2 dy.
    \end{align*}

    \textbf{General strategy:} In each of the following steps, we will be dealing with one of the terms in the above expansion, expanding and upper bounding the corresponding integral, each taking the general form $\bbE_0[\int f(y, \omega) dy \mid \cE_{n,D}]$, where $f\geq 0$ is some random function (randomness arising from $M_0$ and $E_0$) of $y$. We wish to exploit stochasticity (e.g. $\eta_i'$s are i.i.d. from $g_0$ and $\eta_i'$s are unconditionally independent of $E_0$), which are not necessarily true under $\cE_{n,D}$. However, we also want to use the use concentrations such as $s_{\min} (M_0)\asymp \sqrt{n}$, which is not true unconditionally. To resolve this, using $\bbP_0(\cE_{n,D})\geq 1-c/n$ (and hence $\geq 1/2$ for sufficiently large $n$) we first bound
    \begin{equation}\label{eq: conditional to indicator}
        \bbE_0[I \mid \cE_{n,D}] = \frac{\bbE_0[I \ind(\cE_{n,D})]}{\bbP_0(\cE_{n,D})} \leq 2 \bbE_0\left[\int f(y;\omega) dy \ind(\cE_{n,D})\right].
    \end{equation}
    where $I$ is the integral. Next, decompose $\cE_{n,D}=\cE_\eta\cap \cE_{\rm noise}$, where $\cE_\eta\in \sigma(M_0)$ is the score-only good event (on which $\|M_0\|\asymp \sqrt{n}$) and $\cE_{\rm noise}$ is the event containing noise conditions (such as bounds for $\|E_0\|$). This implies $\ind(\cE_{n,D})\leq \ind(\cE_{\eta})$ for all $(M_0, E_0)$. As such, using $f\geq 0$, we get $f(y)\ind(\cE_{n,D})\leq f(y) \ind(\cE_\eta)$. By the tower property of expectation,
    \begin{equation}\label{eq: conditional to nested}
        \bbE_0\left[\int f(y;\omega)dy\mid \cE_{n,D}\right] \leq 2\bbE_\eta \left[\bbE_{E_0}\left[\int f(y;\omega) dy\mid M_0 \right] \ind(\cE_\eta)\right]
    \end{equation}
    where the expectation inside $\bbE_{E_0}$ is conditional on the realization of $M_0$, with expectation over the randomness from $E_0$, while the outer expectation is with respect to the randomness from $M_0$ itself. Conditional on $M_0$, the quantities depending only on the latent scores are deterministic, while $E_0$ retains its original Gaussian law by independence of $M_0$ and $E_0$ under $\bbP_0^n$. Thus, \eqref{eq: conditional to nested} allows us to combine score-only concentration on $\cE_\eta$ with exact conditional moments of the noise. We use \eqref{eq: conditional to nested} only when no subsequent bound requires a condition from $\cE_{\rm noise}$. When a pointwise noise-good event bound is needed, we retain the indicator $\ind(\cE_{n,D})$, use \eqref{eq: conditional to indicator}, apply the necessary bound first and then drop or enlarge the indicator. It is worth noting that since $f\geq 0$, Tonelli's theorem guarantees that we can interchange the integral and expectation in \eqref{eq: conditional to nested}, and hence, also have
    \begin{equation}\label{eq: conditional to nested interchanged}
        \bbE_0\left[\int f(y;\omega)dy\mid \cE_{n,D}\right] \leq 2\bbE_\eta \left[\int \bbE_{E_0}\left[f(y;\omega) \mid M_0 \right]dy \ind(\cE_\eta)\right].
    \end{equation}

    We will also repeatedly use the following translation-scaling identities. For every fixed $\eta\in \bbR^d$, we have
    \begin{equation}\label{eq: kernel derivative translation identities}
        \begin{aligned}
            \int \left\|\nabla K\left(\frac{\eta - y}{h}\right) \right\|^2 dy &= h^d\|\nabla K\|^2_{L^2},\\
            \int \left\|\nabla^2 K\left(\frac{\eta - y}{h}\right)\right\|_F^2 dy &= h^d\|\nabla^2 K\|_{F,L^2}^2
        \end{aligned}
    \end{equation}
    making the substitution $z = (\eta - y)/h\iff dy=h^d dz$, where for matrix valued function $H$, $\|H\|_{F,L^2}^2 = \int \|H(z)\|_{F}^2dz$.

    Consequently, for any $\eta_i, \eta_j\in\bbR^d$, Cauchy-Schwarz in $L^2(dy)$ yields the following inequalities.
    \begin{equation}\label{eq: kernel derivative cross cs}
    \begin{aligned}
    \int_{\mathbb R^d}
    \left\|
    \nabla K\left(\frac{\eta_i-y}{h}\right)
    \right\|
    \left\|
    \nabla K\left(\frac{\eta_j-y}{h}\right)
    \right\|dy
    &\leq
    h^d\|\nabla K\|_{L^2}^2,\\
    \int_{\mathbb R^d}
    \left\| \nabla^2K\left(\frac{\eta_i-y}{h}\right)
    \right\|_F
    \left\| \nabla^2K\left(\frac{\eta_j-y}{h}\right)
    \right\|_Fdy
    &\leq
    h^d\|\nabla^2K\|_{F,L^2}^2.
    \end{aligned}
\end{equation}
These bounds hold path-wise and therefore require neither independence (between $\eta_i$ and $\eta_j$), nor expectation.
However, using stochasticity of $\eta$ helps obtaining sharper bounds. To illustrate, consider the term $T_1=\bbE_\eta [\int \|\nabla K((\eta_1 -y)/h)\|\|\nabla K((\eta_2-y)/h)\| dy]$, where the expectation is over $\eta_1,\eta_2\sim g_0$ i.i.d. One option is to use the inequality in \eqref{eq: kernel derivative cross cs} to get $T_1 \leq h^{d}\|\nabla K\|_{L^2}^2$. Alternately, we could interchange expectation and integral to get
    $$T_1 = \int \bbE_\eta \left[\|\nabla K((\eta_1 -y)/h)\|\|\nabla K((\eta_2-y)/h)\|\right] dy = \int \psi_h(y)^2 dy,$$
    using the iid nature of $\eta_1,\eta_2$, where $\psi_h(y) = \bbE_\eta \left[\left\|\nabla K\left(\frac{\eta - y}{h}\right)\right\|\right]$. Now, note that $\psi_h(y) = (f_h \star g_0)(y)$ where $f_h(u)=\|\nabla K(u/h)\|$ (which is an even function). Using Young's convolution inequality in $L^1\star L^2 \to L^2$, we get
    $$\|\psi_h\|_{L^2} \leq \|f_h\|_{L^1} \|g_0\|_{L^2}.$$
    Since $\|f_h\|_{L^1} = h^d \|\nabla K\|_{L^1}$, we get
    \begin{equation}\label{eq: cross term linear young}
        \bbE_\eta\left[\int \left\|\nabla K\left(\frac{\eta_1-y}{h}\right)\right\| \left\|\nabla K\left(\frac{\eta_2 - y}{h}\right)\right\| dy\right] \leq h^{2d}\|\nabla K\|_{L^1}^2 \|g_0\|_{L^2}^2.
    \end{equation}
    Note that compared to the simpler $h^d\|\nabla K\|_{L^2}^2$ bound, we get a sharper scaling $h^{2d}$ utilizing the iid structure of $\eta$. This is useful when summing over $O(n^2)$ off-diagonal pairs. The $O(n)$ diagonal terms do not admit this factorization and are instead controlled using \eqref{eq: kernel derivative translation identities}. Similarly, we have the Hessian analogue of the above display
    \begin{equation}\label{eq: cross term hessian young}
        \bbE_\eta\left[\int \left\|\nabla^2 K\left(\frac{\eta_1-y}{h}\right)\right\|_F \left\|\nabla^2 K\left(\frac{\eta_2 - y}{h}\right)\right\|_F dy\right] \leq h^{2d}\|\nabla^2 K\|_{F,L^1}^2 \|g_0\|_{L^2}^2.
    \end{equation}
    
    \medskip
    \noindent
    \textbf{Step 1: Bounding the first-order term.} Firstly, note $R\hat{R}^\top\xi = (\bar{V}_0\bar{R}\hat{R})\hat{R}^\top \xi_i = \Theta_0 \xi_i$, where $\Theta_0 = \bar{V}_0\bar{R}$ is a fixed orthogonal matrix. So, we split the linear term into two parts
    $$\|R_{3,1}\|_{L^2}^2 \lesssim \underbrace{\int \left(\frac{1}{nh^{d+1}}\sum_i \nabla K(u_i(y))^\top \Theta_0\xi_i\right)^2 dy}_{I_{3,1}^{\rm lead}} + \underbrace{\int \left(\frac{1}{nh^{d+1}}\sum_i \nabla K(u_i(y))^\top R\Delta_i\right)^2 dy}_{I_{3,1}^{\Delta}}.$$

    \medskip
    \noindent
    \textbf{Step 1(a): Leading term} Applying \eqref{eq: conditional to nested interchanged}, we have
    $$\bbE_0[I_{3,1}^{\rm lead}\mid \cE_{n,D}] \leq 2\bbE_\eta\left[\bbE_{E_0}[I_{3,1}^{\rm lead}\mid M_0] \ind(\cE_\eta)\right].$$
    For a fixed realization of $M_0$, the randomness in $\xi_i$ arises only from $E_0$. By the definition of $\xi$ as $\xi_i = \sqrt{n}S_0^{-1}V_0^\top E_0^\top P_0^\perp e_i$, we have
    $$\bbE_{E_0}[\xi_i\xi_j^\top\mid M_0] = (P_0^\perp)_{ij} \Sigma_\xi, \quad \Sigma_\xi = nS_0^{-1} V_0^\top \Sigma_0 V_0 S_0^{-1},$$
    where we used the identity $\bbE_{E_0}[E_0^\top ab^\top E_0]=(a^\top b) \Sigma_0$ for deterministic $a,b\in\bbR^n$. Moreover, on $\cE_{\eta}$, $$\|\Sigma_\xi\|\leq n \|S_0^{-1}\|^2 \|\Sigma_0\| \lesssim 1/D$$
    using $\|S_0^{-1}\|\asymp (nD)^{-1/2}$ (concentration of $M_0$ under $\cE_\eta$ and assumption on $\Lambda_0$) and $\|\Sigma_0\|\lesssim 1$ (by assumption). Let $v_i(y) = \Theta_0^\top \nabla K(u_i(y))$ and $V(y)=[v_1(y),\dots,v_n(y)]$. Then
    $$\bbE_{E_0}[I_{3,1}^{\rm lead} \mid M_0] = \frac{1}{n^2 h^{2d+2}} \int_{\bbR^d} \tr\left(\Sigma_\xi V(y) P_0^\perp V(y)^\top\right) dy.$$
    Since $\Sigma_\xi$ is positive semi-definite and $0\preceq P_0^\perp \preceq I_n$, 
    $$\tr\left(\Sigma_\xi V(y) P_0^\perp V(y)^\top\right) \leq \|\Sigma_\xi\| \tr\left(V(y)P_0^\perp V(y)^\top\right) \leq \|\Sigma_\xi\| \sum_{i=1}^n \|\nabla K(u_i(y))\|^2,$$
    where we also used orthogonality of $\Theta_0$. Using the bound for $\|\Sigma_\eta\|$ on $\cE_{\eta}$, followed by \eqref{eq: kernel derivative translation identities}, 
    \begin{align*}
        \bbE_0[I_{3,1}^{\rm lead} \mid \cE_{n,D}] &\lesssim \frac{1}{Dn^2 h^{2d+2}} \sum_{i=1}^n \bbE_\eta\left[\int \|\nabla K(u_i(y))\|^2 dy \right] \\
        &= \frac{1}{Dn^2h^{2d+2}} \sum_{i=1}^n h^d \|\nabla K\|_{L^2}^2 \lesssim \frac{1}{nDh^{d+2}}.
    \end{align*}
    Thus,
    \begin{equation}
        \bbE_0[I_{3,1}^{\rm lead} \mid \cE_{n,D}] \lesssim \frac{1}{nDh^{d+2}}.
    \end{equation}

    \medskip
    \noindent
    \textbf{Step 1(b): Residual term} By Theorem \ref{thm: spectral recovery}(ii), on $\cE_{n,D}$, 
    $$\max_i \|R\Delta_i\| = \max_i \|\Delta_i\| \lesssim \frac{\sqrt{\log(nD)}}{D} + \frac{\log n}{\sqrt{n}}=:\delta_{n,D}.$$
Consequently, pointwise of $\cE_{n,D}$,
$$I_{3,1}^{\Delta} \lesssim \frac{\delta_{n,D}^2}{n^2 h^{2d+2}} \int \left(\sum_{i=1}^n \left\|\nabla K\left(\frac{\eta_i - y}{h}\right)\right\|\right)^2 dy.$$
Using Equation \eqref{eq: conditional to indicator}, applying the preceding bound on $\cE_{n,D}$ and subsequently dropping the indicator, we obtain
$$\bbE_0[I_{3,1}^{\Delta} \mid \cE_{n,D}] \lesssim \frac{\delta_{n,D}^2}{n^2 h^{2d+2}}\sum_{i,j=1}^n \bbE_\eta \left[\int \left\|\nabla K\left(\frac{\eta_i - y}{h}\right)\right\| \left\|\nabla K\left(\frac{\eta_j - y}{h}\right)\right\| dy\right].$$

Focus on the sum. For the $n$ diagonal terms, use identity \eqref{eq: kernel derivative translation identities} pointwise, and for $i\neq j$, since $\eta_i$ and $\eta_j$ are iid from $g_0$, use inequality \eqref{eq: cross term linear young} to bound the expectation. There are $n$ diagonal terms in the sum, each $\lesssim h^d$, and $O(n^2)$ off-diagonal terms, each $\lesssim h^{2d}$ in expectation. Combining with the above display, we get
\begin{equation}
    \bbE_0[I_{3,1}^\Delta \mid \cE_{n,D}] \lesssim \frac{\delta_{n,D}^2}{n^2h^{2d+2}}\left(nh^d + n^2 h^{2d}\right) = \delta_{n,D^2}\left(\frac{1}{nh^{d+2}} + \frac{1}{h^2}\right).
\end{equation}

Combining the leading and residual terms, we obtain the following bound for the linear term in 
\begin{equation}\label{eq: kde R_31}
    \bbE_0\left[\|R_{3,1}\|_{L^2}^2\mid \cE_{n,D}\right] \lesssim \frac{1}{nDh^{d+2}} + \delta_{n,D}^2 \left(\frac{1}{nh^{d+2}} + \frac{1}{h^2}\right).
\end{equation}
    
    \medskip
    \noindent
    \textbf{Step 2: Bounding the quadratic term.}
    Recall (and use $R\hat{R}^\top=\Theta_0\in\cO(d)$):
    $$\|R_{3,2}\|_{L^2}^2 = \int_{\bbR^d} \left[\frac{1}{2nh^{d+2}}\sum_{i=1}^n \left(\Theta_0\xi_i + R\Delta_i\right)^\top \nabla^2 K\left(\frac{\eta_i - y}{h}\right) \left(\Theta_0\xi_i + R\Delta_i\right)\right]^2 dy$$
    and we want to bound $\bbE_0[\|R_{3,2}\|_{L^2}^2\mid \cE_{n,D}]$. Firstly, let $H_i(y)=\nabla^2 K(u_i(y))$. Expanding the integrand, we get
    \begin{equation}
        \|R_{3,2}\|_{L^2}^2 \lesssim I_{3,2}^{\rm lead} + I_{3,2}^{\rm cross} + I_{3,2}^{\Delta}
    \end{equation}
    where
    \begin{align*}
        I_{3,2}^{\rm lead} &= \int_{\bbR^d} \left(\frac{1}{2nh^{d+2}} \sum_{i=1}^n \xi_i^\top \Theta_0^\top H_i(y) \Theta_0\xi_i\right)^2dy \\
        I_{3,2}^{\rm cross} &= \int_{\bbR^d} \left(\frac{1}{nh^{d+2}}\sum_{i=1}^n \xi_i^\top\Theta_0^\top H_i(y) R\Delta_i\right)^2 dy \\
        I_{3,2}^{\Delta} &= \int_{\bbR^d} \left(\frac{1}{2nh^{d+2}} \sum_{i=1}^n \Delta_i^\top R^\top H_i(y) R\Delta_i\right)^2 dy.
    \end{align*}

    We bound these separately.

    \medskip
    \noindent
    \textbf{Step 2(a): Leading term} Using \eqref{eq: conditional to nested}, we have $\bbE_0[I_{3,2}^{\rm lead}\mid \cE_{n,D}] \leq 2 \bbE_\eta[\bbE_{E_0}[I_{3,2}^{\rm lead}\mid M_0] \ind(\cE_\eta)]$. For fixed $M_0$, set $A_i(y)=\Theta_0^\top H_i(y) \Theta_0$. So, we get
    \begin{equation}\label{eq: I_32 lead 1}
        \bbE_{E_0}[I_{3,2}^{\rm lead}\mid M_0] = \frac{1}{4n^2h^{2d+4}}\sum_{i,j=1}^n \int \bbE_{E_0} \left[(\xi_i^\top A_i(y) \xi_i)(\xi_j^\top A_j(y) \xi_j)\mid M_0\right] dy.
    \end{equation}
    Conditional on $M_0$, the vectors $\xi_i$ are jointly centered Gaussian with conditional covariance
    $$\Sigma_{ij}=\Cov(\xi_i, \xi_j\mid M_0) = (P_0^\perp)_{ij}\Sigma_\xi, \quad \Sigma_\xi = nS_0^{-1} V_0^\top \Sigma_0 V_0 S_0^{-1}.$$
    As established above, on $\cE_\eta$, $\|\Sigma_\xi\|\lesssim 1/D$. By Isserlis' formula,
    $$\bbE_{E_0}[(\xi_i^\top A_i\xi_i)(\xi_j^\top A_j \xi_j)\mid M_0] = \underbrace{\tr(A_i \Sigma_{ii})\tr(A_j \Sigma_{jj})}_{\text{first Isserlis contribution}} + \underbrace{2\tr(A_i\Sigma_{ij} A_j \Sigma_{ji})}_{\text{second Isserlis contribuion}}.$$
    Thus, there are two contributions. For the first one, since
    $$|\tr(A_i\Sigma_\xi)|\leq \|\Sigma_\xi\|\|A_i\|_* \leq d\|\Sigma_\xi\|\|A_i\| = d\|\Sigma_\xi\|\|H_i\|,$$
    and $(P_0^\perp)_{ii}\leq 1$, on $\cE_\eta$, the corresponding double sum is bounded by
    $$\frac{C}{D^2}\left(\sum_{i=1}^n \|H_i(y)\|\right)^2.$$
    At this point, we invoke the diagonal/off-diagonal calculation established previously. The diagonal terms contribute $O(nh^d)$ using \eqref{eq: kernel derivative translation identities}, while independence of $\eta_i,\eta_j$ for $i\neq j$, together with \eqref{eq: cross term hessian young}, gives off-diagonal contribution $O(n^2h^{2d})$. Consequently, after dividing by $n^2h^{2d+4}$
    $$\text{first Isserlis contribution towards }  \bbE_0[I_{3,2}^{\rm lead}\mid \cE_{n,D}]\lesssim \frac{1}{D^2}\left(\frac{1}{nh^{d+4}} + \frac{1}{h^4}\right).$$

    Now, for the second Isserlis contribution, use
    $$|\tr(A_i\Sigma_\xi A_j \Sigma_\xi)| \leq \|\Sigma_\xi\|^2 \|A_i\|_F \|A_j\|_F = \|\Sigma_\xi\|^2 \|H_i\|_F \|H_j\|_F.$$
    Thus, we need to control 
    $$\sum_{i,j} (P_0^\perp)_{ij}^2 \int \|H_i(y)\|_F \|H_j(y)\|_F dy.$$
    By the pointwise bound in \eqref{eq: kernel derivative cross cs} used for every realization of $M_0$ and every $i,j$, we get
    \begin{align*}
        \sum_{i,j} (P_0^\perp)_{ij}^2 \int \|H_i(y)\|_F\|H_j(y)\|_F dy &\leq h^d \|\nabla^2 K\|_{F, L^2}^2 \sum_{i,j} (P_0^\perp)_{ij}^2 \\
        &= h^d \|\nabla^2 K\|_{F,L^2}^2 \|P_0^\perp \|_F^2 \\
        &= h^d \|\nabla^2 K\|_{F,L^2}^2 \tr(P_0^\perp) \lesssim nh^d.
    \end{align*}
    Note here we cannot use \eqref{eq: cross term hessian young} since both $P_0^\perp$ and $\|H_i\|_F\|H_j\|_F$ depend on $M_0$ and these two parts are not independent. Thus, after dividing appropriately
    $$\text{second Isserlis contribution towards } \bbE_0[I_{3,2}^{\rm lead}\mid \cE_{n,D}] \lesssim \frac{1}{D^2n^2 h^{2d+4}}(nh^d) = \frac{1}{D^2 nh^{d+4}},$$
    which is already contained in the bound from the first contribution. Hence,
    \begin{equation}\label{eq: kde new I_32 lead}
        \bbE_0[I_{3,2}^{\rm lead}\mid \cE_{n,D}] \lesssim \frac{1}{D^2}\left(\frac{1}{nh^{d+4}} + \frac{1}{h^4}\right).
    \end{equation}

    \medskip
    \noindent
    \textbf{Step 2(b): Residual term.} Moving on to $I_{3,2}^{\Delta}$. On $\cE_{n,D}$, we have
    $$|\Delta_i^\top R^\top H_i(y) R\Delta_i|\leq \|R\Delta_i\|^2\|H_i(y)\| \leq \delta_{n,D}^2 \|H_i(y)\|.$$
    Substituting this into the $L^2$ integral, we have
    $$I_{3,2}^\Delta \ind(\cE_{n,D})\leq \frac{\delta_{n,D}^4}{n^2h^{2d+4}}\int_{\bbR^d} \left(\sum_{i=1}^n \left\|\nabla^2 K\left(\frac{\eta_i - y}{h}\right)\right\|\right)^2dy$$
    where we dropped the indicator $\ind\leq1$ at the end. We have already bounded the unconditional expectation of the remaining part (in Step 2(a)). This gives
    \begin{equation}\label{eq: kde new I_32 Delta}
        \bbE_0[I_{3,2}^\Delta\mid \cE_{n,D}] \lesssim \delta_{n,D}^4\left(\frac{1}{nh^{d+4}} + \frac{1}{h^4}\right).
    \end{equation}

    \medskip
    \noindent
        \textbf{Step 2(c): Cross term.} Finally, focus on $I_{3,2}^{\rm cross}$. Note, here $\Delta_i$ (which depends on $\hat{R}$, which is further a function of noise $E_0$) and $\xi_i$ are not independent. We instead use a pointwise Cauchy-Schwarz inequality on the good event
    $$|\xi_i^\top \Theta_0^\top H_i(y) R\Delta_i| \leq \|\xi_i\|\|H_i(y)\|\|\Delta_i\|\leq \delta_{n,D}\|\xi_i\|\|H_i(y)\|, \quad\text{ on }\cE_{n,D}.$$
    This gives
    \begin{align}\label{eq: kde new I_32 step}
        I_{3,2}^{\rm cross} .\ind(\cE_{n,D})\leq \frac{\delta_{n,D}^2}{n^2h^{2d+4}}\int_{\bbR^d}\left(\sum_{i=1}^n \|\xi_i\|\|H_i(y)\|\right)^2 dy .\ind(\cE_{n,D}).
    \end{align}
    Now, use the fact $\bbE_0[J\ind(\cE_{n,D})]\leq \bbE_\eta[\bbE_{E_0}[J\mid M_0] \ind(\cE_{\eta})]$ by first using $\cE_{n,D}\subseteq \cE_\eta$ and then the tower property of expectation, where $J\geq 0$ is the quantity in the above display without the indicator. Thus, the problem boils down to bounding $S_1 = \bbE_{E_0}\left[\int \left(\sum_{i=1}^n \|\xi_i\|\|H_i(y)\|_F\right)^2 dy\mid M_0\right].$

    Conditional on $M_0$, the hessians $H_i$ are deterministic while $\xi_i$ are jointly Gaussian with covariance $\Cov(\xi_i,\xi_j\mid M_0) = (P_0^\perp)_{ij}\Sigma_\xi$, as used earlier. In particular,
    $$\bbE_{E_0}[\|\xi_i\|^2\mid M_0] = (P_0^\perp)_{ii}\tr(\Sigma_\xi)\leq \tr(\Sigma_\xi) \lesssim 1/D$$
    where the last inequality uses the indicator of $\cE_\eta$. Therefore, by conditional Cauchy-Schwarz, 
    $$\bbE_{E_0}[\|\xi_i\|\|\xi_j\|\mid M_0] \lesssim 1/D$$
    on $\cE_\eta$.
    
    Now, expanding the sum $S_1$, recalling that $H_i$ are deterministic conditional on $M_0$ and interchanging integral and sum, we get that on $\cE_\eta$
    \begin{align*}
        \bbE_{E_0}\left[\int \left(\sum_{i=1}^n \|\xi_i\|\|H_i(y)\|_F\right)^2 dy\mid M_0\right] &= \sum_{i,j} \int \|H_i(y)\|_F \|H_j(y)\|_F \bbE_{E_0}[\|\xi_i\|\|\xi_j\|\mid M_0] dy \\
        &\lesssim \frac{1}{D} \sum_{i,j} \int \|H_i(y)\|_F\|H_j(y)\|_F dy.
    \end{align*}
    Hence,
    \begin{align*}
        \bbE_0[I_{3,2}^{\rm cross}\mid \cE_{n,D}] &\lesssim \frac{\delta_{n,D}^2}{Dn^2h^{2d+4}} \bbE_\eta \left[\sum_{i,j} \int \|H_i(y)\|_F\|H_j(y)\|_F dy\right],
    \end{align*}
    where we finally drop $\ind(\cE_\eta)\leq 1$. Now, the remaining expectation is completely unconditional in the scores. Hence, we can control the diagonal terms by \eqref{eq: kernel derivative translation identities}, and the off-diagonal by \eqref{eq: cross term hessian young} to obtain
    $$\bbE_\eta \left[\sum_{i,j} \int \|H_i(y)\|_F\|H_j(y)\|_F dy\right]\lesssim nh^d + n^2 h^{2d}.$$
    Consequently,
    \begin{equation}\label{eq: kde new I_32 cross}
        \bbE_0[I_{3,2}^{\rm cross} \mid \cE_{n,D}] \lesssim \frac{\delta_{n,D}^2}{D} \left(\frac{1}{nh^{d+4}} + \frac{1}{h^4}\right).
    \end{equation}

    \medskip
    Combining the bounds in Equations \eqref{eq: kde new I_32 lead}, \eqref{eq: kde new I_32 Delta} and \eqref{eq: kde new I_32 cross} and since $\delta_{n,D}^2/D \lesssim 1/D^2 + \delta_{n,D}^4$, we get the following bound for the quadratic term.
    \begin{equation}\label{eq: kde R_32}
        \bbE_0\left[\|R_{3,2}\|_{L^2}^2 \mid \cE_{n,D}\right] \lesssim \left(\frac{1}{D^2} +  \delta_{n,D}^4\right) \left(\frac{1}{nh^{d+4}} + \frac{1}{h^4}\right).
    \end{equation}

    \medskip
    \noindent
    \textbf{Step 3: Bounding the remainder term. }
    For $R_{3,\rm rem}$, use $|\nabla^3 K(z)[a,a,a]| \leq \|\nabla^3K(z)\|_{\rm op} \|a\|^3$. By Minkowski, first over $t$ and then over $i$, we have
    $$\|R_{3,\rm rem}\|_{L^2} \leq \frac{1}{2nh^{d+3}}\sum_{i=1}^n \|\delta_i\|^3 \int_0^1 (1-t)^2 \left[\int \left\|\nabla^3 K\left(z_i(x) + t\frac{\delta_i}{h}\right)\right\|_{\rm op}^2 dx\right]^{1/2} dt.$$
    Now, note that for every $i, t, \eta_i, \delta_i$
    $$z_i(x) + t\frac{\delta_i}{h} = \frac{R^\top \eta_i +t\delta_i - x}{h}.$$
    Therefore, by translation and scaling
    $$\int\left\|\nabla^3 K\left(\frac{R^\top \eta_i +t\delta_i - x}{h}\right)\right\|_{\rm op}^2 dx = h^d \int \|\nabla^3 K(u)\|_{\rm op}^2 du \lesssim h^d$$
    Crucially, the right-hand side is free of $\eta_i, \delta_i$ and $t$. Thus, 
    $$\|R_{3,\rm rem}\|_{L^2} \lesssim \frac{h^{d/2}}{2nh^{d+3}}\sum_i \|\delta_i\|^3 \int_0^1 (1-t)^2 dt \Longrightarrow \|R_{3,\rm rem}\|_{L^2}^2 \lesssim \frac{1}{h^{d+6}} \max_i \|\delta_i\|^6.$$
    On $\cE_{n,D}$, by combining Theorem \ref{thm: spectral recovery}(ii) and Lemma \ref{lemma: bound xi_i}, we have 
    $$\max_i \|\delta_i\|\leq \max_i \|\xi_i\| + \max_i \|\Delta_i\| \lesssim \sqrt{\frac{\log n}{D}} + \delta_{n,D}.$$
     So, we get
    \begin{equation}\label{eq: kde R_3 rem}
        \bbE_0\left[\|R_{3,\rm rem}\|_{L^2}^2 \mid \cE_{n,D}\right] \lesssim \frac{1}{h^{d+6}}\left(\frac{\log^3 n}{D^3} + \delta_{n,D}^6\right).
    \end{equation}

    \medskip
    \noindent
    \textbf{Final step: Summarizing all the bounds together.} Now, combining all the bounds, we have
    \begin{equation}
        \bbE_0[\|\hat{g}_h - g_0^R\|_{L^2}^2 \mid \cE_{n,D}] \lesssim h^{2\beta} + \frac{1}{nh^d} + \tilde{R} =:R(h)
    \end{equation}
    where, by Equations \ref{eq: kde R_31}, \ref{eq: kde R_32} and \ref{eq: kde R_3 rem},  the perturbation bound $\tilde{R}$ is decomposed into the following three terms
    \begin{equation}
        \begin{aligned}
            \tilde{R} &\leq R_1 + R_2 + R_{\rm rem} \\
            R_1 &\lesssim \frac{1}{nDh^{d+2}} + \delta_{n,D}^2\left(\frac{1}{nh^{d+2}} + \frac{1}{h^2}\right) \\
            R_2 &\lesssim \left(\frac{1}{D^2} + \delta_{n,D}^4\right)\left(\frac{1}{nh^{d+4}} + \frac{1}{h^4}\right) \\
            R_{\rm rem} &\lesssim \frac{1}{h^{d+6}}\left(\frac{\log^3 n}{D^3} +  \delta_{n,D}^6\right)
        \end{aligned}
    \end{equation}
    where $\delta_{n,D} = \frac{\sqrt{\log(nD)}}{D} + \frac{\log n}{\sqrt{n}}$. 

    Combining all the parts, we get the following upper bound
    \begin{equation}
        \begin{aligned}
            R(h) = h^{2\beta} &+ \frac{1}{nh^d} \\
            &+  \frac{1}{nDh^{d+2}} + \delta_{n,D}^2\left(\frac{1}{nh^{d+2}} + \frac{1}{h^2}\right) \\
            &+\left(\frac{1}{D^2} + \frac{\delta_{n,D}^2}{D} + \delta_{n,D}^4\right)\left(\frac{1}{nh^{d+4}} + \frac{1}{h^4}\right) \\
            &+ \frac{\log^3n}{D^3h^{d+6}} + \frac{\log^6 n}{n^3h^{d+6}},
        \end{aligned}
    \end{equation}
    where $\delta_{n,D}=\frac{\sqrt{\log(nD)}}{D} + \frac{\log n}{\sqrt{n}}$ as before. 
\end{proof}

\begin{lemma}[Concentration of $\xi_i$]\label{lemma: bound xi_i}
    For any $\varepsilon\in(0,1)$, with joint probability  at least $1-\varepsilon-\bbP_0(\cE_\eta^c)$, we have
    $$\max_i \|\xi_i\|^2 \leq \frac{C_0}{D}\left(d + 2\sqrt{d\log(n/\varepsilon)} + 2\log(n/\varepsilon)\right)$$
    for some deterministic large constant $C_0>0$. Plugging $\varepsilon=1/n$, we have $\max_i \|\xi_i\|^2 \lesssim (\log n)/D$ with joint probability at least $1-c/n$ for some constant $c>0$.
\end{lemma}

\begin{proof}
    Finally, recall $\xi_i = \sqrt{n}S_0^{-1}V_0^\top E_0^\top P_0^\perp e_i$. Note that $\bbE_0[\xi_i]=\bbE_0[\xi_i\mid\bM_0]=0$. Further, $\Cov(\xi_i\mid\bM_0)=\bbE_0[\xi_i\xi_i^\top\mid\bM_0]=(P_0^\perp)_{ii}\Sigma_\xi$, where $\Sigma_\xi = nS_0^{-1}V_0^\top\Sigma_0V_0S_0^{-1}$. The $\xi_i'$s are not conditionally independent, but on $\cE_{\eta}$ (see Step 1(a) in the proof of Theorem \ref{thm: new kde}), we have $\|\Cov(\xi_i\mid \bM_0)\|\lesssim 1/D$. Hence, $\bbE_0[\|\xi_i\|^2\mid \bM_0]\leq \tr(\Sigma_\xi)\lesssim 1/D$ on $\cE_\eta$. Thus, 
    $$\xi_i\mid \bM_0 \sim N_d(0, (P_0^\perp)_{ii}\Sigma_\xi).$$
    Applying Hanson-Wright inequality for Gaussian vectors, we have
    $$\bbP_0\left(\|\xi_i\|^2 - \bbE_0[\|\xi_i\|^2\mid \bM_0] \geq 2\sqrt{\tr\left(((P_0^\perp)_{ii}\Sigma_\xi)^2\right)t} + 2\|(P_0^\perp)_{ii}\Sigma_\xi\|t \mid \bM_0\right)\leq e^{-t}.$$
    Using $(P_0^\perp)_{ii}\leq 1$ and the $\bbE_0[\|\xi_i\|^2\mid \bM_0]\leq \tr(\Sigma_\xi)$, we get
    $$\bbP_0\left(\|\xi_i\|^2 - \tr(\Sigma_\xi) \geq 2\sqrt{\tr(\Sigma_\xi^2)t} + 2\|\Sigma_\xi\|t\mid \bM_0\right) \leq e^{-t}.$$

    Since the event $\cE_\eta$ depends solely on $\bM_0$, its indicator $\ind(\cE_\eta)$ is $\bM_0-$measureable. Thus, application of union bound (and integrating out the randomness of $\bM_0$) gives
    $$\bbP\left(\left\{\max_i \|\xi_i\|^2 \geq \tr(\Sigma_\xi) + 2\sqrt{\tr(\Sigma_\xi^2)t} + 2\|\Sigma_\xi\|t\right\} \cap \cE_\eta\right) \leq \bbE_\eta\left[\sum_{i=1}^n e^{-t}\ind(\cE_\eta)\right].$$
    Setting $t=\log(n/\varepsilon)$ fixes the total failure probability of the union bound to $\varepsilon$. Now, on $\cE_{\eta}$, we can use $\|\Sigma_\xi\|\leq C_0/D$, hence $\tr(\Sigma_\xi)\leq dC_0/D$ and $\tr(\Sigma_\xi^2)\leq dC_0^2/D^2$. Plugging these back, we get that with probability 
    $$\bbP_0\left(\left\{\max_i \|\xi_i\|^2 \geq \frac{C_0}{D}[d + 2\sqrt{d\log(n/\varepsilon)} + 2\log(n/\varepsilon)] \right\} \cap \cE_\eta\right) \leq \varepsilon.$$
    Hence, the result follows. 

    \textbf{Note:} Plugging in $\varepsilon=1/n$, we get that $\max_i\|\xi_i\| \lesssim \sqrt{\frac{\log n}{D}}$ holds with probability at least $1-c/n$ for some $c>0$.
\end{proof}

\subsection{Proofs for Section \ref{subsec: latent density optimized rate} (Bandwidth selection and optimized rate)}\label{app: bandwidth}

Recall from Theorem \ref{thm: new kde},
\begin{equation}
        \begin{aligned}
            \cR_{n,D}(h) = h^{2\beta} &+ \frac{1}{nh^d} \\
            &+  \frac{1}{nDh^{d+2}} + \rho_{n,D}^2\left(\frac{1}{nh^{d+2}} + \frac{1}{h^2}\right) \\
            &+\left(\frac{1}{D^2} + \rho_{n,D}^4\right)\left(\frac{1}{nh^{d+4}} + \frac{1}{h^4}\right) \\
            &+ \frac{\log^3n}{D^3h^{d+6}} + \frac{\rho_{n,D}^6}{h^{d+6}}.
        \end{aligned}
    \end{equation}
    where $\rho_{n,D}=\sqrt{\log (nD)} / D + (\log n)/\sqrt{n}$. The following corollary shows the optimal bandwidth choice and the corresponding risk.

    \begin{corollary}\label{cor: optimal bandwidth 1}
        Define the following
        \begin{equation}
            \begin{aligned}
                h_{\rm or} &:= n^{-1/(2\beta+d)},\\
                h_{D,1} &:= D^{-1/(\beta+2)},\\
                h_{D,2} &:= D^{-1/(\beta+1)}\log^{\frac{1}{2(\beta+1)}} (nD), \\
                h_{D,3} &:= D^{-3/(2\beta+d+6)} \log^{\frac{3}{2\beta+d+6}} n,\\
                h_{n,1} &:= n^{-1/(2\beta+2)}\log^{\frac{1}{\beta+1}} n, \\
                h_{n,2} &= n^{-3/(2\beta+d+6)} \log^{\frac{6}{2\beta+d+6}} n,\\
                h_{D,4} &:= D^{-6/(2\beta+d+6)} \log^{\frac{3}{2\beta+d+6}} (nD).
            \end{aligned}
        \end{equation}
        Then it is sufficient to choose
        $$h^* \asymp \max\left\{h_{\rm or}, \; h_{D,j}, h_{n,i}:j\in[4], i\in[2]\right\}$$
        For this choice, $\cR_{n,D}(h^*) \lesssim (h^*)^{2\beta}.$ 
    \end{corollary}
    \begin{proof}
        Starting with the form of $\cR_{n,D}(h)$ from Theorem \ref{thm: new kde}, use $(a+b)^k\lesssim a^k +b^k$ for $k=2,4,6$ on $\rho_{n,D}$. Write $a_{n,D}=\sqrt{\log(nD)}/D$ and $b_n=(\log n)/\sqrt{n}$, and $q=2\beta+d+6$.

        First, we argue that the terms carrying an additional factor $n^{-1}$ do not introduce new bandwidth scales. Because $h^* \geq h_{D,1}$, $h^* \geq D^{-1/(\beta+2)}\geq D^{-1/2}$ for $D$ sufficiently large. Hence
        $$\frac{1}{nD(h^*)^{d+2}}\leq \frac{1}{n(h^*)^d}, \quad \frac{1}{nD^2(h^*)^{d+4}} \leq \frac{1}{n(h^*)^d}.$$

        Likewise, since $h^*\geq a_{n,D}^{1/(\beta+1)}=h_{D,2}$ and $h^*\geq b_n^{1/(\beta+1)}=h_{n,1}$ and $a_{n,D}, b_n = o(1)$, we have $a_{n,D} \lesssim h^*$ and $b_n\lesssim h^*$. Consequently,
        $$\frac{\rho_{n,D}^2}{n(h^*)^{d+2}} \lesssim \frac{1}{n(h^*)^d}, \quad \frac{\rho_{n,D}^4}{n(h^*)^{d+4}} \lesssim \frac{1}{n(h^*)^d}.$$
        Moreover, $\rho_{n,D}^4/(h^*)^4 \lesssim \rho_{n,D}^2/(h^*)^2$ as $\rho_{n,D} \lesssim h^*$. Therefore, it is enough to control the terms
        $$h^{2\beta}, \; \frac{1}{nh^d}, \; \frac{1}{D^2h^4},\; \frac{a_{n,D}^2}{h^2}, \; \frac{b_n^2}{h^2}, \; \frac{\log^3 n}{D^3 h^{d+6}}, \; \frac{a_{n,D}^6}{h^{d+6}}, \; \frac{b_n^6}{h^{d+6}}.$$
        Balancing each decreasing term with $h^{2\beta}$ gives precisely the bandwidths in the corollary. Since $h^*$ is at least each of these balancing bandwidths, every term in the preceding display is bounded by a constant multiple of $(h^*)^{2\beta}$, which implies $R(h^*)\lesssim (h^*)^{2\beta}$.
    \end{proof}

\begin{proof}[Proof of Corollary \ref{cor: optimized latent kde rate}]
    Recall the notations from Corollary \ref{cor: optimal bandwidth 1}. Since $\log n = D^{o(1)}$ by assumption, and automatically $\log D = D^{o(1)}$, we have $\log (nD) = \log n + \log D = D^{o(1)}$, where $a_n = b_n^{o(1)}$ means that $\log a_n = o(\log b_n)$. Hence, every fixed power of $\log n$ and $\log(nD)$ are $D^{o(1)}$. Also, every fixed power of $\log n$ is automatically $n^{o(1)}$. Also, recall 
    \begin{equation}
        \begin{aligned}
            h_D^*:=\max \{h_{D,1}, h_{D,2},h_{D,3}, h_{D,4}\}, \quad h_n^*:=\max\{h_{\rm or}, h_{n,1}, h_{n,2}\}.
        \end{aligned}
    \end{equation}

    These together imply $h_{D,2}/h_{D,1}\to 0$ and $h_{D,4}/h_{D,3}\to 0$. Comparing $h_{D,1}$ and $h_{D,3}$, the polynomial parts tell us
    $$\frac{1}{\beta+2}\lesseqgtr \frac{3}{2\beta+d+6} \iff \beta \gtreqless d.$$
    Since a smaller exponent gives a larger bandwidth, we have the form of $h_D^*$ as given. Note that at $\beta=d$, the polynomial parts match, but $h_{D,3}$ with the log term dominates.

    For the $n-$dependent scales, first suppose $d\geq 3$. Then, $1/(2\beta+d)\leq 1/(2\beta+2)$ and $1/(2\beta+d) \leq 3/(2\beta+d+6)$, which follows from the fact that $2\beta+d\geq 3$. Thus $h_{\rm or}$ is the largest $n-$dependent bandwidth for $d\geq 3$.

    Next, suppose $d\leq 2$. Then $1/(2\beta+2)\leq 1/(2\beta+d)$. So $h_{n,1}$ is at least as large as $h_{\rm or}$. Finally 
    $$\frac{1}{2\beta+2} \lesseqgtr \frac{3}{2\beta+d+6} \iff 4\beta \gtreqless d.$$
    This gives the $h_n^*$, and the remaining conclusions follow immediately.
\end{proof}

The following corollary supplies the phase diagram corresponding to the last corollary in the main text.

\begin{corollary}\label{cor: optimal bandwidth 3}
    Under the conditions of the preceding corollary, define the critical $D-$ thresholds as:
    \begin{equation}
        \begin{aligned}
            D_c(n):=\begin{cases}
                n^{\frac{\beta+2}{2\beta+d}}, &d\geq 3, \beta > d,\\
                n^{\frac{2\beta+d+6}{6\beta+3d}}(\log n), &d\geq 3, \beta\leq d,\\
                n^{\frac{\beta+2}{2(\beta+1)}}(\log n)^{-\frac{\beta+2}{\beta+1}}, & d\leq 2, \beta>d,\\
                n^{\frac{2\beta+d+6}{6(\beta+1)}} (\log n)^{\frac{\beta-d-3}{3(\beta+1)}}, & d\leq 2, d/4 \leq \beta \leq d,\\
                n (\log n)^{-1}, &d\leq 2, 0<\beta<d/4.
            \end{cases}     
        \end{aligned}
    \end{equation}
    Denoting $q:=2\beta+d+6$,  the corresponding risk $\cR_{n,D}^*:=\cR_{n,D}(h^*)=(h^*)^{2\beta}$ satisfies the following:
    \begin{enumerate}[label={Case (\arabic*)}]
        \item If $d\geq 3, \beta > d$, then
        \begin{align*}
            \cR_{n,D}^* \asymp \begin{cases}
                n^{-2\beta/(2\beta+d)}, & D\gtrsim D_c(n), \\
                D^{-2\beta/(\beta+2)}, & D\lesssim D_c(n).
            \end{cases}
        \end{align*}
        \item If $d\geq 3$ and $\beta \leq d$, then
        \begin{align*}
            \cR_{n,D}^* \asymp \begin{cases}
                n^{-2\beta/(2\beta+d)}, & D\gtrsim D_c(n), \\
                D^{-6\beta/q}(\log n)^{6\beta/q}, & D\lesssim D_c(n).
            \end{cases}
        \end{align*}
        \item If $d\leq 2$ and $\beta > d$, then
        \begin{align*}
            \cR_{n,D}^* \asymp \begin{cases}
                n^{-\beta/(\beta+1)}(\log n)^{2\beta/(\beta+1)}, & D\gtrsim D_c(n), \\
                D^{-2\beta/(\beta+2)}, & D\lesssim D_c(n).
            \end{cases}
        \end{align*}
        \item If $d\leq 2$ and $d/4 \leq \beta \leq d$, then
        \begin{align*}
            \cR_{n,D}^* \asymp \begin{cases}
                n^{-\beta/(\beta+1)}(\log n)^{2\beta/(\beta+1)}, & D\gtrsim D_c(n), \\
                D^{-6\beta/q}(\log n)^{6\beta/q}, & D\lesssim D_c(n).
            \end{cases}
        \end{align*}
        \item If $d\leq 2$ and $0<\beta < d/4$, then
        \begin{align*}
            \cR_{n,D}^* \asymp \begin{cases}
                n^{-6\beta/q}(\log n)^{12\beta/q}, & D\gtrsim D_c(n), \\
                D^{-6\beta/q}(\log n)^{6\beta/q}, & D\lesssim D_c(n).
            \end{cases}
        \end{align*}
    \end{enumerate}
\end{corollary}

\begin{proof}[Proof of Corollary \ref{cor: optimal bandwidth 3}]
    From the preceding Corollary, $h^*\asymp \max\{h_n^*, h_D^*\}$, with the forms of $h_n^*$ and $h_D^*$ given in that Corollary, for each of the five regimes here. Since $h_D^*$ decreases with $D$, the $n-$dependent bandwidth dominates when $D$ is larger than the value at which $h_D^*=h_n^*$, whereas the $D-$dependent bandwidth dominates when $D$ is smaller. We therefore obtain $D_c(n)$ by equating the appropriate expressions within each regime, which gives the results here.
\end{proof}

\subsection{Proofs for Section \ref{subsec: latent density lower bound} (Lower bound for latent density estimation)}

 The main idea is to perturb the Gaussian density by $\phi_d(x)\cos(Tx_1)$. This has two useful properties. First, 
$$|\phi_d(x)\cos(Tx_1)|\leq \phi_d(x),$$
so sufficiently small multiplicative perturbations retain Gaussian tails. So for sufficiently small $a>0$, 
$$\phi_d(x)\{1 \pm aT^{-\beta} \cos(T x_1)\} \geq 0.$$
The factor $T^{-\beta}$ controls its amplitude. We choose large $T\asymp \sqrt{D\log n}$.

Second, let us look at the effect of this perturbation when we pass from latent density to the observation density. Consider the oracle experiment $X = \eta + D^{-1/2} Z$. If $\eta \sim g$, then $X$ has density $q_D = g\star \phi_{D^{-1}I_d}.$ Therefore, if the two latent alternatives differ approximately by
$$g_{+,T} - g_{-,T} \approx 2aT^{-\beta} \phi_d(x) \cos (Tx_1)$$
then the corresponding observed densities differ by
$$q_{+,T} - q_{-,T} \approx 2a T^{-\beta} \left[\{\phi_d(\cdot) \cos(T \cdot_1)\} \star \phi_{D^{-1} I_d}\right].$$
To see how Gaussian convolution attenuates the oscillation, consider $d=1$, and suppose the perturbation is $\phi(\eta) \cos(T\eta)$. After adding Gaussian noise, the perturbation becomes exactly
$$\phi(x\mid 0, 1+1/D) \exp\left\{-\frac{T^2}{2(D+1)}\right\} \cos \left(\frac{DT}{D+1}x\right).$$
The key factor is $\exp\{-T^2/2(D+1)\}$, which for large $D$, is $\approx e^{-T^2/2D}$. Thus, convolution with Gaussian noise attenuates the oscillation exponentially. This is precisely the supersmooth deconvolution phenomenon responsible for the $D-$dependent lower bound.

Note that the raw oscillatory perturbation does not preserve mass and covariance exactly. Its defects, however, are exponentially small:
$$\int \phi_d(x) \cos (Tx_1) dx = e^{-T^2/2}$$
and second moment defects are at most polynomial in $T$ times $e^{-T^2/2}$. We use the following elementary fact.

\begin{lemma}\label{lemma: moment correction}
    There exist bounded, smooth, even functions $b_0,\dots,b_d$ depending only on $d$, such that for every sufficiently large $T$, one can choose coefficients $c_k(T)=O((1+T^2)e^{-T^2/2})$ for which
    $$F_T(x):= \cos(Tx_1) - \sum_{k=0}^d c_k(T) b_k(x)$$
    satisfies
    $$\bbE [F_T(X)] = 0, \quad \bbE [X_j^2 F_T(X)] = 0, \quad j=1,\dots, d.$$
    The functions may be chosen even in every coordinate, so $\bbE_0[X_j F_T(X)] = 0$ and all off-diagonal second moments vanish as well. Note all expectations are under $X\sim N(0,I_d)$.
\end{lemma}
\begin{proof}
    To see this, choose bounded even $b_0,\dots,b_d$ for which the $(d+1)\times (d+1)$ matrix
    $$\left(\bbE[b_k(X)], \bbE[X_1^2 b_k(X)],\dots, \bbE[X_d^2 b_k(X)]\right)_{k=0}^d$$
    is nonsingular (call this matrix $M$). Such functions can be obtained, for example, by choosing $d+1$ points $z_0,\dots,z_d$ such that the vectors $(1, z_{k1}^2,\dots,z_{kd}^2)$ are linearly independent, and choose sufficiently localized even $C^\infty$ bumps around their coordinatewise sign orbits. Solving the resulting fixed finite-dimensional linear system $M c(T) = r(T)$ gives the coefficients $c_k(T)$, where $$r(T)=(\bbE[\cos(TX_1)], \bbE[X_1^2 \cos(TX_1)],\dots, \bbE[X_d^2\cos(TX_1)]).$$ Note that $\bbE[X_j^2 \cos(TX_1)]$ is $(1-T^2)e^{-T^2/2}$ for $j=1$ and $e^{-T^2/2}$ for $j\neq 1$, while $\bbE[\cos(TX_1)]=e^{-T^2/2}$ (using the characteristic function of standard normal), which implies $r(T)=e^{-T^2/2}(1, 1-T^2,1,\dots,1)$ and consequently $\|r(T)\|\lesssim (1+T^2)e^{-T^2/2}$. Thus, $\|c(T)\|\leq \|M^{-1}\|\|r(T)\|\lesssim (1+T^2)e^{-T^2/2}$. Hence, for every $k$
    $$|c_k(T) |\lesssim (1+T^2) e^{-T^2/2}.$$
\end{proof}

Thus, 
\begin{equation}\label{eq: moment corrected oscillation bounded}
    \sup_x |F_T(x)| \leq C
\end{equation}
uniformly for large $T$, while the correction to $\cos(T x_1)$ is exponentially small.

\medskip
    \noindent
    \textbf{Proof for the lower bound for latent density estimation}

\begin{proof}[Proof of Theorem \ref{thm: latent lower bound}]
    Let $\Theta_{D,\beta}$ denote the full parameter space of triples $\theta=(g,\Lambda,\Sigma)$ satisfying the assumptions in the theorem. For the $n^{-2\beta/(2\beta+d)}$, we use a modification of the classical Assouad construction for nonparametric density estimation lower bounds, together with Fano's lemma. For the $(D\log n)^{-\beta}$ part, we use Le Cam's two point technique for the lower bound with the quotient metric $d_{\cO}^2(g_1, g_2) = \inf_{R\in \cO(d)} \|g_1 - g_2^R\|_{L^2}^2$.

    \medskip
    \noindent
    \textbf{Step 1: Reduction to an admissible oracle submodel}. A lower bound over any submodel of $\Theta_{D,\beta}$ is a lower bound for the full problem. We therefore fix $\Sigma=I_D$ and a known loading matrix $\Lambda_D$ satisfying $\Lambda_D^\top\Lambda_D = DI_d$ and $\max_j \|\lambda_{Dj}\|^2 \leq C_0$. Such a dense loading matrix can be constructed by taking $\Lambda_D = \sqrt{D} U_D$, where $U_D\in\bbR^{D\times d}$ has orthonormal columns and $\max_j \|(U_D)_{j\cdot}\|^2 \lesssim D^{-1}$. For instance, one may take $d$ columns of a Hadamard-type orthogonal matrix. Hence $s_1(\Lambda_D)=\dots=s_d(\Lambda_D)=\sqrt{D}$ while its row norms remain uniformly bounded. Thus, this oracle submodel satisfies the loading assumptions. 

    Under this submodel, $Y_i = \Lambda_D \eta_i + \epsilon_i$ with $\epsilon_i\sim N(0,I_D)$. Consider the projection
\begin{equation}\label{eq: modified experiment}
        X_i = \frac1D \Lambda_D^\top Y_i = \eta_i + \underbrace{\frac1D \Lambda_D^\top \epsilon_i}_{\sim N_d(0, \frac1D I_d)}.
    \end{equation}
    The component of $Y_i$ orthogonal to $\rm col(\Lambda_D)$ is independent Gaussian noise and carries no information about $g$. Thus, estimation of $g$ in this oracle submodel is statistically equivalent to estimation from 
    $$X_1,\dots, X_n \overset{iid}{\sim} q_{g,D} := g \star \phi_{D^{-1} I_d}.$$
     We establish the two terms in the lower bound separately.

    \medskip
    \noindent
    \textbf{Step 2: Usual nonparametric lower bound}.
    Considering the more informative experiment in which estimator 
    $\tilde{g}$ is based on 
    directly observed
    $\eta_1,\dots,\eta_n$. The experiment \eqref{eq: modified experiment} is obtained from these observations by applying independently the Markov kernel $\eta \mapsto \eta + D^{-1/2} Z$. Consequently, by monotonicity of minimax risk under such a transform:
    $$\mathfrak{R}_{n,D} \geq \inf_{\tilde{g}}\sup_{g\in \cG_{\beta}(L,C)} \bbE_g \left[ \inf_{R\in\cO(d)} \|\tilde{g} - g^R\|_{L^2}^2\right].$$
    To show that the right side is the usual classical nonparametric rate $n^{-2\beta/(2\beta+d)}$, we argue that a slight modification of the classical Assouad construction (accounting for the quotient metric) works. See Lemma \ref{lemma: low dim minimax rotation}.

    \medskip
    \noindent
    \textbf{Step 3: Construction of two high-frequency alternatives}. Take as baseline $g_0(x)=\phi_d(x)$. For $T\to \infty$, let $F_T$ be the moment-corrected oscillatory function from Lemma \ref{lemma: moment correction} and define
    \begin{equation}
        g_{\pm, T}(x) = \phi_d(x)\left\{ 1 \pm a T^{-\beta} F_T(x)\right\}
    \end{equation}
    where $a>0$ is a sufficiently small fixed constant.

    By construction, $\bbE[F_T(X)]=0$ and hence $\int g_{\pm,T}(x)dx =1$. Similarly, $\bbE[X]=0$ and $\bbE[XX^\top]=I_d$ under $g_{\pm,T}$. Moreover, by \eqref{eq: moment corrected oscillation bounded}, $|aT^{-\beta} F_T(x)| \leq aC$ for $T\geq 1$. Choosing $a$ sufficiently small gives
    $$\frac12\phi_d(x) \leq g_{\pm, T}(x) \leq \frac32 \phi_d(x).$$
    Thus, the alternatives are valid densities and sub-Gaussian. 

    Next, we verify the Sobolev condition. First, ignore the exponentially small moment correction and consider $v_T(x)=T^{-\beta} \phi_d(x) \cos(Tx_1)$. Its Fourier transform is, up to the fixed Fourier normalization, 
    $$\varphi_{v_T}(t) = \frac{T^{-\beta}}{2} \left[e^{-\|t-Te_1\|^2/2} + e^{-\|t + Te_1\|^2/2} \right].$$
    Thus its Fourier energy is concentrated in $O(1)$ neighborhoods of $Te_1$ and $-Te_1$ (where $e_1=(1,0\dots,0)^\top\in\bbR^d$ is the unit vector). On these neighborhoods, $1+\|t\|^2 \asymp T^2$. Consequently,
    $$\int (1+\|t\|^2)^\beta |\varphi_{v_T}(t)|^2 dt \lesssim T^{-2\beta} T^{2\beta} \lesssim 1.$$
    The correction in $F_T$ is $-\sum_{k=0}^d c_k(T)b_k(x)$, where $|c_k(T)|\lesssim (1+T^2)e^{-T^2/2}$. Since the $b_k's$ are fixed smooth functions, their contribution to the Sobolev norm after multiplication by $\phi_d$ is exponentially small. It follows that
    $$\sup_T \int (1+\|t\|^2)^\beta |\varphi_{g_{\pm,T}}(t)|^2 dt <\infty.$$
    Taking $a$ sufficiently small therefore ensures $g_{\pm, T}\in \cG_{\beta}(L,C)$ uniformly in $T$. Since $\cG_{\beta}(L,C)$ is invariant under the action of $\cO(d)$, $g_{\pm, T}^R\in \cG_{\beta}(L,C)$ for any rotation $R\in\cO(d)$.

    \medskip
    \noindent
    \textbf{Step 4: Separation of latent densities modulo rotations}. For any $R\in \cO(d)$, let $v=R^\top e_1\in \bbS^{d-1}$ (the unit sphere). Since $\phi_d$ is rotationally invariant,
    $$g_{-,T}^R(x) = \phi_d(x) \left[1 - a T^{-\beta} F_T(Rx)\right].$$
    Since $F_T(Rx) = \cos(Tv^\top x) + o(1)$ uniformly, as $(1+T^2)e^{-T^2/2} = o(1)$ as $T\to\infty$, consequently
    \begin{equation}\label{eq: g pm T difference}
        g_{+,T}(x) - g_{-,T}^R(x) = aT^{-\beta} \phi_d(x) \left[\cos(Te_1^\top x) + \cos(Tv^\top x)\right] + o(T^{-\beta}) \phi_d(x),
    \end{equation}
    uniformly over $R$. Now,
    \begin{align*}
        &\int \phi_d(x)^2\left[\cos(Te_1^\top x)+\cos(Tv^\top x)\right]^2 dx \\
        &=\int \phi_d(x)^2 \cos^2(Te_1^\top x) dx + \int \phi_d(x)^2 \cos^2(T v^\top x)dx \\
        &\quad \quad + 2\int \phi_d(x)^2 \cos(Te_1^\top x) \cos(T v^\top x) dx.
    \end{align*}
    By rotational invariance, the first two terms are equal, and 
    $\int \phi_d(x)^2 \cos^2 (Te_1^\top x) dx = \frac12 (1+e^{-T^2})\int \phi_d(x)^2 dx.$
    The cross term is nonnegative, since $\cos a \cos b = \frac12[\cos (a-b) + \cos (a+b)]$ and the Fourier transform of $\phi_d^2$ is a positive Gaussian. Thus, uniformly over $R\in\cO(d)$, 
    the integral in the above display is $\geq c_2>0$ for sufficiently large $T$. Using Equation \eqref{eq: g pm T difference} and the exponentially small correction,
    \begin{equation}
        \inf_{R\in \cO(d)} \|g_{+,T} - g_{-,T}^R\|_{L^2}^2 \asymp T^{-2\beta}.
    \end{equation}
    
    Note that this also shows that $\inf_{R\in \cO(d)} \|g_{+,T}^R - g_{-,T}\|_{L^2}^2 \asymp T^{-2\beta}$ since $\|f\|_{L^2}^2 = \|f^R\|_{L^2}^2$. Thus, using the definition of the quotient metric $d_{\cO}$, we have 
    \begin{equation}
        d_{\cO}^2(g_{+,T}, g_{-,T}) \asymp T^{-2\beta}.
    \end{equation}

    \medskip
    \noindent
    \textbf{Step 5: Attenuation under Gaussian convolution}. Let $\tau_D^2 = D^{-1}$ and $s_D^2 = 1+1/D$. Note the following key identity in one dimension:
    \begin{equation}\label{eq: convolution identity}
        \left(\left\{\phi(\cdot)\cos(T\cdot)\right\} \star \phi_{\tau_D^2}\right)(y) = \phi_{s_D^2}(y) \exp\left\{-\frac{T^2}{2(D+1)}\right\} \cos\left(\frac{DT}{D+1}y\right),
    \end{equation}
    where $\phi_{r^2}(y)=\phi(y\mid 0, r^2)$. To verify this, let $X\sim N(0,1)$ and $Y=X+D^{-1/2}Z$. Then $X|Y=y\sim N(Dy/(D+1), 1/(D+1))$, and hence $\bbE[\cos(TX)\mid Y=y]=\exp\{-T^2/2(D+1)\}\cos(DTy/(D+1))$ which gives the above display. The same identity can be used in our case in dimension $d$, since the oscillation only occurs in the first coordinate.

    Let $q_{\pm,T}=g_{\pm,T}\star \phi_{D^{-1} I_d}$. Since $g_{-,T}\geq \phi_d/2$ (from Step 3), convolution preserves the inequality:
    $$q_{-,T}(y) \geq \frac12 \phi_{s_D^2 I_d}(y).$$
    The leading part of $q_{+,T}-q_{-,T}$ is, by \eqref{eq: convolution identity}, $2aT^{-\beta} \phi_{s_D^2 I_d}(y) \exp\{ -T^2/(2(D+1))\}\cos (DTy_1 / (D+1))$. The moment correction part has coefficients $O((1+T^2)e^{-T^2/2})$. Since each $b_k$ is bounded, $|\phi_d b_k| \star \phi_{D^{-1}I_d}\leq \|b_k\|_{\infty} \phi_{s_D^2 I_d}$. Consequently, using the lower bound for $q_{-,T}$, we have
    \begin{equation}\label{eq: chi square separation}
        \chi^2(q_{+,T}, q_{-,T}) \lesssim T^{-2\beta} \left[e^{-T^2/(D+1)} + (1+T^2)^2 e^{-T^2}\right] \lesssim T^{-2\beta} e^{-T^2/(D+1)},
    \end{equation}
    where $\chi^2(p,q)=\int (p-q)^2/q$ is the chi squared divergence.

    \medskip
    \noindent
    \textbf{Step 6: Application of Le Cam's lemma}. Take $T^2=A(D+1)\log n$ for sufficiently large fixed constant $A>1$. Then $e^{-T^2/(D+1)}=n^{-A}$, and hence
    $$n\chi^2(q_{+,T}, q_{-,T}) \lesssim n^{1-A}T^{-2\beta}\to 0.$$
    Using $1+\chi^2(p^{\otimes n}, q^{\otimes n}) =[1 + \chi^2(p,q)]^n$ and $\TV(p, q) \leq \frac12\sqrt{\chi^2(p,q)}$, we get
    $$\TV(q_{+,T}^{\otimes n}, q_{-,T}^{\otimes n}) \leq c_0 < 1$$
    for all sufficiently large $n$.

    Le Cam's two-point lemma now gives
    \begin{align*}
        \mathfrak{R}_{n,D} &\geq \inf_{\hat{g}}\max_{\theta\in \{+,-\}} \bbE_\theta d_{\cO}^2(\hat{g}, g_{\theta, T}) \\
        &\gtrsim d_{\cO}^2(g_{+,T}, g_{-,T})[1 - \TV(q_{+,T}^{\otimes n}, q_{-,T}^{\otimes n})]\\
        &\gtrsim T^{-2\beta} \asymp (D\log n)^{-\beta}.
    \end{align*}

    Combining the two mechanisms, we get the desired result.
\end{proof}

\begin{lemma}\label{lemma: rotation lemma}
    Let $g_{\theta}$ be as defined in Equation \eqref{eq: Assouad alternatives}. There exists constants $C_1,\delta_0>0$ independent of $h, \theta, R$ and $S$ such that
    $$\|g_\theta^R - g_{\theta}^S\|_{L^2} \leq C_1(1+h^{\beta-1}) \|R - S\|$$
    for all $R, S\in \cO(d)$, whenever $\|R - S\|\leq \delta_0$.
\end{lemma}
\begin{proof}
    Let $f\in C^1$ with $\|\|x\| \nabla f(x)\|_{L^2} <\infty$. For $R, S\in \cO(d)$ define
    $$R_t = S + t(R-S), \quad t\in [0,1].$$
    By the fundamental theorem of calculus, 
    \begin{align*}
        f(Rx) - f(Sx) &= \int_0^1 \nabla f(R_tx)^\top (R-S) x dt \\
        &\Longrightarrow \; |f(Rx) - f(Sx)| \leq \|R-S\| \|x\| \int_0^1 \|\nabla f(R_t x)\|dt.
    \end{align*}

    Choose $\delta_0>0$ sufficiently small. If $\|R-S\|\leq \delta_0$ then
    $R_t = S[I + t(S^\top R - I)]$ and $\|S^\top R - I\|=\|R-S\|\leq \delta_0$. Thus, the matrices $R_t$ are uniformly invertible over $t\in [0,1]$ with
    $$\sup_{t\in[0,1]} \|R_t^{-1}\| \leq C, \quad \inf_{t\in[0,1]} |\det R_t|\geq c>0$$
    where $c, C$ depend only on $\delta_0$. Taking $L^2$ norm in the previous display and applying Minkowski's integral inequality,
    $$\|f^R - f^S\|_{L^2} \leq \|R-S\| \int_0^1 \left(\int \|x\|^2 \|\nabla f(R_tx)\|^2 dx\right)^{1/2} dt.$$
    For fixed $t$, making the change of variable $y=R_tx$ gives
    \begin{align*}
        \int \|x\|^2 \|\nabla f(R_tx)\|^2 dx &= |\det R_t|^{-1} \int \|R_t^{-1} y\|^2 \|\nabla f(y)\|^2 dy \lesssim \int \|y\|^2 \|\nabla f(y)\|^2 dy.
    \end{align*}
    Therefore,
    $$\|f^R - f^S\|_{L^2} \lesssim \|R - S\| \|\|x\|\nabla f(x)\|_{L^2}.$$

    We apply this to $f=g_\theta$. Since $g_*$ is fixed and Gaussian-tailed, $\|\|x\|\nabla g_*(x)\|<\infty$. For the perturbation, $p_{\theta,h}(x)=ah^\beta\sum_{k=1}^M \theta_k \psi_{k,h}(x)$, the supports are contained in the fixed compact set $K$, so $\|x\|\lesssim 1$. By disjointness, 
    $$\|\nabla p_{\theta,h}\|_{L^2}^2 = a^2 h^{2\beta}\sum_{k=1}^M \|\nabla \psi_{k,h}\|_{L^2}^2 \lesssim h^{2\beta-2}$$
    since $\|\nabla \psi_{k,h}\|_{L^2}^2 = h^{d-2} \|\nabla\psi\|_{L^2}^2$ and $M\asymp h^{-d}$. Thus, 
    $$\|\|x\| \nabla g_\theta\|_{L^2} \leq \| \|x\|\nabla g_*\|_{L^2} + \|\|x\| \nabla p_{\theta,h}\|_{L^2} \leq O(1) + C_K \|\nabla p_{\theta,h}\|_{L^2}\lesssim 1 + h^{\beta-1}$$
    uniformly over $\theta$, using $g_*$ being smooth, fixed with Gaussian tail to control the first term and $\|x\|\leq C_K$ as all bumps are contained in a fixed compact $K$. Hence,
    $$\|g_\theta^R - g_\theta^S\|_{L^2} \leq C_1(1+h^{\beta-1}) \|R - S\|.$$
\end{proof}

\begin{lemma}\label{lemma: low dim minimax rotation}
    Consider the task of estimating a probability density $g\in\cG_{\beta}(L,C)$ on $\bbR^d$ up to a rotation from $X_1,\dots,X_n\sim g$ iid under the metric $d_{\cO}^2(\hat{g}, g)=\inf_{R\in\cO(d)} \|\hat{g} - g^R\|_{L^2}^2$, where $g^R(x)=g(Rx)$. Then the minimax risk is lower bounded by
    $$\inf_{\hat{g}} \sup_{g\in \cG_\beta(L,C)} \bbE_g [d_{\cO}^2(\hat{g}, g)] \gtrsim n^{-2\beta/(2\beta+d)}.$$
\end{lemma}

\begin{proof}
    We construct a finite family of alternatives (by taking an appropriate subset of the classical Assouad construction) that remain exponentially large after quotienting by the orthogonal group. 

    Let $g_*\in\cG_\beta(L,C)$ be a fixed smooth density and choose a compact set $K\subset \bbR^d$ such that $\inf_{x\in K} g_*(x) \geq c_* >0$. Choose a nonzero function $\psi\in C_c^\infty (\bbR^d)$, supported in a sufficiently small cube, satisfying the moment conditions $\int \psi(u)du=0$, $\int u \psi(u) du = 0$ and $\int uu^\top \psi(u)du = 0$. For sufficiently small $h>0$ (to be chosen later), set $M\asymp h^{-d}$ centers $x_1,\dots,x_M\in K$ such that the functions 
    $$\psi_{k,h}(x) =\psi\left(\frac{x-x_k}{h}\right), \quad k=1,\dots,M,$$
    have mutually disjoint supports contained in $K$. 

    For $\theta=(\theta_1,\dots,\theta_M)\in \{-1,1\}^M$, define
    \begin{equation}\label{eq: Assouad alternatives}
        g_{\theta}(x) = g_*(x) + ah^\beta \sum_{k=1}^M \theta_k \psi_{k,h}(x),
    \end{equation}
    where $a>0$ is a sufficiently small constant. By the moment conditions and the standard smooth bump construction, $a$ may be chosen independently of $h$ such that $g_\theta\in \cG_{\beta}(L,C)$ for every $\theta$ and all sufficiently small $h$.

    We first construct an ordinary $L^2-$packing of $\{-1,1\}^M$. By the Varshamov-Gilbert bound, there exists a subset $\Theta_0\subset \{-1,1\}^M$ such that $\log |\Theta_0|\geq cM$ and $H(\theta, \theta')\geq c_0 M$ for all distinct $\theta,\theta'\in\Theta_0$, where $H$ is the Hamming distance. Now, since the bump supports are disjoint, 
    $$\|g_\theta - g_{\theta'}\|_{L^2}^2 = a^2 h^{2\beta} \sum_{k=1}^M (\theta_k - \theta_{k'})^2\|\psi_{k,h}\|_{L^2}^2 =4a^2 h^{2\beta}H(\theta,\theta') h^d \|\psi\|_{L^2}^2.$$
    Since $H(\theta,\theta')\gtrsim M \asymp h^{-d}$, there exists $c_1>0$ such that
    $$\|g_\theta - g_{\theta'}\|_{L^2} \geq c_1 h^{\beta}, \quad \theta\neq\theta', \quad \theta, \theta'\in \Theta_0.$$

    We next extract from $\Theta_0$ an exponentially large packing under the quotient metric $d_{\cO}$. Fix $0<\gamma < c_1/4$ and set $L_h=C_1(1+h^{\beta-1})$, and $\delta_h = \gamma h^\beta/L_h$. Then
    $$\delta_h \asymp \begin{cases}
        h^\beta, & \beta\geq 1,\\
        h, &0<\beta<1
    \end{cases}$$
    and hence $\delta_h\to 0$. In particular, for sufficiently small $h$, $\delta_h \leq \delta_0$, where $\delta_0, C_1$ are the constants from Lemma \ref{lemma: rotation lemma}.

    \medskip
    \noindent
    Since $\cO(d)$ is a compact manifold of fixed dimension $d(d-1)/2$, it admits a $\delta_h-$net $\{R_1,\dots,R_{N_h}\}$ with $N_h\leq C_d\delta_h^{-q}$ for some fixed $q<\infty$. For example, one may take $q=d(d-1)/2$ up to an irrelevant constant accounting for the two connected components of $\cO(d)$. Consequently, 
    $$\log N_h = O(\log (1/h)) = o(h^{-d}) = o(M).$$

    Fix $\theta\in \Theta_0$ and define
    $$\cN_\theta = \{\theta'\in \Theta_0 : d_{\cO}(g_\theta, g_{\theta'}) < \gamma h^\beta\}.$$
    For each $\theta'\in \cN_\theta$, there exists $R\in \cO(d)$ such that $\|g_\theta - g_{\theta'}^R\|_{L^2} < \gamma h^\beta$. Choose $R_j$ from the $\delta_h-$net
 satisfying $\|R-R_j\| \leq \delta_h$. By Lemma \ref{lemma: rotation lemma}, 
 $$\|g_{\theta'}^R - g_{\theta'}^{R_j}\|_{L^2} \leq L_h\delta_h =\gamma h^\beta\; \Longrightarrow \; \|g_\theta - g_{\theta'}^{R_j}\|_{L^2} < 2\gamma h^\beta.$$

 For each fixed $R_j$, at most one $\theta'\in \cN_\theta$ can satisfy this inequality. To see this, suppose two distinct $\theta_1', \theta_2'$ both corresponded to the same $R_j$. Using the invariance of the $L^2$ norm under rotations,
 \begin{align*}
     \|g_{\theta_1'} - g_{\theta_2'}\|_{L^2} & = \left\|g_{\theta_1'}^{R_j} - g_{\theta_2'}^{R_j}\right\|_{L^2} \\
     &\leq \left\|g_{\theta_1'}^{R_j} - g_\theta\right\|_{L^2} + \left\|g_{\theta_2'}^{R_j} - g_\theta\right\|_{L^2} \\
     &< 4\gamma h^\beta < c_1 h^\beta,
 \end{align*}
 contradicting the ordinary $L^2$-packing property. Therefore, for every $\theta\in \Theta_0$, $|\cN_\theta| \leq N_h$.

 \medskip
 \noindent
 Now, construct $\Theta_1 \subset \Theta_0$ greedily as follows. Start with $\Theta_1 = \{\}$ and $B=\Theta_0$. At each stage, pick $\theta\in B$ and $\Theta_1 \to \Theta_1 \cup \{\theta\}$, and remove all elements at quotient distance less than $\gamma h^\beta$, i.e., $B \to B \setminus \cN_\theta$. Continue until $B$ is empty. Since each selected point removes at most $N_h$ elements, 
 $$|\Theta_1| \geq \frac{|\Theta_0|}{N_h}.$$

 Moreover 
 $$d_{\cO}(g_\theta, g_{\theta'})\geq \gamma h^\beta,\quad \text{for all } \theta \neq \theta' \in \Theta_1.$$
 Finally
 $$\log |\Theta_1|\geq \log |\Theta_0| - \log N_h \gtrsim h^{-d} - O(\log (1/h)) \gtrsim h^{-d}.$$
 Thus, $\{g_\theta: \theta\in \Theta_1\}$ is an exponentially large packing under $d_{\cO}$ with mutual separation of order $h^\beta$.

 It remains to verify the information bound. Since all perturbations are supported on $K$, on which $g_*\geq c_* > 0$,
 $$\KL(g_\theta, g_*) \leq \int \frac{(g_\theta - g_*)^2}{g_*} \lesssim \|g_\theta - g_*\|_{L^2}^2.$$

 By disjointness,
 \begin{align*}
     \|g_\theta - g_*\|_{L^2}^2 &= a^2 h^{2\beta} \sum_{k=1}^M \|\psi_{k,h}\|_{L^2}^2 = a^2h^{2\beta} Mh^d \|\psi\|_{L^2}^2 \lesssim h^{2\beta}\;\Rightarrow \; \KL\left(P_{\theta}^{\otimes n} \Vert P_*^{\otimes n}\right) \lesssim nh^{2\beta}.
 \end{align*}

 Choose $h=c n^{-1/(2\beta+d)}$ with $c>0$ sufficiently small. Since $nh^{2\beta} =c^{2\beta+d}h^{-d}$ and $\log |\Theta_1|\gtrsim h^{-d}$, $c$ can be chosen so that 
 $$\max_{\theta \in \Theta_1} \KL\left(P_{\theta}^{\otimes n} \Vert P_*^{\otimes n}\right) \leq \alpha \log |\Theta_1|$$
 for some sufficiently small fixed $\alpha\in (0,1)$. Fano's lemma applied with $\Theta_1$ and metric $d_{\cO}$ gives
 $$\inf_{\hat{g}} \sup_{\theta\in \Theta_1} \bbE_\theta \left[d_{\cO}^2(\hat{g}, g_\theta)\right] \gtrsim h^{2\beta}.$$
 Since $\{g_\theta: \theta\in \Theta_1\} \subset \cG_\beta (L,C)$, 
$$\inf_{\hat{g}} \sup_{g\in \cG_{\beta}(L,C)} \bbE_g\left[d_{\cO}^2(\hat{g}, g)\right] \gtrsim h^{2\beta}.$$
 This proves the result.
 \end{proof}

 \subsection{Additional results}
For completeness, we present a few relatively standard concentration results, and some subspace perturbation results that were used in the proof above. The following is a  standard concentration result for sub-Gaussian random variables \citep{vershynin2019high}.

\begin{lemma}\label{lemma: subgaussian concentration}[Sub-Gaussian concentration]
    Let $A$ be an $m\times n$ matrix  whose rows $A_i$ are independent zero-mean sub-Gaussian isotropic random vectors in $\bbR^n$. Then for any $t\geq 0$, we have
    $$\sqrt{m} - CK^2(\sqrt{n}+t)\leq s_{\min}(A)\leq s_{\max}(A)\leq \sqrt{m} + CK^2 (\sqrt{n}+t)$$
    with probability at least $1-2\exp(-t^2)$, where $K=\max_i \|A_i\|_{\psi_2}$. In fact, with the same probability, the following holds
    $$\norm{\frac{1}{m}A^\top A - I_n}\leq K^2 \max\{\delta,\delta^2\}, \quad \delta = C(\sqrt{n/m} + t/\sqrt{m}).$$
\end{lemma}

Using the above concentration result, we have a high-probability concentration of the matrix $(1/n)M_0^\top M_0$ noting that each row of $M_0$ is sub-Gaussian. In fact, taking $t=\sqrt{\log n}$, we get that with $g_0$-probability at least $1-2/n$, 
$$\norm{\frac{1}{n}M_0^\top M_0 - I_d} \leq C_\eta^2 \sqrt{\frac{\log n}{n}}$$
and with the same probability, $s_{\min}(M_0)\asymp s_{\max}(M_0)\asymp \sqrt{n}$.

We discuss some existing results from matrix perturbation theory for completeness. First, suppose $U_0$ and $U_1$ are two $n\times d$ orthogonal matrices, i.e., $U_0^\top U_0 = U_1^\top U_1 = I_d$. They span two $d-$dimensional subspaces $\cU_0=\rm span(U_0)$ and $\cU_1=\rm span(U_1)$. The principal angles between these subspaces are defined as follows. Let $U_0^\top U_1 = W\Sigma V^\top$ be the SVD, with $\Sigma=\diag(\sigma_1,\dots,\sigma_d)$ with $1\geq \sigma_1\geq \dots \geq \sigma_d\geq 0$. The principal angles are defined by
$$\sigma_i = \cos \theta_i,\qquad i=1,\dots,d$$
with $0\leq \theta_1 \leq \cdots \leq \theta_d \leq \pi/2$, and $V\in \bbR^{d\times d}.$

\begin{lemma}[Principal angles]\label{lemma: projection sine}
    Using the notations above, we have
    $$\|(I_n - U_0U_0^\top) U_1\| = \|\sin \Theta(U_0, U_1)\|$$
    where $\sin \Theta(U_0,U_1) = \diag(\sin \theta_1,\dots,\sin \theta_d)$ is a $d\times d$ diagonal matrix containing the sines of the principal angles.
\end{lemma}
\begin{proof}
    Note that
    \begin{align*}
        ((I_n-U_0U_0^\top)U_1)^\top ((I_n-U_0U_0^\top)U_1) & = U_1^\top (I_n - U_0U_0^\top) U_1 \\
        &= I_d - (U_0^\top U_1)^\top (U_0^\top U_1) \\
        &= I_d - V \cos^2 \Theta V^\top \\
        &= V(I_d - \cos^2 \Theta) V^\top = V \sin^2 \Theta V^\top.
    \end{align*}
    Here we used the fact that $I_n-U_0U_0^\top$ is a projection matrix, and hence idempotent. Also $U_1^\top U_1 = I_d$. This shows that the singular values of $(I_n-U_0U_0^\top)U_1$ are precisely $\sin \theta_1,\dots,\sin \theta_d$, which proves the result.
\end{proof}

So, we see that the singular values of $U^\top V$ (for $U, V\in\bbR^{n\times d}$ with orthonormal columns) are precisely cosines of the principal angles $\theta_1,\dots,\theta_d$. Now, let $U_\perp$ be the matrix whose columns are a basis for the orthogonal complement of $\rm span(U)$. Then, $I_n = UU^\top + U_\perp U_\perp^\top$. Then we get
$$I_d = V^\top I_n V = V^\top(UU^\top + U_\perp U_\perp^\top) V = AA^\top + BB^\top$$
where $A=V^\top U$ and $B=V^\top U_\perp$. Since $BB^\top = I_d - AA^\top$, the singular values of $B$ are precisely $\sqrt{1 - \sigma_j(A)^2}=\sqrt{1-\cos^2(\theta_j)}=\sin(\theta_j)$ since $A=V^\top U$ with singular values $\cos (\theta_j)$, for $j=1,\dots,d$. This proves the following lemma.

\begin{lemma}
    For orthogonal $U, V$, we have 
    $$\|U_\perp^\top V \| = \|\sin \Theta(U,V)\|.$$
\end{lemma}

\begin{theorem}[Mirsky's theorem -- Theorem 4.11 \cite{stewart1990matrix}]\label{thm: mirsky}
    Let $X$ and $\tilde{X}$ be matrices of the same dimensions with singular values
    \begin{align*}
        s_1\geq s_2 \geq \dots \geq s_p, \\
        \tilde{s}_1 \geq \tilde{s}_2 \geq \dots \geq \tilde{s}_p
    \end{align*}
    then for any unitarily invariant norm $\|\cdot\|$ (and hence for the spectral norm in particular),
    $$\|\diag(\tilde{s}_i - s_i)_{i=1}^p\| \leq \|\tilde{X} - X\|.$$
\end{theorem}

\begin{theorem}[Wedin's theorem -- Theorem 4.4 \cite{stewart1990matrix}]\label{thm: wedin}
    Let $A$ and $\tilde{A}=A + E$ be two $n\times D$ matrices. Let $A=UDV^\top + U_\perp D_\perp V_\perp^\top$ and $\tilde{A}=\tilde{U}\tilde{D}\tilde{V}^\top + \tilde{U}_\perp \tilde{D}_\perp\tilde{V}_\perp^\top$ be their respective SVD, where the first parts correspond to the top $d$ singular values. Let $\Phi_1 = \Theta(U, \tilde{U})$ and $\Phi_2 = \Theta(V, \tilde{V})$ be the canonical angles between the subspaces (as diagonal matrices). Additionally, suppose 
    $$s_d(\tilde{A}) \geq \alpha + \delta, \quad s_{d+1}(A) \leq \alpha$$
    for some $\alpha,\delta>0$. Then
    $$\max \{\|\sin\Phi_1\|, \|\sin\Phi_2\|\} \leq \frac{\|E\|}{\delta}.$$
\end{theorem}

While applying this theorem, often $E$ is a noise matrix and as such, and it is difficult to control the singular values of the perturbed matrix $\tilde{A}$ to verify the condition directly. Instead we show something like
$$s_d(A) - s_{d+1}(A) \geq \delta_0.$$
Then we use this spectral gap, together with Mirsky's inequality Theorem \ref{thm: mirsky}. This gives
$$s_d(\tilde{A})-s_{d+1}(A)\geq s_d(A) - \|E\| -s_{d+1}(A) \geq \delta_0 -\|E\|.$$
Thus, as long as $\|E\|<\delta_0$, the condition required to invoke Wedin's theorem is satisfied with $\delta=\delta_0-\|E\|>0$.

\begin{lemma}[Wedin's Identity]\label{lemma: wedin identity}
    Let $A,B\in\bbR^{n\times D}$ be two matrices having the same rank $d$, where $A=B+E$. Let $A^+, B^+$ be the Moore-Penrose pseudo-inverses of $A$ and $B$ respectively. Then we have the following decomposition:
    $$A^+ - B^+ = -A^+ E B^+ + A^+(A^+)^\top E^\top (I-BB^+) + (I-A^+A)E^\top (B^+)^\top B^+.$$
\end{lemma}
\begin{proof}
    First, insert the identities $I=BB^+ + (I-BB^+)$ and $I=A^+A + (I-A^+A)$ to split:
    $$A^+-B^+ = \underbrace{A^+(I-BB^+)}_{\text{term 1}} - \underbrace{(I-A^+A)B^+}_{\text{term 2}} - \underbrace{A^+(A-B)B^+}_{=A^+EB^+}.$$
    The last term is already in the target form. We focus on the first two terms.

    Recall the fundamental property of pseudo-inverse $A^+=A^+(A^+)^\top A^\top$. Plugging this into term 1 gives
    $$A^+(I-BB^+) = A^+(A^+)^\top A^\top (I-BB^+).$$
    Using $A^\top = B^\top + E^\top$ and expanding the right side,
    $$A^\top(I-BB^+) = B^\top(I-BB^+) + E^\top(I-BB^+).$$
    By property of the pseudo-inverse, $B^\top = B^\top B B^+$, hence the first term is 0. Thus, we get the simplification for term 1 as
    $$A^+(I-BB^+) = A^+(A^+)^\top E^\top (I-BB^+).$$
    The second term 2 also simplifies similarly. Note that $B^+=B^\top (B^+)^\top B^+$ and $B^\top = A^\top -E^\top$. Hence,
    \begin{align*}
        (I-A^+A)B^+ &= (I - A^+A)B^\top (B^+)^\top B^+ \\
        &= (I - A^+ A) (A^\top - E^\top )(B^+)^\top B^+ \\
        &= -(I-A^+A)E^\top (B^+)^\top B^+,
    \end{align*}
    because $(I-A^+ A)A^\top = 0$. Since the original decomposition subtracts term 2, this produces the required positive term $+(I-A^+A)E^\top (B^+)^\top B^+$, matching the corresponding term in the statement of the lemma.
\end{proof}

\begin{remark}
    Note that the fact that $A$ and $B$ have the same rank is not used in the proof of the above lemma. Indeed, the algebraic identity holds for any $A$ and $B$. However, the pseudo-inverse operation is  continuous only on the set of matrices with the same rank. If $\rank(A)\neq \rank(B)$, then as the perturbation $E\to 0$, the pseudo-inverse $A^+$ does not converge to $B^+$. Hence, in the following lemma, we introduce truncation to the top $d$ singular values.
\end{remark}

\begin{lemma}[Spectral stability of truncated pseudo-inverse]\label{lemma: spectral stability truncated pseudo-inverse}
    Given $\bY = \bY_0 + E_0$ and their rank$-d$ pseudo-inverses $\bY^+$ and $\bY_0^+$, as defined earlier, if $s_d(\bY)>0$ then
    $$\|\bY^+ - \bY_0^+\| \leq 2\|E_0\| \left(\frac{1}{s_d(\bY)^2} + \frac{1}{s_d(\bY_0)^2} + \frac{1}{s_d(\bY)s_d(\bY_0)} \right).$$
\end{lemma}

\begin{proof}
    Recall $\bY_0^+$ is the pseudo-inverse of $\bY_0=\bM_0\Lambda_0^\top$, the rank $d$ true signal matrix. Also, recall $\bY=USV^\top + U_\perp S_\perp V_\perp^\top$ and $\bY^+=VS^{-1}U^\top$ is the pseudo-inverse of $\tilde{\bY} = USV^\top$, the best rank$-d$ approximation as given by SVD (top $d$ singular values). Note that while $\bY=\bY_0+E_0$ may not be rank $d$, $\tilde{\bY}$ is of rank $d$ by definition. Since both $\tilde{\bY}$ and $\bY_0$ have rank $d$ and are related as $\tilde{\bY} = \bY_0 + \tilde{E}_0$ (where $\tilde{E}_0$ is the effective perturbation matrix), we can use Lemma \ref{lemma: wedin identity} to obtain
    $$\bY^+ - \bY_0^+ = -\bY^+\tilde{E}_0\bY_0^+ + \bY^+(\bY^+)^\top \tilde{E}_0^\top (I - \bY_0\bY_0^+) + (I-\tilde{\bY}^+\bY)\tilde{E}_0^\top (\bY_0^+)^\top \bY_0^+.$$
    Note that both $(I-\bY_0\bY_0^+)$ and $(I-\tilde{\bY}^+\bY)$ are orthogonal projection operators, and hence have spectral norm at most 1. Now, we use triangle inequality on the above display to get:
    $$\|\bY^+-\bY_0^+\| \leq \|\bY^+\tilde{E}_0 \bY_0^+\| + \|\bY^+(\bY^+)^\top \tilde{E}_0^\top\| + \|\tilde{E}_0^\top (\bY_0^+)^\top \bY_0^+\|.$$
    We can bound each of these terms as follows
    \begin{align*}
        \|\bY^+\tilde{E}_0 \bY_0^+\| &\leq \frac{\|\tilde{E}_0\|}{s_d(\tilde{\bY})s_d(\bY_0)} \\
        \|\bY^+(\bY^+)^\top \tilde{E}_0^\top\| &\leq \frac{\|\tilde{E}_0\|}{s_d(\tilde{\bY})^2}\\
        \|\tilde{E}_0^\top (\bY_0^+)^\top \bY_0^+\| &\leq \frac{\|\tilde{E}_0\|}{s_d(\bY_0)^2}.
    \end{align*}
    By definition of $\tilde{\bY}$ in relation to $\bY$, we have $s_d(\tilde{\bY}) = s_d(\bY)$. Now, to relate the norm of $\tilde{E}_0$ in terms of the original noise matrix $E_0=\bY-\bY_0$, we can use triangle inequality:
    $$\tilde{E}_0=E_0 - U_\perp S_\perp V_\perp^\top \Longrightarrow \|\tilde{E}_0\| \leq \|E_0\| + s_{d+1}(\bY).$$
    Now, since $s_{d+1}(\bY_0) = 0$, by Mirsky's inequality \ref{thm: mirsky}, $s_{d+1}(\bY) \leq \|E_0\|$. Thus, combining everything, we get
    $$\|\bY^+ - \bY_0^+\| \leq 2\|E_0\| \left(\frac{1}{s_d(\bY)^2} + \frac{1}{s_d(\bY_0)^2} + \frac{1}{s_d(\bY)s_d(\bY_0)} \right).$$
\end{proof}

\section{Proofs for Section \ref{sec: observation density theory}}

\subsection{Additional Lemmas}\label{app: additional lemmas}
We present and proof the two small technical lemmas. The first is used to transfer the $L_2$ rate obtained in the previous section to an $L_1$ rate using the sub-Gaussian assumption on $g_0$. The second result is used to deal with the positivity correction, for higher-order kernels, as discussed in the main paper.

\medskip
\textbf{$L^2$ to $L^1$ conversion} For $0<r<1$, write $\log_+(1/r)=1\vee \log(1/r)$. Recall the kernel tail assumption \eqref{assume: K tail}.

\begin{lemma}\label{lemma: L2 to L1 signed K}
    Suppose $g_0$ is sub-Gaussian and $K$ satisfies the kernel-tail condition in \eqref{assume: K tail}. Let $\cE_{n,D}$ be the event in Theorem \ref{thm: new kde}, and suppose 
    $$\bbE_0[\|\widehat g_h - g_0^R\|_{L^2}^2 \mid \cE_{n,D}] \lesssim r_{n,D}$$
    where $r_{n,D}\to 0$. For the bandwidths $h\lesssim 1$, we have
    $$\bbE_0[\|\widehat g_h - g_0^R\|_{L^1} \mid \cE_{n,D}] \lesssim r_{n,D}^{1/2} \{\log_+ (1/r_{n,D})\}^{d/4}.$$
\end{lemma}

\begin{proof}[Proof of Lemma \ref{lemma: L2 to L1 signed K}]
    Write $f(z)=\widehat{g}_h(z) - g_0^R(z)$. For any $T>0$, let $B_T=\{z: \|z\| < T\}$
    \begin{equation}\label{eq: L2 to L1 proof decomposition}
        \begin{aligned}
            \|f\|_{L^1} &= \int_{B_T} |f(z)|dz + \int_{B_T^c} |f(z)| dz \\
        &\leq \vol(B_T)^{1/2} \|f\|_{L^2} + \int_{B_T^c} |\widehat{g}_h(z)|dz + \int_{B_T^c} g_0^R(z)dz.
        \end{aligned}
    \end{equation}
    Since $\vol(B_T)=v_d T^d$, we have $\int_{B_T} |f| \lesssim T^{d/2} \|f\|_{L^2}$. Because $K$ may be signed, $\widehat{g}_h$ may be negative. So, for the second term in \eqref{eq: L2 to L1 proof decomposition} we split
    \begin{align*}
        \int_{\|z\|>T} |\widehat{g}_h(z)|dz &\leq \frac{1}{n} \sum_{i=1}^n \int_{\|z\|>T} h^{-d} \left|K\left(\frac{z-\widehat{\eta}_i}{h}\right)\right| dz \\
        &= \frac1n \sum_i \int \ind(\|\widehat{\eta}_i +hu\|>T) |K(u)| du.
    \end{align*}
    If $\|\widehat{\eta}_i\|\leq T/2$ and $\|\widehat{\eta}_i +hu\|>T$, then necessarily $h\|u\|>T/2$. Hence
    $$\ind(\|\widehat{\eta}_i + hu\|>T) \leq \ind(\|\widehat{\eta}_i\| >T/2) + \ind(\|u\|>T/(2h)).$$
    Thus, 
    $$\int_{B_T^c} |\widehat{g}_h| \leq \|K\|_{L^1} \frac1n \sum_{i=1}^n \ind(\|\widehat{\eta}_i\| >T/2) + \int_{\|u\|>T/2h} |K(u)| du.$$
    Now, for the first term, we use the score-estimation event. On $\cE_{n,D}$, we know that $\max_i \|\widehat{\eta}_i - R^\top \eta_i\| \leq a_{n,D}\to 0$. For $T$ sufficiently large compared with $a_{n,D}\to 0$, $\|\widehat{\eta}_i\|>T/2$ implies $\|\eta_i\|\geq \|\widehat{\eta}_i\| - a_{n,D}>T/4$. Therefore, on $\cE_{n,D}$,
    $$\frac1n \sum_i \ind(\|\widehat{\eta}_i\|>T/2) \leq \frac1n \sum_i \ind(\|\eta_i\|>T/4).$$
    Taking conditional expectation, 
    $$\bbE_0\left[\frac1n \sum_i \ind(\|\eta_i\|>T/4)\mid \cE_{n,D}\right] \leq \frac{\bbE_0[n^{-1}\sum_i \ind(\|\eta_i\|>T/4)]}{\bbP_0^{(n)}(\cE_{n,D})}\lesssim e^{-c'T^2},$$
    where we used $\bbP_0(\cE_{n,D})\geq 1-c/n$ and the fact that $\eta_i'$s are iid from $g_0$ with sub-Gaussian tails.

    For the second term we use the K-tail condition assumed on the kernel $K$. This gives
    $$\int_{\|u\|>T/2h} |K(u)|du \lesssim \exp(-c_KT^2/h^2).$$
    For $h\leq 1$, $\exp(-c_KT^2/h^2)\leq e^{-c_K T^2}.$ Consequently,
    $$\bbE_0\left[\int_{B_T^c}|\widehat{g}_h|\mid \cE_{n,D}\right] \lesssim e^{-c'' T^2}.$$

    Similarly, sub-Gaussianity of $g_0$ gives
    $$\int_{B_T^c} g_0^R(z)dz \lesssim e^{-c'T^2},$$
    for the last term in \eqref{eq: L2 to L1 proof decomposition}, where the same statement holds after rotation because $\|Rz\|=\|z\|$. Putting everything together,
    $$\bbE_0[\|f\|_{L^1} \mid \cE_{n,D}] \lesssim T^{d/2} \bbE_0[\|f\|_{L^2} \mid \cE_{n,D}] + e^{-c_0 T^2},$$
    for some suitable constant $c_0>0$. By Cauchy-Schwarz,
    $$\bbE_0[\|f\|_{L^2} \mid \cE_{n,D}] \leq \left\{\bbE_0[\|f\|_{L^2}^2 \mid \cE_{n,D}]\right\}^{1/2} \lesssim r_{n,D}^{1/2}.$$
    Plugging this into the last display and choosing $T^2 = C\log (1/r_{n,D})$ for sufficiently large constant $C>0$, we obtain the required bound.
\end{proof}

\medskip
\noindent
\textbf{Positive correction} For KDE $\widehat g_h$, recall the positive corrected version $\widetilde g_h$, defined in \eqref{eq: positive correction}. Since $\int \widehat g_h = 1$, we have $\int \widehat g_h^+ \geq 1$, and hence, $\widetilde g_h$ is well-defined and is a probability density, by construction.

\begin{lemma}\label{lemma: positive correction preserves L1}
    Let $\widehat g$ be an integrable, possibly signed function satisfying $\int \widehat g = 1$, and define $\widetilde g =\widehat g^+ / \int \widehat g^+$. Then, for every probability density $g$, we have
    $$\|\widetilde g - g \|_{L^1} \leq 2\|\widehat g - g\|_{L^1}.$$
    
    Moreover, for $\widehat g_h$ as in \eqref{eq: fade latent kde} for an integrable kernel $K$ satisfying $\int \|u\| |K(u)|du <\infty$, then,
    $$\int \|z\| \tilde{g}(z) dz \leq \sqrt{d} \|K\|_{L^1} + h\int \|u\| |K(u)| du.$$
    In particular, this first moment is uniformly bounded for $h\lesssim 1$.
\end{lemma}

\begin{proof}[Proof of Lemma \ref{lemma: positive correction preserves L1}]
    Since $\int K = 1$, $\int \widehat{g}_h = 1$. Write $\widehat{g} = \widehat{g}^+ - \widehat{g}^-$ where $\widehat{g}^- = \max\{-\widehat{g}, 0\}$. Let $N = \int \widehat{g}^-(z)dz$ and $A = \int \widehat{g}^+(z)dz$. Since the signed integral of $\widehat{g}$ equals 1, $A - N = 1$. Hence, $A = 1+ N\geq 1$. Thus, $\widetilde{g}$ is well-defined, non-negative and integrates to 1. 

    Next, we establish the pointwise $L^1$ bound, with $g$ a probability density. Because $g\geq 0$, $\|\widehat{g}_h^+ - g\|_{L^1}\leq \|\widehat{g} - g\|_{L^1}$. Indeed, on $\{\widehat{g}\geq 0\}$, the two integrands coincide, whereas on $\{\widehat{g}<0\}$, $|\widehat{g}^+-g|=g\leq g-\widehat{g}=|\widehat{g} - g|$. Also, 
    $$N = \int_{\widehat{g} <0} |\widehat{g}(z)|dz \leq \int_{\widehat{g} <0} |\widehat{g}(z)-g(z)|dz \leq \|\widehat{g} - g\|_{L^1}$$
    since $g\geq 0$, which implies $|\widehat{g}(z)|\leq |\widehat{g}(z) - g(z)|$ whenever $\widehat{g}(z) <0$. Now,
    $$\|\widetilde{g} - g\|_{L^1} \leq \left\|\frac{\widehat{g}^+}{A} - \widehat{g}^+\right\|_{L^1} + \|\widehat{g}^+ - g\|_{L^1}.$$
    The second term is already shown to be bounded by $\|\widehat{g} - g\|_{L^1}$. For the first term, note that
    $$\left\|\frac{\widehat{g}^+}{A} - \widehat{g}^+\right\|_{L^1} = |A^{-1} - 1| \int \widehat{g}^+(z) dz = |1-A| = N \leq \|\widehat{g} - g\|_{L^1}.$$
    Thus, $\|\widetilde{g} - g\|_{L^1} \leq 2\|\widehat{g} - g\|_{L^1}$.

    It remains to control the first moment for the case $\widehat g = \widehat g_h$, the KDE estimator in \eqref{eq: fade latent kde}. Denote the corresponding quantities with the suffix $h$. Since $A_h\geq 1$, $\tilde{g}_h (z)\leq \hat{g}_h^+(z)\leq |\hat{g}_h(z)|$. Thus, $\int \|z\|\tilde{g}_h(z)dz \leq \int \|z\| |\hat{g}_h(z)| dz$. Using the definition of the KDE and triangle inequality,
    $$|\hat{g}_h(z)| \leq \frac{1}{nh^d} \sum_{i=1}^n \left|K\left(\frac{z-\hat{\eta}_i}{h}\right)\right|.$$
    Therefore, we have
    \begin{align*}
        \int \|z\|\tilde{g}_h(z)dz &\leq \frac1n \sum_{i=1}^n \frac{1}{h^d}\int \|z\| \left|K\left(\frac{z-\hat{\eta}_i}{h}\right)\right|dz \\
        &= \frac1n \sum_{i=1}^n \int \|\hat{\eta}_i + hu\| |K(u)| du, \qquad u=(z-\hat{\eta}_i)/h \iff z=\hat{\eta}_i + hu \\
        &\leq \frac1n \sum_{i=1}^n \left\{\|\hat{\eta}_i\|\int |K(u)|du + h \int\|u\| |K(u)| du\right\} \\
        &= \|K\|_1 \left(\frac1n \sum_{i=1}^n \|\hat{\eta}_i\|\right) + h \int \|u\| |K(u)| du.
    \end{align*}
    Recalling $\hat{\eta}_i'$s are the rows of $\widehat{M}=\sqrt{n} U$. Thus, $\|\widehat{M}\|_F^2 = nd$, and hence, $\frac1n\sum_{i=1}^n \|\hat{\eta}_i\|^2 = d$. Thus, $\frac1n\sum_i \|\hat{\eta}_i\| \leq \sqrt{d}$ by Cauchy Schwarz.
\end{proof}

\medskip
\noindent
\textbf{$L^1$ between Gaussians.} Next, we present a simply lemma that provides an upper bound for the $L^1$ distance between two Gaussian distributions in terms of Euclidean distances between the parameters.

\begin{lemma}\label{lemma: L1 between normals}
    Suppose $\lambda_{\min}(\Sigma)\land \lambda_{\min}(\Sigma_0)\geq c> 0$. Then, for any $z\in\bbR^d$
    $$\|\phi(\cdot\mid \Lambda z, \; \Sigma) - \phi(\cdot \mid \Lambda_0 z,\; \Sigma_0)\|_{L^1} \lesssim \|\Lambda - \Lambda_0\|_F \|z\| +  \|\Sigma - \Sigma_0\|_F.$$
\end{lemma}

\begin{proof}
    By triangle inequality
    $$\|\phi(\cdot|\Lambda z, \Sigma) - \phi(\cdot|\Lambda_0 z, \Sigma_0)\|_{L^1} \leq \|\phi(\cdot|\Lambda z, \Sigma) - \phi(\cdot|\Lambda_0 z, \Sigma)\|_{L^1} + \|\phi(\cdot|\Lambda_0 z, \Sigma) - \phi(\cdot|\Lambda_0 z, \Sigma_0)\|_{L^1}.$$
    We control the mean and covariance perturbations separately.

    \medskip
    \noindent
    \textbf{Part (A): Mean perturbation}. Translation invariance and Pinsker's inequality give
    \begin{align*}
        \|\phi(\cdot|\Lambda z, \Sigma) - \phi(\cdot|\Lambda_0 z, \Sigma)\|_{L^1} &\leq \sqrt{z^\top(\Lambda-\Lambda_0)^\top  \Sigma^{-1}(\Lambda-\Lambda_0)z} \leq c^{-1/2} \|(\Lambda - \Lambda_0)z\|\lesssim_c \|\Lambda-\Lambda_0\|_F \|z\|.
    \end{align*}

    \medskip
    \noindent
    \textbf{Part (B): Covariance perturbation}. Translation invariance allows us to take the common mean to be 0. Set $\Delta=\Sigma-\Sigma_0$ and $\Sigma_t = \Sigma_0 + t\Delta$ for $t\in [0,1]$. Since the cone of matrices bounded below by $cI_D$ is convex, $\Sigma_t \succeq cI_D$ for all such $t$. Let $p_t(y)=\phi(y\mid 0, \Sigma_t)$. By the fundamental theorem of calculus
    $$p_1(y) - p_0(y) = \int_0^1 \dot{p}_t(t) dt,$$
    so $\|p_1 - p_0\|_{L^1} \leq \int_0^1 \|\dot{p}_t\|_{L^1} dt$. Differentiating the Gaussian density (with respect to $t$), 
    $$\dot{p}_t (y) = \frac12 p_t(y) \left[y^\top \Sigma_t^{-1} \Delta \Sigma_t^{-1} y - \tr(\Sigma_t^{-1}\Delta)\right].$$

    Writing $Y_t = \Sigma_t^{1/2} Z$ with $Z\sim N(0, I_D)$ and $A_t := \Sigma_t^{-1/2} \Delta \Sigma_t^{-1/2}$, we have
    $$\|\dot{p}_t\|_{L^1} = \frac12 \bbE\left|Z^\top A_t Z - \tr(A_t)\right|.$$

    By Cauchy-Schwarz and standard variance formula for Gaussian quadratic forms,
    $$\bbE\left|Z^\top A_t Z - \tr(A_t)\right| \leq \left\{\Var(Z^\top A_t Z)\right\}^{1/2} = \sqrt{2} \|A_t\|_F.$$
    Moreover,
    $$\|A_t\|_F \leq \|\Sigma_t^{-1/2}\|\|\Delta\|_F \|\Sigma_t^{-1/2}\| \leq \frac1c \|\Sigma - \Sigma_0\|_F$$

    Hence, uniformly in $t\in[0,1]$, $\|\dot{p}_t\|_{L^1} \lesssim_c \|\Sigma - \Sigma_0\|_F.$ Integrating over $t$, we get
    $$\|\phi(\cdot \mid 0, \Sigma) - \phi(\cdot \mid 0, \Sigma_0)\|_{L^1} \lesssim_c  \|\Sigma - \Sigma_0\|_F.$$

    Combining the mean and covariance bounds, we get the desired result.
\end{proof}

\subsection{Proofs for Section \ref{subsec: obs density upper bound}}

\begin{proof}[Proof of Theorem \ref{thm: observation kde rate}]\label{proof of observation estimation upper bound}
    By rotational invariance of the factor representation, we may write the target observation density as
    $$p_0(y) = \int_{\bbR^d} \phi_D(y\mid \Lambda_0 R z, \Sigma_0)g_0^R(z)dz,$$
    using the rotation $R$ used in Theorem \ref{thm: new kde}. Introduce the intermediate density
    $$\tilde{p}(y) = \int_{\bbR^d} \phi_D(y\mid \Lambda_0 Rz, \Sigma_0) \tilde{g}(z) dz$$
    which uses the true parameters $\Lambda_0, \Sigma_0$, but the latent density estimator $\tilde{g}$. By the triangle inequality,
    $$\|\hat{p} - p_0\|_{L^1} \leq \|\hat{p} - \tilde{p}\|_{L^1} + \|\tilde{p} - p_0\|_{L^1}.$$
    We control these two terms separately. 

    \medskip
    \noindent
    \textbf{Step 1: Latent density estimation error} Note that, by Fubini's theorem 
    \begin{align*}
        \|\tilde{p} - p_0\|_{L^1} & = \int_{\bbR^D} \left|\int_{\bbR^d} \phi_D(y\mid \Lambda_0R z, \Sigma_0) \left\{\tilde{g}(z) - g_0^R(z)\right\}\right| dy \\
        &\leq \int_{\bbR^d}\int_{\bbR^D} \phi_D(y \mid \Lambda_0 Rz, \Sigma_0)\left|\tilde{g}(z) - g_0^R(z)\right| dy dz \\
        &= \|\tilde{g} - g_0^R\|_{L^1}.
    \end{align*}
    Under the $L^2$ rate assumed in the statement, apply Lemma \ref{lemma: positive correction preserves L1}, followed by Lemma \ref{lemma: L2 to L1 signed K} to get
    $$\bbE_0\left[\|\tilde{p} - p_0\|_{L^1} \mid \cE_{n,D}\right] \lesssim r_{n,D}^{1/2} \left(\log \frac{1}{r_{n,D}}\right)^{d/4}.$$

    \medskip
    \noindent
    \textbf{Step 2: Loading and covariance estimation error}. By the definitions of $\hat{p}$ and $\tilde{p}$, we have
    \begin{align*}
        \|\hat{p} - \tilde{p}\|_{L^1} &= \int_{\bbR^D} \left|\int_{\bbR^d} \left\{\phi_D(y\mid \widehat{\Lambda}z, \widehat{\Sigma}) - \phi_D(y \mid \Lambda_0 R z, \Sigma_0)\right\} \tilde{g}(z) dz\right| dy \\
        &\leq \int_{\bbR^d} \tilde{g}(z) \left\|\phi_D(\cdot \mid \widehat{\Lambda}z, \widehat{\Sigma}) - \phi_D(\cdot \mid \Lambda_0 R z, \Sigma_0)\right\|_{L^1} dz.
    \end{align*}
    Next, we apply Lemma \ref{lemma: L1 between normals}. Before applying it, we verify the lower eigenvalue condition for $\widehat \Sigma$. Define
    $$b_{n,D} = \sqrt{\frac Dn} + \sqrt{\frac{\ell_D}{D}}.$$
    We divide the argument into two cases. If $b_{n,D}$ is bounded below by a sufficiently small fixed constant $b_0>0$, then the desired conclusion follows immediately from 
    $$\|\widehat p - p_0\|_{L^1} \leq 2 \lesssim b_{n,D}.$$
    Thus, it remains to consider the case $b_{n,D}<b_0$. On $\cE_{n,D}$, Theorem \ref{thm: spectral recovery} gives $\|\widehat \Sigma - \Sigma_0\|\lesssim \sqrt{\frac{\log(nD)}{n}}+\frac{\ell_D}{D}$. Since $\log n = o(D)$ and $\log D \lesssim D$, we have $\log (nD) \lesssim D$ for all sufficiently large $n,D$. Hence, 
    $$\|\widehat \Sigma - \Sigma_0\| \lesssim \sqrt{\frac Dn} + \frac{\ell_D}{D} \leq b_{n,D} + b_{n,D}^2.$$
    Choosing $b_0>0$ sufficiently small, depending on the constants in the assumptions, ensures that $\|\widehat \Sigma - \Sigma_0 \| \leq c_\sigma /2$ on $\cE_{n,D}$, where $\lambda_{\min}(\Sigma_0)\geq c_\sigma$ by Assumption \ref{assume: noise}. Weyl's inequality therefore yields $\lambda_{\min}(\widehat \Sigma) \geq c_\sigma/2$. The covariance matrices $\widehat \Sigma$ and $\Sigma_0$ consequently satisfy the conditions of Lemma \ref{lemma: L1 between normals}.

    Applying Lemma \ref{lemma: L1 between normals}, we get
    \begin{align*}
        \|\hat{p} - \tilde{p}\|_{L^1} &\lesssim \int_{\bbR^d} \tilde{g}(z)\left\{\|\widehat{\Lambda} - \Lambda_0 R\|_F\|z\| +  \|\widehat{\Sigma} - \Sigma_0\|_F\right\} dz.
    \end{align*}
    By Lemma \ref{lemma: positive correction preserves L1}, for $h^*\lesssim 1$, $\int \|z\|\tilde{g}(z) dz \lesssim 1$. Therefore, 
    $$\|\hat{p} - \tilde{p}\|_{L^1} \lesssim \|\widehat{\Lambda} - \Lambda_0 R\|_F + \|\widehat{\Sigma} - \Sigma_0\|_F.$$
    Taking conditional expectations and applying Corollary \ref{cor: expected factor error}, we get 
    $$\bbE_0\left[\|\hat{p} - \tilde{p}\|_{L^1}\mid \cE_{n,D}\right] \lesssim \sqrt{\frac{D}{n}} + \sqrt{\frac{\ell_D}{D}},$$
    when $\ell_D\gtrsim 1$. Combining both parts, 
    $$\bbE_0\left[\|\hat{p} - p_0\|_{L^1}\mid \cE_{n,D}\right]\lesssim r_{n,D}^{1/2}\left(\log \frac{1}{r_{n,D}}\right)^{d/4} + \sqrt{\frac{D}{n}} + \sqrt{\frac{\ell_D}{D}}.$$
    The unconditional version is obtained by noting that the $L^1$ distance between two probability densities is bounded by 2 on $\cE_{n,D}^c$, and $\bbP_0^{(n)}(\cE_{n,D}^c) \lesssim 1/n$.
\end{proof}

\subsection{Proofs for Section \ref{subsec: obs density lower bound}}

\begin{lemma}\label{lemma: local L1 separation covariance}
    There exist universal constants $c_0, c_1>0$ such that the following holds. Let $P_0 = N_m(0,\Gamma_0)$ and $P_1=N_m(0,\Gamma_1)$, where $\frac12I_m \preceq \Gamma_0, \Gamma_1\preceq 2I_m$. If $\|\Gamma_1 - \Gamma_0\|_F \leq c_0$, then
    $$\|p_1 - p_0 \|_{L^1} \geq c_1 \|\Gamma_1 - \Gamma_0\|_F.$$
\end{lemma}

More generally, if the eigenvalues of $\Gamma_0, \Gamma_1$ are in a fixed compact subset of $(0,\infty)$, the same conclusion holds with the constants depending on that compact set.
\begin{proof}
    By the invertible change of variables $x=\Gamma_0^{1/2}z$, the $L^1$ distance is unchanged. Thus, it suffices to compare $N_m(0, I_m)$ and $N_m(0, I_m+A)$ where
    $$A = \Gamma_0^{-1/2} (\Gamma_1 - \Gamma_0)\Gamma_0^{-1/2}.$$
    Under the assumed eigenvalue bound, $\|A\|_F\asymp \|\Gamma_1 - \Gamma_0\|_F$, $\|A\|\lesssim c_0$. Diagonalize $A = U \diag(a_1,\dots,a_m) U^\top$. A further orthogonal transformation leaves $N_m(0,I_m)$ invariant, so we may assume without loss of generality that $A = \diag(a_1,\dots,a_m)$. Let $P=N(0,I_m)$ and $Q=N(0, I_m +A)$, and write $L = dQ/dP$.

    For $Z=(Z_1,\dots,Z_m)\sim N(0,I_m)$, 
    $$L(Z) = \prod_{j=1}^m (1+a_j)^{-1/2} \exp \left\{\frac12 \frac{a_j}{1+a_j} Z_j^2\right\}.$$
     Define $S_A = \frac12\sum_{j=1}^m a_j (Z_j^2 - 1)$. Since the $Z_j$ are iid standard normal,
    $$\|S_A\|_{L^2(P)}^2 =\frac14 \sum_{j=1}^m a_j^2 \bbE[(Z_j^2 - 1)^2] =\frac12 \sum_{j=1}^m a_j^2 = \frac12 \|A\|_F^2.$$
    With $X_j = Z_j^2-1$ independent and centered, and we have the moments $\bbE Z^2=1, \bbE Z^4=3$, $\bbE Z^6=15$ and $\bbE Z^8=105$, we have $\bbE (Z^2-1)^4 = 60$. Now, for the fourth moment, $\bbE S_A^4 = \frac{1}{16}\bbE \left(\sum_j a_j X_j \right)^4$. Expanding and noting that only two types of terms have nonzero expectation, we get
    $$\bbE S_A^4 = \frac{1}{16}\left[60 \sum_j a_j^4 + 24\sum_{j<k} a_j^2a_k^2\right] \lesssim \left(\sum_j a_j^2\right)^2.$$
    Therefore,
    $$\|S_A\|_{L^4(P)} \lesssim \left(\sum_j a_j^2 \right)^{1/2} \asymp \|S_A\|_{L^2}\asymp \|A\|_F.$$

    Now, we perform a Taylor expansion of $L$ for sufficiently small $\|A\|_F$. For $c_0$ sufficiently small, $\|A\|\leq \|A\|_F$ is also sufficiently small. For $t\in[0,1]$ define
    $$P_t=N(0, I_m + t A), \quad L_t = \frac{dP_t}{dP}$$
    so that $L_0=1$ and $L_1 = L$. Writing $\ell_t = \log L_t$, we have
    $$L_t' = L_t \ell_t', \quad L_t'' = L_t \left\{(\ell_t')^2 + \ell_t''\right\}.$$
    The likelihood ratio is
    \begin{align*}L_t(Z) &= \prod_{j=1}^m (1+ta_j)^{-1/2} \exp\left\{\frac12 \frac{ta_j}{1+ta_j} Z_j^2 \right\}\\
    \;&\Longrightarrow \; \ell_t(Z)=-\frac12 \sum_{j=1}^m \log (1+ta_j) +\frac12 \sum_{j=1}^m \frac{ta_j}{1+ta_j} Z_j^2.\end{align*}
    Differentiating, we have
    \begin{align*}
        \ell_t' &= \frac12 \sum_{j=1}^m \left[-\frac{a_j}{1+ta_j} + \frac{a_j}{(1+ta_j)^2} Z_j^2\right], \\
        \ell_t'' &= \frac12 \sum_{j=1}^m \left[\frac{a_j^2}{(1+ta_j)^2} - \frac{2a_j^2}{(1+ta_j)^3} Z_j^2 \right].
    \end{align*}
    Note that $L_0' = \ell_0' = S_A$. Since $\|A\|$ is sufficiently small, $1+ta_j$ is uniformly bounded away from 0 for all $t\in[0,1]$. Standard moment bounds for sums of independent centered variables therefore give, for any fixed $q<\infty$,
    $$\sup_{t\in[0,1]} \|\ell_t'\|_{L^q(P)} \lesssim_q \|A\|_F, \quad \sup_{t\in[0,1]} \|\ell_t''\|_{L^q(P)} \lesssim_q \|A\|_F^2.$$
    To see this, after separating the centered terms $Z_j^2-1$ from their means, the coefficients in $\ell_t'$ are of order $a_j$ as $1+ta_j\asymp 1$ uniformly and coefficients of the centered $Z_j^2-1$ are $O(a_j)$, while the deterministic term is $O(\sum_j a_j^2)$. Likewise, for the $\ell_t''$, use $(\sum_j a_j^4)^{1/2} \leq \sum_j a_j^2$. Next, we argue a uniform moment bound for $L_t$. For any fixed $r>1$, we have
    $$\bbE_P [L_t^r] = \prod_{j=1}^m (1+t a_j)^{-r/2} \left( 1 - \frac{rta_j}{1+t a_j}\right)^{-1/2}.$$
For sufficiently small $\|A\|$ all factors are finite. Taking logarithm and expanding around $a_j=0$, the linear terms cancel, so we get
$$\log \bbE_P [L_t^r] \leq C_r \sum_{j=1}^m a_j^2 = C_r \|A\|_F^2.$$
To see this, taking the logarithm of the expression for $\bbE_P[L_t^r]$, each term in the sum in $\log \bbE_P[L_t^r]$ can be expressed as $f_r(x)=-\frac{r-1}{2}\log (1+x) - \frac12 \log (1-(r-1)x)$ with $x=t a_j$. Note that $f_r(0)=0$ and $f_r'(0)=0$. Hence, for $|x|$ sufficiently small, $f_r(x) \leq C_r x^2$. Thus, since $|t a_j|\leq \|A\|$ is uniformly small, we have $\log \bbE_P[L_t^r]\leq C_r \sum_j t^2 a_j^2 \leq C_r \|A\|_F^2$. Hence, $\sup_{t\in[0,1]} \|L_t\|_{L^r(P)} \leq C_r$.

Using H\"{o}lder's inequality together with this bound for $p=2,4$, we have
\begin{align*}
    \|L_t''\|_{L^p(P)} &\leq \|L_t (\ell_t')^2 \|_{L^p(P)} + \|L_t \ell_t''\|_{L^p(P)} \\
    &\leq \|L_t\|_{L^{2p}(P)} \|\ell_t'\|_{L^{4p}(P)}^2 + \|L_t\|_{L^{2p}(P)} \|\ell_t''\|_{L^{2p}(P)} \\
    &\lesssim \|A\|_F^2,
\end{align*}
uniformly over $t\in[0,1]$ for $p=2,4$ using the bounds above with $r=2p$. Finally, Taylor expansion with integral remainder gives
$$L_1 = L_0 + L_0' + \int_0^1 (1-t)L_t'' dt.$$
Thus, $L-1 = S_A + R_A$ where $R_A=\int_0^1 (1-t)L_t'' dt$. By Minkowski's inequality
$$\|R_A\|_{L^p(P)} \leq \int_0^1 (1-t) \|L_t''\|_{L^p(P)} dt \lesssim \|A\|_F^2, \quad p=2,4.$$
Therefore, we have the expansion,
$$L - 1 = \underbrace{\frac12\sum_{j=1}^m a_j (Z_j^2 - 1)}_{S_A} + R_A, \quad \text{with }\|R_A\|_{L^2(P)} + \|R_A\|_{L^4(P)} \lesssim \|A\|_F^2.$$
and we have already shown that $\|S_A\|_{L^2(P)}\asymp \|A\|_F$ and $\|S_A\|_{L^4(P)} \lesssim \|A\|_F$. Therefore, we can conclude
\begin{align*}
    \|L-1\|_{L^2(P)} & = \|S_A + R_A\|_{L^2(P)} \\
    &\geq \|S_A\|_{L^2(P)} - \|R_A\|_{L^2(P)} \\
    &\gtrsim \|A\|_F
\end{align*}
while
\begin{align*}
    \|L-1\|_{L^4(P)} &\leq \|S_A\|_{L^4(P)} + \|R_A\|_{L^4(P)} \\
    &\lesssim \|A\|_F + \|A\|_F^2 \lesssim \|A\|_F.
\end{align*}

    Since $1/p=\theta/p_0 + (1-\theta)/p_1$ with $p_0=1, p=2, p_1=4$ and $\theta=1/3$, by H\"{o}lder's inequality $\|f\|_{L^p}\leq \|f\|_{L^{p_0}}^\theta \|f\|_{L^{p_1}}^{1-\theta}$, we have
    $$\|L-1\|_{L^2(P)}\leq \|L-1\|_{L^1(P)}^{1/3} \|L-1\|_{L^4(P)}^{2/3} \;\Longrightarrow\; \|L-1\|_{L^1(P)} \geq \frac{\|L-1\|_{L^2(P)}^3}{\|L-1\|_{L^4(P)}^2} \gtrsim \|A\|_F.$$
    Finally, 
    $$\|p_1 - p_0\|_{L^1} = \bbE_P |L-1|=\|L-1\|_{L^1(P)} \gtrsim \|A\|_F \asymp \|\Gamma_1 - \Gamma_0\|_F.$$
\end{proof}

\begin{proof}[Proof of Theorem \ref{thm: obs lower bound}]
    \textbf{Step 1: Reduction}. We first restrict to a particular loading matrix satisfying the assumptions. Since $d$ is fixed, for every sufficiently large $D$, choose a deterministic matrix $U_D\in\bbR^{D\times d}$ such that $U_D^\top U_D = I_d$ and $\max_j \|(U_D)_{j\cdot}\|^2\lesssim D^{-1}$. For example, take $d$ columns of a Hadamard-type orthogonal matrix. Then set $\Lambda_D=\sqrt{D} U_D$ so that $\Lambda_D^\top \Lambda_D= DI_d$ and hence, $s_j(\Lambda_D)=\sqrt{D}$ $j=1,\dots,d$, and $\|\lambda_{Dj}\|^2 =O(1)$ (thus satisfying $\ell_D\asymp 1$). 

    We shall use this fixed loading matrix for the lower bound argument. Since the minimax risk over the full parameter space is bounded below by the minimax risk over any submodel obtained by fixing $\Lambda = \Lambda_D$, it suffices to derive lower bounds on such submodels.

    Let $P_D=\frac1D \Lambda_D \Lambda_D^\top = U_DU_D^\top$. Then $P_D$ is the orthogonal projection onto $\rm col(\Lambda_D)$. Let $W_D\in \bbR^{D\times (D-d)}$ have orthonormal columns spanning $\rm col(\Lambda_D)^\perp$. Thus,
    $$W_D^\top W_D =I_{D-d}, \quad W_D W_D^\top = I_D - P_D,\quad W_D^\top \Lambda_D = 0.$$
    These identities will be used in Step 2.

    \medskip
    \noindent
    \textbf{Step 2: Parametric lower bound (fix $g$ and vary $\Sigma$)}. Fix any $g_*\in\cG_{\beta}(L,C)$ and fix the loading matrix $\Lambda_D$ from Step 1. For one observation $Y=\Lambda_D \eta + \epsilon$ with $\eta\sim g_*, \epsilon\sim N_D(0,\Sigma)$. Project $Y$ onto the orthogonal complement of the factor space: $Z=W_D^\top Y$. Since $W_D^\top \Lambda_D=0$, we have $Z=W_D^\top \epsilon \sim N_{D-d}(0, \Gamma_{\Sigma})$ where $\Gamma_{\Sigma}=W_D^\top \Sigma W_D$. Thus the latent component disappears under this projection.

    \medskip
    \noindent
    \textbf{Construction} Consider the family of diagonal covariance matrices
    $$\Sigma_\vartheta = I_D + \delta \diag(\vartheta_1,\dots,\vartheta_D), \quad \vartheta\in \{0,1\}^D$$
    where $\delta>0$ will be chosen below. For sufficiently small fixed $\delta$, every $\Sigma_\vartheta$ satisfies the noise-variance assumptions. The corresponding projected covariance is
    $$\Gamma_\vartheta = W_D^\top \Sigma_\vartheta W_D.$$
    For two vertices $\vartheta, \vartheta'$, set $A_{\vartheta,\vartheta'}=\diag(\vartheta - \vartheta')$. Then,
    $$\Gamma_{\vartheta} - \Gamma_{\vartheta'} = \delta  W_D^\top A_{\vartheta,\vartheta'} W_D.$$

    \medskip
    \noindent
    \textbf{Separation}. For any diagonal $A$, 
    \begin{align*}
        \|W_D^\top A W_D\|_F^2 &= \tr(W_D^\top A W_D W_D^\top A W_D) \\
        &= \tr((I_D-P_D)A (I_D-P_D)A) \\
        &= \tr(A^2) -2\tr(P_D A^2) + \tr(P_D A P_D A).
    \end{align*}
    The final term is non-negative as $\tr(P_D A P_D A) = \tr[(P_D A P_D)^2] \geq 0$. Take $A=A_{\vartheta,\vartheta'}$ and write $H=H(\vartheta, \vartheta')$ as the Hamming distance. Since nonzero diagonal entries of $A$ are $\pm 1$, $\tr(A^2) = H$. Moreover, $(P_D)_{jj}=\|\lambda_{Dj}\|^2/D \lesssim 1/D$. Thus $\tr(P_D A^2) = \sum_{j: \vartheta_j \neq \vartheta_j'} (P_D)_{jj} \lesssim H/D$. It follows that 
    $$\|W_D^\top A_{\vartheta, \vartheta'} W_D\|_F^2 \geq \left(1 - \frac{2C}{D}\right) H.$$
    Hence, for all sufficiently large $D$, $\|W_D^\top A_{\vartheta,\vartheta'} W_D\|_F \gtrsim \sqrt{H(\vartheta, \vartheta')}$ and therefore,
    $$\|\Gamma_\vartheta - \Gamma_{\vartheta'}\| \gtrsim \delta \sqrt{H(\vartheta,\vartheta')}.$$
    We now translate this covariance separation into $L^1$ separation. Since $I_D\preceq \Sigma_\vartheta \preceq (1+\delta) I_D$, we have $I_{D-d} \preceq \Gamma_\vartheta \preceq (1+\delta)I_{D-d}$. Thus, all eigenvalues of all $\Gamma_\vartheta$ are uniformly bounded above and below. Furthermore, $\|\Gamma_\vartheta - \Gamma_{\vartheta'}\|_F \leq \delta \|A_{\vartheta, \vartheta'}\|_F \leq \delta \sqrt{D}$. Hence, provided $\delta\sqrt{D}$ is smaller than the constant $c_0$ in Lemma \ref{lemma: local L1 separation covariance}, that lemma gives
    $$\|p_\vartheta^Z - p_{\vartheta'}^Z\|_{L^1} \gtrsim \|\Gamma_\vartheta - \Gamma_{\vartheta'}\|_F \gtrsim \delta \sqrt{H(\vartheta, \vartheta')}.$$
    Since $Z=W_D^\top Y$ is a measurable function of $Y$, the total variation contracts under the projection $Y\mapsto Z$. Hence for any $\vartheta, \vartheta' \in \{0,1\}^D$, we have
    $$\|p_\vartheta^Y - p_{\vartheta'}^Y\|_{L^1} \gtrsim \delta \sqrt{H(\vartheta, \vartheta')}.$$
    Here $p_{\vartheta}^Z$ and $p_\vartheta^Y$ are respectively the distribution on $Z$ and $Y$ under $\Sigma_\vartheta$.
    
    \medskip
    \noindent
    \textbf{Controlling KL}. Let $P_0^Y$ denote the distribution of $Y$ under $\Sigma_0=I_D$. Since $g_*$ and $\Lambda_D$ are fixed, the two models $P_\vartheta^Y$ and $P_0^Y$ differ only through the Gaussian noise distribution. Adding the independent common random vector $\Lambda_D\eta$ is a Markov kernel, so
    $$\KL(P_\vartheta^Y\Vert P_0^Y) \leq \KL(N_D(0, \Sigma_\vartheta) \;\Vert \; N(0, I_D)).$$
    For centered Gaussian distributions
    $$\KL(N_D(0, \Sigma_\vartheta) \;\Vert \; N(0, I_D)) = \frac12 \left[\tr(\Sigma_\vartheta) - D - \log \det (\Sigma_\vartheta)\right] = \frac12 \sum_{j=1}^D \left[\delta \vartheta_j - \log (1+\delta \vartheta_j)\right].$$
    For sufficiently small $\delta$, $x-\log(1+x)\lesssim x^2$ and hence $\KL(P_\vartheta^Y\Vert P_0^Y)\lesssim D\delta^2$. For $n$ independent observations
    $$\KL\left((P_\vartheta^Y)^{\otimes n} \; \Vert \; (P_0^Y)^{\otimes n} \right) \lesssim nD\delta^2.$$

    Now, by the Varshamov-Gilbert bound, there exists a subset $\cV_D\subset \{0,1\}^D$ with $\log |\cV_D| \gtrsim D$ and $H(\vartheta,\vartheta') \gtrsim D$ for every distinct $\vartheta,\vartheta' \in \cV_D$. We may include $0$ as a reference element in the construction. Choose 
    $$\delta = \frac{c_\delta}{\sqrt{n\vee D}},$$
    where $c_\delta>0$ is a sufficiently small constant. Then $\delta\sqrt{D}\leq c_\delta$. Thus, by choosing $c_\delta$ sufficiently small, the local covariance-separation Lemma \ref{lemma: local L1 separation covariance} applies uniformly over the packing. For distinct $\vartheta,\vartheta'\in \cV_D$, 
    $$\|p_\vartheta^Y - p_{\vartheta'}^Y\|_{L^1} \gtrsim \delta \sqrt{H} \gtrsim \delta \sqrt{D} = c_\delta \sqrt{\frac{D}{n\vee D}}.$$
    At the same time, the upper bound for KL satisfies
    $$nD\delta^2 = c_\delta^2 \frac{nD}{n \vee D} \leq c_\delta^2 D.$$
    Since $\log |\cV_D|\gtrsim D$, choosing $c_\delta$ sufficiently small ensures that the KL divergence is bounded by a sufficiently small multiple of the logarithm of the packing size. Fano's lemma therefore gives
    $$\inf_{\hat{p}} \sup_{\vartheta\in \cV_D} \bbE_\vartheta \left[\|\hat{p} - p_\vartheta^Y\|_{L^1}\right] \gtrsim \delta \sqrt{D}\asymp \min\left\{1, \sqrt{\frac{D}{n}}\right\}.$$
    Since this covariance family is a submodel of $\Theta_{D,\beta}$, we conclude that
    \begin{equation}\label{eq: lower bound observation parametric}
        \inf_{\hat{p}} \sup_{(g,\Lambda,\Sigma)\in \Theta_{D,\beta}} \bbE_{g,\Lambda,\Sigma}\left[\|\hat{p} - p_{g,\Lambda,\Sigma}\|_{L^1}\right] \gtrsim \min \left\{1, \sqrt{D/n}\right\}.
    \end{equation}
    
    \bigskip
    \noindent
    \textbf{Part 3: Nonparametric lower bound (fix $\Sigma$ and vary $g$)}. Fix $\Sigma = I_D$. Thus, for this submodel $Y=\Lambda_D \eta + \epsilon$, with $\epsilon\sim N(0,I_D)$ and $\eta\sim g$ with only $g\in \cG_\beta(L,C)$ allowed to vary. Define $X=\frac1D \Lambda_D^\top Y$. Since $\Lambda_D^\top \Lambda_D = D I_d$, we have
    $$X = \eta + \frac1D \Lambda^\top \epsilon = \eta + D^{-1/2} Z, \quad Z\sim N(0, I_d).$$
    Hence $X$ has the density
    $$q_{g,D} = g \star \phi_{D^{-1} I_d}$$
    where $\phi_{D^{-1}I_d}$ denotes the $N(0, D^{-1} I_d)$ density. Using the definitions of $P_D$ and $W_D$ from Step 1, 
    $$W_D^\top Y = W_D^\top \epsilon \sim N(0,  I_{D-d}) \; \perp X.$$
    Moreover, the map $Y\mapsto (D^{-1} \Lambda_D^\top Y, W_D^\top Y)$ is invertible. Therefore, the $Y$ experiment is equivalent, for inference on $g$, to observing $X$ together with an ancillary Gaussian variable whose distribution does not depend on $g$. In particular, for two latent densities $g, \tilde{g}$,
    $$\|p_{g,\Lambda_D,I_D} - p_{\tilde{g},\Lambda_D,I_D}\|_{L^1} = \|q_{g,D} - q_{\tilde{g},D}\|_{L^1}, \qquad \KL\left(p_{g,\Lambda_D,I_D} \; \Vert \; p_{\tilde{g},\Lambda_D,I_D}\right) = \KL(q_{g,D} \;\Vert \; q_{\tilde{g},D}).$$
    We therefore construct a lower bound in the $d-$dimensional convolution model $X=\eta + D^{-1/2}Z$.

    \medskip
    \noindent
    \textbf{Construction}. Choose a smooth density $g_*$ in the interior of $\cG_\beta(L,C)$. Fix a compact cube $K\subset \bbR^d$ such that $\inf_{x\in K} g_*(x) \geq c_*>0$. Choose a nonzero function $\psi\in C_c^\infty(\bbR^d)$ supported in the unit cube satisfying $\int \psi = 0$, $\int u\psi(u)du=0$ and $\int uu^\top \psi(u)=0$. For $h> 0$ sufficiently small, choose points $x_1,\dots,x_M\in K$ such that the functions
    $$\psi_{j,h}(x) = \psi\left(\frac{x-x_j}{h}\right), \quad j=1,\dots, M$$
    have pairwise disjoint supports in $K$, with $M\asymp h^{-d}$. For $\theta=(\theta_1,\dots,\theta_M)\in\{-1,1\}^M$, define the alternatives
    $$g_\theta = g_* + \epsilon_h \sum_{j=1}^M \theta_j \psi_{j,h}, \quad \epsilon_h = a \min\left\{h^\beta, \frac{1}{\sqrt{nh^d}}\right\}$$
    and $a>0$ is a sufficiently small constant. By the moment conditions on $\psi$, $\int g_\theta = 1$, $\bbE_{g_\theta}[X]=0$ and $\bbE_{g_\theta}[XX^\top]=I_d$. Since all perturbations are supported on the fixed compact cube $K$ and $g_*$ is bounded away from 0 on $K$, choosing $a$ sufficiently small guarantees $g_\theta \geq 0$. The Gaussian-tail is also preserved as the perturbation vanishes outside $K$. Finally, standard scaling of Sobolev norms give $\|\sum_j \theta_j \psi_{j,h}\|_{\cW_\beta}^2 \lesssim M h^{d-2\beta}$. Since $M\asymp h^{-d}$, we get $\|\epsilon_h \sum_j \theta_j \psi_{j,h}\|_{\cW_\beta} \lesssim \epsilon_h h^{-\beta} \leq a$. Because $g_*$ is in the interior, $a$ can be chosen sufficiently small so that $g_\theta\in \cG_\beta(L,C)$ for every $\theta$.

    \medskip
    \noindent
    \textbf{Separation} Let $r_D=\frac{1}{\sqrt{D}h}$.  By scaling 
    $$\|\psi_{j,h}\star \phi_{D^{-1}I_d} - \psi_{j,h}\|_{L^1} = h^d \|\psi \star \phi_{r_D^2 I_d} - \psi\|_{L^1}.$$
    Since $\|\psi\star \phi_{r^2 I_d} - \psi\|_{L^1} \to 0$  as $r \to 0$, there is a constant $r_0>0$ such that whenever $r_D\leq r_0$, $\|\psi\star \phi_{r_D^2 I_d} - \psi\|_{L^1} \leq \frac12 \|\psi\|_{L^1}$.

    Let $\theta,\theta'\in \{-1,1\}^M$. Write $q_{\theta,D}=q_{g_\theta, D}$. Since 
    $$q_{\theta,D} - q_{\theta',D} = \epsilon_h \sum_{j=1}^M (\theta_j - \theta_j') \psi_{j,h} \star \phi_{D^{-1}I_d},$$
    by reverse triangle inequality
    \begin{align*}
        \|q_{\theta,D} - q_{\theta',D}\|_{L^1} &\geq \epsilon_h \left\|\sum_{j=1}^M (\theta_j - \theta_j')\psi_{j,h}\right\|_{L^1} - \epsilon_h \sum_{j=1}^M |\theta_j-\theta_j'| \|\psi_{j,h} \star \phi_{D^{-1} I_d} - \psi_{j,h} \|_{L^1} \\
        &\geq 2\epsilon_h h^d \|\psi\|_{L^1} H(\theta,\theta') -  \epsilon_h h^d \|\psi\|_{L^1} H(\theta,\theta') \gtrsim \epsilon_h h^d H(\theta,\theta')
    \end{align*}
    whenever $r_D\leq r_0\iff \sqrt{D} h \geq 1/r_0$, using the disjointness of the supports of $\psi_{j,h}$ for the first term, and the upper bound we obtained above for $r_D\leq r_0$ in the second term. Here $H(\theta,\theta')$ is the Hamming distance. By the equivalence of the $Y-$ and $X-$ experiments, 
    $$\|p_\theta - p_{\theta'}\|_{L^1} \gtrsim \epsilon_h h^d H(\theta,\theta')$$
    with $p_\theta = p_{g_\theta,\Lambda_D, I_D}.$

    \medskip
    \noindent
    \textbf{Controlling KL} Consider two neighboring alternatives $\theta,\theta'$ such that $H(\theta,\theta')=1$. Then $g_\theta$ and $g_{\theta'}$ differ only on one bump. Since $g_*$ is bounded below on $K$ and $a$ is sufficiently small, $g_\theta(x)\asymp g_*(x)$ uniformly on the support of the perturbation. Using $\KL(f\Vert g) \leq \chi^2(f,g)= \int (f-g)^2/g$, we obtain
    $$\KL(g_\theta \Vert g_{\theta'}) \lesssim \epsilon_h^2 \int \psi_{j,h}(x)^2 dx \lesssim \epsilon_h^2 h^d.$$
    Since Gaussian convolution is a Markov kernel, $\KL(q_{\theta,D} \Vert q_{\theta',D})\leq \KL(g_\theta \Vert g_{\theta'})$. Thus, for $n$ independent observations,
    $$\KL\left(q_{\theta,D}^{\otimes n}\;\Vert\; q_{\theta',D}^{\otimes n} \right) \lesssim n\epsilon_h^2 h^d  \leq a^2$$
    by using the definition of $\epsilon_h$. Hence, choosing $a$ sufficiently small makes the total variation distance between every pair of neighboring $n-$sample experiments uniformly bounded away from 1. Assouad's lemma therefore yields
    $$\inf_{\hat{p}} \sup_{\theta\in \{-1,1\}^M} \bbE_\theta \left[\|\hat{p} - p_\theta\|_{L^1}\right] \gtrsim M\epsilon_h h^d \gtrsim \epsilon_h, \quad \text{since }M\asymp h^{-d}.$$
    Take
    $$h\asymp \max \{n^{-1/(2\beta+d)}, D^{-1/2}\}$$
    where the constant multiplying $D^{-1/2}$ is chosen sufficiently large that $r_D \leq r_0$. 

    Now, if $D\gtrsim n^{2/(2\beta+d)}$, then $D^{-1/2}\lesssim n^{-1/(2\beta+d)}$. So, $h\asymp n^{-1/(2\beta+d)}$. At this bandwidth, $h^\beta \asymp 1/\sqrt{nh^d}$ and therefore, $\epsilon_h \asymp n^{-\beta/(2\beta+d)}$.

    If, instead $D\lesssim n^{2/(2\beta+d)}$, then $h\asymp D^{-1/2}$. Since in this regime $h\gtrsim n^{-1/(2\beta+d)}$, we have $1/\sqrt{nh^d} \lesssim h^\beta$. Hence, $\epsilon_h \asymp 1/\sqrt{nh^d} = D^{d/4}/\sqrt{n}$.

    Combining the two regimes, we have
    \begin{equation}\label{eq: lower bound observation nonparametric}
    \inf_{\hat{p}} \sup_{(g,\Lambda, \Sigma)\in \Theta_{D,\beta}} \bbE_{g,\Lambda,\Sigma}\left[\|\hat{p} - p_{g,\Lambda,\Sigma}\|_{L^1}\right] \gtrsim \min \left\{ n^{-\beta/(2\beta+d)}, \frac{D^{d/4}}{\sqrt{n}}\right\}.
    \end{equation}

    Combine Equations \eqref{eq: lower bound observation parametric} and \eqref{eq: lower bound observation nonparametric} to conclude the desired result.
\end{proof}

\section{Extension to an unknown observation mean}
\label{app: sample centering}
In this section, we go back to the raw observation model $Y_i^{\rm raw}=\mu + \Lambda \eta_i + \epsilon_i$ discussed in Remark \ref{remark: unknown mean}, and demonstrate that sample centering does not affect any of the bounds established in the prior sections for the centered factor model formulation. To recall, suppose $\mu\in\bbR^D$ is unknown, but the latent variables satisfy $\bbE_g[\eta_i]=0$ and $\Cov_g(\eta_i)=I_d$ without any loss of generality. For this section, we denote $\theta=(\mu, g,\Lambda, \Sigma)$ as the parameters, which satisfy the Assumptions in the main paper. In matrix notation, with the sample centering matrix defined as
$$H = I_n  - \frac1n \mathbf1_n \mathbf 1_n^\top,$$
sample centering reduces the problem to a centered factor model with  
$$Y^c := HY^{\rm raw} = (H M) \Lambda^\top + HE$$
where the $n$ rows of $HM$ are $\eta_i - \overline \eta$ for $i=1,\dots,n$. No restriction on $\mu$ is required and the conclusions below hold uniformly over $\mu\in\bbR^D$.

Writing $M_{0,c}=HM_0$ and $E_{0,c}=H E_0$ with the $i-$th rows of $M_{0,c}$ and $E_{0,c}$ given by $(\eta_i-\overline \eta)^\top$ and $(\epsilon_i - \overline \epsilon)^\top$ respectively, the true data generating model is $Y^c = M_{0,c}\Lambda_0^\top + E_{0,c}$. The FADE procedure is applied on the centered data matrix $Y^c$.

\subsection{Spectral recovery after sample centering}
\label{app: sample centering spectral}

We first record the basic properties of the centered signal and noise matrices. Since $\frac1n M_{0,c}^\top M_{0,c} = \frac1n M_0^\top M_0 - \overline \eta \overline \eta^\top$, we have
\begin{equation}\label{eq: centered score gram}
    \left\|\frac1n M_{0,c}^\top M_{0,c} - I_d \right\| \leq \left\|\frac1n M_0^\top M_0 - I_d\right\| + \|\overline \eta\|^2.
\end{equation}
Under the sub-Gaussian assumption on $\eta_i$, there is an event with probability at least $1-C/n$ on which 
\begin{equation}\label{eq: centered score concentration}
    \|\overline \eta\| \lesssim \sqrt{\frac{\log n}{n}}, \qquad \left\|\frac1n M_{0,c}^\top M_{0,c} - I_d\right\| \lesssim \sqrt{\frac{\log n}{n}},
\end{equation}
and 
$$\max_{i\in [n]} \|\eta_i - \overline \eta \| \lesssim \sqrt{\log n}.$$

Consequently, $s_d(M_{0,c})\asymp |M_{0,c}| \asymp \sqrt n$ and the singular-value and incoherence properties used in the proof of Theorem \ref{thm: spectral recovery} continue to hold.

Although the rows of $E_{0,c}=HE_0$ are no longer independent, this does not change the Gaussian calculations used in the spectral analysis. The columns remain independent and satisfy
$$E_{0,c}^{(j)} \sim N_n(0, \sigma_{0j}^2 H).$$
Equivalently, if $Q\in \bbR^{n\times (n-1)}$ has orthonormal columns spanning $\mathbf 1_n^\perp$ so that $Q^\top Q = I_{n-1}$ and $QQ^\top = H$, then 
$$Q^\top Y^{\rm raw} = (Q^\top M_0) \Lambda_0^\top + Q^\top E_0,$$
where the rows of $Q^\top E_0$ are independent $N_D(0, \Sigma_0)$ random vectors. Thus, sample centering removes only the one-dimensional intercept direction.

Let $Y_{0,c}:=M_{0,c}\Lambda_0^\top=U_{0,c}S_{0,c}V_{0,c}^\top$ be the compact SVD of the centered true signal, and let $P_{0,c}=U_{0,c} U_{0,c}^\top$. Because $\rm col(M_{0,c})\subseteq \mathbf 1_n^\perp$, we have $HP_{0,c} = P_{0,c}H=P_{0,c}$. The effective orthogonal-complement projection is therefore $H-P_{0,c}$, which is an orthogonal projection of rank $n-d-1$. The Gaussian quadratic-form identities in the proof of Theorem \ref{thm: spectral recovery} remain valid after replacing $P_0^\perp$ by $H-P_{0,c}$. 

For completeness, the leave-one-out argument can be preserved by defining
$$Y_c^{(i)} = Y^c - He_i \epsilon_i^\top,$$
similar to \eqref{eq: def leave-one-out}. 
Applying the spectral procedure to $Y^c$, let
$$Y^c = USV^\top + \text{ remaining singular components}, \qquad \widehat M = \sqrt n U, \qquad \widehat \Lambda = \frac{1}{\sqrt n} Y_c^\top U$$
and define $\widehat \Sigma$ for the corresponding centered residuals. Both $Y^c$ and $Y_c^{(i)}$ have columns in $\mathbf 1_n^\perp$. Consequently, their leading left singular vectors satisfy $HU=U$ and $HU^{(i)}=U^{(i)}$, and their truncated pseudo-inverses satisfy
$$(Y^c)^+ H = (Y^c)^+, \qquad (Y_c^{(i)})^+ H = (Y_c^{(i)})^+.$$
The same identifies hold for the centered population signal. Thus the left complement projections become $H-P_{0,c}$ and $H-U^{(i)}U^{(i)\top}$. In the decomposition corresponding to \eqref{eq: 2.34}, the perturbation $e_i\epsilon_i^\top$ is replaced by $He_i \epsilon_i^\top$; the identities above remove the additional factors of $H$ while $|He_i|\leq 1$ and $\|H e_i \epsilon_i^\top\|\leq \|\epsilon_i\|$. Moreover, conditional on $M_0$, $Y_c^{(i)}$ is independent of $\epsilon_i$, so the leave-one-out Gaussian concentration arguments remain applicable. 

The centered noise row is $\epsilon_i - \bar \epsilon$. The additional average-noise term in the remainder analysis is controlled using the independent of $\bar\epsilon \sim N_D(0, \Sigma_0/n)$ and $HE_0$. Specifically, for
$$B = \left[ (Y_c)^+ - Y_{0,c}^+ \right] U_{0,c},$$
conditional Gaussian concentration and pseudoinverse bound give, with failure probability $O(1/n)$,
$$\sqrt n |\bar \epsilon^\top B | \lesssim \sqrt{\log n} \|B\|_F \lesssim \sqrt{\log (nD)} \frac{\sqrt n +  \sqrt D}{nD} \lesssim \sqrt{\log (nD)} \left(\frac 1n + \frac 1D\right).$$
This is absorbed by the existing rowwise remainder bound. Intersecting with this additional event therefore preserves both the remainder order and failure probability order. The preceding observations give the following extension of Theorem \ref{thm: spectral recovery}.

\begin{proposition}[Spectral recovery with unknown mean]\label{prop: unknown mean spectral recovery}
    Under the conditions of Theorem \ref{thm: spectral recovery}, there is an event $\cE_{n,D}^\mu$ satisfying $\bbP_0^{(n)}(\cE_{n,D}^\mu) \geq 1 - \frac Cn$ and a measurable rotation $R\in \cO(d)$ such that on $\cE_{n,D}^\mu$, 
    \begin{equation}
        \begin{aligned}
            \|\widehat \Sigma - \Sigma_0 \| &\lesssim \sqrt{\frac{\log (nD)}{n}} + \frac{\ell_D}{D},\\
            \|\widehat M - M_0 R\|_F &\lesssim \sqrt n\left(\frac{1}{\sqrt D} + \sqrt{\frac{\log n}{n}}\right), \\
            \|\widehat \Lambda - \Lambda_0 R\|_F &\lesssim \sqrt D \left(\frac1D + \sqrt{\frac{\log (nD)}{n}}\right).
        \end{aligned}
    \end{equation}
    Moreover, the rowwise score error admits the decomposition
    \begin{equation}\label{eq: centered structured score expansion}
        \widehat \eta_i - R^\top \eta_i = \widehat R^\top \xi_{i,c} + \Delta_{i,c}, \qquad i=1,\dots,n
    \end{equation}
    where $\xi_{i,c}=\sqrt n S_{0,c}^{-1} V_{0,c}^\top E_0^\top (H - P_{0,c}) e_i$. Conditionally on $M_0$, these leading perturbations are jointly Gaussian with 
    \begin{equation}\label{eq: centered xi cov}
        \bbE [\xi_{i,c}\xi_{j,c}^\top \mid M_0] = (H - P_{0,c})_{ij} \Sigma_{\xi,c}, \qquad \Sigma_{\xi, c} = n S_{0,c}^{-1} V_{0,c}^\top \Sigma_0 V_{0,c} S_{0,c}^{-1}
    \end{equation}
    and $\|\Sigma_{\xi,c}\| \lesssim 1/D$. The remainder satisfies
    \begin{equation}
        \max_{i\in [n]} \|\Delta_{i,c}\| \lesssim \frac{\sqrt{\log (nD)}}{D} + \frac{\log n}{\sqrt n}.
    \end{equation}
    The constants are independent of $\mu, n, D$.
\end{proposition}

\begin{proof}
    Applying the proof of Theorem \ref{thm: spectral recovery} to $Y^c$ gives the stated bounds initially with $M_{0,c}$ in place of $M_0$. The discussing preceding the proposition verifies all singular-value, Gaussian concentration, projection, and leave-one-out ingredients required by that proof.

    For the global score bound, use $M_{0,c}=M_0 - \mathbf 1_n \overline \eta^\top$ to obtain
    $$\|\widehat M - M_0 R\|_F \leq \|\widehat M - M_{0,c}R\|_F + \|\mathbf 1_n \overline \eta^\top R\|_F = \|\widehat M - M_{0,c}R\|_F + \sqrt n \|\overline \eta\|.$$
    By \eqref{eq: centered score concentration}, the additional term is $O(\sqrt{\log n})$, and is therefore absorbed by the existing $n-$dependent part of the bound.

    The centered rowwise expansion from the proof of Theorem \ref{thm: spectral recovery} first gives
    $$\widehat \eta_i - R^\top (\eta_i - \overline \eta) = \widehat R^\top \xi_{i,c} + \widetilde \Delta_{i,c}$$
    with $\max_i \|\widetilde\Delta_{i,c}\| \lesssim \sqrt{\log (nD)}/D + (\log n)/\sqrt n$. Therefore, 
    $$\widehat \eta_i - R^\top \eta_i = \widehat R^\top \xi_{i,c} + \underbrace{\widetilde \Delta_{i,c} - R^\top \overline \eta}_{\Delta_{i,c}}$$
    where $\|\overline \eta\|\lesssim \sqrt{(\log n)/n}$ ensures that $\Delta_{i,c}$ has the stated uniform bound. 

    The loading proof is unchanged because the centered population signal retains the loading matrix $\Lambda_0$. For the residual variance estimator, sample centering removes one additional degree of freedom. The corresponding deterministic bias is $O(1/n)$, which is absorbed by the existing bounds for the noise variance estimation.
\end{proof}

The expectation bounds in Corollary \ref{cor: expected factor error} also remain unchanged. 

\subsection{Latent density estimation after sample centering}
\label{app: sample centering latent density}

The leading score perturbation in \eqref{eq: centered structured score expansion} has the same structure as the one used in the proof of Theorem \ref{thm: new kde}. The only change is that $P_0^\perp=I_n-P_0$ is replaced by the projection $H-P_{0,c}$. In particular, 
\begin{align*}
    \sum_{i=1}^n (H-P_{0,c})_{ii} &= n-d-1, \\
    \sum_{i,j=1}^n (H-P_{0,c})_{ij}^2 &= \tr((H -P_{0,c})^2) = n-d-1
\end{align*}
These are the only projection-rank identities used when aggregating the conditionally centered Gaussian perturbations. The convolution and kernel calculations continue to involve the original i.i.d. factors $\eta_1,\dots,\eta_n$, because the common displacement $\overline \eta$ has already been incorporated into the remainder $\Delta_{i,c}$. Consequently, the proof of Theorem \ref{thm: new kde} applies without any change, leading to the same rate. It follows that the optimized bandwidths, the threshold $D_c(n)$ and all oracle-transition conclusions in Section \ref{subsec: latent density optimized rate} remain unchanged.

The latent-density minimax lower bound also continues to hold, since the centered model with $\mu=0$ is a submodel of the model considered here. Hence, all upper- and lower-bound comparisons from Section \ref{subsec: latent density lower bound} remain valid.

\subsection{Observation density estimation after sample centering}
\label{app: sample centering observation density}

Let 
$$p_{0,c}(y) = \int_{\bbR^d} \phi_D(y\mid \Lambda_0 z, \Sigma_0) g_0(z) dz$$
denote the centered observation density. Under the model considered in this section, the raw observation density is
$$p_{0,\mu}(y) = p_{0,c}(y - \mu).$$
Let $\widehat p_c$ denote the FADE plug-in estimator constructed from the sample-centered data matrix $Y^c$, including the positive correction of the latent-density estimator. Define the estimator of the raw observation density by
\begin{equation}
    \widehat p_{\rm raw}(y) = \widehat p_c(y - \widehat \mu), \qquad \widehat \mu = \overline Y.
\end{equation}
The mean estimator satisfies $\widehat \mu - \mu = \Lambda_0\overline \eta + \overline \epsilon$. Using independence, $\Cov(\eta_i)=I_d$ and the loading and noise assumptions
\begin{align*}
    \bbE_0 \|\widehat \mu - \mu\|^2 &= \frac1n\tr\left(\Lambda_0 \Lambda_0^\top + \Sigma_0\right) \\
    &= \frac1n \left(\|\Lambda_0\|_F^2 + \tr(\Sigma_0)\right) \\
    &\lesssim \frac Dn.
\end{align*}
By translation invariance of $L^1$,
\begin{align}\label{eq: raw density triangle}
    \|\widehat p_{\rm raw} - p_{0,\mu}\|_{L^1} &\leq \|\widehat p_c - p_{0,c}\|_{L^1} + \|p_{0,c}(\cdot - (\widehat \mu - \mu)) - p_{0,c}\|_{L^1}.
\end{align}
For any $a\in \bbR^D$, convexity of the $L^1$ norm gives 
\begin{align*}
    \|p_{0,c}(\cdot - a) - p_{0,c}\|_{L^1} &\leq \int_{\bbR^d} \left\|\phi_D(\cdot \mid \Lambda_0 z + a, \Sigma_0) - \phi_D(\cdot \mid \Lambda_0 z, \Sigma_0) \right\|_{L^1} g_0(z) dz \\
    &\leq \|\Sigma_0^{-1/2} a\| \lesssim \|a\|
\end{align*}
where the first inequality uses the uniform lower bound on the eigenvalues of $\Sigma_0$. Taking expectations in \eqref{eq: raw density triangle} and applying the last display yields
$$\bbE_0 \|\widehat p_{\rm raw} - p_{0,\mu}\|_{L^1} \lesssim \bbE_0 \|\widehat p_c - p_{0,c}\|_{L^1} + \sqrt{\frac Dn}.$$
Combining this with Theorem \ref{thm: observation kde rate}, we observe that we obtain the exact same order of the upper bound for the raw observation density estimator, as the additional $\sqrt{D/n}$ is absorbed by the already existing $\sqrt{D/n}$ term in the bound in Theorem \ref{thm: observation kde rate}. 

\section{Additional simulation details}\label{app: simulation details}

\begin{figure}
    \centering
    \includegraphics[width=1.0\linewidth]{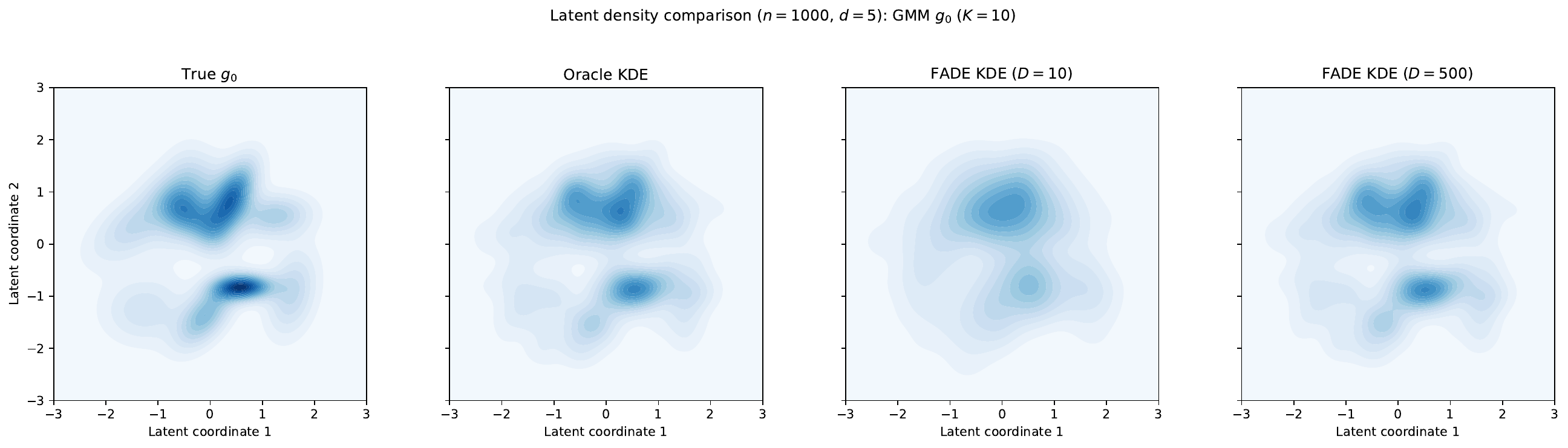}
    \caption{Same phenomenon as in Figure \ref{fig: demonstration}, when the underlying latent density is a Gaussian mixture model in $d=5$ with $K=10$ components, visualized in the first two coordinates. Performance of FADE is nearly the same as that of the oracle as the ambient dimension $D$ grows large. This latent density is used in the simulations in Sections \ref{subsec: simulations spectral recovery} and \ref{subsec: simulations latent density}}
    \label{fig: latent gmm example}
\end{figure}

\subsection{Spectral recovery simulations}

\subsubsection{Details about Section \ref{subsec: simulations spectral recovery}}
\label{app: spectral recovery simulation details}

We provide the complete design for the spectral recovery experiment reported in Section \ref{subsec: simulations spectral recovery}. The latent dimension is fixed at \(d=5\), and \(g_0\) is a fixed ten-component Gaussian-mixture density. The mixture is centered and linearly transformed so that $\mathbb E_{g_0}(\eta)=0$ and $\operatorname{Cov}_{g_0}(\eta)=I_d$. The same mixture distribution is used at every design point and in every
Monte Carlo replication. An illustration of the resulting latent density is given in Figure \ref{fig: latent gmm example}. The idiosyncratic noise
covariance is fixed at \(\Sigma_0=I_D\).

To generate loadings satisfying the bounded-row and pervasiveness conditions
of Section \ref{subsec: assumptions}, we draw
\[
    \lambda_{0j}
    =
    \sqrt d\,\frac{z_j}{\|z_j\|},
    \qquad
    z_j\stackrel{\mathrm{iid}}{\sim}N_d(0,I_d),
    \qquad j=1,\ldots,D.
\]
It follows that $\|\lambda_{0j}\|^2=d$ for every $j$ and $\ell_D=d\asymp 1$.
Moreover, \(D^{-1}\Lambda_0^\top\Lambda_0\) concentrates around \(I_d\).
For each design point, the loading matrix is generated once and held fixed
across all Monte Carlo replications. Thus, the reported variation is due to
the sampled latent scores and observation noise rather than changes in the
underlying factor parameters.

We consider three joint-growth regimes indexed by \(M\):
\[
\begin{array}{c|cc}
\text{regime} & n & D\\ \hline
\text{balanced growth} & M & M\\
\text{sample size grows faster} & M^2 & M\\
\text{ambient dimension grows faster} & M & M^2.
\end{array}
\]
For the balanced regime, \(M\) varies logarithmically from \(100\) to \(3000\). For each of the two unbalanced regimes, \(M\) varies logarithmically from \(50\) to \(250\). At each design point, we perform \(100\) independent replications. In every
replication, we draw
\[
    \eta_1,\ldots,\eta_n
    \stackrel{\mathrm{iid}}{\sim}g_0,
    \qquad
    \epsilon_1,\ldots,\epsilon_n
    \stackrel{\mathrm{iid}}{\sim}N_D(0,I_D),
\]
and form
\[
    Y_i=\Lambda_0\eta_i+\epsilon_i,
    \qquad i=1,\ldots,n.
\]
Applying the spectral step of FADE gives the estimated score matrix
\(\widehat M\), loading matrix \(\widehat\Lambda\), and diagonal noise
covariance $\widehat\Sigma=\diag(\widehat\sigma_1^2,\ldots,\widehat\sigma_D^2)$.  Because the score and loading matrices are identifiable only up to a common
orthogonal transformation, we compute $\widehat R \in \argmin_{R\in\mathcal O(d)} \|\widehat M-M_0R\|_F$.
The same rotation is used when evaluating the loading matrix. The reported
errors are
\[
    \mathcal E_M
    =
    \frac{\|\widehat M-M_0\widehat R\|_F}
         {\|M_0\|_F},
    \qquad
    \mathcal E_\Lambda
    =
    \frac{\|\widehat\Lambda-\Lambda_0\widehat R\|_F}
         {\|\Lambda_0\|_F}, \qquad
    \mathcal E_\Sigma
    =
    \frac1D\sum_{j=1}^D
    |\widehat\sigma_j^2-\sigma_{0j}^2|.
\]
The first two are relative Frobenius errors. The third is the average
absolute error across the \(D\) diagonal noise variances. Theorem \ref{thm: spectral recovery} gives, after division by
\(\|M_0\|_F\asymp\sqrt n\) and omission of logarithmic factors,
$\mathcal E_M
    \lesssim
    n^{-1/2}+D^{-1/2}$.
Similarly, Corollary \ref{cor: expected factor error} and
\(\|\Lambda_0\|_F\asymp\sqrt D\) give $\mathbb E_0\mathcal E_\Lambda
    \lesssim
    n^{-1/2}+D^{-1}$.
For the diagonal covariance estimator, the average absolute error satisfies $\mathcal E_\Sigma
    \leq
    \frac{1}{\sqrt D}
    \|\widehat\Sigma-\Sigma_0\|_F$.
The expected Frobenius bound in Corollary
\ref{cor: expected factor error}, together with \(\ell_D\asymp1\), therefore
yields
\[
    \mathbb E_0\mathcal E_\Sigma
    \lesssim
    n^{-1/2}+D^{-1}.
\]

Figure \ref{fig: parameter estimation simulation gmm truth} reports the
Monte Carlo mean of each error. The shaded region is the mean plus or minus
one Monte Carlo standard error. To facilitate visual comparison of decay
rates, each panel also includes reference curves proportional to
\(M^{-1/2}\) and \(M^{-1}\). These curves are rescaled to agree with the
empirical mean at the smallest value of \(M\); they indicate slopes only and
are not theoretical upper bounds with specified constants.

\subsection{Latent density estimation simulations}

\subsubsection{Details about Section \ref{subsec: simulations latent density}}
\label{app: latent density simulation details}

We provide the complete design for the experiment reported in Section
\ref{subsec: simulations latent density}. The latent dimension is fixed at
\(d=5\), and the latent density \(g_0\) is a fixed ten-component Gaussian
mixture. The mixture is centered and linearly transformed so that
$\mathbb E_{g_0}(\eta)=0$ and
$\Cov_{g_0}(\eta)=I_d$.
The same density is used in every simulation regime and replication. We set
\(\Sigma_0=I_D\). For every replication, loading rows are generated independently according to
\[
    \lambda_{0j}
    =
    \sqrt d\,\frac{z_j}{\|z_j\|},
    \qquad
    z_j\stackrel{\mathrm{iid}}{\sim}N_d(0,I_d).
\]
Thus, \(\|\lambda_{0j}\|^2=d\) for every \(j\), implying
\(\ell_D=d\asymp1\). The normalized Gram matrix
\(D^{-1}\Lambda_0^\top\Lambda_0\) concentrates around \(I_d\). We record its
smallest and largest eigenvalues in each replication and at every value of
\(D\) to monitor finite-sample conditioning. Loading matrices are independent
between replications but nested across the varying dimension within a
replication.

\medskip
\noindent
For the fixed-\(n\) experiment, we consider $n\in\{100,500,5000\}$
and a logarithmically spaced grid of ambient dimensions between \(10\) and
\(500\). For each \(n\) and replication, we generate one latent score matrix
\(M_0\in\mathbb R^{n\times d}\), one master loading matrix
\(\Lambda_0^{\max}\in\mathbb R^{D_{\max}\times d}\), and one master noise
matrix \(E_0^{\max}\in\mathbb R^{n\times D_{\max}}\), where
\(D_{\max}=500\). At ambient dimension \(D\), we use
$\Lambda_{0,D}
    =
    \Lambda_{0,1:D}^{\max}$ and $
    E_{0,D}
    =
    E_{0,:,1:D}^{\max}$,
and form $Y_D=M_0\Lambda_{0,D}^{\top}+E_{0,D}$.
Consequently, \(Y_D\) is exactly the first \(D\) columns of
\(Y_{D_{\max}}\). Moving along the \(D\)-grid therefore adds observed
coordinates without changing the latent subjects, loading rows or noise
coordinates already present.

\medskip
\noindent
For the fixed-\(D\) experiment, we take $D\in\{10,100,500\}$
and a logarithmically spaced grid of sample sizes between \(200\) and
\(5000\). For each \(D\) and replication, we generate one loading matrix
\(\Lambda_0\in\mathbb R^{D\times d}\), a master score matrix
\(M_0^{\max}\in\mathbb R^{n_{\max}\times d}\), and a master noise matrix
\(E_0^{\max}\in\mathbb R^{n_{\max}\times D}\), where \(n_{\max}=5000\).
At sample size \(n\), we retain their first \(n\) rows and form
$Y_n
    =
    M_{0,1:n}^{\max}\Lambda_0^\top
    +
    E_{0,1:n}^{\max}$.
Hence, increasing \(n\) adds subjects while leaving the loading matrix and
all previously observed subjects unchanged.

\medskip
\noindent
For every dataset, FADE computes the rank-\(d\) spectral score estimates
\(\widehat\eta_1,\ldots,\widehat\eta_n\) and applies an isotropic Gaussian
KDE to these scores. The oracle estimator applies the same type of KDE to
the true scores \(\eta_1,\ldots,\eta_n\). The two estimators select their
bandwidths separately but use the same selection rule. Specifically, we use
five-fold held-out log likelihood over \(50\) logarithmically spaced
bandwidths between \(0.1h_{\mathrm{sc}}\) and \(10h_{\mathrm{sc}}\), where
\[
    h_{\mathrm{sc}}
    =
    \widehat s\,n^{-1/(d+4)}
\]
and \(\widehat s\) is the median of the coordinatewise sample standard
deviations for the scores used by the corresponding estimator. Thus, the
oracle and FADE bandwidth grids are scaled separately to their respective
score samples.

Oracle and FADE use the same subject-level fold assignments. These
assignments are retained across \(D\) in the fixed-\(n\) experiment, and
existing subjects retain their assignments as \(n\) increases in the
fixed-\(D\) experiment. Spectral scores are estimated once from the complete
dataset before bandwidth selection. Cross-validation therefore applies only
to the second-stage KDE: the factor model is not refitted within individual
folds. In the fixed-\(n\) experiment, the oracle KDE, its selected bandwidth,
and its integrated squared error are computed once per replication and
reused exactly across all values of \(D\).

To evaluate FADE, we calculate the Procrustes rotation $\widehat R
    \in
    \argmin_{R\in\mathcal O(d)}
    \|\widehat M-M_0R\|_F$
and measure
\[
    \operatorname{ISE}_{\mathrm{FADE}}
    =
    \|\widehat g_{\mathrm{FADE}}-g_0^{\widehat R}\|_{L^2}^2,
    \qquad
    g_0^{\widehat R}(x)=g_0(\widehat R x).
\]
While $\operatorname{ISE}_{\rm FADE}$ is not equal to the loss $\inf_{R\in\cO(d)} \|\widehat g - g_0^R\|_{L^2}^2$ used in the accompanying theory, it upper bounds this loss. The oracle error is
\[
    \operatorname{ISE}_{\mathrm{oracle}}
    =
    \|\widehat g_{\mathrm{oracle}}-g_0\|_{L^2}^2.
\]
Because the Gaussian kernel is isotropic and \(L^2\) loss is invariant under
orthogonal transformations, the two losses correspond to the same
identifiable latent-density target.

Both \(g_0\) and the fitted Gaussian KDEs are finite Gaussian mixtures. Their
integrated squared errors are therefore evaluated analytically. In
particular, if
\[
    f(x)=\sum_{r=1}^{R}a_r\phi_d(x\mid\mu_r,A_r),
    \qquad
    q(x)=\sum_{s=1}^{S}b_s\phi_d(x\mid\nu_s,B_s),
\]
then
\[
    \int f(x)q(x)\,dx
    =
    \sum_{r=1}^{R}\sum_{s=1}^{S}
    a_rb_s\phi_d(\mu_r\mid\nu_s,A_r+B_s).
\]
Applying this identity to
\[
    \|f-q\|_{L^2}^2
    =
    \int f^2+\int q^2-2\int fq
\]
gives the exact ISE without Monte Carlo integration. We report the mean ISE
over \(50\) independent replications, with uncertainty bands corresponding
to one Monte Carlo standard error.

\begin{figure}
    \centering
    \includegraphics[width=1.0\linewidth]{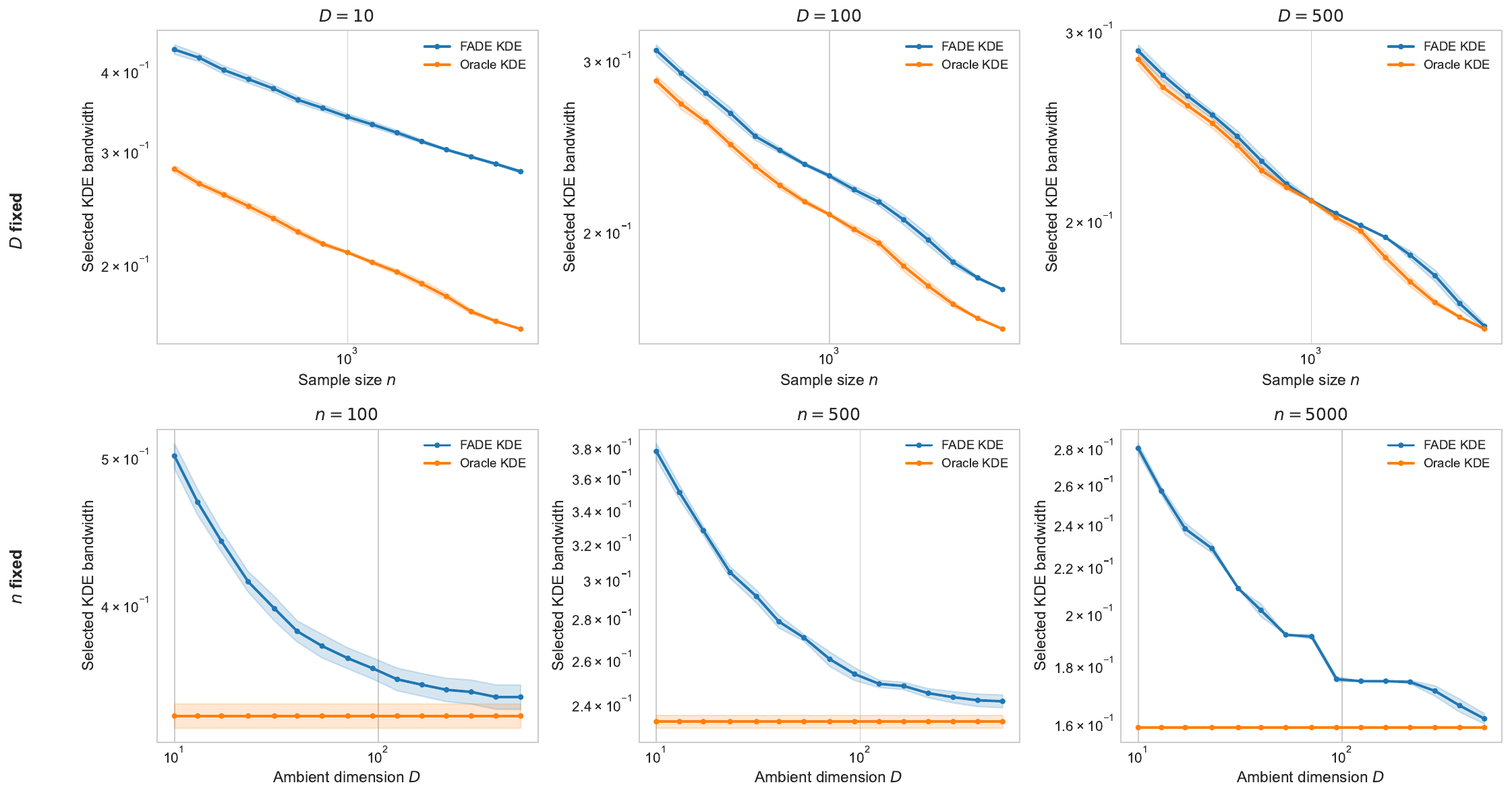}
    \caption{Bandwidths chosen by oracle KDE and FADE KDE for latent density estimation simulations of Section \ref{subsec: simulations latent density}. Despite cross validation being done on $n\times d$ estimated score matrix, the chosen bandwidth responds to $D$, in a manner similar to that required for the optimal bandwidth as discussed in Section \ref{subsec: latent density optimized rate}.}
    \label{fig:placeholder}
\end{figure}

\subsubsection{Additional simulation studies}

We perform another experiment in $d=1$. For each $(n,D)$ setting as in the simulation in the preceding subsection, we compare oracle KDE with FADE KDE. The fixed underlying true $g_0$ is a Gaussian mixture model in $\bbR$ with 10 components, while the other parts remain identical to the previous simulation study. Instead of a nested setting, we generate the loadings and noise errors independently across all the replicates and across the different $n,D$ settings (which is why there is some random fluctuation with large standard errors in the low $n=100$ case). The results are shown in Figure  \ref{fig: simulation latent density d1}, and are very similar to the oracle transition phenomenon observed in Section \ref{subsec: simulations latent density} (for example, compare with Figure \ref{fig: latent density estimation simulation}). The experiment does not exhibit an obvious deterioration of FADE in $d=1$, suggesting that the gap in the available upper bound may be conservative.

\begin{figure}
    \centering
    \includegraphics[width=1.0\linewidth]{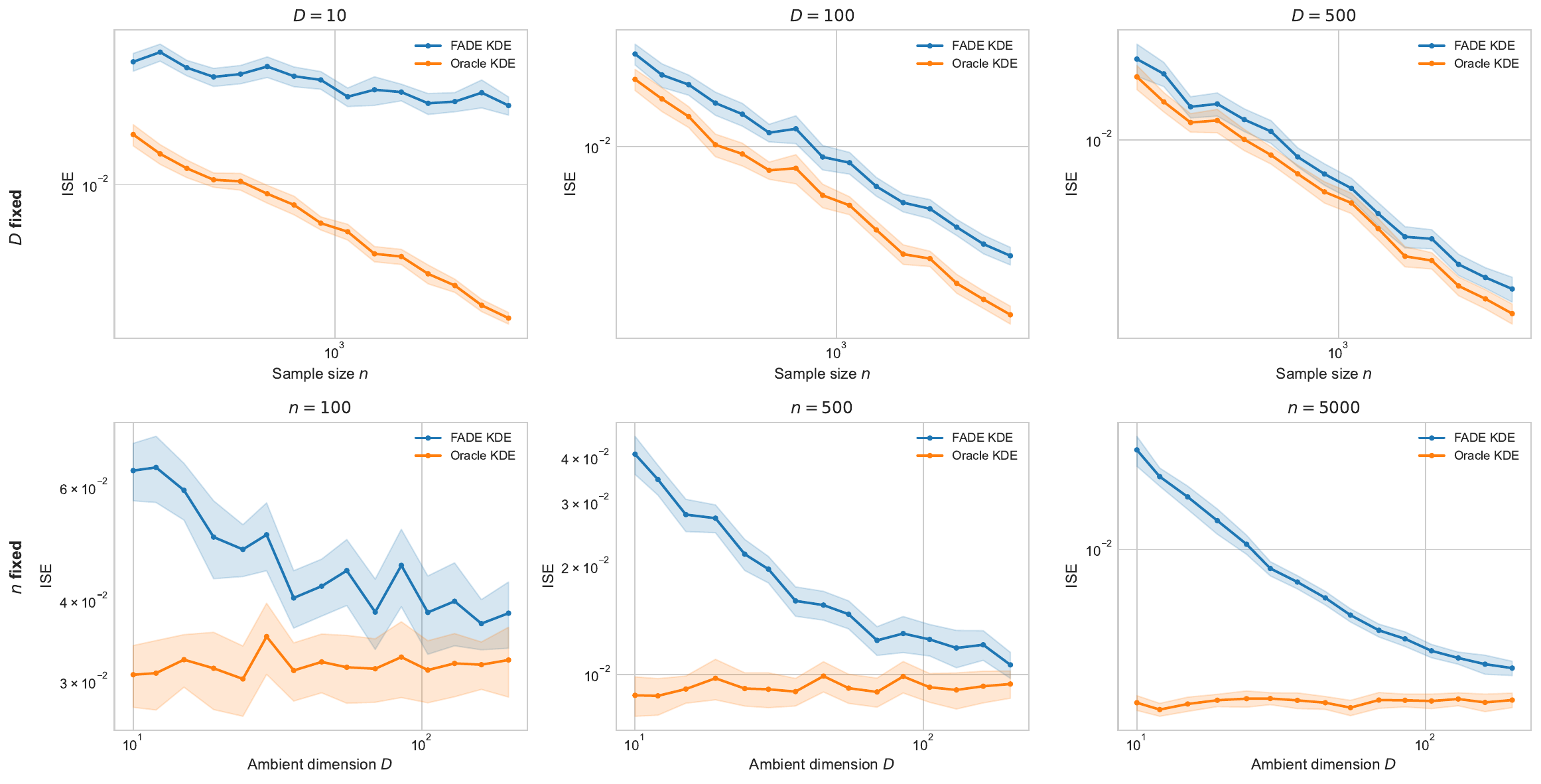}
    \caption{Simulation results evaluating latent density estimation in $d=1$}
    \label{fig: simulation latent density d1}
\end{figure}

\subsection{Observation density estimation simulations}

\subsubsection{Details about Section \ref{subsec: simulations obs density}}
\label{app: obs density simulation details}

We provide the complete construction and fitting details for the observation-density experiment reported in Section \ref{subsec: simulations obs density}. For the ground truth $g_0$, we use a probability density in $\bbR^d$ ($d=3$) obtained as a Gaussian convolution of a probability measure supported on two curves 
$$c_1(\theta)=
  \begin{pmatrix}
  \cos\theta\\
  \sin\theta\\
  \tau\cos(3\theta)
  \end{pmatrix}, \quad c_2(\theta)=
  \begin{pmatrix}
  (1-\cos\theta)\\
  -\sin\theta-0.5\\
  -\tau\cos(3\theta)
  \end{pmatrix}$$
with $\tau=0.65$ and $\theta\sim \text{Unif}(0,1)$ (equal weights on both parts), with a centered Gaussian with standard deviation $0.06$ added to make it a density. The population mean and covariance were then computed by approximating the curve integrals via 1000 equally weighted midpoint quadrature points on each curve, and thereafter, standardizing the distribution so that $\bbE_{g_0}[\eta]=0$ and $\Cov_{g_0}(\eta)=I_d$. Figure \ref{fig: demonstration} in the Introduction visualizes a two-dimensional projection of this density $g_0$, on the first two coordinates. 

\medskip
\noindent
\textbf{Observation model.}
For each
$n\in\{100,500,5000\}$ and $D\in\{10,100,500\}$, we generate observations from
$$Y_i=\Lambda_0\eta_i+\epsilon_i,
\qquad
\epsilon_i\sim N_D(0,\Sigma_0).$$
We use \(\Sigma_0=I_D\) and generate the loading rows according to
$\lambda_{0j}\sim N_d(0, I_d)$ for $j=1,\dots, D$. Within each replication, all methods are fitted to the same training sample and evaluated on the same independent test sample. The experiment is repeated independently $30$ times.

\medskip
\noindent
\textbf{FADE and oracle estimators.}
FADE obtains spectral estimates
\(\widehat\eta_1,\ldots,\widehat\eta_n\),
\(\widehat\Lambda\), and \(\widehat\Sigma\), and fits an isotropic Gaussian KDE
\(\widehat g_{\rm FADE}\) to the estimated scores. Its induced observation-density estimator is
\[
\widehat p_{\mathrm{FADE}}(y)
=
\int
\phi_D(y\mid\widehat\Lambda z,\widehat\Sigma)
\widehat g_{\mathrm{FADE}}(z)dz.
\]
Because the latent KDE uses a Gaussian kernel, this integral is available as a finite Gaussian mixture.

The oracle procedure observes the true scores and knows
\(\Lambda_0\) and \(\Sigma_0\). It fits a Gaussian KDE
\(\widehat g_{\mathrm{oracle}}\) to
\(\eta_1,\ldots,\eta_n\) and uses
\[
\widehat p_{\mathrm{oracle}}(y)
=
\int
\phi_D(y\mid\Lambda_0z,\Sigma_0)
\widehat g_{\mathrm{oracle}}(z)dz.
\]
It is included only as an unattainable benchmark for the error arising from estimation of the latent distribution.

For both KDE procedures, bandwidths are selected by five-fold held-out log likelihood using the bandwidth grid described in Section \ref{app: latent density simulation details}. The factor model is fitted before bandwidth cross-validation and is not re-estimated within individual folds.

\medskip
\noindent
\textbf{Joint latent Gaussian-mixture model.}
J-GMM($K$) models the latent density as
\[
g(z)
=
\sum_{k=1}^K
w_k\phi_d(z\mid\mu_k,\Gamma_k).
\]
After integrating out the latent variable, the observation density is
\[
p_{\mathrm{J\text{-}GMM}(K)}(y)
=
\sum_{k=1}^K
w_k
\phi_D
\left(
y\mid
\Lambda\mu_k, \,
\Lambda\Gamma_k\Lambda^\top+\Sigma
\right).
\]
Unlike FADE, this method estimates the latent-mixture parameters, common loading matrix, and diagonal noise covariance jointly. We take \(K=10\) as the number of components in the latent density GMM model. The algorithm is initialized using the FADE loading and noise-covariance estimates and a ten-component Gaussian mixture fitted to the FADE scores. We use three initial mixture fits and retain the initialization with the largest likelihood. The joint EM algorithm is then run for at most \(200\) iterations.

\medskip
\noindent
\textbf{Mixture of factor analyzers.}
MFA(\(K\)) uses the observation-space mixture
\[
p_{\mathrm{MFA}(K)}(y)
=
\sum_{k=1}^K
\pi_k
\phi_D
\left(
y\mid
\nu_k, \,
\Lambda_k\Lambda_k^\top+\Psi_k
\right),
\]
where each \(\Lambda_k\in\mathbb R^{D\times d}\) and each \(\Psi_k\) is diagonal. Unlike J-GMM, MFA allows every component to have its own loading space and idiosyncratic covariance. We use \(K=10\) and component rank \(d=3\). The EM algorithm is initialized by PCA followed by \(K\)-means clustering, uses three independent restarts, and is run for at most \(200\) iterations. The fit with the largest training likelihood is retained.

\medskip
\noindent
\textbf{Standard Gaussian factor analysis.}
The standard factor-analysis benchmark assumes
\(\eta_i\sim N_d(0,I_d)\),
so that
\[
Y_i\sim N_D
\left(
0,\Lambda\Lambda^\top+\Sigma
\right).
\]
This method estimates a single rank-\(d\) loading matrix and diagonal noise covariance but cannot represent multimodality or curved latent structure.

\medskip
\noindent
\textbf{Evaluation metric.}
Each fitted estimator is evaluated on an independent test sample
\[
Y_1^{\mathrm{test}},\ldots,Y_{N_{\mathrm{test}}}^{\mathrm{test}}
\stackrel{\mathrm{iid}}{\sim}p_0,
\qquad
N_{\mathrm{test}}=5000.
\]
We report the per-coordinate excess held-out log loss
\[
\widehat L(\widehat p)
=
\frac{1}{DN_{\mathrm{test}}}
\sum_{m=1}^{N_{\mathrm{test}}}
\left[
\log p_0(Y_m^{\mathrm{test}})
-
\log\widehat p(Y_m^{\mathrm{test}})
\right].
\]
Conditional on \(\widehat p\), its population expectation is
\[
\frac{1}{D}
\operatorname{KL}(p_0\Vert\widehat p).
\]
The normalization by \(D\) places predictive discrepancies at different ambient dimensions on a comparable per-coordinate scale. All mixture likelihoods are evaluated using log-sum-exp calculations for numerical stability. 

\end{appendix}
\end{document}